\documentclass[12 pt]{amsart}
\usepackage{fullpage, amsthm, amsmath, amsfonts, amssymb, url, mathrsfs}
\usepackage{esint}
\usepackage{hyperref}
\usepackage{enumerate}
\newcommand{\R}{\mathbb R}

\newcommand{\CF}{\widetilde{CF}}

\def \colonequals {\mathrel{\mathop:}=}
\def \equalscolon {\mathrel =:}
 \numberwithin{equation}{section}
 \newcommand{\Lip}{\mathrm{Lip}}
 \newcommand{\supp}{\mathrm{supp}\;}
\usepackage{standalone}

\newtheorem{thm}{Theorem}[section]
\newtheorem{defin}[thm]{Definition}
\newtheorem{lem}[thm]{Lemma}
\newtheorem{cor}[thm]{Corollary}
\newtheorem{prop}[thm]{Proposition}
\newtheorem{rem}[thm]{Remark}
\newtheorem*{thm*}{Theorem}
\newtheorem*{lem*}{Lemma}

\title{A Free Boundary Problem for the Parabolic Poisson Kernel}
\date{\today}
\author{Max Engelstein}
\address{Department of Mathematics, University of Chicago, 5734 S. University Avenue, Chicago, IL, 60637}
\email{maxe@math.uchicago.edu}

\begin{document}

\maketitle

\begin{abstract}
We study parabolic chord arc domains, introduced by Hofmann, Lewis and Nystr\"om \cite{hlncaloricmeasure}, and prove a free boundary regularity result below the continuous threshold.  More precisely, we show that a Reifenberg flat, parabolic chord arc domain whose Poisson kernel has logarithm in VMO must in fact be a vanishing chord arc domain (i.e. satisfies a vanishing Carleson measure condition). This generalizes, to the parabolic setting, a result of Kenig and Toro \cite{kenigtoro} and answers in the affirmative a question left open in the aforementioned paper of Hofmann et al. A key step in this proof is a classification of ``flat" blowups for the parabolic problem. 
\end{abstract}

\tableofcontents

\section{Introduction}\label{introduction}
In this paper we prove a parabolic analogue of Kenig and Toro's Poisson kernel characterization of vanishing chord arc domains, \cite{kenigtoro}. This continues a program started by Hofmann, Lewis and Nystr\"om, \cite{hlncaloricmeasure}, who introduced the concept of parabolic chord arc domains and proved parabolic versions of results in \cite{kenigtoroannals} and \cite{kenigtoroduke} (see below for more details). Precisely, we show that if $\Omega$ is a $\delta$-Reifenberg flat parabolic chord arc domain, with $\delta > 0$ small enough, and the logarithm of the Poisson kernel has vanishing mean oscillation, then $\Omega$ actually satisfies a vanishing Carleson measure condition (see \eqref{vanishingcarlesoncondition}).  The key step in this proof is a classification of ``flat" blowups (see Theorem \ref{poissonkernel1impliesflat} below), which itself was an open problem of interest (see, e.g., the remark at the end of Section 5 in \cite{hlncaloricmeasure}).

Let us recall the defintions and concepts needed to state our main results. In these we mostly follow the conventions established in \cite{hlncaloricmeasure}. We then briefly sketch the contents of the paper, taking special care to highlight when the difficulties introduced by the parabolic setting require substantially new ideas. Throughout, we work in two or more spacial dimensions ($n\geq 2$); the case of one spacial dimension is addressed in \cite{nequals2}. Finally, for more historical background on free boundary problems involving harmonic or caloric measure we suggest the introduction of \cite{hlncaloricmeasure}. 

We denote points $(x_1,...,x_n, t) = (X, t)\in \mathbb R^{n+1}$ and the parabolic distance between them is $d((X,t), (Y,s)) \colonequals |X-Y| + |t-s|^{1/2}$. For $r > 0$, the parabolic cylinder $C_r(X,t) \colonequals \{(Y,s)\mid |s-t| < r^2, |X-y| < r\}$. Our main object of study will be $\Omega$, an unbounded, connected open set in $\mathbb R^{n+1}$ such that $\Omega^c$ is also unbounded and connected. As the time variable has a ``direction" we will often consider $\Omega^{t_0} \colonequals \Omega \cap \{(X,s) \mid s < t_0\}$.  Finally, for any Borel set $F$ we will define $\sigma(F) = \int_{F} d\sigma_t dt$ where $d\sigma_t  \colonequals \mathcal H^{n-1}|_{\{s= t\}}$, the $(n-1)$-dimension Hausdorff measure restricted to the time-slice $t$. We normalize $\mathcal H^{n-1}$ so that $\sigma(C_1(0,0)\cap V) = 1$ for any $n$-plane through the origin containing a direction parallel to the time axis (we also normalize the Lebesgue measure, $dXdt$, by the same multiplicative factor). 

\begin{defin}\label{reifenbergflat}
We say that $\Omega$ is {\bf $\delta$-Reifenberg flat}, $\delta > 0$, if for $R > 0$ and $(Q,\tau) \in \partial \Omega$ there exists a $n$-plane $L(Q,\tau, R)$, containing a direction parallel to the time axis and passing through $(Q,\tau)$,  such that \begin{equation}\label{seperationequation}
\begin{aligned}
\{(Y,s) + r\hat{n} \in C_R(Q,\tau)\mid r > \delta R, (Y,s) \in L(Q,\tau, R)\} &\subset \Omega\\
\{(Y,s) - r\hat{n} \in C_R(Q,\tau)\mid r > \delta R, (Y,s) \in L(Q,\tau, R)\} &\subset \mathbb R^{n+1}\backslash \overline{\Omega}.
\end{aligned}
\end{equation}
Where $\hat{n}$ is the normal vector to $L(Q,\tau, R)$ pointing into $\Omega$ at $(Q,\tau)$. 
\end{defin}

The reader may be more familiar with a definition of Reifenberg flatness involving the Hausdorff distance between two sets (recall that the Hausdorff distance between $A$ and $B$ is defined as $D(A,B) = \sup_{a \in A} d(a, B) + \sup_{b\in B} d(b,A)$). These two notions are essentially equivalent as can be seen  in the following remark (which follows from the triangle inequality).

\begin{rem}
If $\Omega$ is a $\delta$-Reifenberg flat domain, then for any $R > 0$ and $(Q,\tau) \in \partial \Omega$ there exists a plane $L(Q,\tau, R)$, containing a line parallel to the time axis and through $(Q,\tau)$, such that $D[C_{R}(Q,\tau)\cap L(Q,\tau, R), C_R(Q,\tau) \cap \partial \Omega] \leq 4\delta R$.

Similarly, if \eqref{seperationequation} holds for some $\delta_0$ and there always exists an $L(Q,\tau, R)$ such that $D[C_{R}(Q,\tau)\cap L(Q,\tau, R), C_R(Q,\tau) \cap \partial \Omega] \leq \delta R$ then \eqref{seperationequation} holds for $2\delta$. 
\end{rem}

Let $\theta((Q,\tau), R) \colonequals \inf_{P} \frac{1}{R} D[C_{R}(Q,\tau)\cap P, C_R(Q,\tau) \cap \partial \Omega]$ where the infimun is taken over all planes containing a line parallel to the time axis and through $(Q,\tau)$. 

\begin{defin}\label{vanishingreifenbergflat}
We say that $\Omega$ (or $\Omega^{t_0}$) is {\bf vanishing Reifenberg flat} if $\Omega$ is $\delta$-Reifenberg flat for some $\delta > 0$ and for any compact set $K$ (alternatively $K \subset\subset \{t < t_0\}$), \begin{equation}\label{vanishingrefflatdef}\lim_{r\downarrow 0} \sup_{(Q,\tau) \in K\cap \partial \Omega} \theta((Q,\tau),r) = 0.\end{equation}
\end{defin}

\begin{defin}\label{ahlforsregularity}
$\partial \Omega$ is {\bf Ahflors regular} if there exists an $M \geq 1$ such that for all $(Q,\tau) \in \partial \Omega$ and $R > 0$ we have $$\left(\frac{R}{2}\right)^{n+1} \leq \sigma(C_R(Q,\tau)\cap \partial \Omega) \leq MR^{n+1}.$$
\end{defin}

Note that left hand inequality follows immediately in a $\delta$-Reifenberg flat domain for $\delta > 0$ small enough (as Hausdorff measure decreases under projection, and $C_{R/2}(Q,\tau)\cap L(Q,\tau,R) \subset \mathrm{proj}_{L(Q,\tau,R)}(C_R(Q,\tau)\cap \partial \Omega)$). 

Following \cite{hlncaloricmeasure}, define, for $r > 0$ and $(Q,\tau) \in \partial \Omega$, \begin{equation}\label{gammainuniformrectifiability}
\gamma(Q,\tau, r) = \inf_P \left(r^{-n-3}\int_{\partial \Omega \cap C_r(Q,\tau)} d((X,t), P)^2 d\sigma(X,t)\right)
\end{equation} where the infimum is taken over all $n$-planes containing a line parallel to the $t$-axis and going through $(Q,\tau)$. This is an $L^2$ analogue of Jones' $\beta$-numbers (\cite{jones}).  We want to measure how $\gamma$, ``on average", grows in $r$ and to that end introduce \begin{equation}\label{whatisnu} d\nu(Q,\tau, r) = \gamma(Q,\tau, r)d\sigma(Q,\tau)r^{-1}dr.\end{equation} Recall that $\mu$ is a Carleson measure with norm $\|\mu\|_+$ if \begin{equation}\label{carlesonmeasure}
\sup_{R > 0} \sup_{(Q,\tau) \in \partial \Omega} \mu((C_R(Q,\tau) \cap \partial \Omega) \times [0,R]) \leq \|\mu\|_+ R^{n+1}.
\end{equation}

In analogy to David and Semmes \cite{davidandsemmes} (who defined uniformly rectifiable domains in the isotropic setting) we define a parabolic uniformly rectifiable domain;

\begin{defin}\label{uniformlyrectifiable}
If $\Omega \subset \mathbb R^{n+1}$ is such that $\partial \Omega$ is Ahlfors regular and $\nu$ is a Carleson measure then we say that $\Omega$ is a {\bf (parabolic) uniformly rectifiable domain.}

As in \cite{hlncaloricmeasure}, if $\Omega$ is a parabolic uniformly rectifiable domain which is also $\delta$-Reifenberg flat for some $\delta > 0$ we say that $\Omega$ is a {\bf parabolic regular domain}. We may also refer to them as {\bf parabolic chord arc domains}. 

Finally, if $\Omega$ is a parabolic regular domain and satisfies a vanishing Carleson measure condition, \begin{equation}\label{vanishingcarlesoncondition}\lim_{R \downarrow 0} \sup_{0 < \rho < R} \sup_{(Q,\tau) \in \partial \Omega} \rho^{-n-1} \nu((C_\rho(Q,\tau) \cap \partial \Omega) \times [0,\rho]) = 0\end{equation} we call $\Omega$ a {\bf vanishing chord arc domain}. Alternatively, if \eqref{vanishingcarlesoncondition} holds when the $\partial \Omega$ is replaced by $K$ for any $K \subset \subset \{s < t_0\}$ then we say that $\Omega^{t_0}$ is a vanishing chord arc domain. 
\end{defin}

Readers familiar with the elliptic theory will note that these definitions differ from, e.g. Definition 1.5 in \cite{kenigtoroduke}. It was observed in \cite{hlnbigpieces} that these definitions are equivalent in the time independent setting whereas the elliptic definition is weaker when $\Omega$ changes with time. Indeed, in the time independent setting, uniform recitifiability with small Carleson norm and being a chord arc domain with small constant are both equivalent to the existence of big pieces of Lipschitz graphs, in the sense of Semmes \cite{semmeslipschitz}, at every scale (see Theorem 2.2 in \cite{kenigtoroduke} and Theorem 1.3 in Part IV of \cite{davidandsemmes}). On the other hand, in the time dependent case, even $\sigma(\Delta_R(Q,\tau))  \equiv R^{n+1}$ does not imply the Carleson measure condition (see the example at the end of \cite{hlnbigpieces}). 

The role of this Carleson measure condition becomes clearer when we consider domains of the form $\Omega = \{(X,t)\mid x_n > f(x,t)\}$ for some function $f$. Dahlberg \cite{dahlberg} proved that surface measure and harmonic measure are mutually absolutely continuous in a Lipschitz domain. However, Kaufman and Wu \cite{kaufmanandwu} proved that surface measure and caloric measure are not necessarily mutually absolutely continuous when $f \in \Lip(1, 1/2)$. To ensure mutual absolute continuity one must also assume that the $1/2$ time derivative of $f$ is in $\mathrm{BMO}$ (see \cite{lewisandmurray}). In \cite{hlncaloricmeasure} it is shown that the BMO norm of the $1/2$ time derivative of $f$ can be controlled by the Carleson norm of $\nu$. Morally, the growth of $\sigma(\Delta_R(Q,\tau))$ controls the $\Lip(1,1/2)$ norm of $f$ but cannot detect the BMO norm of the $1/2$-time derivative of $f$ (for $n=1$ this is made precise in \cite{nequals2}). 

For $(X,t) \in \Omega$, the caloric measure with a pole at $(X,t)$, denoted $\omega^{(X,t)}(-)$, is the measure associated to the map $f \mapsto U_f(X,t)$ where $U_f$ solves the heat equation with Dirichlet data $f\in C_0(\partial \Omega)$. If $\Omega$ is Reifenberg flat the Dirichlet problem has a unique solution and this measure is well defined (in fact, weaker conditions on $\Omega$ suffice to show $\omega^{(X,t)}$ is well defined c.f. the discussion at the bottom of page 283 in \cite{hlncaloricmeasure}). Associated to $\omega^{(X,t)}$ is the parabolic Green function $G(X,t,-,-)\in C(\Omega \backslash \{(X,t)\})$, which satisfies \begin{equation}\label{fpgreenfunction}\tag{FP} \left\{ \begin{aligned}
G(X,t,Y,s) \geq& 0, \;\forall (Y,s) \in \Omega \backslash \{(X,t)\},\\
G(X,t,Y,s) \equiv& 0,\; \forall (Y,s) \in \partial \Omega,\\
-(\partial_s + \Delta_Y)G(X,t, Y,s) =& \delta_0((X,t) - (Y,s)),\\
\int_{\partial \Omega} \varphi d\omega^{(X,t)} =& \int_\Omega G(X,t, Y,s) (\Delta_Y - \partial_s)\varphi dYds, \; \forall \varphi \in C^\infty_c(\mathbb R^{n+1}).
\end{aligned} \right. \end{equation} (Of course there are analogous objects for the adjoint equation; $G(-,-, Y,s)$ and $\hat{\omega}^{(Y,s)}$.)  We are interested in what the regularity of $\omega^{(X,t)}$ can tell us about the regularity of $\partial \Omega$. Observe that by the parabolic maximum theorem the caloric measure with a pole at $(X,t)$ can only ``see" points $(Q,\tau)$ with $\tau < t$. Thus, any regularity of $\omega^{(X,t)}$ will only give information about $\Omega^{t}$ (recall $\Omega^{t} \colonequals \Omega \cap \{(X,s) \mid s < t\}$). Hence, our results and proofs will often be clearer when we work with $\omega$, the caloric measure with a pole at infinity, and $u \in C(\Omega)$, the associated Green function, which satisfy  \begin{equation}\label{ipgreenfunction}\tag{IP} \left\{ \begin{aligned}
u(Y,s) \geq& 0, \;\forall (Y,s) \in \Omega,\\
u(Y,s) \equiv& 0,\; \forall (Y,s) \in \partial \Omega,\\
-(\partial_s + \Delta_Y)u(Y,s) =& 0, \; \forall (Y,s) \in \Omega\\
\int_{\partial \Omega} \varphi d\omega =& \int_\Omega u(Y,s) (\Delta_Y - \partial_s)\varphi dYds, \; \forall \varphi \in C^\infty_c(\mathbb R^{n+1}).
\end{aligned} \right. \end{equation} (For the existence, uniqueness and some properties of this measure/function, see Appendix \ref{appendix:caloricmeasure}). However, when substantial modifications are needed, we will also state and prove our theorems in the finite pole setting.  

Let us now recall some salient concepts of ``regularity" for  $\omega^{(X,t)}$. Denote the {\it surface ball} at a point $(Q,\tau) \in \partial \Omega$ and for radius $r$ by $\Delta_r(Q,\tau) \colonequals C_r(Q,\tau)\cap \partial \Omega$. 

\begin{defin}\label{doubling}
Let $(X_0,t_0) \in \Omega$. We say $\omega^{(X,t)}$ is a {\bf doubling measure} if, for every $A \geq 2$, there exists an $c(A) > 0$ such that for any $(Q,\tau) \in \partial \Omega$ and $r > 0$, where $|X_0-Q|^2 < A(t_0 - \tau)$ and $t_0- \tau \geq 8r^2$, we have \begin{equation}\label{fpdoublingmeasure} \omega^{(X_0,t_0)}(\Delta_{2r}(Q,\tau)) \leq c(A) \omega^{(X_0,t_0)}(\Delta_r(Q,\tau)).\end{equation}

Alternatively, we say $\omega$ is a doubling measure if there exists a $c > 0$ such that $\omega(\Delta_{2r}(Q,\tau)) \leq c\omega(\Delta_r(Q,\tau))$ for all $r > 0$ and $(Q,\tau) \in \partial \Omega$. 
\end{defin}

\begin{defin}\label{apweight}
Let $(X_0,t_0) \in \Omega$ and let $\omega^{(X_0,t_0)}$ be a doubling measure such that $\omega^{(X_0,t_0)} << \sigma$ on $\partial \Omega$ and $k^{(X_0,t_0)}(Q,\tau) \colonequals \frac{d\omega^{(X_0,t_0)}}{d\sigma}(Q,\tau)$.  We say that $\omega^{(X_0,t_0)} \in A_\infty(d\sigma)$ ({\bf is an $A_\infty$-weight}) if it satisfies a ``reverse H\"older inequality." That is, if there exists a $p > 1$ such that if $A \geq 2, (Q,\tau)\in \partial \Omega, r > 0$ are as in Definition \ref{doubling} then there exists a $c \equiv c(p, A) > 0$ where \begin{equation}\label{reverseholderinequality} \fint_{\Delta_{2r}(Q,\tau)} k^{(X_0,t_0)}(Q,\tau)^p d\sigma(Q,\tau) \leq c \left(\fint_{\Delta_r(Q,\tau)} k^{(X_0,t_0)}(Q,\tau)d\sigma(Q,\tau)\right)^p.\end{equation}

We can similarly say $\omega \in A_\infty(d\sigma)$ if $\omega << \sigma$ on $\partial \Omega, h(Q,\tau) \colonequals \frac{d\omega}{d\sigma}$, and there exists a $c > 0$ such that \begin{equation}\label{reverseholderinequality2} \fint_{\Delta_{2r}(Q,\tau)} h(Q,\tau)^p d\sigma(Q,\tau) \leq c \left(\fint_{\Delta_r(Q,\tau)} h(Q,\tau)d\sigma(Q,\tau)\right)^p.\end{equation}
\end{defin}

In analogy to the results of David and Jerison \cite{davidandjerison}, it was shown in \cite{hlncaloricmeasure} that if $\Omega$ is a ``flat enough" parabolic regular domain, then the caloric measure is an $A_\infty$ weight (note that \cite{hlncaloricmeasure} only mentions the finite pole case but the proof works unchanged for a pole at infinity, see Proposition \ref{hlnformeasureatinfinity}). 

\begin{thm*}[Theorem 1 in \cite{hlncaloricmeasure} ]
If $\Omega$ is a parabolic regular domain with Reifenberg constant $\delta_0 > 0$ sufficiently small (depending on $M, \|\nu\|_+$) then $\omega^{(X_0,t_0)}$ is an $A_\infty$-weight. 
\end{thm*}

Closely related to $A_\infty$ weights are the $\mathrm{BMO}$ and $\mathrm{VMO}$ function classes. 

\begin{defin}\label{BMOVMOdefinitions}
We say that $f\in \mathrm{BMO}(\partial \Omega)$ with norm $\|f\|_*$ if $$\sup_{r > 0} \sup_{(Q,\tau) \in \partial \Omega} \fint_{C_r(Q,\tau)}|f(P,\eta) - f_{C_r(Q,\tau)}| d\sigma(P,\eta) \leq \|f\|_*,$$ where $f_{C_r(Q,\tau)} \colonequals \fint_{C_r(Q,\tau)} f(P,\eta) d\sigma(P,\eta)$, the average value of $f$ on $C_r(Q,\tau)$. 

Define $\mathrm{VMO}(\partial \Omega)$ to be the closure of uniformly continuous functions vanishing at infinity in $\mathrm{BMO}(\partial \Omega)$ (analogously we say that $k^{(X_0,t_0)}\in \mathrm{VMO}(\partial \Omega^{t_0})$ if $k^{(X_0,t_0)} \in \mathrm{VMO}(\Delta_r(Q,\tau))$ for any $(Q,\tau)\in \partial \Omega, r > 0$ which satisfies the hypothesis of Definition \ref{doubling} for some $A \geq 2$). 
\end{defin}

This definition looks slightly different than the one given by equation 1.11 in \cite{hlncaloricmeasure}. In the infinite pole setting it gives control over the behavior of $f$ on large scales. In the finite pole setting it is actually equivalent to the definition given in \cite{hlncaloricmeasure} as can be seen by a covering argument. 

In analogy with the elliptic case, if $\Omega$ is a vanishing chord arc domain then we expect control on the small scale oscillation of $\log(k^{(X_0,t_0)})$.  

\begin{thm*}[Theorem 2 in \cite{hlncaloricmeasure}]
If $\Omega$ is chord arc domain with vanishing constant and $(X_0,t_0)\in \Omega$ then $\log(k^{(X_0,t_0)}) \in \mathrm{VMO}(\partial \Omega^{t_0})$. 
\end{thm*}

Our main result is the converse to the above theorem and the parabolic analogue of the Main Theorem in \cite{kenigtoro}. 

\begin{thm}\label{loghvmoisvanishingchordarc}[Main Theorem]
Let $\Omega\subset \mathbb R^{n+1}$ be a $\delta$-Reifenberg flat parabolic regular domain with $\log(h) \in \mathrm{VMO}(\partial \Omega)$ (or $\log(k^{(X_0,t_0)}) \in \mathrm{VMO}(\partial \Omega^{t_0})$). There is a $\delta_n > 0$ such that if $\delta < \delta_n $, then $\Omega$ is a parabolic vanishing chord arc domain (alternatively, $\Omega^{t_0}$ is a vanishing chord arc domain). 
\end{thm}

Contrast this result to Theorem 3 in \cite{hlncaloricmeasure};

\begin{thm*}[Theorem 3 in \cite{hlncaloricmeasure}]
Let $\Omega$ be a $\delta$-Reifenberg flat parabolic regular domain with $(\hat{X}, \hat{t}) \in \Omega$.  Assume that
\begin{enumerate}[(i)]
\item $\omega^{(\hat{X}, \hat{t})}(-)$ asymptotically optimally doubling,
\item $\log k^{(\hat{X}, \hat{t})} \in \mathrm{VMO}(\partial \Omega^{\hat{t}})$,
\item $\|\nu\|_+$ small enough.
\end{enumerate}
Then $\Omega^{\hat{t}}$ is a vanishing chord arc domain.
\end{thm*}

Our main theorem removes the asymptotically optimally doubling and small Carleson measure hypotheses. As mentioned above, this requires a classification of  the ``flat" limits of pseudo-blowups (Definition \ref{parabolicpseudoblowup} below), which was heretofore open in the parabolic setting.  

\begin{thm}\label{poissonkernel1impliesflat}[Classification of ``flat" Blowups]
Let $\Omega_\infty$ be a $\delta$-Reifenberg flat parabolic regular domain with Green function at infinity, $u_\infty$, and associated parabolic Poisson kernel, $h_\infty$ (i.e.  $h_\infty = \frac{d\omega_\infty}{d\sigma}$). Furthermore, assume that $|\nabla u_\infty| \leq 1$ in $\Omega_\infty$ and $|h_\infty| \geq 1$ for $\sigma$-almost every point on $\partial \Omega_\infty$. There exists a $\delta_n > 0$ such that if $\delta_n \geq \delta > 0$ we may conclude that, after a potential rotation and translation, $\Omega_\infty = \{(X,t)\mid x_n > 0\}$. 
\end{thm}

Nystr\"om \cite{nystromindiana} proved a version of Theorem \ref{poissonkernel1impliesflat} under the additional assumptions that $\Omega$ is a graph domain and that the Green function is comparable with the distance function from the boundary. Furthermore, under the additional assumption that $\Omega$ is a graph domain, Nystr\"om \cite{nystromgraphdomains} also proved that Theorem \ref{poissonkernel1impliesflat} implies Theorem \ref{loghvmoisvanishingchordarc}.  Our proof of Theorem \ref{poissonkernel1impliesflat} (given in Appendix \ref{proofofflatprop}) is heavily inspired by the work of Andersson and Weiss \cite{anderssonweiss}, who studied a related free boundary problem arising in combustion. However, we are unable to apply their results directly as they consider solutions in the sense of ``domain variations" and it is not clear if the parabolic Green function is a solution in this sense. For example, solutions in the sense of domain variations satisfy the bound $$\int_{C_r(X,t)} |\partial_t u|^2 \leq C_1 r^n,\; \forall C_r(X,t) \subset \Omega,$$ and it is unknown if this inequality holds in a parabolic regular domain (see, e.g., the remark at the end of Section 1 in \cite{nystromgraphdomains}). Furthermore, the results in \cite{anderssonweiss} are local, whereas Theorem \ref{poissonkernel1impliesflat} is a global result. Nevertheless, we were able to adapt the ideas in \cite{anderssonweiss} to our setting. For further discussion of exactly how our work fits in with that of \cite{anderssonweiss} see the beginning of Appendix \ref{proofofflatprop} below.

 Let us now briefly outline this paper and sketch the contents of each section. The paper follows closely the structure, and often the arguments, of \cite{kenigtoro}. In Section \ref{preliminaryestimates} we prove some technical estimates which will be used in Sections \ref{boundingthegradient} and \ref{blowupanalysis}. Section \ref{boundingthegradient} is devoted to proving an integral bound for the gradient of the Green function. The arguments in this section are much like those in the elliptic case. However, we were not able to find the necessary results on non-tangential convergence in parabolic Reifenberg flat domains (e.g. Fatou's theorem) in the literature.  Therefore, we prove them in Appendix \ref{fatouetcforparaboliccase}. Of particular interest may be Proposition \ref{parabolicsawtoothdomains} which constructs interior ``sawtooth" domains (the elliptic construction does not seem to generalize to the parabolic setting). Section \ref{blowupanalysis} introduces the blowup procedure and uses estimates from Sections \ref{preliminaryestimates} and \ref{boundingthegradient} to show that the limit of this blowup satisfies the hypothesis of Theorem \ref{poissonkernel1impliesflat}. This allows us to conclude that $\Omega$ is vanishing Reifenberg flat and, after an additional argument, gives the weak convergence of surface measure under pseudo-blowup. 
 
By combining the weak convergence of $\sigma$ with the weak convergence of $\hat{n}\sigma$ under pseudo-blowups (the latter follows from the theory of sets with finite perimeter) we can conclude easily that $\hat{n} \in \mathrm{VMO}$ (morally, bounds on the growth of $\sigma(\Delta_r(Q,\tau))$ give bounds on the $\mathrm{BMO}$ norm of $\hat{n}$, see Theorem 2.1 in \cite{kenigtoroduke}). Therefore, in the elliptic setting, the weak convergence of surface measure is essentially enough to prove that $\Omega$ is a vanishing chord arc domain. On the other hand, to show that $\Omega$ is a parabolic vanishing chord arc domain one must establish a vanishing Carleson measure condition (equation \eqref{vanishingcarlesoncondition}). Furthermore, the aforementioned example at the end of \cite{hlnbigpieces} shows that control of the growth of $\sigma(\Delta_r(Q,\tau))$ does not necessarily give us control on $\|\nu\|_+$. In Section \ref{sec: vanishingcarlesoncondition} we use purely geometric measure theoretical arguments to prove Theorem \ref{chordarctovanishingchordarc}; that a vanishing Reifenberg flat parabolic chord arc domain whose surface measure converges weakly under pseudo-blowups must be a parabolic vanishing chord arc domain. To establish this (and thus finish the proof of Theorem \ref{loghvmoisvanishingchordarc}), we adapt approximation theorems of Hofmann, Lewis and Nystr\"om, \cite{hlnbigpieces}, and employ a compactness argument. 

 The remainder of our paper is devoted to free boundary problems with conditions above the continuous threshold. In particular, we prove (stated here in the infinite pole setting),
 \begin{thm}\label{loghholdercontinuous}
Let $\Omega\subset \mathbb R^{n+1}$ be a parabolic regular domain with $\log(h) \in \mathbb C^{k+\alpha, (k+\alpha)/2}(\mathbb R^{n+1})$ for $k \geq 0$ and $\alpha \in (0,1)$.  There is a $\delta_n > 0$ such that if $\delta_n \geq \delta > 0$ and $\Omega$ is $\delta$-Reifenberg flat then $\Omega$ is a $\mathbb C^{k+1+\alpha, (k+1+\alpha)/2}(\mathbb R^{n+1})$ domain. 

Furthermore, if $\log(h)$ is analytic in $X$ and in the second Gevrey class (see Definition \ref{gevrydef}) in $t$ then, under the assumptions above, we can conclude that $\Omega$ is the region above the graph of a function which is analytic in the spatial variables and in the second Gevrey class in $t$. Similarly, if $\log(h) \in C^{\infty}$ then $\partial \Omega$ is locally the graph of a $C^\infty$ function.
\end{thm}

The case of $k = 0$ follows in much the same manner as the proof of Theorem \ref{poissonkernel1impliesflat} but nevertheless is done in full detail in Section \ref{initialholderregularity}. For larger values of $k$, we use the techniques of Kinderlehrer and Nirenberg (see e.g. \cite{kinderlehrernirenberganalyticity}), parabolic Schauder-type estimates (see e.g. \cite{liebermanintermediateschauder}) and an iterative argument inspired by Jerison \cite{jerison}. These arguments are presented in Section \ref{higherregularity}. 

Finally, let us comment on the hypothesis of Theorem \ref{loghholdercontinuous}. For $n \geq 3$, this theorem is sharp. In particular, Alt and Caffarelli, \cite{altcaf}, constructed an Ahlfors regular domain $\Omega \subset \mathbb R^3$ with $\log(h) = 0$ but which is not a $C^1$ domain (it has a cone point at the origin). A cylinder over this domain shows that the flatness condition is necessary. On the other hand, Keldysh and Lavrentiev (see \cite{keldyshandlavrentiev})  constructed a domain in $\mathbb R^2$ which is rectifiable but not Ahlfors regular, where $h \equiv 1$ but the domain is not a $C^1$ domain. A cylinder over this domain shows that the Parabolic regular assumption is necessary. In one spatial dimension, our upcoming preprint \cite{nequals2} shows that the the flatness condition is not necessary (as topology implies that a parabolic NTA domain is a graph domain). When $n =2$ it is not known if the flatness assumption is necessary and we have no intuition as to what the correct answer should be. 

\medskip

\noindent {\bf Acknowledgements:} This research was partially supported by the National Science Foundation's Graduate Research Fellowship, Grant No. (DGE-1144082). We thank Abdalla Nimer for helpful comments regarding Section \ref{sec: vanishingcarlesoncondition} and Professor Tatiana Toro for helping us overcome a technical difficulty in Section \ref{initialholderregularity}. Finally, we owe a debt of gratitude to Professor Carlos Kenig who introduced us to free boundary problems and whose patience and guidance made this project possible.

\section{Notation and Preliminary Estimates}\label{preliminaryestimates}

As mentioned above, all our theorems will concern a $\delta$-Reifenberg flat, parabolic regular domain $\Omega$. Throughout, $\delta > 0$ will be small enough such that $\Omega$ is a {\it non-tangentially accesible} (NTA) domain  (for the definition see \cite{lewisandmurray}, Chapter 3, Sec 6, and Lemma 3.3 in \cite{hlncaloricmeasure}). In particular, for each $(Q,\tau) \in \partial \Omega$ and $r > 0$ there exists two ``corkscrew" points, $A^{\pm}_r(Q,\tau) \colonequals (X^\pm_r(Q,\tau), t^{\pm}_r(Q,\tau))\in C_r(Q,\tau) \cap \Omega$ such that $d(A^{\pm}_r(Q,\tau), \partial \Omega) \geq r/100$ and $\min\{t^+_r(Q,\tau) - \tau, \tau - t^-_r(Q,\tau)\} \geq r^2/100$. 

Our theorems apply both to finite and infinite pole settings. Unfortunately, we will often have to treat these instances seperately. Fix (for the remainder of the paper) $(X_0,t_0) \in \Omega$ and define $u^{(X_0,t_0)}(-,-) = G(X_0,t_0,-,-)$, the Green function (which is adjoint-caloric), with a pole at $(X_0,t_0)$.  As above, $\omega^{(X_0,t_0)}$ is the associated caloric measure and $k^{(X_0,t_0)}$ the corresponding Poisson kernel (which exists by \cite{hlncaloricmeasure}, Theorem 1).  In addition, $u$ is the Green function with a pole at $\infty$, $\omega$ the associated caloric measure and $h$ the corresponding Poisson kernel. We will always assume (unless stated otherwise) that $\log(h) \in \mathrm{VMO}(\partial \Omega)$ or $\log(k^{(X_0,t_0)}) \in \mathrm{VMO}(\partial \Omega^{t_0})$. 

Finally, define, for convenience, the distance from $(X,t) \in \Omega$ to the boundary $$\delta(X,t) = \inf_{(Q,\tau) \in \partial \Omega} \|(X,t) - (Q,\tau)\|.$$

\subsection{Estimates for Green Functions in Parabolic Reifenberg Flat Domains}

Here we will state some estimates on the Green function of a parabolic Reifenberg flat domain that will be essential for the gradient bounds of Section \ref{boundingthegradient}. Corresponding estimates for the Green function with a pole at infinity are discussed in Appendix \ref{appendix:caloricmeasure}. 

We begin by bounding the growth of caloric functions which vanish on surface balls. The reader should note this result appears in different forms elsewhere in the literature  (e.g. \cite{lewisandmurray}, Lemma 6.1 and \cite{hlncaloricmeasure} Lemma 3.6), so we present the proof here for the sake of completeness.

\begin{lem}\label{growthattheboundary}
Let $\Omega$ be a $\delta$-Reifenberg flat domain and $(Q,\tau) \in \partial \Omega$. Let $w$ be a continuous non-negative solution to the (adjoint)-heat equation in $C_{2r}(Q,\tau)\cap \Omega$ such that $w = 0$ on $C_{2r}(Q,\tau)\cap \partial \Omega$. Then for any $\varepsilon > 0$ there exists a $\delta_0 = \delta_0(\varepsilon) > 0$ such that if $\delta < \delta_0$ there exists a $c = c(\delta_0) > 0$ such that \begin{equation}\label{growthatboundaryequation}
w(X,t) \leq c \left(\frac{d((X,t), (Q,\tau))}{r}\right)^{1-\varepsilon}\sup_{(Y,s) \in C_{2r}(Q,\tau)} w(Y,s)
\end{equation}
whenever $(X,t) \in C_{r}(Q,\tau)\cap \Omega$. 
\end{lem}

\begin{proof}
We argue is in the proof of Lemma 2.1 in \cite{kenigtoro}. Let $(Q,\tau) \in \partial \Omega$ and $r > 0$. Let $v_0$ be adjoint caloric in $C_{2r}(Q,\tau)\cap \Omega$ such that $v_0 = 0$ on $\Delta_{2r}(Q,\tau)$ and $v_0 \equiv 1$ on $\partial_p C_{2r}(Q,\tau) \cap \Omega$. By the maximum principle, \begin{equation}\label{transfertov} w(X,t) \leq [\sup_{(Y,s) \in C_{2r}(Q,\tau)} w(Y,s)] v_0(X,t).\end{equation} We will now attempt to bound $v_0$ from above. 

Assume, without loss of generality, that the plane of best fit at $(Q,\tau)$ for scale $2r$ is $\{x_n = 0\}$ and that $(Q,\tau) = (0,0)$. Define $\Lambda  = \{(X,t) = (x,x_n, t)\mid x_n \geq - 4r\delta\}$. It is a consequence of Reifenberg flatness that $C_r(0,0) \cap \Omega \subset C_r(0,0)\cap \Lambda$. Define $h_0$ to be an adjoint-caloric function in $\Lambda\cap C_{2r}(0,0)$ such that $h_0 = 0$ on $\partial \Lambda\cap C_{2r}(0,0)$ and $h_0 = 1$ on $\partial_pC_{2r}(0,0)\cap \Lambda$. By the maximum principle $h_0(X,t) \geq v_0(X,t)$ for all $(X,t) \in C_{2r}(0,0)\cap \Omega$. 

Finally, consider the function $g_0$ defined by $g_0(x, x_n, t) = x_n+ 4\delta r$. It is clear that $g_0$ is an adjoint caloric function on $\Lambda \cap C_{2r}(0,0)$. Furthermore, $h_0, g_0$ both vanish on  $\partial \Lambda \cap C_{2r}(Q,\tau)$. Recall that (adjoint-)caloric functions in a cylinder satsify a comparison principle (see Theorem 1.6 in \cite{fabesgarofalosalsa}). Hence, there is a constant $C > 0$ such that \begin{equation}\label{comparehandg} \frac{h_0(X,t)}{g_0(X,t)} \leq C\frac{h_0(0,r/2, 0)}{r},\; \forall (X,t) \in C_{r/4}(0,0)\cap \Lambda. \end{equation} Let $(X,t) = (x, a, t)$. Then equation \eqref{comparehandg} becomes \begin{equation}\label{boundingh} h_0(X,t) \leq C\frac{a+4\delta r}{r}.\end{equation} It is then easy to see, for any $\theta < 1$ and $(X,t) \in C_{\theta r}(0,0)$, that $v_0(X,t) \leq h_0(X,t) \leq C(\theta + \delta)$. Let $\theta = \delta$ and iterate this process. The desired result follows. 
\end{proof}

Using the parabolic Harnack inequality, we can say more about $\sup_{(Y,s)\in C_{2r}(Q,\tau)} w(Y,s)$.

\begin{lem}\label{biggeratNTApoint}[Lemma 3.7 in \cite{hlncaloricmeasure}]
Let $\Omega, w, (Q,\tau), \delta_0$ be as in Lemma \ref{growthattheboundary}. There is a universal constant $c(\delta_0) \geq 1$ such that if $(Y,s) \in \Omega\cap C_{r/2}(Q,\tau)$ then $$w(Y,s) \leq cw(A^{\pm}_r(Q,\tau)),$$ where we choose $A^-$ if $w$ is a solution to the adjoint-heat equation and $A^+$ otherwise. 
\end{lem}

As the heat equation is anisotropic, given a boundary point $(Q,\tau)$ it will behoove us to distinguish the points in $\Omega$ which are not much closer to $(Q,\tau)$ in time than in space. 

\begin{defin}\label{farenoughaway}
For $(Q,\tau) \in \partial \Omega$ and $A \geq 100$ define the {\bf time-space cone at scale $r$ with constant $A$}, $T_{A,r}^{\pm}(Q,\tau)$, by $$T_{A,r}^{\pm}(Q,\tau) \colonequals \{(X,t) \in \Omega\mid |X-Q|^2 \leq A |t-\tau|, \pm (t-\tau) \geq 4r^2\}.$$
\end{defin}

The next four estimates are presented, and proven, in \cite{hlncaloricmeasure}. We will simply state them here. The first compares the value of the Green function at a corkscrew point with the caloric or adjoint caloric measure of a surface ball. 

\begin{lem}\label{comparisontheorem}[Lemma 3.10 in \cite{hlncaloricmeasure}]
Let $\Omega, (Q,\tau), \delta_0$ be as in Lemma \ref{growthattheboundary}. Additionally suppose from some $A\geq 100, r > 0$ that $(X,t) \in T_{A,r}^+(Q,\tau)$.  There exists some $c = c(A) \geq 1$ (independent of $(Q,\tau)$) such that $$c^{-1}r^nG(X,t, A^+_r(Q,\tau)) \leq \omega^{(X,t)}(\Delta_{r/2}(Q,\tau)) \leq cr^nG(X,t, A^-_r(Q,\tau)).$$ Similarly if $(X,t) \in T_{A,r}^-(Q,\tau)$ we have $$c^{-1}r^nG(A^-_r(Q,\tau),X,t) \leq \hat{\omega}^{(X,t)}(\Delta_{r/2}(Q,\tau)) \leq cr^nG(A^+_r(Q,\tau),X,t).$$
\end{lem}

We now recall what it means for an (adjoint-)caloric function to satisfy a backwards in time Harnack inequality (see e.g. \cite{fabesgarofalosalsa}). 

\begin{defin}\label{backwardsharnack}
If $(Q,\tau) \in \partial \Omega$ and $\rho > 0$ we say that $w > 0$ satisfies a backwards Harnack inequality in $C_\rho(Q,\tau)\cap \Omega$ provided $w$ is a solution to the (adjoint-)heat equation in $C_\rho(Q,\tau)\cap \Omega$ and there exists $1 \leq \lambda < \infty$ such that $$w(X,t) \leq \lambda w(\tilde{X},\tilde{t}),\; \forall (X,t), (\tilde{X},\tilde{t}) \in C_r(Z,s),$$ where $(Z,s), r$ are such that $C_{2r}(Z,s) \subset C_\rho(Q,\tau)\cap \Omega$. 
\end{defin}

In Reifenberg flat domains, the Green function satisfies a backwards Harnack inequality.

\begin{lem}\label{backwardsharnack}[Lemma 3.11 in \cite{hlncaloricmeasure}]
Let $\Omega, (Q,\tau), \delta_0$ be as in Lemma \ref{growthattheboundary}. Additionally suppose from some $A\geq 100, r > 0$ that $(X,t) \in T_{A,r}^+(Q,\tau)$.  There exists some $c = c(A) \geq 1$ such that $$G(X,t, A^-_r(Q,\tau)) \leq c G(X,t, A^+_r(Q,\tau)).$$ On the other hand, if  $(X,t) \in T_{A,r}^-(Q,\tau)$ we conclude $$G(A^+_r(Q,\tau), X,t) \leq c G(A^-_r(Q,\tau),X,t).$$ 
\end{lem}

Lemmas \ref{comparisontheorem} and \ref{backwardsharnack}, imply that (adjoint-)caloric measure is doubling.

\begin{lem}\label{doublingmeasure}[Lemma 3.17 in \cite{hlncaloricmeasure}]
Let $\Omega,(Q,\tau), \delta_0$ be as in Lemma \ref{growthattheboundary}. Additionally suppose from some $A\geq 100, r > 0$ that $(X,t) \in T_{A,r}^+(Q,\tau)$. Then there exists a constant $c= c(A)\geq 1$ such that $$\omega^{(X,t)}(\Delta_r(Q,\tau)) \leq c \omega^{(X,t)}(\Delta_{r/2}(Q,\tau)).$$ If $(X,t) \in T_{A,r}^-(Q,\tau)$ a similar statement holds for $\hat{\omega}$. 
\end{lem}

In analogy to Lemma 4.10 in \cite{jerisonandkenig}, there is a boundary comparison theorem for (adjoint-)caloric functions in Reifenberg flat domains (see also Theorem 1.6 in \cite{fabesgarofalosalsa}, which gives a comparison theorem for caloric functions in cylinders). 

\begin{lem}\label{quotienttheorem}[Lemma 3.18 in \cite{hlncaloricmeasure}]
Let $\Omega, (Q,\tau), \delta_0$ be as in Lemma \ref{growthattheboundary}. Let $w, v \geq 0$ be continuous solutions to the (adjoint)-heat equations in $\overline{C}_{2r}(Q,\tau)\cap \overline{\Omega}$ with $w,v > 0$ in $\Omega \cap C_{2r}(Q,\tau)$ and $w= v = 0$ on $C_{2r}(Q,\tau)\cap \partial \Omega$. If $w,v$ satisfy a backwards Harnack inequality in $C_{2r}(Q,\tau)\cap \Omega$ for some $\lambda \geq 1$ then $$\frac{w(Y,s)}{v(Y,s)} \leq c(\lambda)\frac{w(A^{\pm}_r(Q,\tau))}{v(A^{\pm}_r(Q,\tau))},\; \forall (Y,s) \in C_{r/2}(Q,\tau)\cap \overline{\partial \Omega}.$$ Where we choose $A^-$ if $w,v$ are solutions to the adjoint heat equation and $A^+$ otherwise. 
\end{lem}

As in the elliptic setting, a boundary comparison theorem leads to a growth estimate. 

\begin{lem}\label{growhtofquotientatboundary}[Lemma 3.19 in \cite{hlncaloricmeasure}]
Let $\Omega, (Q,\tau), \delta_0, w, v$ be as in Lemma \ref{quotienttheorem}. There exists a $0 < \gamma \equiv \gamma(\lambda)\leq 1/2$ and a $c \equiv c(\lambda) \geq 1$ such that $$\left|\frac{w(X,t)v(Y,s)}{w(Y,s)v(X,t)}-1\right| \leq c\left(\frac{\rho}{r}\right)^\gamma,\; \forall (X,t), (Y,s) \in C_\rho(Q,\tau)\cap \Omega$$ whenever $0 < \rho \leq r/2$. 
\end{lem}

\subsection{VMO functions on Parabolic Chord Arc Domains}
Here we state some consequences of the condition $\log(h) \in \mathrm{VMO}(\partial \Omega)$ or $\log(k^{(X_0,t_0)}) \in \mathrm{VMO}(\partial \Omega^{t_0})$. Our first theorem is a reverse H\"older inequality for every exponent. This is a consequence of the John-Nirenberg inequality \cite{johnnirenberg}, in the Euclidean case (see Garnett and Jones, \cite{garnettandjones}). However, per a remark in \cite{garnettandjones}, the result remains true in our setting as $\partial \Omega$ is a ``space of homogenous type". For further remarks and justification, see Theorem 2.1 in \cite{kenigtoro}, which is the analogous result for the elliptic problem. 

\begin{lem}\label{hinvmosatisfiesreverseholder}
Let $\Omega\subset \mathbb R^{n+1}$ be a parabolic chord arc domain and $\log(f)\in \mathrm{VMO}(\partial \Omega)$. Then for all $(Q,\tau) \in \partial \Omega$ and $r > 0$ and $1 < q < \infty$ we have \begin{equation}\label{reverseholderinequality} \left(\fint_{\Delta_{r}(Q,\tau)} f^qd\sigma\right)^{1/q} \leq C \fint_{\Delta_r(Q,\tau)} fd\sigma. \end{equation} Here $C$ depends only on the $\mathrm{VMO}$ character of $f$, the chord arc constants of $\Omega$, $n$ and $q$.
\end{lem}

For the Poisson kernel with finite pole a localized analogue of the above Lemma holds (and is proved in much the same way):

\begin{lem}\label{hinvmosatisfiesreverseholderfinitepole}
Let $(X_0,t_0)\in \Omega$ with $(Q,\tau) \in \partial \Omega, A \geq 100, r > 0$ such that $(X_0,t_0) \in T^+_{A,r}(Q,\tau)$. If $\log(k^{(X_0,t_0)}) \in \mathrm{VMO}(\Omega^{t_0})$ then, for any $1 < q < \infty$  \begin{equation}\label{reverseholderinequalityfinitepole} \left(\fint_{\Delta_{r}(Q,\tau)} (k^{X_0,t_0})^qd\sigma\right)^{1/q} \leq C \fint_{\Delta_r(Q,\tau)} k^{(X_0,t_0)}d\sigma. \end{equation} Here $C > 0$ depends on $n,q, A$ the $\mathrm{VMO}$ character of $k^{(X_0,t_0)}$ and the chord arc constants of $\Omega$. 
\end{lem}

Substitute the poisson kernel, $h \colonequals \frac{d\omega}{d\sigma}$, for $f$ in Lemma \ref{hinvmosatisfiesreverseholder} to glean information on the concentration of harmonic measure in balls. This is the parabolic analogue of Corollary 2.4 in \cite{kenigtoro}. 

\begin{cor}\label{vmomeasurecomparison}
If $\Omega, h$ are as above, then for all $\varepsilon > 0, (Q,\tau) \in \partial \Omega, r> 0$ and $E \subset \Delta_{r}(Q,\tau)$ \begin{equation}\label{measurecomparison} C^{-1} \left(\frac{\sigma(E)}{\sigma(\Delta_r(Q,\tau))}\right)^{1+\varepsilon} \leq \frac{\omega(E)}{\omega(\Delta_r(Q,\tau))} \leq C \left(\frac{\sigma(E)}{\sigma(\Delta_r(Q,\tau))}\right)^{1-\varepsilon}. \end{equation} Here $C$ depends on $n, \varepsilon$, the chord arc constants of $\Omega$ and the $\mathrm{VMO}$ character of $h$. 
\end{cor}

Similarly, in the finite pole case we can conclude: 

\begin{cor}\label{vmomeasurecomparisonfinitepole}
Let $(X_0,t_0) \in \Omega, \log(k^{(X_0,t_0)}) \in \mathrm{VMO}(\partial \Omega^{t_0})$, and $A \geq 100, r >0$ and $(Q,\tau) \in \partial \Omega$ such that $(X_0,t_0) \in T^+_{A,r}(Q,\tau)$. Then for all $\varepsilon > 0$ and $E \subset \Delta_{r}(Q,\tau)$ \begin{equation}\label{measurecomparisonfinitepole} C^{-1} \left(\frac{\sigma(E)}{\sigma(\Delta_r(Q,\tau))}\right)^{1+\varepsilon} \leq \frac{\omega^{(X_0,t_0)}(E)}{\omega^{(X_0,t_0)}(\Delta_r(Q,\tau))} \leq C \left(\frac{\sigma(E)}{\sigma(\Delta_r(Q,\tau))}\right)^{1-\varepsilon}. \end{equation} Here $C$ depends on $n, \varepsilon, A$, the chord arc constants of $\Omega$ and the $\mathrm{VMO}$ character of $k^{(X_0,t_0)}$. 
\end{cor}

Finally, the John-Nirenberg inequality and the definition of VMO lead to the following decomposition (see the discussion in the proof of Lemma 4.3 in \cite{kenigtoro} for more detail--specifically equations 4.95 and 4.96). 

\begin{lem}\label{vmodecomposition}
Let $\Omega, h$ be as above. Given $\varepsilon > 0$ and $(Q_0,\tau_0) \in \partial \Omega$ there exists an $r(\varepsilon) > 0$ such that for $\rho \in (0, r(\varepsilon))$ and $(Q,\tau) \in \Delta_1(Q_0,\tau_0)$ there exists a $G(Q,\tau, \rho) \subset \Delta_{\rho}(Q,\tau)$ such that $\sigma(\Delta_\rho(Q,\tau)) \leq (1+\varepsilon)\sigma(G(Q,\tau, \rho))$ and, for all $(P,\eta) \in G(Q,\tau, \rho)$, \begin{equation}\label{goodpoints} (1+\varepsilon)^{-1}\fint_{\Delta_\rho(Q,\tau)} hd\sigma \leq h(P,\eta) \leq (1+\varepsilon)\fint_{\Delta_\rho(Q,\tau)} hd\sigma.\end{equation}
\end{lem}

And in the finite pole setting: 

\begin{lem}\label{vmodecompositionfinitepole}
Let $\Omega, k^{(X_0,t_0)}, (Q,\tau)\in \partial \Omega, r > 0, A\geq 100$ be as above. Given $\varepsilon > 0$ exists an $r(\varepsilon) > 0$ such that for $\rho \in (0, r(\varepsilon))$ and $(\tilde{Q},\tilde{\tau}) \in C_r(Q,\tau)$ there exists a $G(\tilde{Q},\tilde{\tau}, \rho) \subset \Delta_{\rho}(\tilde{Q},\tilde{\tau})$ such that $\sigma(\Delta_\rho(\tilde{Q},\tilde{\tau})) \leq (1+\varepsilon)\sigma(G(\tilde{Q},\tilde{\tau}, \rho))$ and, for all $(P,\eta) \in G(\tilde{Q},\tilde{\tau}, \rho)$, \begin{equation}\label{goodpoints} (1+\varepsilon)^{-1}\fint_{\Delta_\rho(\tilde{Q},\tilde{\tau})} k^{(X_0,t_0)}d\sigma \leq k^{(X_0,t_0)}(P,\eta) \leq (1+\varepsilon)\fint_{\Delta_\rho(\tilde{Q},\tilde{\tau})} k^{(X_0,t_0)}d\sigma.\end{equation}
\end{lem}

\section{Bounding the Gradient of the Green Function}\label{boundingthegradient}
As mentioned in the introduction, the first step in our proof is to establish an integral bound for $\nabla u$ (and $\nabla G(X_0,t_0, -,-)$). Later, this will aid in demonstrating that our blowup satisfies the hypothesis of the classification result, Theorem \ref{poissonkernel1impliesflat}.

We begin by estimating the non-tagential maximal function of the gradient. Recall the definition of a non-tangetial region: 

\begin{defin}\label{nontangentialregion}
For $\alpha > 0, (Q,\tau) \in \partial \Omega$ define, $\Gamma_\alpha(Q,\tau)$, the non-tangential region at $(Q,\tau)$ with aperture $\alpha$, as $$\Gamma_\alpha(Q,\tau) = \{(X,t) \in \Omega \mid \|(X,t) - (Q,\tau)\| \leq (1+\alpha)\delta(X,t)\}.$$ For $R > 0$ let $\Gamma_\alpha^R(Q,\tau) \colonequals \Gamma_\alpha(Q,\tau) \cap C_{R}(Q,\tau)$ denote the truncated non-tangential region. 

Associated with these non-tangential regions are maximal functions $$\begin{aligned} N_\alpha(f)(Q,\tau) \colonequals& \sup_{(X,t) \in \Gamma_\alpha(Q,\tau)} |f(X,t)|\\ N^R_\alpha(f)(Q,\tau) \colonequals& \sup_{(X,t) \in \Gamma^R_\alpha(Q,\tau)} |f(X,t)|.\end{aligned}$$

Finally, we say that $f$ has a non-tangential limit, $L$, at $(Q,\tau) \in \partial \Omega$ if for any $\alpha > 0$ $$\lim_{\stackrel{(X,t)\rightarrow (Q,\tau)}{(X,t) \in \Gamma_\alpha(Q,\tau)}} f(X,t) = L.$$
\end{defin}

In order to apply Fatou's theorem and Martin's repesentation theorem (see Appendix \ref{fatouetcforparaboliccase})  we must bound the non-tagential maximal function by a function in $L^2$. We argue as in the proof of Lemma 3.1 in \cite{kenigtoro} (which proves the analogous result in the elliptic setting). 

\begin{lem}\label{nontangetiallybounded}
For any $\alpha > 0, R > 0, N^R_\alpha(|\nabla u|) \in L^2_{loc}(d\sigma)$. 
\end{lem}

\begin{proof}
Let $K \subset \R^{n+1}$ be a compact set and $\hat{K}$ be the compact set of all points parabolic distance $\leq 4R$ away from $K$.  Pick $(X,t) \in \Gamma_\alpha^R(Q,\tau)$.  Standard estimates for adjoint-caloric functions, followed by Lemmas \ref{biggeratNTApoint} and  \ref{comparisonformeasureatinfinity} yield $$|\nabla u(X,t)| \leq C\frac{u(X,t)}{\delta(X,t)} \leq C\frac{u(A^-_{4\|(Q,\tau)- (X,t)\|}(Q,\tau))}{\delta(X,t)} \leq C\frac{\omega(\Delta_{2\|(Q,\tau)-(X,t)\|}(Q,\tau))}{\delta(X,t)\|(Q,\tau)-(X,t)\|^n}.$$ In the non-tangential region, $\delta(X,t)\sim_\alpha \|(Q,\tau) - (X,t)\|$, which, as $\sigma$ is Ahlfors regular and $\omega$ is doubling, implies $$|\nabla u(X,t)|\leq C_\alpha \fint_{\Delta_{\|(Q,\tau) - (X,t)\|}(Q,\tau)} hd\sigma \leq C_\alpha M_R(h)(Q,\tau).$$ Here, $M_R(h)(Q,\tau) \colonequals \sup_{0 < r \leq R} \fint_{C_r(Q,\tau)} |h(P,\eta)| d\sigma(P,\eta)$ is the truncated Hardy-Littlewood maximal operator at scale R. $\partial \Omega$ is a space of homogenous type and $h \in L^2_{\mathrm{loc}}(d\sigma)$, so we may apply the Hardy-Littlewood maximal theorem to conclude $$\int_{K} M_R(h)^2d\sigma \leq C\int_{\hat{K}} h^2 d\sigma < \infty.$$
\end{proof}

The result in the finite pole setting follows in the same way;

\begin{lem}\label{nontangentiallyboundedfinitepole}
For $(X_0,t_0) \in \Omega$ let $(Q,\tau) \in \partial \Omega, R> 0$ and $A \geq 100$ be such that $(X_0,t_0) \in T_{A,R}^+(Q,\tau)$. Then for any $\alpha > 0, N^{R/8}_\alpha(|\nabla u^{(X_0,t_0)}|)|_{\Delta_{R/2}(Q,\tau)} \in L^2(d\sigma)$.
\end{lem}

\begin{proof}
Let $(P,\eta) \in \Delta_{R/2}(Q,\tau)$ and pick $(X,t) \in \Gamma_\alpha^{R/8}(P,\eta)$.  Standard estimates for adjoint-caloric functions, followed by Lemma \ref{biggeratNTApoint} yield $$|\nabla u^{(X_0,t_0)}(X,t)| \leq C\frac{u^{(X_0,t_0)}(X,t)}{\delta(X,t)} \leq C\frac{u^{(X_0,t_0)}(A^-_{4\|(P,\eta)- (X,t)\|}(P,\eta))}{\delta(X,t)}.$$

Note $(X,t) \in \Gamma_\alpha^{R/8}(P,\eta)$ hence $4\|(P,\eta) - (X,t)\| \leq R/2$. By our assumption on $(Q,\tau)$ and $(P,\eta) \in \Delta_{R/2}(Q,\tau)$ we can compute that $R/2, (P,\eta), (X_0,t_0)$ satisfy the hypothesis of  Lemmas \ref{backwardsharnack} and \ref{comparisontheorem} for some $A \geq 100$ which can be taken uniformly over $(P,\eta) \in \Delta_{R/2}(Q,\tau)$. Therefore, \begin{equation}\label{boundbymeasurequotient} |\nabla u^{(X_0,t_0)}(X,t)| \leq C(A)\frac{\omega^{(X_0,t_0)}(\Delta_{2\|(P, \eta)-(X,t)\|}(P,\eta))}{\delta(X,t)\|(P,\eta)-(X,t)\|^n}.\end{equation}

In the non-tangential region, $\delta(X,t)\sim_\alpha \|(P,\eta) - (X,t)\|$, which, as $\sigma$ is Ahlfors regular and $\|(P,\eta) - (X,t)\| \leq R/8$  implies $$|\nabla u^{(X_0,t_0)}(X,t)|\leq C_{\alpha,A} \fint_{\Delta_{2\|(P,\eta)- (X,t)\|}(P,\eta)} k^{(X_0,t_0)}d\sigma \leq C_{\alpha, A} M_{R/2}(k^{(X_0,t_0)})(P,\eta).$$ The result then follows as in Lemma \ref{nontangetiallybounded}.
\end{proof}

Unfortunately, the above argument only bounds the {\it truncated} non-tangential maximal operators. We need a cutoff argument to transfer this estimate to the untruncated non-tangential maximal operator. We will do this argument first for the infinite pole case and then in the finite pole setting. The following lemma is a parabolic version of Lemma 3.3 in  \cite{kenigtoro} or Lemma 3.5 in \cite{kenigtorotwophase}, whose exposition we will follow quite closely.

\begin{lem}\label{cutoffargument}
Assume that $(0,0) \in \partial \Omega$ and fix $R > 1$ large. Let $\varphi_R \in C_c^\infty(\R^{n+1}), \varphi_R \equiv 1$ on $C_R(0,0)$, $0 \leq \varphi_R \leq 1$ and assume $\mathrm{spt}(\varphi_R) \subset C_{2R}(0,0)$. It is possible to ensure that $|\nabla \varphi_R| \leq C/R$ and $|\partial_t \varphi_R| , |\Delta \varphi_R| \leq C/R^2$. For $(X,t) \in \Omega$ define \begin{equation}\label{definitionofw}
w_R(X,t) = \int_{\Omega}G(Y,s,X,t) (\partial_s+\Delta_Y)[\varphi_R(Y,s)\nabla u(Y,s)]dYds,
\end{equation}
here, as before $G(Y,s,X,t)$ is the Green's function for the heat equation with a pole at $(X,t)$ (i.e. $(\partial_s- \Delta_Y)(G(Y,s,X,t)) = \delta_{(X,t),(Y,s)}$ and $G(Y,s,-,-) \equiv 0$ on $\partial \Omega$). Then, $w_R|_{\partial \Omega} \equiv 0$ and $w_R\in C(\overline{\Omega})$. Furthermore, if $\|(X,t)\| \leq \frac{R}{2}$, then
\begin{equation}\label{growthofwclose}
|w_R(X,t)| \leq C \frac{\delta(X,t)^{3/4}}{R^{1/2}}.
\end{equation}
Additionally, if $\|(X,t)\| \geq 4R$, there is a constant $C \equiv C(R) > 0$ such that \begin{equation}\label{growthofwfar}
|w_R(X,t)| \leq C
\end{equation}
\end{lem}

\begin{proof}
That $w_R|_{\partial \Omega} \equiv 0$ and $w_R\in C(\overline{\Omega})$ follows immediately from the definition of $w_R$. For ease of notation let $V(x,t) \colonequals \nabla u(X,t)$. By the product rule $$(\partial_s+\Delta_Y)[\varphi_R(Y,s)V(Y,s)] = V(Y,s)(\Delta_Y + \partial_s)\varphi_R(Y,s) + 2\left\langle \nabla \varphi_R(Y,s), \nabla V(Y,s)\right\rangle.$$  Split $w_R(X,t) = w^1_R(X,t) + w^2_R(X,t)$ where, \begin{equation*}\begin{aligned}w_R^1(X,t)&\colonequals \int_{\Omega}G(Y,s, X,t)V(Y,s)(\Delta_Y + \partial_s)\varphi_R(Y,s)dYds\\
w_R^2(X,t)&\colonequals 2\int_{\Omega}G(Y,s, X,t) \left\langle \nabla \varphi_R(Y,s), \nabla V(Y,s)\right\rangle dYds.
  \end{aligned}
  \end{equation*}
  
Regularity theory gives $|\nabla u(Y,s)| \leq C\frac{u(Y,s)}{\delta(Y,s)}$ and $|\nabla \nabla u(Y,s)| \leq C\frac{u(Y,s)}{\delta^2(Y,s)}$. This, along with our bounds on $\varphi_R$, yields \begin{equation}\label{firstestimateforw}\begin{aligned}
  |w_R^1(X,t)| &\leq \frac{C}{R^2}\int_{\{(Y,s)\in \Omega\mid R < \|(Y,s)\| \leq 2R\}}\frac{u(Y,s)}{\delta(Y,s)}G(Y,s,X,t)dYds\\
    |w_R^2(X,t)| &\leq \frac{C}{R}\int_{\{(Y,s)\in \Omega\mid R < \|(Y,s)\| \leq 2R\}} \frac{u(Y,s)}{\delta^2(Y,s)}G(Y,s,X,t) dYds.
  \end{aligned}
  \end{equation}

Any upper bound on $w_R^2$ will also be an upper bound for $w_R^1$ (as $R \gtrsim \delta(Y,s)$). Assume first, in order to prove \eqref{growthofwclose}, that $\|(X,t)\| < R/2$.  We start by showing that there is  a universal constant, $C> 1$ such that  \begin{equation}\label{harnackchainargument}
G(Y,s, X,t) \leq C \left(\frac{\delta(Y,s)}{R}\right)^{3/4} G(A^+_{3R}(0,0),X,t)
\end{equation}
 for all $(Y,s) \in \Omega \cap (C_{2R}(0,0)\backslash C_R(0,0))$ and $(X,t) \in \Omega \cap C_{R/2}(0,0)$. To prove this, first assume that $\delta(Y,s) \geq R/10$.  We would like to construct a Harnack chain between $A^+_{3R}(0,0)$ and $(Y,s)$. To do so, we need to verify that the parabolic distance between the two points is less than 100 times the square root of the distance between the two points along the time axis. As we are in a $\delta$-Reifenberg flat domain the $t$ coordinate of $A^+_{3R}(0,0)$ is equal to $(3R)^2$ and so $A^+_{3R}(0,0)$ and $(Y,s)$ are seperated in the $t$-direction by a distance of $5R^2$. On the other hand $\|A^+_{3R}(0,0) - (Y,s)\| \leq 20R < 100(5R^2)^{1/2}$. So there is a Harnack chain connecting $(Y,s)$ and $A^+_{3R}(0,0)$. In a $\delta$-Reifenberg flat domain the chain can be constructed to stay outside of $C_{R/2}(0,0)$ (see the proof of Lemma 3.3 in \cite{hlncaloricmeasure}). Furthermore, as $\delta(Y,s)$ is comparable to $R$, the length of this chain is bounded by some constant. Therefore, by the Harnack inequality, we have equation \eqref{harnackchainargument} (note in this case $\left(\frac{\delta(Y,s)}{R}\right)^{3/4} $ is greater than some constant, and so can be included on the right hand side). 

If $\delta(Y,s) < R/10$, there is a point $(Q,\tau) \in \partial \Omega$ such that $C_{R/5}(Q,\tau) \cap C_{R/2}(0,0) = \emptyset$, and $(Y,s) \in C_{R/10}(Q,\tau)$. Lemma \ref{growthattheboundary} yields $G(Y,s, X,t) \leq \left(\frac{\delta(Y,s)}{R}\right)^{3/4} G(A^+_{R/5}(Q,\tau), X,t).$ We can then create a Harnack chain, as above, connecting $A^+_{R/5}(Q,\tau)$ and $A^+_{3R}(0,0)$ to obtain equation \eqref{harnackchainargument}.

$G(A^+_{3R}(0,0), -,-)$ is an adjoint caloric function in $\Omega \cap C_{R}(0,0)$, whence, 
\begin{equation}\label{estimateongreensfunctionforclosext}\begin{aligned}
G(Y,s,X,t) \leq& C \left(\frac{\delta(Y,s)}{R}\right)^{3/4} G(A^+_{3R}(0,0),X,t)\\
 \leq& C \left(\frac{\delta(Y,s)}{R}\right)^{3/4}\left(\frac{\delta(X,t)}{R}\right)^{3/4}  G(A^+_{3R}(0,0), A^-_{R}(0,0))\\ \leq& C \left(\frac{\delta(Y,s)}{R}\right)^{3/4}\left(\frac{\delta(X,t)}{R}\right)^{3/4}R^{-n},\end{aligned}
\end{equation} where the penultimate inequality follows from Lemma \ref{growthattheboundary} applied in $(X,t)$. The bound on $G(A^+_{3R}(0,0), A^-_{R}(0,0))$ and, therefore, the last inequality above, is a consequence of Lemmas \ref{backwardsharnack} and \ref{comparisontheorem}: $G(A^+_{3R}(0,0), A^-_{R}(0,0)) \leq  G(A^+_{3R}(0,0), A^+_{R}(0,0)) \leq cR^{-n}\omega^{A^+_{3R}(0,0)}(C_{R/2}(0,0)) \leq cR^{-n}$.

Lemma \ref{growthattheboundary} applied to $u(Y,s)$ and equation \eqref{estimateongreensfunctionforclosext} allow us to bound $$\frac{C}{R}\int_{C_{2R}(0,0)\backslash C_R(0,0)} \frac{u(Y,s)}{\delta^2(Y,s)}G(Y,s,X,t) dYds \leq \frac{C}{R^{n+2+1/2}}\left(\frac{\delta(X,t)}{R}\right)^{3/4}\int_{\Omega\cap C_{2R}(0,0)} \frac{u(A^-_{2R}(0,0))}{\delta(Y,s)^{1/2}} dYds.$$

Ahlfors regularity implies, for any $\beta > (1/(2R))^{1/2}$, that $$|\{(Y,s) \in \Omega\cap C_{2R}(0,0)\mid \delta(Y,s)^{-1/2} > \beta\}| = |\{(Y,s) \in \Omega\cap C_{2R}(0,0)\mid \delta(Y,s) < \beta^{-2} \}| \lesssim R^{n+1} \beta^{-2}.$$ Therefore, $$\int_{\Omega\cap C_{2R}(0,0)} \frac{1}{\delta(Y,s)^{1/2}}dYds \lesssim R^{n+1} \int_{(1/(2R))^{1/2}}^\infty \frac{1}{\beta^2}d\beta \simeq R^{n+1+1/2}.$$

Putting everything together we get that \begin{equation}\label{puttingitalltogetherforwclose}
|w(X,t)| \leq C\frac{u(A^-_{2R}(0,0))}{R} \left(\frac{\delta(X,t)}{R}\right)^{3/4} \leq C \frac{\delta(X,t)^{3/4}}{R^{1/2}},
\end{equation}
where the last inequality follows from the fact that $u(A^-_{2R}(0,0))$ cannot grow faster than $R^{1+\alpha}$ for any $\alpha > 0$. This can be established by arguing as in proof of Lemma \ref{nontangetiallybounded} and invoking Corollary \ref{vmomeasurecomparison}. 

We turn to proving equation \eqref{growthofwfar}, and assume that $\|(X,t)\| \geq 4R$. Following the proof of equation \eqref{harnackchainargument} we can show \begin{equation}\label{farharnackchainargument}
G(Y,s, X,t) \leq C G(Y,s,A^-_{2^jR}(0,0)).
\end{equation} Above, $j$ is such that $2^{j-2}R \leq \|(X,t)\|\leq 2^{j-1}R$. Note that $G(Y,s, A^-_{2^jR})$ is a caloric function in $C_{2^{j-1}R}(0,0)$ and apply Lemma \ref{growthattheboundary} to obtain $$G(Y,s, X,t) \leq C\left(\frac{\delta(Y,s)}{2^{j-1}R}\right)^{3/4} G(A^+_{2^{j-1}R}(0,0), A^-_{2^jR}(0,0)) \leq C\delta(Y,s)^{3/4}(2^jR)^{-n-3/4}.$$ The last inequality follows from estimating the Green's function as we did in the proof of equation \eqref{estimateongreensfunctionforclosext}. Proceeding as in the proof of equation \eqref{growthofwclose} we write $$\frac{C}{R}\int_{\Omega \cap (C_{2R}(0,0)\backslash C_R(0,0))} \frac{u(Y,s)}{\delta^2(Y,s)}G(Y,s,X,t) dYds \leq (2^jR)^{-n-3/4} R^{-1-3/4} R^{n+1 + 1/2} u(A^-_{4R}(0,0)).$$ Putting everything together, \begin{equation}\label{wgetssmallfarout} |w_R(X,t)| \leq C \frac{u(A^-_{4R}(0,0))}{R} 2^{-nj} \leq C(R).\end{equation}
\end{proof}

\begin{cor}\label{nontangentiallimiteverywhere}
For any $(Y, s) \in \Omega$, $\nabla u$ has a non-tangential limit, $F(Q,\tau)$, for $d\hat{\omega}^{(Y,s)}$-almost every $(Q,\tau) \in \Omega$. In particular, the non-tangential limit exists for $\sigma$-almost every $(Q,\tau) \in \partial \Omega$. Furthermore, $F(Q,\tau) \in L^1_{\mathrm{loc}}(d\hat{\omega}^{(Y,s)})$ and $F(Q,\tau) \in L^2_{\mathrm{loc}}(d\sigma)$. 
\end{cor}

\begin{proof}
Theorem 1 in \cite{hlncaloricmeasure} implies that for any compact set $K \subset \partial \Omega$ there exists $(Y,s) \in \Omega$ such that $\hat{\omega}^{(Y,s)}|_K \in A_\infty(\sigma|_{K})$. Therefore, if the non-tangential limit exists $d\hat{\omega}^{(Y,s)}$-almost everywhere for any $(Y,s) \in \Omega$ we can conclude that it exists $\sigma$-almost everywhere. Additionally, Lemma \ref{nontangetiallybounded} implies that if $\nabla u$ has a non-tangential limit, that limit is in $L^2_{\mathrm{loc}}(d\sigma)$ and therefore $L^1_{\mathrm{loc}}(d\hat{\omega}^{(Y,s)})$ for any $(Y,s) \in \Omega$.

Thus it suffices to prove, for any $(Y,s) \in \Omega$, that $\nabla u$ has a non-tangential limit $d\hat{\omega}^{(Y,s)}$-almost everywhere. Let $R > 0$ and define, for $(X,t) \in \Omega$, $H_R(X,t) = \varphi_R(X,t)\nabla u(X,t) - w_R(X,t)$, where $w_R,\varphi_R$ were introduced in Lemma \ref{cutoffargument}. Equation \eqref{growthofwfar} and $w_R(X,t) \in C(\overline{\Omega})$ imply that $w_R(X,t) \in L^\infty(\Omega)$, which, with Lemma \ref{nontangetiallybounded}, gives that $N(H_R)(X,t) \in L^1(d\hat{\omega}^{(Y,s)})$ for any $(Y,s) \in \Omega$. By construction, $H_R$ is a solution to the adjoint heat equation in $\Omega$, hence, by Lemma \ref{nontangentiallimitsdae}, $H_R(X,t)$ has a non-tangential limit $\hat{\omega}^{(Y,s)}$-almost everywhere. Finally, because $w_R, \varphi_R \in C(\overline{\Omega})$ we can conclude that $\nabla u$ has a non-tangential limit for $\hat{\omega}^{(Y,s)}$-almost every point in $C_{R}(0,0)$. As $R$ is arbitrary the result follows. 
\end{proof}

If we assume higher regularity in $\partial \Omega$, it is easy to conclude that $\nabla u(Q,\tau) = h(Q,\tau)\hat{n}(Q,\tau)$ for every $(Q,\tau) \in \partial \Omega$. The following lemma, proved in Appendix \ref{appendix:nontangentiallimit}, says that this remains true in our (low regularity) setting.

\begin{lem}\label{equaltoh}
For $\sigma$-a.e. $(Q,\tau)\in \partial \Omega$ we have $F(Q,\tau) = h(Q,\tau)\hat{n}(Q,\tau)$
\end{lem}

Finally, we can prove the integral estimate.  

\begin{lem}\label{gradientrepresentation}
Let $\Omega$ be a $\delta$-Reifenberg flat parabolic regular domain. Let $u,h$ be the Green function and parabolic Poisson kernel with poles at infinity respectively.  Fix $R >>1$, then for any $(X,t) \in \Omega$ with $\|(X,t)\| \leq R/2$ \begin{equation}\label{gradientbound}
|\nabla u(X,t)| \leq \int_{\Delta_{2R}(0,0)} h(Q,\tau) d\hat{\omega}^{(X,t)}(Q,\tau) + C\frac{\|(X,t)\|^{3/4}}{R^{1/2}}.
\end{equation}
\end{lem}

\begin{proof}
For $(X,t) \in \Omega$ define $H_R(X,t) = \varphi_R(X,t)\nabla u(X,t) - w_R(X,t)$ where $w_R,\varphi_R$ were introduced in Lemma \ref{cutoffargument}. $H_R(X,t)$ is a solution to the adjoint heat equation (by construction) and $N(H_R) \in L^1(d\hat{\omega}^{(Y,s)})$ for every $(Y,s) \in \Omega$ (as shown in the proof of Corollary \ref{nontangentiallimiteverywhere}). Hence, by Proposition \ref{generalrepresentationtheorem} we have $H_R(X,t) = \int_{\partial \Omega} g(Q,\tau) d\hat{\omega}^{(X,t)}(Q,\tau)$, where $g(Q,\tau)$ is the non-tangential limit of $H_R$. 

Lemma \ref{equaltoh} and $w_R|_{\partial \Omega} \equiv 0$ imply that $g(Q,\tau) = \varphi_R(Q,\tau)h(Q,\tau)\hat{n}(Q,\tau)$.  Estimate \eqref{growthofwclose} on the growth of $w_R$ allows us to conclude
\begin{equation*}\label{nontangentiallimitofh}
|\nabla u(X,t)| \leq |H_R(X,t)| + |w_R(X,t)| \leq \int_{\partial \Omega \cap C_{2R}(0,0)} h(Q,\tau)d\hat{\omega}^{(X,t)}(Q,\tau) + C\frac{\|(X,t)\|^{3/4}}{R^{1/2}}.
\end{equation*}
\end{proof}

The finite pole case begins similarly; we start with a cut-off argument much in the style of Lemma \ref{cutoffargument}. 

\begin{lem}\label{cutoffargumentfinitepole}
Let $(X_0, t_0) \in \Omega$ and fix any $(Q,\tau) \in \partial \Omega, R > 0, A \geq 100$ such that $(X_0,t_0) \in T_{A,R}^+(Q,\tau)$. Let $\varphi \in C^\infty_c(C_{R/2}(Q,\tau))$. Furthermore, it is possible to ensure that $\varphi \equiv 1$ on $C_{R/4}(Q,\tau)$, $0 \leq \varphi \leq 1, |\nabla \varphi| \leq C/R$ and $|\partial_t \varphi| , |\Delta \varphi| \leq C/R^2$. 

For $(X,t) \in \Omega$ define \begin{equation}\label{definitionofwfinitepole}
W(X,t) = \int_{\Omega}G(Y,s,X,t) (\partial_s+\Delta_Y)[\varphi(Y,s)\nabla u^{(X_0,t_0)}(Y,s)]dYds.
\end{equation}
Then, $W|_{\partial \Omega} \equiv 0$ and $W\in C(\overline{\Omega})$. Additionally, if $\|(X,t)-(Q,\tau)\| \leq R/8$ then
\begin{equation}\label{growthofwclosefinitepole}
|W(X,t)| \leq C(A) \left(\frac{\delta(X,t)}{R}\right)^{3/4}\frac{\omega^{(X_0,t_0)}(\Delta_{R}(Q,\tau))}{R^{n+1}}.
\end{equation}
Finally, if $\|(X,t) - (Q, \tau)\| \geq 32R$ there is a constant $C > 0$ (which might depend on $(X_0,t_0), (Q,\tau)$ but is independent of $(X,t)$) such that \begin{equation}\label{growthofwfarfinitepole}
|W(X,t)| \leq C
\end{equation}
\end{lem}

\begin{proof}
Using the notation from Lemma \ref{cutoffargumentfinitepole}, observe that $\varphi(X,t)\equiv \varphi_{R/4}((X,t)-(Q,\tau))$. Therefore, the continuity and boundary values of $W$ follows as in the infinite pole case. Furthermore, arguing exactly as in the proof of equation \eqref{growthofwclosefinitepole} (taking into account our modifications on $\varphi$) we establish an analogue to equation \eqref{puttingitalltogetherforwclose} in the finite pole setting;
\begin{equation}\label{puttingitalltogetherforwclosefinitepole}
|W(X,t)| \leq C\frac{u^{(X_0,t_0)}(A^-_{R/4}(Q,\tau))}{R} \left(\frac{\delta(X,t)}{R}\right)^{3/4}.
\end{equation}
By assumption, $(X_0,t_0), (Q,\tau)$ and $R$ satisfy the hypothesis of Lemmas \ref{comparisontheorem} and \ref{backwardsharnack}. As such, we may apply these lemmas to obtain the desired inequality \eqref{growthofwclosefinitepole}. 

To prove equation \eqref{growthofwfarfinitepole} we follow the proof of equation \eqref{growthofwfar} to obtain an analogue of \eqref{wgetssmallfarout};
\begin{equation}\label{wgetssmallfaroutfinitepole}
|W(X,t)| \leq C\frac{u^{(X_0,t_0)}(A^-_{R/4}(Q,\tau))}{R} \frac{R}{\|(X,t)- (Q,\tau)\|^n} \leq C.
\end{equation}
\end{proof}

\begin{cor}\label{nontangentiallimiteverywherefinitepole}
For $(X_0,t_0) \in \Omega$, let $(Q,\tau) \in \partial \Omega, R > 0, A \geq 100$ be as in Lemma \ref{cutoffargumentfinitepole}.  For any $(Y, s) \in \Omega$, $\nabla u^{(X_0,t_0)}(-,-)$ has a non-tangential limit, $F^{(X_0,t_0)}(P,\eta)$, for $d\hat{\omega}^{(Y,s)}$-almost every $(P,\eta) \in \Delta_{R/4}(Q,\tau)$. In particular, the non-tangential limit exists for $\sigma$-almost every $(P,\eta) \in \Delta_{R/4}(Q,\tau)$. Furthermore, $F^{(X_0,t_0)}|_{\Delta_{R/4}(Q,\tau)} \in L^1(d\hat{\omega}^{(Y,s)})$ and $F^{(X_0,t_0)}|_{\Delta_{R/4}(Q,\tau)} \in L^2(d\sigma)$. 
\end{cor}

\begin{proof}
Theorem 1 in \cite{hlncaloricmeasure} implies that there exists $(Y,s) \in \Omega$ such that $\hat{\omega}^{(Y,s)}|_{\Delta_{\delta(X_0,t_0)/4}(Q,\tau)} \in A_\infty(\sigma|_{\Delta_{R}(Q,\tau)})$. Therefore, if the non-tangential limit exists $d\hat{\omega}^{(Y,s)}$-almost everywhere on $\Delta_{R/4}(Q,\tau)$ for any $(Y,s) \in \Omega$ we can conclude that it exists $\sigma$-almost everywhere on $\Delta_{R/4}(Q,\tau)$. Additionally, Lemma \ref{nontangentiallyboundedfinitepole} implies that if $\nabla u^{(X_0,t_0)}(-,-)$ has a non-tangential limit on $\Delta_{R/4}(Q,\tau)$, that limit is in $L^2(d\sigma)$-integrable on $\Delta_{R/4}(Q,\tau)$ and therefore $L^1(d\hat{\omega}^{(Y,s)})$-integrable on $\Delta_{R/4}(Q,\tau)$ for any $(Y,s) \in \Omega$.

Thus it suffices to prove, for any $(Y,s) \in \Omega$, that $\nabla u^{(X_0,t_0)}(-,-)$ has a non-tangential limit $d\hat{\omega}^{(Y,s)}$-almost everywhere on $\Delta_{R/4}(Q,\tau)$. Let $\varphi, W$ be as in Lemma \ref{cutoffargumentfinitepole} and define, for $(X,t) \in \Omega$, $H(X,t) = \varphi(X,t)\nabla u^{(X_0,t_0)}(X,t) - W(X,t)$. Equation \eqref{growthofwfarfinitepole} and $W(X,t) \in C(\overline{\Omega})$ imply that $W(X,t) \in L^\infty(\Omega)$. Lemma \ref{nontangentiallyboundedfinitepole} implies that $N(H)(P,\eta) \in L^1(d\hat{\omega}^{(Y,s)})$ for any $(Y,s) \in \Omega$  (as outside of $C_{R/2}(Q,\tau)$ we have $H = -W$, which is bounded).  By construction, $H$ is a solution to the adjoint heat equation in $\Omega$, hence, by Lemma \ref{nontangentiallimitsdae}, $H(X,t)$ has a non-tangential limit $d\hat{\omega}^{(Y,s)}$-almost everywhere. Finally, because $W, \varphi \in C(\overline{\Omega})$ we can conclude that $\nabla u^{(X_0,t_0)}(-,-)$ has a non-tangential limit for $d\hat{\omega}^{(Y,s)}$-almost every point in $\Delta_{R/4}(Q,\tau)$.
\end{proof}

As in the infinite pole case, if we assume higher regularity in $\partial \Omega$, it is easy to conclude that $\nabla u^{(X_0,t_0)}(-,-)(P,\eta) = k^{(X_0,t_0)}(P,\eta)\hat{n}(P,\eta)$ for every $(P,\eta) \in \partial \Omega$. The following lemma, proved in Appendix \ref{appendix:nontangentiallimit}, says that this remains true in our (low regularity) setting.

\begin{lem}\label{equaltohfinitepole}
For $(X_0,t_0) \in \Omega$ let $(Q,\tau) \in \partial \Omega, R> 0, A\geq 100$ be as in Lemma \ref{cutoffargumentfinitepole}. Then for $\sigma$-a.e. $(P,\eta)\in \Delta_{R/4}(Q,\tau)$ we have $F^{(X_0,t_0)}(P,\eta)= k^{(X_0,t_0)}(P,\eta)\hat{n}(P,\eta)$
\end{lem}

Finally, we have the integral estimate (the proof follows as in the infinite pole case and so we omit it). 

\begin{lem}\label{gradientrepresentationfinitepole}
For $(X_0,t_0) \in \Omega$ let $(Q,\tau) \in \partial \Omega, R > 0, A \geq 100$ be as in Lemma \ref{cutoffargumentfinitepole}. Then for any $(X,t) \in \Omega$ with $\|(X,t)-(Q,\tau)\| \leq \delta R/8$ \begin{equation}\label{gradientbound}
|\nabla u^{(X_0,t_0)}(X,t)| \leq \int_{\Delta_{R/2}(Q,\tau)} k^{(X_0,t_0)}d\hat{\omega}^{(X,t)} + C \left(\frac{\delta(X,t)}{R}\right)^{3/4}\frac{\omega^{(X_0,t_0)}(\Delta_{R}(Q,\tau))}{R^{n+1}}.
\end{equation}
Here $C \equiv C(A) < \infty$. 
\end{lem}

\section{$\Omega$ is Vanishing Reifenberg Flat}\label{blowupanalysis}

In this section we use a blowup argument to prove Proposition \ref{vanishingreifenbergflat}, that $\Omega$ is vanishing Reifenberg flat, and Lemma \ref{weakconvergenceinpseudoblowups}, that $\lim_{r\downarrow 0} \sup_{(Q,\tau) \in K\cap \partial \Omega}\frac{\sigma(C_r(Q,\tau))}{r^{n+1}} = 1$. To do this, we invoke Theorem \ref{poissonkernel1impliesflat}, the classification of ``flat blow-ups". 

We now describe the blowup process,
\begin{defin}\label{parabolicpseudoblowup}
Let $K$ be a compact set (in the finite pole case we require $K = \Delta_R(Q,\tau)$ where $(Q,\tau) \in \partial \Omega, R > 0$ satisfy $(X_0,t_0) \in T_{A,4R}^+(Q,\tau)$ for some $A\geq 100$), $(Q_i,\tau_i) \in K \cap \partial \Omega$ and $r_i \downarrow 0$. Then we define 

\begin{subequations}
\label{eq:blowups}
\begin{align}
\Omega_i \colonequals& \{(X,t)\mid (r_iX + Q_i, r_i^2t + \tau_i)\in \Omega\} \\
u_i(X,t) \colonequals& \frac{u(r_iX + Q_i, r_i^2t+\tau_i)}{r_i \fint_{C_{r_i}(Q_i,\tau_i)\cap \partial \Omega} h d\sigma} \label{ublowup}\\
\omega_i(E) \colonequals& \frac{\sigma(C_{r_i}(Q_i,\tau_i))}{r_i^{n+1}} \frac{\omega(\{(P,\eta)\in \Omega\mid ((P-Q_i)/r_i, (\eta-\tau_i)/r_i^2)\in E\})}{\omega(C_{r_i}(Q_i,\tau_i))}\label{omegablowup}\\
h_i(P,\eta) \colonequals& \frac{h(r_iP + Q_i, r_i^2\eta+\tau_i)}{\fint_{C_{r_i}(Q_i,\tau_i)\cap \partial \Omega} h d\sigma}\label{hblowup}\\
\sigma_i \colonequals& \sigma|_{\partial \Omega_i}. \label{sigmablowup}
    \end{align}
\end{subequations}
Similarly we can define $u_i^{(X_0,t_0)}, \omega_i^{(X_0,t_0)}$ and $k_i^{(X_0,t_0)}$.
\end{defin}

\begin{rem}\label{whattheblowupsmean}
Using the uniqueness of the Green function and caloric measure it follows by a change of variables that $u_i$ is the adjoint-caloric Green's function for $\Omega_i$ with caloric measure $\omega_i$ and $$d\omega_i = h_i d\sigma_i.$$

Similarly, $u_i^{(X_0,t_0)}$ is the Green function for $\Omega_i$ with a pole at $(\frac{X_0-Q_i}{r_i}, \frac{t_0-\tau_i}{r_i^2})$ with associated caloric measure $\omega_i^{(X_0,t_0)}$ and Poisson kernel $k_i^{(X_0,t_0)}$. 
\end{rem}

We first need to show that (perhaps passing to a subsequence) the blowup process limits to a parabolic chord arc domain. In the elliptic setting this is Theorem 4.1 in \cite{kenigtoro}. Additionally, in \cite{nystrom1}, Nystr\"om considered a related parabolic blowup to the one above and proved similar convergence results. 

\begin{lem}\label{limitsexistforblowups}
Let $\Omega_i, u_i, h_i, \omega_i$ (or $u_i^{(X_0,t_0)}, k_i^{(X_0,t_0)}$ and $\omega_i^{(X_0,t_0)}$) be as in Definition \ref{parabolicpseudoblowup}. Then there exists a subsequence (which we can relabel for convenience) such that for any compact $K$ \begin{equation}\label{limitofthedomains}
D[K\cap\Omega_i, K\cap \Omega_\infty]\rightarrow 0
\end{equation} where $\Omega_\infty$ is a parabolic chord arc domain which is $4\delta$-Reifenberg flat. 

Moreover there is a $u_\infty \in C(\overline{\Omega}_\infty)$ such that  \begin{equation}\label{limitofus}u_i \rightarrow u_\infty\end{equation} uniformly on compacta. Additionally, there is a Radon measure $\omega_\infty$ supported on $\partial \Omega_\infty$ such that \begin{equation}\label{limitofomegas}
\omega_i \rightharpoonup \omega_\infty.
\end{equation}
Finally, $u_\infty, \omega_\infty$ are the Green function and caloric measure with poles at infinity for $\Omega_\infty$ (i.e. they satisfy equation \eqref{ipgreenfunction}).\end{lem}

\begin{proof}
Lemma 16 in \cite{nystrom1} proves that $\Omega_j \rightarrow \Omega_\infty$ and that $\Omega_\infty$ is $4\delta$-Reifenberg flat. In the same paper, Lemma 17 proves that $u_j \rightarrow u_\infty, \omega_j \rightarrow \omega_\infty$ and that $u_\infty, \omega_\infty$ satisfy equation \eqref{ipgreenfunction}. A concerned reader may point out that their blowup differs slightly from ours (as their $\Omega$ is not necessarily a chord arc domain). However, using Ahlfors regularity their argument works virtually unchanged in our setting (see also the proof of Theorem 4.1 in \cite{kenigtoro}). 

Therefore, to finish the proof it suffices to show that $\Omega_\infty$ is a parabolic regular domain. That is, $\sigma_\infty \equiv \sigma|_{\partial \Omega_\infty}$ is Ahlfors regular and $\partial \Omega$ is uniformly parabolic rectifiable. Let us first concentrate on $\sigma_\infty$. Note that for each $t_0 \in \mathbb R$ we have that $(\Omega_j)_{t_0} \equiv \Omega_j \cap \{(Y,s)\mid s= t_0\}$ is $\delta$-Reifenberg flat (and thus  the topological boundary coincides with measure theoretic boundary). Furthermore, we claim that $(\Omega_j)_{t_0} \rightarrow (\Omega_\infty)_{t_0}$ in the Hausdorff distance sense.  This follows from the observation that, in a Reifenberg flat domain $\Omega$, the closest point on $\partial \Omega$ to $(X,t) \in \Omega$ is also at time $t$ (see Remark \ref{closestatsametimepoint}). 

For almost every $s_0$ we know that $\Omega_{s_0}$ is a set of locally finite perimeter in $X\in \mathbb R^n$ and thus $(\Omega_j)_{t_0}$ is a set of locally finite perimeter for almost every $t_0$. We claim, for those $t_0$, that $\chi_{\Omega_j}(X,t_0) \rightarrow \chi_{\Omega_\infty}(X,t_0)$ in $L^1_{\mathrm{loc}}(\mathbb R^n)$. Indeed, by compactness, there exists an $E_{t_0}$ such that $\chi_{\Omega_j}(X,t_0) \rightarrow \chi_{E_{t_0}}$. That $E_{t_0} = (\Omega_\infty)_{t_0}$ is a consequence of $(\Omega_j)_{t_0}\rightarrow (\Omega_\infty)_{t_0}$ in the Hausdorff distance sense (for more details, see the bottom of page 351 in \cite{kenigtoro}). Hence, for almost every $t_0$, $(\Omega_\infty)_{t_0}$ is a set of locally finite perimeter. In addition, lower semicontinuity and Fatou's lemma imply \begin{equation}\label{limitsigmaisahlforsregular}\begin{aligned} \sigma_\infty(\Delta_R(P,\eta)) &= \int_{\eta-R^2}^{\eta+R^2} \mathcal H^{n-1}(\{(X,s) \mid (X,s) \in \partial \Omega_\infty, |X-P| \leq R\})ds\\ &\leq \int_{\eta-R^2}^{\eta+R^2} \liminf_{i\rightarrow \infty} \mathcal H^{n-1}(\{(X,s) \mid (X,s) \in \partial \Omega_i, |X-P| \leq R\})ds\\ &\leq \liminf_{i\rightarrow \infty} \int_{\eta-R^2}^{\eta+R^2} \mathcal H^{n-1}(\{(X,s) \mid (X,s) \in \partial \Omega_i, |X-P| \leq R\})ds\\ &\leq MR^{n+1}.\end{aligned}\end{equation} (The last inequality above follows from the fact that $\sigma$ is Ahflors regular and the definition of the blowup). The lower Ahlfors regularity is given immediately by the $\delta$-Reifenberg flatness of $\Omega_\infty$.

It remains to show that $\nu_\infty$ (defined as in \eqref{whatisnu} but with respect to $\Omega_\infty$) is a Carleson measure. Define $\gamma^{(\infty)}(Q,\tau, r) \colonequals \inf_P r^{-n-3}\int_{C_{r}(Q,\tau)\cap \partial \Omega_\infty} d((Y,s), P)^2 d\sigma_\infty(Y,s)$ where the infinum is taken over all $n$-planes $P$ containing a line parallel to the $t$-axis. Similarly define $\gamma^{(i)}(Q,\tau, r)$. 
We claim that \begin{equation}\label{settingupfatouslemma} \gamma^{(\infty)}(P,\eta, r)\leq \liminf_{i\rightarrow \infty} \gamma^{(i)}(P_i,\eta_i, r+\varepsilon_i), \forall (P,\eta) \in \partial \Omega_\infty, \end{equation} where $(P_i,\eta_i) \in \partial \Omega_i$ is the closest point in $\partial \Omega_i$ to $(P,\eta)$ and $\varepsilon_i \downarrow 0$ is any sequence such that $\varepsilon_i \geq 2D[\partial \Omega_i\cap C_{2r}(P,\eta), \partial \Omega_\infty \cap C_{2r}(P,\eta)] $. 

 Let $V_i$ be a plane which achieves the infinum in $\gamma^{(i)}(P_i,\eta_i, r+\varepsilon_i)$. Passing to a subsequence, the $V_i$ converge in the Hausdorff distance to some $V_\infty$. As such, there exists $\delta_i \downarrow 0$ with $D[V_i\cap C_1(0,0), V_\infty\cap C_1(0,0)] < \delta_i$. Estimate, \begin{equation*}\label{gammaibiggerthangammainfinity} \begin{aligned} \gamma^{(i)}(P_i,\eta_i, r+\varepsilon_i) &= (r+ \varepsilon_i)^{-n-3}\int_0^\infty 2\lambda\sigma_i(\{(Y,s) \in C_{r + \varepsilon_i}(P_i,\eta_i) \mid d((Y,s), V_i) > \lambda\})d\lambda\\ &\geq (r+\varepsilon_i)^{-n-3} \int_0^\infty 2\lambda \sigma_i(\{(Y,s)\in C_r(P,\eta) \mid d((Y,s), V_i) > \lambda\})d\lambda \\ &\geq (r+\varepsilon_i)^{-n-3} \int_0^\infty 2\lambda \sigma_i(\{(Y,s) \in C_r(P,\eta) \mid d((Y,s), V_\infty) > \lambda + r\delta_i\})d\lambda\\ \stackrel{\gamma = \lambda +r\delta_i}{\geq}& (r+\varepsilon_i)^{-n-3} \int_{o(1)}^\infty 2(\gamma - o(1)) \sigma_i(\{(Y,s) \in C_r(P,\eta) \mid d((Y,s), V_\infty) > \gamma\})d\gamma.\end{aligned}\end{equation*} Take $\liminf$s of both sides and recall, as argued above, that for all open $U$, $\sigma_\infty(U) \leq \liminf_{i\rightarrow \infty} \sigma_i(U)$. Equation \eqref{settingupfatouslemma} then follows from dominated convergence theorem, Fatou's lemma and Ahlfors regularity. 

 We claim, for any $\rho > 0$, \begin{equation}\label{settingupfatouslemma2} \int_{C_\rho(P,\eta)} \gamma^{(\infty)}(Y,s, r) d\sigma_\infty(Y,s) \leq \liminf_i \int_{C_{\rho+ \varepsilon_i}(P_i,\eta_i)}\gamma^{(i)}(Y,s, r+ \varepsilon_i)d\sigma_i(Y,s)\equalscolon F^i(r) \end{equation}where $\rho > r > 0,\; (P,\eta) \in \partial \Omega_\infty$, the $(P_i,\eta_i)$ are as above and $\varepsilon_i \downarrow 0$ with $\varepsilon_i \geq 3D[\partial \Omega_i\cap C_{3\rho}(P,\eta), \partial \Omega_\infty \cap C_{3\rho}(P,\eta)] \rightarrow 0$. The proof of equation \eqref{settingupfatouslemma2} is in the same vein as that of equation \eqref{settingupfatouslemma}, and thus we will omit it. Observing that the $\|\nu_i\|_+$ are bounded uniformly in $i$, Fatou's lemma implies \begin{equation}\label{fatouonthenus} \begin{aligned} \nu_\infty(C_\rho(P,\eta)\times [0,\rho)) =& \int_0^\rho F^\infty(r)\frac{dr}{r} \stackrel{\mathrm{eq}\; \eqref{settingupfatouslemma2}}{\leq} \int_0^\rho \liminf_{i\rightarrow \infty} F^i(r) \frac{dr}{r}\\ \leq &\liminf_{i\rightarrow \infty} \nu_i(C_{\rho+\varepsilon_i}(P_i,\eta_i)\times [0,\rho + \varepsilon_i)) \leq C\rho^{n+1}\limsup_{i} \|\nu_i\|_+.\end{aligned}\end{equation}
 \end{proof}

We now want to show a bound on $\nabla u_\infty$ (in hopes of applying Proposition \ref{poissonkernel1impliesflat}). Here we follow \cite{nystromgraphdomains} (see, specifically, the proof of Lemma 3.3 there). 

\begin{prop}\label{maingradientbound}
$|\nabla u_\infty(X,t)| \leq 1$ for all $(X,t) \in \Omega_\infty$. 
\end{prop}

\begin{proof}[Proof: Infinite Pole case]
Let $(X,t) \in \Omega_\infty$ and define $d((X,t), \partial \Omega_\infty) \equalscolon \delta_\infty$. There exists an $i_0 > 0$ such that if $i > i_0$ then $C_{3\delta_\infty/4}(X,t) \subset \Omega_i$. By standard parabolic regularity theory, $\nabla u_i \rightarrow \nabla u_\infty$ uniformly on $C_{\delta_\infty/2}(X,t)$.

$\Omega_i$ is a parabolic regular domain with $(0,0) \in \partial \Omega_i$ so by Lemma \ref{gradientrepresentation} we can write \begin{equation}\label{representationofgradui}
|\nabla u_i(X,t)| \leq \int_{\partial \Omega_i \cap C_{M}(0,0)} h_i(Q,\tau) d\hat{\omega}^{(X,t)}_i(Q,\tau) + C\frac{\|(X,t)\|^{3/4}}{M^{1/2}}
\end{equation}
as long as $M/2 \geq \max\{2\|(X,t)\|, 1\}$. For $\varepsilon > 0$, let $M \geq 2$ be such that $C\frac{\|(X,t)\|^{3/4}}{M^{1/2}} < \varepsilon/2$ and let $\varepsilon' = \varepsilon'(M, \delta_\infty, \varepsilon) > 0$ be a small constant to be chosen later. By Lemma \ref{vmodecomposition} there exists an $r(\varepsilon') > 0$ such that if $(Q_i,\tau_i) \in K\cap \partial \Omega$ and $r_iM < r(\varepsilon')$ there is a set $G_i \colonequals G((Q_i,\tau_i), Mr_i) \subset C_{Mr_i}(Q_i,\tau_i)\cap \partial \Omega$ with the properties that $(1+\varepsilon')\sigma(G_i) \geq \sigma(C_{Mr_i}(Q_i,\tau_i))$ and, for all $(P,\eta) \in G_i$, $$(1+\varepsilon')^{-1}\fint_{C_{Mr_i}(Q_i,\tau_i)} hd\sigma \leq h(P,\eta) \leq (1+\varepsilon')\fint_{C_{Mr_i}(Q_i,\tau_i)} hd\sigma.$$ Throughout we will assume that $i$ is large enough such that $Mr_i < r(\varepsilon')$. 

Define $\tilde{G}_i \colonequals \{(P,\eta) \in \partial \Omega_i\mid (r_iP + Q_i, r_i^2\eta + \tau_i) \in G_i\}$, the image of $G_i$ under the blowup. Then \begin{equation}\label{pointsingi}
h_i(P,\eta) = \frac{h(r_iP + Q_i, r_i^2\eta + \tau_i)}{\fint_{\Delta_{r_i}(Q_i,\tau_i)} hd\sigma} \simeq_{\varepsilon'} \frac{\fint_{\Delta_{Mr_i}(Q_i,\tau_i)} hd\sigma}{\fint_{\Delta_{r_i}(Q_i,\tau_i)} hd\sigma},\; \forall (P,\eta) \in \tilde{G}_i
\end{equation} where, as in \cite{kenigtoro}, we write $a\simeq_{\varepsilon'} b$ if $\frac{1}{1+\varepsilon'} \leq \frac{a}{b} \leq (1+\varepsilon')$. 

Observe \begin{equation}\label{hrivshMri}\fint_{\Delta_{r_i}(Q_i,\tau_i)} hd\sigma \geq \frac{1}{\sigma(\Delta_{r_i}(Q_i,\tau_i))} \int_{G_i\cap \Delta_{r_i}(Q_i,\tau_i)}hd\sigma \geq \frac{\sigma(G_i\cap \Delta_{r_i}(Q_i,\tau_i))}{(1+\varepsilon')\sigma(\Delta_{r_i}(Q_i,\tau_i))} \fint_{\Delta_{Mr_i}(Q_i,\tau_i)} hd\sigma.\end{equation} Combining this with equation \eqref{pointsingi} we can conclude $$h_i(P,\eta) \leq (1+\varepsilon')^2 \frac{\sigma(\Delta_{r_i}(Q_i,\tau_i))}{\sigma(G_i\cap \Delta_{r_i}(Q_i,\tau_i))}, \; \forall (P,\eta) \in \tilde{G}_i.$$ Ahlfors regularity implies\begin{equation}\label{sigmaGvssigmari} \begin{aligned} \sigma(G_i\cap \Delta_{r_i}(Q_i,\tau_i)) &= \sigma(\Delta_{r_i}(Q_i,\tau_i)) - \sigma(\Delta_{r_i}(Q_i,\tau_i)\backslash G_i)\\ &\geq\sigma(\Delta_{r_i}(Q_i,\tau_i)) - \sigma(\Delta_{Mr_i}(Q_i,\tau_i)\backslash G_i)\\ &\geq  \sigma(\Delta_{r_i}(Q_i,\tau_i)) - \varepsilon'\sigma(\Delta_{Mr_i}(Q_i,\tau_i))\\ &\geq \sigma(\Delta_{r_i}(Q_i,\tau_i))(1-CM^{n+1}\varepsilon').\end{aligned}\end{equation}

Putting everything together $h_i(P,\eta) \leq (1+\varepsilon')^2 (1-CM^{n+1}\varepsilon')^{-1},\; \forall (P,\eta) \in \tilde{G}_i.$ Hence,
\begin{equation}\label{estimategoodpart}
\int_{\tilde{G}_i}h_id\hat{\omega}_i^{(X,t)} \leq \frac{(1+\varepsilon')^2}{(1-CM^{n+1}\varepsilon')} \hat{\omega}_i^{(X,t)}(\tilde{G}_i) \leq \frac{(1+\varepsilon')^2}{ (1-CM^{n+1}\varepsilon')},
\end{equation}
as $\hat{\omega}_i^{(X,t)}$ is a probability measure. 

Define $F_i \colonequals (C_{Mr_i}(Q_i,\tau_i)\cap \partial \Omega) \backslash G_i$ and $\tilde{F_i}$ analogously to $\tilde{G}_i$. Let $A_i \in \Omega_i$ be the backwards non-tangential point at $(0,0)$ and scale $30M$. We want to connect $A_i$ with $(X,t)$ by a Harnack chain in $\Omega_i$. Thus we need to show that that the square root of the difference in the $t$-coordinates of $A_i$ and $(X,t)$ is greater than $\frac{d(A_i, (X,t))}{100}$. This follows after observing that the $t$-coordinate of $A_i$ is $\leq -9M^2$.  The Harnack inequality then tells us that there is a $C = C(n, M, \delta_\infty) > 0$ such that $d\hat{\omega}^{(X,t)} \leq C d\hat{\omega}^{A_i}$ on $C_{2M}(0,0)\cap \partial \Omega_i$. Furthermore, $A_i \in T^-_{100, M}(0,0)$ which implies, by Theorem 1 in \cite{hlncaloricmeasure}, that there is a $p > 1$ such that $\hat{k}^{A_i} \colonequals \frac{d\hat{\omega}^{A_i}}{d\sigma}$ satisfies a reverse H\"older inequality with exponent $p$ and constant $C$ (as the $\Omega_i$ are uniformly parabolic regular and $\delta$-Reifenberg flat, the arguments in \cite{hlncaloricmeasure} ensure that $p, C$ can be taken independent of $i$). Let $q$ be the dual exponent; then, by H\"older's inequality, \begin{equation}\label{splithebadpart}
\int_{\tilde{F}_i} h_i d\hat{\omega}_i^{(X,t)} \leq C\int_{\tilde{F}_i} h_i d\hat{\omega}_i^{A_i} \leq C\left(\int_{\tilde{F}_i} h_i^qd\sigma_i\right)^{1/q}\left(\int_{\tilde{F}_i} (\hat{k}^{A_i})^p d\sigma_i\right)^{1/p}.
\end{equation}

To bound the first term in the product note, \begin{equation}\label{hiinfi}\int_{\tilde{F}_i} h_i^qd\sigma_i = r_i^{-n-1} \frac{\int_{F_i} h^qd\sigma}{\left(\fint_{\Delta_{r_i}(Q_i,\tau_i)} hd\sigma\right)^q}.\end{equation} Per Lemma \ref{hinvmosatisfiesreverseholder}, $h^2 \in A_2(d\sigma)$  (as $\log(h^2) \in \mathrm{VMO}(\partial \Omega)$). Apply H\"older's inequality and then the reverse H\"older inequality with exponent $2$ to obtain \begin{equation}\label{reverseholderexponent2}\begin{aligned}\int_{F_i} h^qd\sigma &\leq \sigma(\Delta_{Mr_i}(Q_i,\tau_i))^{1/2}\sigma(F_i)^{1/2} \left(\fint_{\Delta_{Mr_i}(Q_i,\tau_i)} h^{2q}d\sigma\right)^{1/2}\\
 &\leq C\left(\frac{\sigma(F_i)}{\Delta_{Mr_i}(Q_i,\tau_i)}\right)^{1/2} \int_{\Delta_{Mr_i}(Q_i,\tau_i)} h^{q}d\sigma\\
 & \leq C\sqrt{\varepsilon'} \sigma(\Delta_{Mr_i}(Q_i,\tau_i)) \fint_{\Delta_{Mr_i}(Q_i,\tau_i)} h^{q}d\sigma,\end{aligned}\end{equation} where that last inequality comes from the fact that $F_i$ is small in $\Delta_{Mr_i}(Q_i,\tau_i)$. Invoking Lemma \ref{hinvmosatisfiesreverseholder} again, $h$ satisfies a reverse H\"older inequality with exponent $q$. This fact, combined with equations \eqref{hiinfi}, \eqref{reverseholderexponent2}, implies \begin{equation}\label{boundonbadparth}
 \begin{aligned}
\int_{\tilde{F}_i} h_i^qd\sigma_i \leq& C\sqrt{\varepsilon'} \frac{\sigma(\Delta_{Mr_i}(Q_i,\tau_i))}{r_i^{n+1}} \left(\frac{\fint_{\Delta_{Mr_i}(Q_i,\tau_i)} hd\sigma}{\fint_{\Delta_{r_i}(Q_i,\tau_i)} hd\sigma}\right)^q\\
\stackrel{\mathrm{eq}\; \eqref{hrivshMri}}{\leq}& C\sqrt{\varepsilon'}M^{n+1} \left(\frac{(1+\varepsilon')\sigma(\Delta_{r_i}(Q_i,\tau_i))}{\sigma(G_i\cap \Delta_{r_i}(Q_i,\tau_i))}\right)^q\\
\stackrel{\mathrm{eq}\; \eqref{sigmaGvssigmari}}{\leq}& C\sqrt{\varepsilon'}M^{n+1}(1-CM^{n+1}\varepsilon')^{-q}.
\end{aligned}
\end{equation}

To bound the second term of the product in equation \eqref{splithebadpart} we recall that $\hat{k}^{A_i}$ satisfies a reverse H\"older inequality with exponent $p$ at scale $M$,
\begin{equation}\label{boundonthebadpartk}
\begin{aligned}
\left(\int_{\tilde{F}_i} (\hat{k}^{A_i})^p d\sigma_i\right)^{1/p} \leq& \left(\int_{C_{M}(0,0)\cap \partial \Omega_i} (\hat{k}^{A_i})^pd\sigma_i \right)^{1/p}\\
\leq& C\sigma_i(C_M(0,0)\cap \partial \Omega_i)^{1/p} \left(\fint_{C_M(0,0)\cap \partial \Omega_i}\hat{k}^{A_i}d\sigma_i\right)\\
\leq& C\sigma_i(C_M(0,0)\cap \partial \Omega_i)^{1/p-1} \hat{\omega}^{A_i}(C_M(0,0))\\
\leq& C\left(\frac{\sigma(\Delta_{Mr_i}(Q_i,\tau_i))}{r_i^{n+1}}\right)^{1/p-1} \leq CM^{-(n+1)/q}.
\end{aligned}
\end{equation} 

Together, equations \eqref{splithebadpart}, \eqref{boundonbadparth} and \eqref{boundonthebadpartk}, say \begin{equation}\label{finalboundonbadpart}
\int_{\tilde{F}_i} h_i d\omega_i^{(X,t)} \leq CM^{-(n+1)/q}\left(C\sqrt{\varepsilon'}M^{n+1}(1-CM^{n+1}\varepsilon')^{-q}\right)^{1/q} =C\frac{(\varepsilon')^{1/(2q)}}{1-CM^{n+1}\varepsilon'}.
\end{equation}

Having estimated the integral over $\tilde{G}_i$ in equation \eqref{estimategoodpart} and the integral over $\tilde{F}_i$ in equation \eqref{finalboundonbadpart} we can invoke equation \eqref{representationofgradui} to conclude \begin{equation}\label{boundonnablaufinal}
\begin{aligned}
|\nabla u_\infty(X,t)| \leq& \limsup_i \int_{\partial \Omega_i \cap C_{M}(0,0)} h_i(Q,\tau) d\hat{\omega}^{(X,t)}_i(Q,\tau) + C\frac{\|(X,t)\|^{3/4}}{M^{1/2}} \\ \leq& \frac{(1+\varepsilon')^2}{ (1-CM^{n+1}\varepsilon')} + C\frac{(\varepsilon')^{1/(2q)}}{1-CM^{n+1}\varepsilon'} + \varepsilon/2\\
\leq& 1+ \varepsilon.
\end{aligned}
\end{equation}
The last inequality above is justified by picking $\varepsilon' > 0$ small so that $\frac{(1+\varepsilon')^2}{ (1-CM^{n+1}\varepsilon')} + C\frac{(\varepsilon')^{1/(2q)}}{1-CM^{n+1}\varepsilon'}< 1+ \varepsilon/2$. 
\end{proof}

\begin{proof}[Proof: Finite Pole Case]
Let $(X,t) \in \Omega_\infty$ and define $d((X,t), \partial \Omega_\infty) \equalscolon \delta_\infty$. There exists an $i_0 > 0$ such that if $i > i_0$ then $C_{3\delta_\infty/4}(X,t) \subset \Omega_i$ but $C_{3\delta_\infty/2}(X,t)\cap \partial \Omega_i \neq \emptyset$. By standard parabolic regularity theory, $\nabla u^{(X_0,t_0)}_i \rightarrow \nabla u_\infty$ uniformly on $C_{\delta_\infty/2}(X,t)$. Let $(\hat{X}_i,\hat{t}_i) \in \partial \Omega_i$ be a point on $\partial \Omega_i$ closest to $(X,t)$. Note that $u^{(X_0,t_0)}_i$ is the Green function of $\Omega_i$ with a pole at $\left(\frac{X_0-Q_i}{r_i}, \frac{t_0-\tau_i}{r_i^2}\right)$. Let $M \geq \max\{4R, 100 \delta_\infty\}$, be arbitrarily large to be chosen later. We want to check that $\left(\frac{X_0-Q_i}{r_i}, \frac{t_0-\tau_i}{r_i^2}\right) \in T_{2A,M}^+(\hat{X}_i, \hat{t}_i)$ for $i$ large enough (recall our assumption that $(Q_i,\tau_i) \in \Delta_R(Q,\tau)$ where $(X_0,t_0) \in T^+_{A,4R}(Q,\tau)$). 

Observe $$\frac{t_0-\tau_i}{r_i^2}-\hat{t}_i > 4M^2 \Leftrightarrow t_0 - \tau_i - r_i^2\hat{t}_i > 4r_i^2M^2$$ But $t_0 -\tau_i > 0$ and $|\hat{t_i}| <C(|t| + \delta_\infty) < C$. So for $i$ large enough $\frac{t_0-\tau_i}{r_i^2}-\hat{t}_i  > 4 M^2$. Similarly, $$\left|\frac{X_0-Q_i}{r_i} - \hat{X}_i\right|^2 \leq 2A\left|\frac{t_0-\tau_i}{r_i^2}-\hat{t}_i\right| \Leftrightarrow |X_0-Q_i - r_i\hat{X}_i|^2 \leq 2A|t_0-\tau_i - r_i^2\hat{t_i}|.$$ As $|X_0-Q_i|^2 \leq \frac{3}{2}A|t_0-\tau_i|$ we may conclude, for large $i$, that $\left(\frac{X_0-Q_i}{r_i}, \frac{t_0-\tau_i}{r_i^2}\right) \in T_{2A,M}^+(\hat{X}_i, \hat{t}_i)$. Invoking Lemma \ref{gradientrepresentationfinitepole},  \begin{equation}\label{representationofgraduifinitepole}
|\nabla u_i^{(X_0,t_0)}(X,t)| \leq \int_{C_M(\hat{X}_i,\hat{t}_i)\cap \partial \Omega_i} k_i^{(X_0,t_0)}d\hat{\omega}_i^{(X,t)} + C \left(\frac{\delta_\infty}{M}\right)^{3/4}\frac{\omega_i^{(X_0,t_0)}(C_{M}(Q,\tau)\cap \partial \Omega_i)}{M^{n+1}}.
\end{equation}

For any $\varepsilon > 0$, pick an $M \equiv M(\varepsilon) > 0$ large such that $$C \left(\frac{\delta_\infty}{M}\right)^{3/4}\frac{\omega_i^{(X_0,t_0)}(C_{M}(Q,\tau)\cap \partial \Omega_i)}{M^{n+1}} \leq \varepsilon/2.$$ For large enough $i$, \begin{equation}\label{wheredoesplive} \begin{aligned}(P,\eta) &\in C_{M}(\hat{X}_i,\hat{t}_i)\cap \partial \Omega_i \Rightarrow \\ (r_iP + Q_i, r_i^2 \eta + \tau_i) &\in \Delta_{r_iM}(r_i\hat{X}_i + Q_i, r_i^2\hat{t}_i + \tau_i) \subset \Delta_{2r_iM}(Q_i,\tau_i) \subset \Delta_{2R}(Q,\tau). \end{aligned}\end{equation} Therefore, we can apply Lemmas \ref{hinvmosatisfiesreverseholderfinitepole}, \ref{vmomeasurecomparisonfinitepole} and \ref{vmodecompositionfinitepole} to $k^{(X_0,t_0)}$ on $\Delta_{2r_iM}(Q_i,\tau_i)$ for large enough $i$. Let $\varepsilon'  \equiv \varepsilon'(M,\varepsilon) > 0$ be small and chosen later. There exists an $i_0 \in \mathbb N$ such that $i \geq i_0$ implies that equations \eqref{representationofgraduifinitepole} and \eqref{wheredoesplive} hold and that $2r_iM \leq r(\varepsilon')$ (where $r(\varepsilon')$ is given by Lemma \ref{vmodecompositionfinitepole}). We may now proceed as in the infinite pole case to get the desired conclusion. 
\end{proof}

To invoke Theorem \ref{poissonkernel1impliesflat} we must also show that $h_\infty \geq 1$ almost everywhere. Here we follow closely the method of \cite{kenigtoro}.

\begin{lem}\label{hatinfinityisgreaterthan1}
Let $\Omega_\infty, u_\infty, \omega_\infty$ be as above. Then $h_\infty = \frac{d\omega_\infty}{d\sigma_\infty}$ exists and $$h_\infty(Q,\tau) \geq 1$$ for $\sigma_\infty$-a.e. $(Q,\tau)\in \partial \Omega_\infty$. 
\end{lem}

\begin{proof}
In Lemma \ref{limitsexistforblowups} we prove that $\Omega_\infty$ is a $\delta$-Reifenberg flat parabolic regular domain. By Theorem 1 in \cite{hlncaloricmeasure} (see Proposition \ref{hlnformeasureatinfinity} for remarks when the pole is at infinity) $\omega_\infty \in A_\infty(d\sigma_\infty)$; thus $h_\infty$ exists. 

By the divergence theorem, the limiting process described in Lemma \ref{limitsexistforblowups}, and $\left\langle e, \vec{n}_\infty\right\rangle = 1 - \frac{1}{2}|\vec{n}_\infty - e|^2$ we have, for any positive $\varphi \in C_c^\infty(\R^{n+1})$ and any $e \in \mathbb S^{n-1}$, \begin{equation}\label{messingaroundwithnormals}
\begin{aligned}
\int_{\partial \Omega_i}\varphi d\sigma_i &\geq \int_{\partial \Omega_i}\varphi \left\langle e, \vec{n}_i \right\rangle d\sigma_i = -\int_{\Omega_i} \mathrm{div}(\varphi e)dXdt\\
&\stackrel{i\rightarrow \infty}{\rightarrow} - \int_{\Omega_\infty} \mathrm{div}(\varphi e) dX dt = \int_{\partial \Omega_\infty} \varphi \left\langle e, \vec{n}_\infty \right\rangle d\sigma_\infty\\
&\geq \int_{\partial \Omega_\infty} \varphi d\sigma_\infty - \frac{1}{2}\int_{\partial \Omega_\infty}\varphi |\vec{n}_\infty - e|^2d\sigma_\infty.
\end{aligned}
\end{equation}

We claim, for any positive $\varphi \in C_c^\infty(\R^{n+1})$,  \begin{equation}\label{basicallyhbiggerthan1}
\int_{\partial \Omega_\infty} \varphi h_\infty d\sigma_\infty \geq \limsup_{i\rightarrow \infty} \int_{\partial \Omega_i} \varphi d\sigma_i.
\end{equation}

If our claim is true then $\int_{\partial \Omega_\infty} \varphi h_\infty d\sigma_\infty \geq \int_{\partial \Omega_\infty} \varphi d\sigma_\infty - \frac{1}{2}\int_{\partial \Omega_\infty}\varphi |\vec{n}_\infty - e|^2d\sigma_\infty.$ For $(Q_0,\tau_0) \in \partial \Omega_\infty$, let $e =\vec{n}_\infty(Q_0,\tau_0)$ and $\varphi \rightarrow \chi_{C_{r}(Q_0,\tau_0)}$ to obtain \begin{equation*}\begin{aligned}\int_{C_r(Q_0,\tau_0)} h_\infty d\sigma_\infty &\geq \sigma_\infty(C_r(Q_0,\tau_0)) - \frac{1}{2} \int_{C_r(Q_0,\tau_0)} |\vec{n}_\infty(P,\eta) - \vec{n}_\infty(Q_0,\tau_0)|^2 d\sigma_\infty(P,\eta) \Rightarrow\\ \fint_{C_{r}(Q_0,\tau_0)} h_\infty d\sigma_\infty &\geq 1 - \frac{1}{2} \fint_{C_r(Q_0,\tau_0)} |\vec{n}_\infty(P,\eta) - \vec{n}_\infty(Q_0,\tau_0)|^2 d\sigma_\infty(P,\eta).\end{aligned}\end{equation*} If $(Q_0,\tau_0)$ is a point of density for $\vec{n}_\infty, h_\infty$ letting $r\downarrow 0$ gives $h_\infty(Q_0,\tau_0) \geq 1$ for $\sigma_\infty$-a.e. $(Q_0,\tau_0)$. 

Thus we need only to establish equation \eqref{basicallyhbiggerthan1}. Pick any positive $\varphi \in C_c^\infty(\R^{n+1}), \varepsilon > 0$ and let $M \geq 0$ be large enough such that $\varphi \in C_c^\infty(C_M(0,0))$ and $\|\varphi\|_{L^\infty} \leq M$. Let $\varepsilon' > 0$ to be choosen later (depending on $M, \varepsilon$). We will prove equation \eqref{basicallyhbiggerthan1} for the infinite pole blowup. However, the arguments in the finite pole setting are completely unchanged; for large enough $i$ we have $C_{Mr_i}(Q_i,\tau_i) \subset C_{2R}(Q,\tau)$ and hence can apply Lemmas \ref{hinvmosatisfiesreverseholderfinitepole}, \ref{vmomeasurecomparisonfinitepole} and \ref{vmodecompositionfinitepole} to $k^{(X_0,t_0)}$ on $\Delta_{Mr_i}(Q_i,\tau_i)$.

$\log(h) \in \mathrm{VMO}(\partial \Omega)$, so Lemma \ref{vmodecomposition} gives an $R = R(\varepsilon') > 0$ such that for $r_iM \leq R$ we can split $C_{Mr_i}(Q_i, \tau_i)\cap \partial \Omega$ into $G_i, F_i$ with $\sigma(\Delta_{Mr_i}(Q_i, \tau_i)) \leq (1+\varepsilon')\sigma(G_i)$ and $\fint_{\Delta_{Mr_i}(Q_i,\tau_i)} hd\sigma \sim_{\varepsilon'} h(P,\eta)$ for $(P,\eta) \in G_i$. Define $\tilde{G}_i$ and $\tilde{F}_i$ as in the proof of Proposition \ref{maingradientbound}. 

For $(P,\eta) \in \tilde{G}_i$ we have $h_i(P,\eta) \sim_{\varepsilon'} \frac{\fint_{\Delta_{Mr_i}(Q_i,\tau_i)} h d\sigma}{\fint_{\Delta_{r_i}(Q_i,\tau_i)} hd\sigma}$, and consequently \begin{equation}\label{comparableong}\int_{\tilde{G}_i} h_i\varphi d\sigma_i \sim_{\varepsilon'} \frac{\fint_{\Delta_{Mr_i}(Q_i,\tau_i)} h d\sigma}{\fint_{\Delta_{r_i}(Q_i,\tau_i)} hd\sigma} \int_{\tilde{G}_i} \varphi d\sigma_i.\end{equation} We can then estimate \begin{equation}\label{gisalmostall}\int_{\tilde{G}_i}\varphi d\sigma_i = \int_{\partial \Omega_i} \varphi d\sigma_i - \int_{\tilde{F}_i} \varphi d\sigma_i \geq \int_{\partial \Omega_i} \varphi d\sigma_i - C\varepsilon' M^{n+2},\end{equation} using the Ahflors regularity of $\partial \Omega_i$ and the definition of $\sigma_i, \tilde{F}_i$. 

Therefore, \begin{equation}\label{allbutthefraction}\begin{aligned} \int_{\partial \Omega_\infty} \varphi h_\infty d\sigma_\infty &= \lim_{i\rightarrow \infty} \int_{\partial \Omega_i}\varphi h_i d\sigma_i \geq \limsup_{i \rightarrow \infty} \int_{\tilde{G}_i} h_i\varphi d\sigma_i\\
& \stackrel{\eqref{comparableong}}{\geq} \limsup_{i\rightarrow \infty} (1+\varepsilon')^{-1} \frac{\fint_{\Delta_{Mr_i}(Q_i,\tau_i)} h d\sigma}{\fint_{\Delta_{r_i}(Q_i,\tau_i)} hd\sigma} \int_{\tilde{G}_i} \varphi d\sigma_i\\ &\stackrel{\eqref{gisalmostall}}{\geq} \limsup_{i\rightarrow \infty} (1+\varepsilon')^{-1} \frac{\fint_{\Delta_{Mr_i}(Q_i,\tau_i)} h d\sigma}{\fint_{\Delta_{r_i}(Q_i,\tau_i)} hd\sigma} \left(\int_{\partial \Omega_i}\varphi d\sigma_i - CM^{n+2}\varepsilon'\right).
\end{aligned}
\end{equation}

To estimate $\frac{\fint_{\Delta_{Mr_i}(Q_i,\tau_i)} h d\sigma}{\fint_{\Delta_{r_i}(Q_i,\tau_i)} hd\sigma}$ from below, we write \begin{equation}\label{dealingwiththefraction}
\begin{aligned}
\fint_{\Delta_{r_i}(Q_i,\tau_i)} hd\sigma =& \frac{1}{\sigma(\Delta_{r_i}(Q_i,\tau_i))}\left(\int_{\Delta_{r_i}(Q_i,\tau_i) \cap G_i} h d\sigma + \int_{\Delta_{r_i}(Q_i,\tau_i) \cap F_i} h d\sigma \right)\\
\leq& (1+\varepsilon') \frac{\sigma(\Delta_{r_i}(Q_i,\tau_i)\cap G_i)}{\sigma(\Delta_{r_i}(Q_i,\tau_i))} \fint_{\Delta_{Mr_i}(Q_i,\tau_i)} hd\sigma + \frac{\omega(F_i\cap \Delta_{r_i}(Q_i,\tau_i))}{\sigma(\Delta_{r_i}(Q_i,\tau_i))}\\
\leq& (1+\varepsilon') \fint_{\Delta_{Mr_i}(Q_i,\tau_i)} hd\sigma + \left(\frac{\sigma(F_i)}{\sigma(\Delta_{r_i}(Q_i,\tau_i))}\right)^{1/2} \left(\fint_{\Delta_{r_i}(Q_i,\tau_i)} h^2d\sigma\right)^{1/2}\\
\leq& (1+\varepsilon') \fint_{\Delta_{Mr_i}(Q_i,\tau_i)} hd\sigma + (C\varepsilon'M^{n+1})^{1/2}\fint_{\Delta_{r_i}(Q_i,\tau_i)} hd\sigma.
\end{aligned}
\end{equation}
To justify the penultimate inequality above note, for any set $E\subset \Delta_{r_i}(Q_i,\tau_i)$, H\"older's inequality gives $$\omega(E) \leq \sigma(E)^{1/2} \left(\int_{\Delta_{r_i}(Q_i,\tau_i)} h^2d\sigma\right)^{1/2}.$$ The last inequality in equation \eqref{dealingwiththefraction} follows from the fact that $F_i$ has small volume and $h$ satisfies a reverse H\"older inequality with exponent 2 (Lemma \ref{hinvmosatisfiesreverseholder}). 

After some algebraic manipulation, equation \eqref{dealingwiththefraction} implies $$\frac{\fint_{\Delta_{Mr_i}(Q_i,\tau_i)} h d\sigma}{\fint_{\Delta_{r_i}(Q_i,\tau_i)} hd\sigma} \geq (1+\varepsilon')^{-1}(1-(C\varepsilon'M^{n+1})^{1/2}).$$ Hence, in light of equation \eqref{allbutthefraction}, and choosing $\varepsilon'$ wisely, $$ \int_{\partial \Omega_\infty} \varphi h_\infty d\sigma_\infty  \geq (1-\varepsilon) \limsup_{i\rightarrow \infty} \int_{\partial \Omega_i}\varphi d\sigma_i -\varepsilon.$$ Let $\varepsilon \downarrow 0$  to prove equation \eqref{basicallyhbiggerthan1}.
\end{proof}

We have shown that our blowup satisfies the hypothesis of  Proposition \ref{poissonkernel1impliesflat} (the classification of flat blowups).

\begin{prop}\label{vanishingreifenbergflat}
After a rotation (which may depend on on the sequences $(Q_i,\tau_i), r_i$), $\Omega_\infty = \{x_n > 0\}, u_\infty = x_n^+$ and $d\omega_\infty = d\sigma_\infty = \mathcal H^{n-1}|_{\{x_n =0\}\cap \{s = t\}}dt$. In particular, $\Omega$ is vanishing Reifenberg flat.
\end{prop}

In the above we have shown that any pseudo-blowup (i.e. a blowup described in Definition \ref{parabolicpseudoblowup}) of $\Omega$ is a half space. However, we will need a slightly stronger result, namely that under this blowup $\sigma_i \rightharpoonup \sigma_\infty$. In the elliptic setting this is Theorem 4.4 in \cite{kenigtoro}.

\begin{prop}\label{weakconvergenceinpseudoblowups}
For any blowup described in Definition \ref{parabolicpseudoblowup}, $\sigma_i \rightharpoonup \sigma_\infty$. In particular, for any compact set $K$ (in the finite pole case $K$ is as in Definition \ref{parabolicpseudoblowup}), we have \begin{equation}\label{ellipticvanishingchordarc}
\lim_{r\downarrow 0} \sup_{(Q,\tau) \in K\cap \partial \Omega} \frac{\sigma(C_r(Q,\tau)\cap \partial \Omega)}{r^{n+1}} = 1.
\end{equation}
\end{prop}

\begin{proof}
Observe that $\sigma_i\rightharpoonup \sigma_\infty$ implies equation \eqref{ellipticvanishingchordarc}: let $(Q_i,\tau_i) \in K\cap \partial \Omega$ and $r_i\downarrow 0$ be such that $$\lim_{i\rightarrow \infty} \frac{\sigma(C_{r_i}(Q_i,\tau_i)\cap \partial \Omega)}{r_i^{n+1}} = \lim_{r\downarrow 0} \sup_{(Q,\tau) \in K\cap \partial \Omega} \frac{\sigma(C_r(Q,\tau)\cap \partial \Omega)}{r^{n+1}}.$$ 

Blowing up along this sequence (possibly passing to subsequences) we get $\Omega_i \rightarrow \Omega_\infty$ and, by Proposition \ref{vanishingreifenbergflat}, we have that $\Omega_\infty = \{x_n > 0\}$ (after a rotation). Since $\sigma_\infty(\partial C_1(0,0)) = 0$, if $\sigma_i \rightharpoonup \sigma_\infty$ we have $\lim_i \sigma_i(C_1(0,0)) = \sigma_\infty(C_1(0,0)) = 1$ (recall our normalization from the introduction). By definition, $\sigma_i(C_1(0,0)) = \frac{1}{r_i^{n+1}} \sigma(C_{r_i}(Q_i,\tau_i))$ which implies equation \eqref{ellipticvanishingchordarc}. 

Proposition \ref{vanishingreifenbergflat} proved that $\sigma_\infty =\omega_\infty$ so to show $\sigma_i \rightharpoonup \sigma_\infty$ it suffices to prove, for any positive $\varphi \in C_{c}^\infty(\R^{n+1})$, $$\lim_{i\rightarrow \infty} \int_{\partial \Omega_i} \varphi d\sigma_i = \int_{\partial \Omega} \varphi d\omega_\infty.$$ Equation \eqref{basicallyhbiggerthan1} says that the right hand side is larger than the left. Hence, it is enough to show $$\liminf_{i\rightarrow \infty} \int_{\partial \Omega_i} \varphi d\sigma_i \geq \int_{\partial \Omega_\infty} \varphi d\omega_\infty.$$ We will work in the infinite pole setting. The finite pole setting follows similarly (i.e. for large $i$ we may assume $C_{Mr_i}(Q_i,\tau_i) \subset C_R(Q,\tau)$, where $(X_0,t_0) \in T_{A,2R}^+(Q,\tau)$, and then adapt the arguments below). 

Keeping the notation from the proof of Lemma \ref{hatinfinityisgreaterthan1}, it is true that, for large $i$, \begin{equation}\label{startingtoshowitsbigger} \begin{aligned} \int_{\partial \Omega_i} \varphi d\sigma_i \geq \int_{\tilde{G_i}} \varphi d\sigma_i &\stackrel{\eqref{comparableong}}{\geq} (1+\varepsilon')^{-1} \frac{\fint_{\Delta_{r_i}(Q_i,\tau_i)} hd\sigma}{\fint_{\Delta_{Mr_i}(Q_i,\tau_i)} h d\sigma} \int_{\tilde{G}_i}h_i\varphi d\sigma_i\\ &= (1+\varepsilon')^{-1} \frac{\fint_{\Delta_{r_i}(Q_i,\tau_i)} hd\sigma}{\fint_{\Delta_{Mr_i}(Q_i,\tau_i)} h d\sigma} \left(\int_{\partial \Omega_i} \varphi d\omega_i - \int_{\tilde{F}_i} h_i\varphi d\sigma_i\right)\\ \stackrel{\eqref{hrivshMri}+\eqref{sigmaGvssigmari}}{\geq}& (1+\varepsilon')^{-2}(1-CM^{n+1}\varepsilon')\left(\int_{\partial \Omega_i} \varphi d\omega_i - \int_{\tilde{F}_i} h_i\varphi d\sigma_i\right).\end{aligned}\end{equation}

We need to bound from above the integral of $h_i\varphi$ on $\tilde{F}_i$, \begin{equation}\label{estimateforthefterm}
\begin{aligned}
\int_{\tilde{F}_i} h_i\varphi d\sigma_i\leq& M \omega_i(\tilde{F}_i) = M  \frac{\sigma(\Delta_{r_i}(Q_i,\tau_i))}{r_i^{n+1}} \frac{\omega(F_i)}{\omega(\Delta_{r_i}(Q_i,\tau_i))}\\
\stackrel{\mathrm{H\ddot{o}lder's\; Inequality}}{\leq}& M  \frac{\sigma(\Delta_{r_i}(Q_i,\tau_i))\sigma(\Delta_{Mr_i}(Q_i,\tau_i))^{1/2}}{\omega(\Delta_{r_i}(Q_i,\tau_i))r_i^{n+1}} \sigma(F_i)^{1/2}\left(\fint_{\Delta_{Mr_i}(Q_i,\tau_i)} h^2d\sigma \right)^{1/2}\\
\stackrel{h \in A_2(d\sigma)}{\leq}& CM \frac{\sigma(\Delta_{r_i}(Q_i,\tau_i))\sigma(\Delta_{Mr_i}(Q_i,\tau_i))^{1/2}}{\omega(\Delta_{r_i}(Q_i,\tau_i))r_i^{n+1}} \sigma(F_i)^{1/2} \fint_{\Delta_{Mr_i}(Q_i,\tau_i)} hd\sigma\\
\stackrel{\sigma(F_i)\leq C\varepsilon (r_iM)^{n+1}}{\leq}& C(\varepsilon')^{1/2}M^{n+2} \frac{\fint_{\Delta_{Mr_i}(Q_i,\tau_i)} h d\sigma}{\fint_{\Delta_{r_i}(Q_i,\tau_i)} hd\sigma}\\
\stackrel{\eqref{hrivshMri}+\eqref{sigmaGvssigmari}}{\leq}& C(\varepsilon')^{1/2}M^{n+2}(1+\varepsilon')(1- CM^{n+1}\varepsilon')^{-1}.
\end{aligned}
\end{equation}

From equations \eqref{startingtoshowitsbigger} and \eqref{estimateforthefterm} we can conclude, for large $i$, that $$\int_{\partial \Omega_i} \varphi d\sigma_i \geq (1+\varepsilon')^{-2}(1- CM^{n+1}\varepsilon')\int_{\partial \Omega_i} \varphi d\omega_i  - CM^{n+2}(\varepsilon')^{1/2}(1+\varepsilon')^{-1}.$$ $\omega_i \rightharpoonup \omega_\infty$, consequently, let $i\rightarrow \infty$ and then $\varepsilon'\downarrow 0$ to obtain the desired result. 
\end{proof}

\section{The Vanishing Carleson Condition}\label{sec: vanishingcarlesoncondition}

In this section we prove the following geometric measure theory proposition to finish our proof of Theorem \ref{loghvmoisvanishingchordarc}. 

\begin{prop}\label{chordarctovanishingchordarc}
Let $\Omega$ be a parabolic uniformly rectifiable domain which is also vanishing Reifenberg flat. Furthermore, assume that $$\lim_{r\downarrow 0} \sup_{(Q,\tau) \in K\cap \partial \Omega} \frac{\sigma(\Delta_r(Q,\tau))}{r^{n+1}} = 1$$ holds for all compact sets $K$. Then $\Omega$ is actually a vanishing chord arc domain.
\end{prop}

Propositions \ref{vanishingreifenbergflat} and \ref{weakconvergenceinpseudoblowups} show that the assumptions of Proposition \ref{chordarctovanishingchordarc} are satisfied and therefore Proposition \ref{chordarctovanishingchordarc} implies Theorem \ref{loghvmoisvanishingchordarc} (restricting to $K \subset \subset \{t < t_0\}$ in Proposition \ref{chordarctovanishingchordarc} implies Theorem \ref{loghvmoisvanishingchordarc} in the finite pole setting).  

In the elliptic case, Proposition \ref{chordarctovanishingchordarc} is also true but the proof is substantially simpler (see the proof beginning on page 366 in \cite{kenigtoro}). This is due to the fact (mentioned in the introduction) that the growth of the ratio $ \frac{\sigma(\Delta_r(Q,\tau))}{r^{n+1}}$ controls the oscillation of $\hat{n}$ (see, e.g. Theorem 2.1 in \cite{kenigtoroduke}). However, as we also alluded to before, the behaviour of $\frac{\sigma(\Delta_r(Q,\tau))}{r^{n+1}}$ does not give information about the Carleson measure $\nu$; see the example at the end of \cite{hlnbigpieces} in which $\sigma(\Delta_r(Q,\tau)) \equiv r^{n+1}$ but $\nu$ is not a Carleson measure. So we cannot hope that the methods in \cite{kenigtoro} can be adapted to prove Proposition \ref{chordarctovanishingchordarc} above. We also mention that the previous example in \cite{hlnbigpieces} shows that Proposition \ref{chordarctovanishingchordarc} is not true without the {\it a priori} assumption that $\nu$ is a Carleson measure (i.e. that the domain is parabolic uniformly rectifiable). 

When $\Omega  = \{(x,x_n, t)\mid x_n \geq \psi(x,t)\}$ and $\psi \in \Lip(1, 1/2)$ with $D^{1/2}_t \psi \in \mathrm{BMO}(\mathbb R^{n+1})$ (see the introduction of \cite{hlncaloricmeasure} for precise definitions), Nystr\"om, in \cite{nystromgraphdomains}, showed that vanishing Reifenberg flatness implies the vanishing Carleson condition.  To summarize his argument, for any $r_i \downarrow 0, (Q_i,\tau_i) \in \partial \Omega \cap K$ we can write \begin{equation}\label{nuoftheblowup2}r_i^{-n-1}\nu(C_{r_i}(Q_i,\tau_i) \times [0,r_i]) \lesssim \int_0^1 \int_{\{(x, t)\mid |x| \leq 1, |t| \leq 1\}} \frac{\gamma_i((y, \psi_i(y,s),s),r)}{r} dyds dr,\end{equation}  where $\psi_i(x, t) \colonequals \frac{\psi(r_ix + q_i, r_i^2t + \tau_i)}{r_i}$ and $\gamma_i$ is defined as in equation \eqref{whatisnu} but with respect to the graph of $\psi_i$.  By vanishing Reifenberg flatness, $\gamma_i(-, -, r)\downarrow 0$ pointwise and the initial assumptions on $\psi$ imply that $\{\gamma_i/r\}$ is uniformly integrable. Hence, we can apply the dominated convergence theorem to get the desired result. 

The argument above relies on the fact that, for a graph domain, $\sigma_i = \sqrt{1+\nabla \psi_i} dyds \lesssim dyds$, where the implicit constant in $\lesssim$ is independent of $i$.  In general, $\Omega$ need not be a graph domain and, although $\sigma_i \rightharpoonup \sigma$ and $\gamma_i \rightarrow 0$ pointwise, we cannot, {\it a priori}, control the integral of $\gamma_id\sigma_i$. Instead, for each $i$, we will approximate $\Omega$, near $(Q_i,\tau_i)$ and at scale $r_i$, by a graph domain and then adapt the preceeding argument.  Our first step is to approximate $\Omega_i$ by graphs whose $\Lip(1,1/2)$ and Carleson measure norms are bounded independently of $i$. The proof follows closely that of Theorem 1 in \cite{hlnbigpieces}, which shows that parabolic chord arc domains contain big pieces of graphs of $f\in\Lip(1, 1/2)$ with $D^{1/2}_t f \in \mathrm{BMO}$.  However, we don't need to bound the $\mathrm{BMO}$ norm of $D^{1/2}_t \psi_i$ so the quantities we focus on are different. 

\begin{lem}\label{adaptedbigpiecesofgraphs}
Let $\Omega$ satisfy the conditions of Proposition \ref{chordarctovanishingchordarc}. Also let $r_i \downarrow 0$ and $(Q_i,\tau_i) \in K\cap \partial \Omega$. Then, for every $\varepsilon > 0$, there exists an $i_0 \equiv i_0(\varepsilon, K) > 1$ where $i > i_0$ implies the existence of a $\psi_i \in \Lip(1,1/2)(\R^{n-1}\times \R)$ such that:
\begin{enumerate}
\item $\sup_i \|\psi_i\|_{\Lip(1,1/2)} \leq C \equiv C(n, \varepsilon) < \infty$. 
\item Let $P_i \equiv P((Q_i,\tau_i), r_i)$ be the plane which best approximates $\Omega$ at scale $r_i$ around $(Q_i,\tau_i)$. Define $\tilde{\Omega}_i$ be the domain above the graph of $\psi_i$ over $P((Q_i,\tau_i),r_i)$. Then \begin{equation}\label{getsalmosteverything}
\sigma(\Delta_{r_i}(Q_i,\tau_i) \backslash \partial \tilde{\Omega}_i)  < \varepsilon r_i^{n+1}
\end{equation}
\item $D[C_{r_i}(Q_i,\tau_i) \cap \partial \tilde{\Omega}_i, P_i\cap C_{r_i}(Q_i,\tau_i)] \leq C(n) D[C_{r_i}(Q_i,\tau_i)\cap \partial \Omega, C_{r_i}(Q_i,\tau_i) \cap P_i]$.
\item If $\tilde{\nu}_i$ is the Carleson measure defined as in equation \eqref{whatisnu} but with respect to $\tilde{\Omega}_i$ then \begin{equation}\label{tildenusarebounded}\tilde{\nu}_i(C_{r_i}(Q_i,\tau_i)\times [0, r_i]) \leq K(n,\varepsilon, \|\nu\|)r_i^{n+1}\end{equation}
\end{enumerate}
\end{lem}

\begin{proof}
Let $\varepsilon > 0$ and $r_i \downarrow 0, (Q_i,\tau_i) \in K\cap \partial \Omega$. By the condition on $\sigma$ there exists an $i_1 > 0$ such that  $ (1-\varepsilon^2)\rho^{n+1}< \sigma(\Delta_\rho(P, \eta)) < (1+\varepsilon^2)\rho^{n+1}$ for all $\rho < 2r_{i_1}, (P,\eta) \in \partial \Omega \cap K$. There is also an $i_2$ such $\rho < r_{i_2}$ implies that $\Delta_{\rho}(P,\eta)$ is contained in a $\varepsilon^2 \rho$ neighborhood of some $n$-plane which contains a line parallel to the $t$-axis. Let $i_0(\varepsilon) = \max\{i_1, i_2\}$ and $i > i_0$. 

Henceforth, we will work at scale $r_i$ and so, for ease of notation, let $r_i \equiv R,(Q_i, \tau_i) \equiv (0,0)$ and $P \equiv P((0,0), R) \equiv \{x_n = 0\}$. If $\tilde{D} \equiv \frac{1}{R}D[C_R(0,0) \cap P, C_R(0,0) \cap \partial \Omega]$, then, by assumption, $\tilde{D} \leq \varepsilon^2$. Let $p: \mathbb R^{n+1}\rightarrow P$ be the orthogonal projection, i.e.  $p(Y,s) \colonequals (y, 0, s)$.  Fix a $\theta \in (0,1)$ to be choosen later (depending on $n, \varepsilon$) and define \begin{equation}\label{defofE} E = \{(P,\eta) \in \Delta_{R}(0,0) \mid \exists \rho < R,\; \mathrm{s.t.},\; \mathcal H^n(p(\Delta_\rho(P,\eta))) \leq \theta \rho^{n+1}\}.\end{equation}

The Vitali covering lemma gives $(P_i,\eta_i) \in E$, such that $E \subset \bigcup_i \Delta_{\rho_i}(P_i, \eta_i) \subset \Delta_{2R}(0,0)$, the $C_{\rho_i/5}(P_i, \eta_i)$ are pairwise disjoint and $\mathcal H^n(p(\Delta_{\rho_i}(P_i,\eta_i))) \leq \theta \rho_i^{n+1}$. Then $$\mathcal H^n(p(E)) \leq \sum_i \mathcal H^n(p(\Delta_{\rho_i}(P_i, \eta_i))) \leq \theta\sum_i \rho_i^{n+1}$$$$ \leq 5^{n+1}(1-\varepsilon^2)^{-1}\theta\sum_i \sigma(\Delta_{\rho_i/5}(P_i, \eta_i)) \leq 10^{n+1}(1-\varepsilon^2)^{-1}(1+\varepsilon^2) \theta R^{n+1}.$$ 

If $F \colonequals \Delta_R(0,0)\backslash E$ then $\sigma(F) \geq \mathcal H^n(p(F)) \geq \mathcal H^n(p(\Delta_R(0,0))\backslash p(E)) \geq (1-\varepsilon^2)R^{n+1} - \sigma(p(E))$ by Reifenberg flatness. This implies $\sigma(E)\leq  (1+\varepsilon^2)R^{n+1} - \sigma(F)  \leq 2\varepsilon^2 R^{n+1} + 10^{n+1}(1-\varepsilon^2)^{-1}(1+\varepsilon^2) \theta R^{n+1}$. So if $\theta = 10^{-(n+1)} \varepsilon/100$ then $\sigma(E) \leq \varepsilon R^{n+1}/2$ (as long is $\varepsilon > 0$ is sufficiently small). 

We want to show that $F$ is the graph of a Lipschitz function over $P$. Namely, that if $(Y,s), (Z,t) \in F$ then $\|(Y,s) - (Z,t)\| \leq C(|y-z| + |s-t|^{1/2})$. Let $\rho \colonequals 2(|y-z| + |s-t|^{1/2})$ and note that if $\rho \geq R$ then the flatness of $\Omega$ at scale $R$ implies the desired estimate.   Write $Z = Z' + Z''$ and $Y = Y' + Y''$ where $(Y',s), (Z',t)$ are the projections of $(Y,s)$ and $(Z,t)$ on $P((Y,s),\rho)$. By vanishing Reifenberg flatness $|Y''|,|Z''| \leq \varepsilon^2 \rho$. We can write \begin{equation}\label{trigcomputation}Y'-Z' -((Y'-Z')\cdot e_n)e_n  = p(Y'-Z') \Rightarrow |Y'-Z'| \leq \frac{|p(Z'-Y')|}{\min_{\hat{v} \in P((Y,s),\rho)} |\hat{v} -(\hat{v}\cdot e_n)e_n|}.\end{equation}  Define $\gamma = \min_{\hat{v} \in P((Y,s),\rho)} |\hat{v} -(\hat{v}\cdot e_n)e_n|$. Combine the above estimates to obtain $$\|(Y,s) - (Z,t)\| \leq |s-t|^{1/2} +|Y''-Z''| + |Y'-Z'| \leq \rho+ 2\varepsilon^2\rho +\gamma^{-1} |p(Z'-Y')| \leq (1+ 2\varepsilon^2 + \gamma^{-1}) \rho.$$ It remains only to bound $\gamma$ from below. As $p(C_\rho(Y,s)\cap P((Y,s),\rho))$ is a convex body in $P$, equation \eqref{trigcomputation} implies $\mathcal H^n(p(C_\rho(Y,s)\cap P((Y,s),\rho))) \leq c\gamma \rho^{n+1}$ for some constant $c$ (depending only on dimension).   As $C_\rho(Y,s)\cap \partial \Omega$ is well approximated by $C_\rho(Y,s)\cap P((Y,s),\rho)$ it must be the case that $\mathcal H^n(p(C_\rho(Y,s)\cap \partial \Omega)) \leq c(\gamma + \varepsilon^2)\rho^{n+1}$. As $(Y,s) \in F$ we know $(c\gamma +\varepsilon^2) \geq \theta = c(n) \varepsilon \Rightarrow \gamma \geq \tilde{c} \varepsilon$. In particular we have shown that for $(Y,s), (Z,t) \in F$ that \begin{equation}\label{lipnormbounded}\|(Y,s) - (Z,t)\| \leq C(n,\varepsilon) (|y-z| + |s-t|^{1/2}).\end{equation}

In order to eventually get the bound, \eqref{tildenusarebounded}, on the Carleson norm we need to shrink $F$ slightly, to $F_1$. $\Omega$ is a parabolic regular domain so if $$f(P, \eta) \colonequals \int_0^{2R} \gamma(P,\eta,r) r^{-1}dr,\; (P,\eta) \in C_R(0,0)$$ then $$\int_{C_{2R}(0,0)} f(P,\eta)d\sigma(P,\eta) \leq \|\nu\|(2R)^{n+1}.$$ Hence, by Markov's inequality, $$\sigma(\{(P,\eta) \in C_{2R}(0,0) \mid f(P,\eta) \geq (2\varepsilon^{-1})^{n+1}\|\nu\|\} ) \leq \varepsilon^{n+1} R^{n+1}.$$ Let $F_1 = F\backslash \{(P,\eta) \in C_{2R}(0,0) \mid f(P,\eta) \geq (2\varepsilon^{-1})^{n+1}\|\nu\|\} $. It is clear that equation \eqref{lipnormbounded} holds for $(Y,s), (Z,t)\in  F_1$ and that \begin{equation}\label{fisnottoobiginf1} f(P,\eta) \leq 2^{n+1}\varepsilon^{-(n+1)}\|\nu\|, \forall (P,\eta) \in F_1.\end{equation} Finally, we have the estimate \begin{equation}\label{notmuchoutsideofF1}\sigma(C_R(0,0)\cap \partial \Omega \backslash F_1) \leq \varepsilon R^{n+1}.\end{equation}

At this point we are ready to construct $\psi$ using a Whitney decomposition of $\{x_n =0\}$. Let $\psi^*$ be such that if $(y,0,s) \in p(F_1)$ then $(y, \psi^*(y,s), s) \in F_1$.  Let $Q_i \colonequals Q_{\rho_i}(\hat{x}_i,\hat{t}_i) \subset \{x_n = 0\}$ be such that \begin{enumerate}
\item $\{x_n = 0\} \backslash p(F_1) = \bigcup \overline{Q}_i$.
\item Each $Q_i$  is centered at $(\hat{x}_i,\hat{t}_i)$ with side length $2\rho_i$ in the spacial directions and $2\rho_i^2$ in the time direction.
\item $Q_i\cap Q_j = \emptyset$ for all $i\neq j$
\item $10^{-10n}d(Q_i, p(F_1)) \leq \rho_i \leq 10^{-8n}d(Q_i,p(F_1))$. 
\end{enumerate}
Then let $v_i$ be a partition of unity subordinate to $Q_i$. Namely,
\begin{enumerate}[(I)]
\item $\sum_i v_i \equiv 1$ on $\R^n \backslash p(F_1)$. 
\item $v_i \equiv 1$ on $\frac{1}{2}Q_i$ and $v_i$ is supported on the double of $Q_i$. 
\item $v_i \in C^\infty(\R^n)$ and $\rho_i^{\ell} |\partial^\ell_x v_i| + \rho_i^{2\ell}|\partial^\ell_t v_i| \leq c(\ell,n)$ for $\ell = 1,2,...$
\end{enumerate}
For each $i$ there exists $(x_i, t_i) \in p(F_1)$ such that $$d_i \colonequals d(p(F_1), Q_i) = d((x_i, t_i), Q_i).$$ Finally let $\Lambda = \{i\mid \overline{Q}_i \cap C_{2R}(0,0) \neq \emptyset\}$ and define \begin{equation}\label{definitionofpsi}
\psi(y,s) \colonequals \begin{cases} \psi^*(y,s),\; (y,s) \in p(F_1)\\
\sum_{i\in \Lambda} (\psi^*(x_i, t_i) + \tilde{D} d_i)v_i(y,s),\; (y,s) \in \R^n \backslash p(F_1)
\end{cases}
\end{equation}
where, as before, $\tilde{D} = \frac{1}{R}D[C_R(0,0) \cap P, C_R(0,0) \cap \partial \Omega] \leq \varepsilon^2$. 

Let $\widetilde{\Omega}$ be the graph of $\psi$ over $\{x_n = 0\}$ and recall the conditions we want $\psi$ and $\widetilde{\Omega}$ to satisfy. Condition (2) is a consequence of equation \eqref{notmuchoutsideofF1}. Condition (3) follows as $|\psi| \leq C(n) \tilde{D}$. 

It remains to show Condition (1): $|\psi(y, s) - \psi(z, t)| \leq C(n,M,\varepsilon)(|y-z| + |s-t|^{1/2})$. Equation \eqref{lipnormbounded} says this is true when $(y,s), (z,t) \in p(F_1)$. When $(y,s) \in p(F_1)$ and $(z,t) \notin p(F_1)$ we can estimate $$|\psi(y, s) - \psi(z, t)| \leq \sum_{\{i\in \Lambda\mid (z,t) \in 2Q_i\}} v_i(z,t)|\psi^*(y,s) - \psi^*(x_i,t_i)| +C(n) \varepsilon^2 d((z,t), p(F_1))$$ as $v_i(z,t) \neq 0$ implies $d_i \leq 10 d((z,t), p(F_1))$. Apply the triangle inequality to conclude \begin{equation}\label{boundingdifferenceinpsis}\begin{aligned} &|\psi(y, s) - \psi(z, t)| \leq C(n)\varepsilon^2\left(|y-z| + |s-t|^{1/2}\right)+ \sum_{\{i\in \Lambda \mid (z,t) \in 2Q_i\}} |y-x_i| + |s-t_i|^{1/2} \\ &\leq C(n,\varepsilon)\left(|y-z| + |s-t|^{1/2}\right) + \sum_{\{i\in \Lambda \mid (z,t) \in 2Q_i\}} |z-x_i| + |t-t_i|^{1/2}\\ &\leq C(n,\varepsilon)\left(|y-z| + |s-t|^{1/2} + \sum_{\{i\in \Lambda \mid (z,t) \in 2Q_i\}} d_i \right)\\
&\leq C(n,\varepsilon) |y-z| + |s-t|^{1/2}.\end{aligned}\end{equation} In the above, we used that $|{\{i\in \Lambda \mid (z,t) \in 2Q_i\}}| \leq C$ and, if $(z,t) \in Q_i$ that $d_i \leq C(n)d((z,t), p(F_1))$.  From now on we write, $a\lesssim b$ if there is a constant $C$, (which can depend on $\varepsilon$, the dimension and the parabolic uniform regularity constants of $\Omega$) such that $a \leq Cb$. 

The last case is if $(y,s), (z,t) \notin p(F_1)$. When $\max\{d((y,s), p(F_1)), d((z,t), p(F_1))\} \leq \|(y,s)- (z,t)\|$ estimate $|\psi(y, s) - \psi(z, t)| \leq |\psi(y,s) - \psi(\tilde{y}, \tilde{s})| + |\psi(\tilde{y},\tilde{s}) - \psi(\tilde{z}, \tilde{t})| + |\psi(\tilde{z}, \tilde{t}) - \psi(z, t)|$ where $(\tilde{y},\tilde{s})$ is the closest point in $p(F_1)$ to $(y,s)$ and similarly $(\tilde{z},\tilde{t})$. The Lipschitz bound is then a trivial consequence of the fact that  $d((\tilde{y},\tilde{s}),(\tilde{z}, \tilde{t})) \leq d((y,s), p(F_1))+ d((z,t), p(F_1)) + \|(y,s)- (z,t)\| \leq 3\|(y,s)- (z,t)\|$ and the above analysis.

Now assume $\min\{d((y,s), p(F_1)), d((z,t), p(F_1))\} \geq \|(y,s)- (z,t)\|$. Recall that if $(y,s) \in 2Q_i$ then $d_i \leq C(n) d((y,s), p(F_1)) \leq \|(y,s)- (z,t)\|$. Similarly if $i, j$ are such that $(y,s) \in 2Q_i$ and $(z,t) \in 2Q_j$ then $\|(x_i, t_i) - (x_j,t_j)\| \leq C(n) \|(y,s)- (z,t)\|$. Then we can write \begin{equation}\label{distancesarebig}\begin{aligned} |\psi(y, s) - \psi(z, t)| =& \left|\sum_{i \in \Lambda} (\psi^*(x_i,t_i) + \tilde{D}d_i)v_i(y,s) - \sum_{j \in \Lambda} (\psi^*(x_j,t_j) + \tilde{D}d_j)v_j(z,t)\right|\\
\leq& \sum_{i,j \in \Lambda} |\psi^*(x_i,t_i) - \psi^*(x_j, t_j)| + \tilde{D}|d_i - d_j|)v_i(y,s)v_j(z,t)\\
\leq& \sum_{\{i,j \in \Lambda\mid (y,s) \in 2Q_i, (z,t) \in 2Q_j\}} C(n,\varepsilon) \|(x_i,t_i) - (x_j, t_j)\|  + C(n)\|(y,s)- (z,t)\|\\
\leq& C(n,\varepsilon) \|(y,s) - (z, t)\| .
\end{aligned}\end{equation}
In the above we use that $|\{i,j \in \Lambda\mid (y,s) \in 2Q_i, (z,t) \in 2Q_j\}| \leq C(n)$ and equation \eqref{lipnormbounded}.

Finally, we may assume, without loss of generality, that $d((y,s), p(F_1)) \leq  \|(y,s)- (z,t)\| \leq d((z,t), p(F_1))$. Then \begin{equation}\label{onedistancesisbig}|\psi(y, s) - \psi(z, t)| \leq \sum_{i\in \Lambda} ((\psi^*(x_i, t_i) -\psi(y,s))  + \tilde{D} d_i) |v_i(y,s) - v_i(z,t)|\end{equation} as $\sum \psi(y,s)(v_i(y,s) - v_i(z,t)) = 0$. Arguing as in equation \eqref{boundingdifferenceinpsis}, $|\psi^*(x_i, t_i) -\psi(y,s)| \leq C(n,\varepsilon)\|(y,s)-(x_i,t_i)\|$. As before, if $(y,s) \in 4Q_i$ then $\|(y,s)-(x_i,t_i)\| \leq \tilde{C}(n)d_i \leq c(n)d((y,s), p(F_1)) \leq c(n) \|(y,s)- (z,t)\|$. If $(y,s) \notin 4Q_i$ then we may assume $(z,t) \in 2Q_i$ (or else $v_i(y,s) - v_i(z,t) = 0$). So $\|(y,s)-(x_i,t_i)\| \leq \|(y,s) - (z,t)\| + \|(z,t) - (x_i,t_i)\| \leq \|(y,s) - (z,t)\|  + c(n) \rho_i \leq \tilde{c}(n) \|(y,s) - (z,t)\|$ (as $\|(y,s) - (z,t)\| \geq c'(n) \rho_i$). Either way, $\|(y,s)-(x_i,t_i)\|, d_i \leq C(n)\|(y,s) - (z,t)\|$. Hence,  $$|\psi(y, s) - \psi(z, t)| \leq C(n,\varepsilon) \|(y,s) - (z,t)\| \sum_{i\in \Lambda} |v_i(y,s) - v_i(z,t)|  \leq C(n,\varepsilon)\|(y,s) - (z,t)\|.$$

It remains only to estimate the Carleson norm of $\tilde{\nu}$. Our first claim in this direction is that if $(Y,s) \in \partial \tilde{\Omega} \cap C_{2R}(0,0)$ then \begin{equation}\label{distancetoomega}
d((Y,s), \partial \Omega) \leq C(n, M, \varepsilon) d((y,s), p(F_1)).
\end{equation}
Indeed, if $(\tilde{y},\tilde{s})\in p(F_1)$ be the point in $p(F_1)$ closest to $(y,s)$ then, $$d((Y,s), \partial \Omega)  \leq d((y,s), p(F_1)) + |\psi(y,s) - \psi(\tilde{y},\tilde{s})| \leq C(n,M,\varepsilon)  d((y,s), p(F_1))$$ by the boundedness of $\psi$'s $\Lip(1,1/2)$ norm. 

Define $\Gamma_i$ to be the graph of $\psi$ over $Q_i$, and, for $r > 0, (X,t) \in F_1$, define $\xi(X,t, r) \colonequals \{i\mid \overline{\Gamma}_i \cap C_r(X,t) \neq \emptyset\}$. By equation \eqref{distancetoomega} and standard covering theory there are constants $k \equiv k(n, M, \varepsilon), \tilde{k} \equiv \tilde{k}(n,M, \varepsilon)$ such that $\overline{\Gamma}_i \subset \bigcup_{j} C_{kd_i}(Z_{i,j},\tau_{i,j})\subset C_{\tilde{k}r}(X,t)$ where $(Z_{i,j},\tau_{i,j}) \in \partial \Omega$ and the $C_{kd_i/5}(Z_{i,j},\tau_{i,j})$ is  disjoint from $C_{kd_i/5}(Z_{i,\ell},\tau_{i,\ell})$ for $j \neq \ell$. For any $(Z,\tau) \in \Gamma_i\cap C_{kd_i}(Z_{i,j},\tau_{i,j})$ and any $n$-plane $\hat{P}$ containing a line parallel to the $t$-axis we have \begin{equation}\label{distancedominatedbynotgraph}d((Z,\tau),\hat{P})^2 \leq C(n)\left(\min_{(Y,s) \in C_{kd_i}(Z_{i,j},\tau_{i,j})\cap \partial \Omega} d((Y,s), \hat{P})^2 + k^2d_i^2\right).\end{equation}

Define $\tilde{\gamma}$ as in equation \eqref{gammainuniformrectifiability} but with respect to $\tilde{\Omega}$. For any $(X,t) \in F_1$, equation \eqref{distancedominatedbynotgraph} gives:  $$\tilde{\gamma}(X,t,r) \leq \frac{1}{r^{n+3}} \left(\int_{C_r(X,t)\cap F_1}d((Z,\tau), P)^2 d\sigma+ \sum_{i\in \xi}\sum_{j} \int_{C_{kd_i}(Z_{i,j},\tau_{i,j})\cap \partial \tilde{\Omega}} d((Z,\tau), P)^2 d\tilde{\sigma} \right)$$$$\stackrel{Eq. \eqref{distancedominatedbynotgraph}}{\leq} \gamma(X,t,r) + C(n,k,\varepsilon)\sum_{i\in \xi} \left(\frac{d_i}{r}\right)^{n+3} + \frac{C(n)}{r^{n+3}}\sum_{i\in \xi} \sum_j \int_{\Delta_{kd_i/5}(Z_{i,j},\tau_{i,j})} d((Z,\tau), P)^2 d\sigma.$$ As $Q_i$ can be adjacent to at most $c(n)$ many other $Q_k$ we can be sure that $C_{kd_i/5}(Z_{i,j},\tau_{i,j})$ intersects at most $\tilde{c}(n)$ other $C_{kd_\ell/5}(Z_{\ell,j},\tau_{\ell,j})$. Additionally, the $C_{kd_i/5}(Z_{i,j},\tau_{i,j}) \subset C_{\tilde{k}r}(X,t)$ for all $i,j$. Hence, we can control $\tilde{\gamma}$ on $F_1$: 

\begin{equation}\label{dominateonthegraphbyoff}
\tilde{\gamma}(X,t,r) \leq c(M,n,\varepsilon)\left(\sum_{i\in \xi} (d_i/r)^{n+3} + \gamma(X,t, \tilde{k}r)\right),\; \forall (X,t) \in F_1.
\end{equation}

Integrating equation \eqref{dominateonthegraphbyoff} over $F_1$ and then in $r$ from $[0, R]$, allows us to conclude \begin{equation}\label{tildenuonfone}\tilde{\nu}(F_1 \times [0,R]) \leq \|\nu\|(\tilde{k}R)^{n+1} +  \int_{(x,t) \in p(F_1)}\int_0^R r^{-1}\sum_{i\in \xi(X,t, r)} \left(\frac{d_i}{r}\right)^{n+3}dr dxdt.\end{equation} Note that $i \in \xi(X,t,r)$ implies that $Q_i \cap C_r(X,t) \neq \emptyset \Rightarrow r \geq d((x,t), Q_i) \geq d_i$.  Therefore, $$\int_{(x,t) \in p(F_1)}\int_0^R r^{-1}\sum_{i\in \xi(X,t, r)} \left(\frac{d_i}{r}\right)^{n+3}dr dxdt  \leq \int_{(x,t) \in p(F_1)} \sum_{i\in \xi(X,t, 2R)} 1 dxdt \leq C(n)R^{n+1},$$ as $F_1 \subset C_R(0,0)$. Putting this together with equation \eqref{tildenuonfone} gives us \begin{equation}\label{donetilenuonf}
\tilde{\nu}(F_1 \times [0,R]) \leq C(\|\nu\|, n, \varepsilon) R^{n+1}
\end{equation}

If $(x,t) \in Q_i$ (i.e. $(X,t) \subset C_R(0,0)\cap \partial \tilde{\Omega} \backslash F_1$) then approximation by affine functions and a Taylor series expansion yields \begin{equation}\label{taylorseriesestimate}\tilde{\gamma}(X,t,r) \leq c(n,M,\varepsilon) r^2 d_i^{-2},\; (x,t) \in Q_i\; \forall r \leq 8d_i, \end{equation}  (see \cite{hlnbigpieces} pp 367, for more details).  When $r \geq 8d_i$ we can lazily estimate $\tilde{\gamma}(X,t,r) \lesssim \tilde{\gamma}(X_i, t_i, \tilde{k}r)$ where $(X_i, t_i) = (x_i, \psi(x_i, t_i),t_i)$ and $(x_i, t_i)\in p(F_1)$ such that $d(p(F_1), Q_i) = d((x_i,t_i), Q_i)$ (as in the definition of $\psi$). This is because $C_r(X,t) \subset C_{\tilde{k}r}(X_i,t_i)$ (where $\tilde{k}$ is as above). Hence, \begin{equation}\label{estoffF1forbigr} \begin{aligned} &\int_{8d_i}^R\tilde{\gamma}(X,t,r) r^{-1}dr \lesssim \int_{8d_i}^R \tilde{\gamma}(X_i, t_i, \tilde{k}r) r^{-1}dr\\ &\stackrel{Eq. \eqref{dominateonthegraphbyoff}}{\lesssim} \int_{8d_i}^R \gamma(X_i, t_i, \tilde{k}^2r)r^{-1}dr +  \int_{8d_i}^Rr^{-1}\sum_{j\in \xi(X_i, t_i,r)} (d_j/r)^{n+3}dr,\; \forall (X,t) \in \Gamma_i.\end{aligned}\end{equation} Note that $(X_i,t_i)\in F_1$, thus $\int_0^R \gamma(X_i, t_i, \tilde{k}^2r)r^{-1}dr \lesssim 2^{n+1}\varepsilon^{-n-1}\|\nu\|$. As before, $j\in \xi(X_i, t_i,r) \Rightarrow r \geq d_j$. So we can bound $$\int_{8d_i}^Rr^{-1}\sum_{j\in \xi(X_i, t_i,r)} (d_j/r)^{n+3}dr \leq c(n)\sum_{j\in \xi(X_i, t_i,2R)} 1.$$ If we combine these estimates and integrate over $\Gamma_i$ we get $$\int_{\Gamma_i}\int_{8d_i}^R\tilde{\gamma}(X,t,r) r^{-1}drd\tilde{\sigma}  \leq C(\|\nu\|, \varepsilon, n)\left(\tilde{\sigma}(\Gamma_i)+  \int_{\Gamma_i} |\{j\in \xi(X_i, t_i,2R)\}| d\tilde{\sigma}\right).$$

Use equation \eqref{taylorseriesestimate} to bound the integral for small $r$ and sum over all $Q_i$s to obtain: \begin{equation}\label{nuofff1}
\tilde{\nu}((C_R(0,0)\backslash F_1) \times [0,R]) \lesssim R^{n+1} + \sum_{j\in \xi(0,0, 2R)} \tilde{\sigma}(\Gamma_i) \lesssim R^{n+1}.
\end{equation}
 
Combine equations \eqref{donetilenuonf} and \eqref{nuofff1} to obtain $$\tilde{\nu}(C_R(0,0)\times [0,R]) \leq C(n, \varepsilon, \|\nu\|) R^{n+1}.$$
\end{proof}

We now want to control the Carleson norm of $\Omega$ by that of the graph domain.

\begin{lem}\label{carlesonnormofomegabounded}
Let $\Omega$ be a parabolic uniformly rectifiable domain and let $\Psi$ be a $\Lip(1,1/2)$ function such that \begin{equation}\label{getsalmosteverything2}
\sigma(C_{1}(0,0)\cap \partial \Omega \backslash \partial \tilde{\Omega})  < \varepsilon,
\end{equation}
where $\tilde{\Omega}$ is the domain above the graph of $\Psi$ over some $n$-plane $\overline{P}$ which contains a line parallel to the $t$-axis.

Then, if $\tilde{\nu}$ is defined as in \eqref{carlesonmeasure} but associated $\partial \tilde{\Omega}$ we have \begin{equation}\label{nucontrolledbynutilde} \nu(C_{1/2}(0,0) \times [0,1/2]) \leq c(\tilde{\nu}(C_{3/4}(0,0) \times [0,3/4])  + \varepsilon^{1/2}),\end{equation} where $c> 0$ depends on $n, \|\nu\|$ and the Ahlfors regularity constant of $\Omega$. 
\end{lem}

\begin{proof}
This proof follows closely the last several pages of \cite{hlncaloricmeasure}. For ease of notation let $\overline{P} = \{x_n =0\}$, so that $\partial \tilde{\Omega} = \{(y,\Psi(y,s),s)\}$. Let $\tilde{\chi}$ be the characteristic function of $\Delta_1(0,0) \backslash \partial \tilde{\Omega}$. The Hardy-Littlewood maximal function of $\tilde{\chi}$ with respect to $\sigma$ is
$$M_\sigma(\tilde{\chi})(Y,s) =\sup_{\rho > 0} \frac{\sigma(C_\rho(Y,s)\cap \partial\Omega\cap C_1(0,0)\backslash \partial \tilde{\Omega})}{\sigma(\partial \Omega \cap C_\rho(Y,s))},\; (Y,s) \in \partial \Omega.$$

The Hardy-Littlewood maximal theorem states $$\sigma(\{(Y,s)\mid M_\sigma(\tilde{\chi})(Y,s) \geq \sqrt{\varepsilon}\}) \leq C(n) \frac{\|\tilde{\chi}\|_{L^1}}{\sqrt{\varepsilon}} \leq C(n) \sqrt{\varepsilon}.$$ As such, there exists a compact set $E \subset \partial \tilde{\Omega} \cap \Delta_1(0,0)$, such that $M_\sigma(\tilde{\chi})(Y,s) \leq \sqrt{\varepsilon}$  for all $(Y,s) \in E$ and \begin{equation}\label{sizewithoutE}\sigma(\Delta_1(0,0)\backslash E) \leq \varepsilon +C(n)\sqrt{\varepsilon} < \tilde{C}(n)\sqrt{\varepsilon}.\end{equation}

Let $\{Q_i\}$ be a Whitney decomposition of $\mathbb R^{n+1}\backslash E$. That is to say,
\begin{enumerate}
\item $Q_i \colonequals Q_{r_i}(P_i,\eta_i)$ is a parallelogram whose cross section at any time is a cube of side length $2r_i$ centered at $P_i$ and whose length (in the time direction) is $2r_i^2$, centered around the time $\eta_i$.
\item $Q_i \cap Q_j = \emptyset, i \neq j$
\item $10^{-10n}d(Q_i, E) \leq r_i \leq 10^{-5n}d(Q_i, E)$, 
\item For each $i, \{j \mid \overline{Q_j}\cap \overline{Q}_i\neq \emptyset\}$ has cardinality at most $c$
\item $\R^{n+1}\backslash E = \bigcup \overline{Q_i}$. 
\end{enumerate}

For $(Q,\tau) \in E$ and $0 < \rho < 1/2$, let $\xi(Q,\tau, \rho) = \{i \mid \overline{Q}_i\cap \Delta_{\rho}(Q,\tau)\neq \emptyset\}$. We claim \begin{equation}\label{boundongamma} \gamma(Q,\tau, \rho) \leq c(n) \left(\tilde{\gamma}(Q,\tau, \frac{3}{2}\rho) +  \sum_{i \in \xi(Q,\tau, \rho) } \left(\frac{r_i}{\rho}\right)^{n+3}\right).
\end{equation}

To wit, let $P$ be a plane containing a line parallel to the $t$ axis such that $\tilde{\gamma}(Q,\tau, 3\rho/2)$ is achieved by $P$. By definition $$\tilde{\gamma}(Q,\tau, \frac{3}{2}\rho) = \left(\frac{3}{2}\right)^{-n-3} \rho^{-n-3} \int_{\partial \tilde{\Omega}\cap C_{3\rho/2}(Q,\tau)} d((Y,s), P)^2d\tilde{\sigma}(Y,s).$$ On the other hand \begin{equation}\label{towardsaboundongamma}\begin{aligned}\gamma(Q,\tau,\rho) \leq& \rho^{-n-3}\left(\int_{E\cap C_{\rho}(Q,\tau)} d((Y,s), P)^2d\sigma + \int_{\Delta_\rho(Q,\tau) \backslash E} d((Y,s), P)^2d\sigma\right)\\ \leq& \tilde{\gamma}(Q,\tau,\rho) + \rho^{-n-3}\sum_{i\in \xi} \int_{\overline{Q}_i\cap \Delta_\rho(Q,\tau)} d((Y,s), P)^2d\sigma,\end{aligned}\end{equation} as $d\sigma = d\tilde{\sigma}$ on $E$. 

Note that the parabolic diameter of $Q_i$ is $\leq c(n) r_i$. Hence if $$\xi_1(Q,\tau,\rho) \colonequals \{i\in \xi(Q,\tau,\rho)\mid \exists (Y,s) \in \overline{Q}_i\; \mathrm{s.t.}\; d((Y,s), P) < r_i\}$$ then $d((Y,s), P) \leq c'(n)r_i$ for all $(Y,s) \in \overline{Q}_i$ and all $i\in \xi_1(Q,\tau,\rho)$. Therefore, \begin{equation}\label{estimateinxi1}\rho^{-n-3}\sum_{i\in \xi_1(Q,\tau,\rho)} \int_{\overline{Q}_i\cap \Delta_\rho(Q,\tau)} d((Y,s), P)^2d\sigma \leq C(n,M) \sum_{i\in \xi_1(Q,\tau,\rho)} (r_i/\rho)^{n+3}.\end{equation}

If $i \in \xi(Q,\tau,\rho)\backslash \xi_1(Q,\tau,\rho)$ let $(Y_i^*,s_i^*) \in \overline{Q_i}$ be such that $d(\overline{Q}_i, E) = d((Y_i^*,s_i^*), E) \equalscolon \delta_i$. By the triangle inequality, $\sup_{(Y,s) \in \overline{Q}_i} d((Y,s), P) \leq d((Y_i^*,s_i^*), P) + c(n)r_i \leq c'(n)d((Y_i^*,s_i^*), P)$ (because $i \notin \xi_1(Q,\tau,\rho)$). This implies, $$\int_{\overline{Q}_i\cap\Delta_\rho(Q,\tau)} d((Y,s), P)^2d\sigma(Y,s) \leq c(n)r_i^{n+1}d((Y_i^*,s_i^*), P)^2.$$

The distance between $(Y_i^*,s_i^*)$ and $E$ is $\delta_i$, as such, $ C_{\delta_i/9}(\tilde{Y}_i,\tilde{s}_i)\subset C_{10\delta_i/9}(Y_i^*,s_i^*)$ for some $(\tilde{Y}_i,\tilde{s}_i)\in E\subset \partial \tilde{\Omega}$. Furthermore, recall $\delta_i \simeq r_i$, hence, $\tilde{\sigma}(C_{\delta/9}(\tilde{Y}_i,\tilde{s}_i)\cap \partial \tilde{\Omega}) \geq c(n) \delta_i^{n+1} \geq c'r_i^{n+1}$. Arguing as above,  $d((Y,s),P) +c(n) \delta _i\geq d((Y_i^*, s_i^*), P)$ for any $(Y,s) \in C_{10\delta_i/9}(Y_i^*,s_i^*)$. Putting all of this together, \begin{equation*}\begin{aligned}&\sum_{i\in \xi(Q,\tau,\rho)\backslash \xi_1(Q,\tau,\rho)}\int_{\overline{Q}_i\cap\Delta_\rho(Q,\tau)} d((Y,s), P)^2d\sigma \leq \sum_{i\in \xi(Q,\tau,\rho)\backslash \xi_1(Q,\tau,\rho)}c(n)r_i^{n+1}d((Y_i^*,s_i^*), P)^2\\ &\leq \sum_{i\in \xi(Q,\tau,\rho)\backslash \xi_1(Q,\tau,\rho)}c(n)\left(\int_{C_{\delta_i/9}(\tilde{Y}_i,\tilde{s}_i)\cap \partial \tilde{\Omega}} d((Y,s), P)^2 d\tilde{\sigma} + \delta_i^2r_i^{n+1}\right).\end{aligned}\end{equation*} 

Observe that, if $i \in \xi(Q,\tau,\rho)$ then, $\delta_i \leq \rho$ and, therefore, $\bigcup_{i\in \xi(Q,\tau,\rho)\backslash \xi_1(Q,\tau,\rho)} C_{\delta_i/9}(\tilde{Y}_i,\tilde{s}_i) \subset C_{3\rho/2}(Q,\tau)$. Furthermore for each $i \in \xi(Q,\tau, \rho)$, $\#\{j\in \xi(Q,\tau,\rho)\backslash \xi_1(Q,\tau,\rho)\mid \overline{C}_{\delta_i/9}(\tilde{Y}_i,\tilde{s}_i) \cap \overline{C}_{\delta_j/9}(\tilde{Y}_j,\tilde{s}_j)\} < c(n)$. Plugging these estimates into the offset equation above yields \begin{equation}\label{estimatenotinxione}\rho^{-n-3}\sum_{i \in \xi\backslash \xi_1} \int_{\overline{Q}_i\cap \Delta_\rho(Q,\tau)} d((Y,s), P)^2d\sigma \leq c(n) \left(\tilde{\gamma}(Q,\tau, \frac{3}{2}\rho) + \sum_{i\in \xi\backslash \xi_1} \left(\frac{r_i}{\rho}\right)^{n+3}\right).\end{equation} Our claim, equation \eqref{boundongamma}, follows from equations \eqref{towardsaboundongamma}, \eqref{estimateinxi1} and  \eqref{estimatenotinxione}.

By definition, if $i \in \xi(Q,\tau, \rho)$ then $\rho \geq d(\overline{Q}_i, (Q,\tau))$. Integrate equation \eqref{boundongamma} in $\rho$ from $0$ to $1/2$ and over $(Q,\tau) \in E\cap C_{1/2}(0,0)$ to obtain \begin{equation}\label{estimateforE}\begin{aligned}  &\nu((E\cap C_{1/2}(0,0))\times [0,1/2]) \leq C(\tilde{\nu}((E\cap C_{3/4}(0,0))\times [0,3/4])\\ &+ \int_{E\cap C_{1/2}(0,0)}\left(\sum_{i\in \xi(Q,\tau,1/2)} \int_{d(\overline{Q}_i, (Q,\tau))}^{1/2} (r_i/\rho)^{n+3} \rho^{-1}d\rho\right)d\sigma(Q,\tau). \end{aligned}\end{equation} Here, and for the rest of the proof, $C$ will refer to a constant which may depend on the dimension, $\|\nu\|$ and the Ahlfors regularity constant of $\Omega$ but not on $\Psi$ or $\varepsilon$. 
  
Evaluate the integral in $\rho$, $$\int_{E\cap C_{1/2}(0,0)}\left(\sum_{i\in \xi(Q,\tau,1/2)} \int_{d(\overline{Q}_i, (Q,\tau))}^{1/2} (r_i/\rho)^{n+3} \rho^{-1}d\rho\right)d\sigma \leq C\sum_{i\in \xi(0,0,1)}\int_{E\cap C_{1/2}(0,0)}\left(\frac{r_i}{d(\overline{Q}_i, (Q,\tau))}\right)^{n+3} d\sigma.$$

For every $\lambda \geq c(n)r_i$, let $$E_\lambda \colonequals \{(Q,\tau) \in E\mid d(\overline{Q}_i, (Q,\tau)) \leq \lambda\}.$$  Trivially $E_\lambda = \{(Q,\tau) \in E\mid r_i/d(\overline{Q}_i, (Q,\tau)) \geq r_i/\lambda\}$. By construction of the Whitney decomposition, $\mathrm{diam}(Q_i) \leq c(n) \lambda$, so Ahlfors regularity implies $$\sigma(E_\lambda) \leq C\lambda^{n+1}.$$ Recall that $r_i/d(Q_i, (Q,\tau)) \leq r_i/\delta_i \leq 10^{-5n}$. Let $\gamma = r_i/\lambda$ and evaluate, \begin{equation*}\begin{aligned}\int_{E\cap C_{1/2}(0,0)}\left(\frac{r_i}{d(\overline{Q}_i, (Q,\tau))}\right)^{n+3} d\sigma &\leq \int_0^{10^{-5n}} (n+3)\gamma^{n+2}\sigma(\{(Q,\tau) \in E\mid r_i/d(\overline{Q}_i, (Q,\tau)) \geq \gamma\})d\gamma\\ &\leq C(n) \int_0^1\gamma^{n+2} r_i^{n+1}\gamma^{-n-1}d\gamma\leq c(n)r_i^{n+1}.\end{aligned}\end{equation*} 

Now recall, that $Q_i \in \xi(0,0,3/4)$ form a cover of $\Delta_{3/4}(0,0)\backslash E$. In light of equation \eqref{sizewithoutE}, $\sum_{i\in \xi(0,0,3/4)} r_i^{n+1} \leq c(n)\sigma(\Delta_1(0,0) \backslash E) \leq C\sqrt{\varepsilon}$. Which allows us to bound, $$\int_{E\cap C_{1/2}(0,0)}\left(\sum_{i\in \xi(Q,\tau,1/2)} \int_{d(\overline{Q}_i, (Q,\tau))}^{1/2} (r_i/\rho)^{n+3} \rho^{-1}d\rho\right)d\sigma \leq C \sqrt{\varepsilon}.$$ Plugging the above inequality into equation \eqref{estimateforE} proves \begin{equation}\label{donewithEpart}\nu((E\cap C_{1/2}(0,0))\times [0,1/2]) \leq C(\tilde{\nu}((E\cap C_{3/4}(0,0))\times [0,3/4])  + \sqrt{\varepsilon}).\end{equation}

It remains to estimate $\nu((\Delta_{1/2}(0,0)\backslash E)\times [0,1/2])$. By the Cauchy-Schwartz inequality \begin{equation}\label{estimateoffE}\begin{aligned}&\nu((\Delta_{1/2}(0,0)\backslash E)\times [0,1/2]) = 
\int_{\Delta_{1/2}(0,0)}\chi_{E^c}(Q,\tau) \left(\int_0^{1/2} \gamma(Q,\tau,\rho) \rho^{-1}d\rho\right) d\sigma(Q,\tau)\\ &\leq\sqrt{\sigma(\Delta_{1/2}(0,0)\backslash E)} \left(\int_{\Delta_{1/2}(0,0)} \left(\int_0^{1/2} \gamma(Q,\tau,\rho) \rho^{-1}d\rho\right)^2 d\sigma(Q,\tau)\right)^{1/2}.\end{aligned}\end{equation}

As above, let $f(Q,\tau) \colonequals \int_0^{1/2} \gamma(Q,\tau,\rho) \rho^{-1}d\rho$. We claim that $f(Q,\tau) \in \mathrm{BMO}(C_{1/2}(0,0))$.  Let $C_{\rho}(P,\eta) \subset C_{1/2}(0,0)$ with $(P,\eta) \in \partial \Omega$. Define $k \equiv k(\rho, P,\eta) = \int_\rho^{1/2} \gamma(P,\eta,r)r^{-1}dr$. Additionally, let $f_1(Q,\tau) \colonequals\int_0^\rho \gamma(Q,\tau,r)r^{-1}dr$ and $f_2(Q,\tau) = f(Q,\tau) - f_1(Q,\tau)$ for $(Q,\tau) \in C_\rho(P,\eta)$. By the triangle inequality,
\begin{equation}\label{showingitsBMO1}\begin{aligned}
\int_{\Delta_{\rho}(P,\eta)} |f(Q,\tau) - k| d\sigma &\leq \int_{\Delta_{\rho}(P,\eta)} f_1(Q,\tau) d\sigma + \int_{\Delta_\rho(P,\eta)} |f_2(Q,\tau) - k|d\sigma\\
&\leq \|\nu\| \rho^{n+1} + \int_{\Delta_\rho(P,\eta)} |f_2(Q,\tau) - k|d\sigma.
\end{aligned}
\end{equation}

If $(Q,\tau) \in \Delta_\rho(P,\eta)$ then $\sigma((\Delta_r(Q,\tau) \cup \Delta_r(P,\eta)) \backslash (\Delta_r(Q,\tau)\cap \Delta_r(P,\eta))) \leq C\rho^{n+1}$ (by Ahlfors regularity).  Therefore, 
 $|\gamma(P,\eta, r) - \gamma(Q,\tau, r)| \leq C \frac{\rho^{n+1}}{r^{n+1}}$. Hence, $$\begin{aligned}\int_{C_\rho(P,\eta)} |f_2(Q,\tau) - k|d\sigma(Q,\tau) \leq& \int_{C_\rho(P,\eta)} \int_\rho^{1/2}|\gamma(P,\eta, r) - \gamma(Q,\tau, r)|\frac{dr}{r} d\sigma(Q,\tau)\\ \leq& C \rho^{n+1} \int_\rho^{1/2} \frac{\rho^{n+1}}{r^{n+1}}r^{-1}dr \leq C \rho^{n+1}.\end{aligned}$$ Together with equation \eqref{showingitsBMO1}, this proves $\|f(Q,\tau)\|_{\mathrm{BMO}(C_{1/2}(0,0))} \leq C$. 

Let $k_{1/2} \colonequals \fint_{C_{1/2}(0,0)} f(Q,\tau)d\sigma$ (hence $k_{1/2} \leq \|\nu\|$). By the definition of $f(Q,\tau) \in \mathrm{BMO}$,  \begin{equation}\label{bmoisgood} \begin{aligned}\int_{C_{1/2}(0,0)} |f(Q,\tau) - k_{1/2}|^2 d\sigma \leq& c(n)\|f(Q,\tau)\|^2_{\mathrm{BMO}(C_{1/2}(0,0))} \Rightarrow\\ \int_{C_{1/2}(0,0)} |f(Q,\tau)|^2 d\sigma(Q,\tau) \leq& c(n)(\|f(Q,\tau)\|^2_{\mathrm{BMO}(C_{1/2}(0,0))}  + \|\nu\|^2).\end{aligned}\end{equation} Combine equations \eqref{estimateoffE} and \eqref{bmoisgood} to produce \begin{equation}\label{outsideEitissmall}
\nu(\{[\partial \Omega\cap  C_{1/2}(0,0)\backslash E]\times [0,1/2]\}) \leq C\varepsilon^{1/4}
\end{equation} which, with \eqref{donewithEpart}, is the desired result. 
\end{proof}

We are now ready to prove Proposition \ref{chordarctovanishingchordarc}, and by extension, complete the proof of Theorem \ref{loghvmoisvanishingchordarc}.

\begin{proof}[Proof of Proposition \ref{chordarctovanishingchordarc}]
Fix an $\varepsilon > 0$ and let $r_i \downarrow 0$ and $(Q_i,\tau_i) \in \partial \Omega \cap K$ for any compact set $K$. For each $i$, apply Lemma \ref{adaptedbigpiecesofgraphs} inside of $C_{r_i}(Q_i,\tau_i)$. This produces a sequence of functions, $\{\psi_i\}$, with bounded Lipschitz norms whose graphs are good approximations to $C_{r_i}(Q_i,\tau_i)\cap \partial \Omega$. We write, for ease of notation, $P_i \equiv P((Q_i,\tau_i), r_i)$. As there is no harm in a rotation (and we will be considering each $i$ seperately) we may assume that $P_i \equiv \{x_n =0\}$. We can define $\Phi_i(x,t)\colonequals \frac{1}{r_i}\psi_i(r_ix+q_i, r_i^2t + \tau_i) - \frac{(Q_i)_n}{r_i}$. Then, after a rotation which possibly depends on $i$, $\Omega_i$ and $\Phi_i$ satisfy the requirements of Lemma \ref{carlesonnormofomegabounded}. In particular, there exists an $i_0(\varepsilon) > 0$ such that for $i \geq i_0$, \begin{equation}\label{boundedateachscale}
\frac{\nu(\Delta_{r_i/2}(Q_i,\tau_i)\times [0,r_i/2])}{r_i^{n+1}}  \leq K_1(n,\|\nu\|,\varepsilon)\int_0^{3/4} \int_{C_{3/4}(0, 0)} \gamma_{\Phi_i}(x,t,r) dxdt\frac{dr}{r} + K_2(n,\|\nu\|)\varepsilon^{1/2}. 
\end{equation} It is important to note that while both constants above can depend on the dimension, $\|\nu\|$ and the Ahlfors regularity of $\Omega$, only $K_1(n, \|\nu\|, \varepsilon)$ will depend on $\varepsilon$ and both constants are independent of $i$. 

Conclusion (4) of Lemma \ref{adaptedbigpiecesofgraphs} implies that $f_i(x,t) \colonequals \int_0^{3/4}\gamma_{\Phi_i}(x,t,r) r^{-1}dr$ is uniformly integrable on $C_{3/4}(0,0) \cap \{x_n = 0\}$. Furthermore the $\Phi_i$ are uniformly bounded in the $\Lip(1,1/2)$ norm so by the Arzel\`a-Ascoli theorem there is some $\Phi_\infty$ such that $\Phi_i \Rightarrow \Phi_\infty$. It follows that $f_i(x,t) \rightarrow f_\infty(x,t) \colonequals  \int_0^{3/4}\gamma_{\Phi_\infty}(x,t,r) r^{-1}dr$. Thus, the dominated convergence theorem implies \begin{equation}\label{boundedlimsup}\limsup_{i\rightarrow \infty} \frac{\nu(C_{r_i/2}(Q_i,\tau_i) \times [0,r_i/2]) }{r_i^{n+1}}\leq K_1(n, \|\nu\|,\varepsilon) \int_{C_{3/4}(0, 0)}  f_\infty(x,t) dxdt + K_2(n,\|\nu\|) \varepsilon^{1/2}.\end{equation} On the other hand, condition (3) in Lemma \ref{adaptedbigpiecesofgraphs} and vanishing Reifenberg flatness tells us that the graph of $\Phi_i$ in the cylinder $C_{1}(0,0)$ is contained in increasingly smaller neighborhoods of $P_i$. Hence, $\Phi_\infty(x,t) \equiv 0$ inside of $C_1(0,0)$, and $f_\infty(x,t) \equiv 0$ in $C_{1/2}(0,0)$. Plugging this into equation \eqref{boundedlimsup} yields the bound,  \begin{equation}\label{limitofnuislessthanepsilon}
\limsup_{i\rightarrow \infty} \frac{1}{r_i^{n+1}}\nu(\Delta_{r_i/2}(Q_i,\tau_i) \times [0,r_i/2])  \leq K_2(n, \|\nu\|) \varepsilon^{1/2}.
\end{equation}
Since $\varepsilon$ is arbitrarily small the result follows.
\end{proof}

\section{Initial H\"older Regularity}\label{initialholderregularity}

We turn our attention to proving Proposition \ref{loghholdercontinuous} and assume that $\log(h)$ (or $\log(k^{(X_0,t_0)})$) is H\"older continuous. As before, $\Omega$ will be a $\delta$-Reifenberg flat parabolic regular domain. We will state and prove all the results in the infinite pole setting, however, almost no modifications are needed for kernels with a finite pole. 

This section is devoted to proving an initial H\"older regularity result:

\begin{prop}\label{initialholderregularityprop}
Let $\Omega\subset \mathbb R^{n+1}$ be a parabolic regular domain with $\log(h) \in \mathbb C^{\alpha, \alpha/2}(\mathbb R^{n+1})$ for $\alpha \in (0,1)$.  There is a $\delta_n > 0$ such that if $\delta_n \geq \delta > 0$ and $\Omega$ is $\delta$-Reifenberg flat then $\Omega$ is a $\mathbb C^{1+\alpha, (1+\alpha)/2}(\mathbb R^{n+1})$ domain. 
\end{prop} 

We follow closely the structure and exposition of Appendix \ref{proofofflatprop}, occasionally dealing with additional complications introduced by the H\"older condition. We should also mention that this section is strongly influenced by the work of Andersson and Weiss in \cite{anderssonweiss} and Alt and Caffarelli in \cite{altcaf}. To begin, we introduce flatness conditions (these are in the vein of  Definitions \ref{strongcurrentflatness} and \ref{weakcurrentflatness} but adapted to the H\"older regularity setting).

First, ``current flatness" (compare to Definition 7.1 in \cite{altcaf}). 
\begin{defin}\label{holdercurrentflatness}
For $0 < \sigma_i \leq 1, \kappa > 0$ we say that $u \in HCF(\sigma_1, \sigma_2, \kappa)$ in $C_\rho(X,t)$ in the direction $\nu \in \mathbb S^{n-1}$ if for $(Y,s) \in C_\rho(X,t)$
\begin{itemize}
\item $(X, t) \in \partial \{u > 0\}$
\item $u((Y,s)) = 0$ whenever $(Y-X)\cdot \nu \leq -\sigma_1 \rho$
\item $u((Y,s)) \geq  h(X,t)\left((Y-X)\cdot \nu-\sigma_2\rho\right)$ whenever  $(Y-X)\cdot \nu - \sigma_2\rho \geq 0$.
\item $|\nabla u(Y,s)| \leq h(X,t)(1+\kappa)$
\item $\mathrm{osc}_{(Q,\tau) \in \Delta_\rho(X,t)} h(Q,\tau) \leq \kappa h(X,t)$. 
\end{itemize}
\end{defin}

In some situations we will not have a estimate on the growth of $u$ in the positive side. Thus ``weak current flatness": 
\begin{defin}\label{weakholdercurrentflatness}
For $0 < \sigma_i \leq 1, \kappa > 0$ we say that $u \in \widetilde{HCF}(\sigma_1, \sigma_2, \kappa)$ in $C_\rho(X,t)$ in the direction $\nu \in \mathbb S^{n-1}$ if for $(Y,s) \in C_\rho(X,t)$
\begin{itemize}
\item $(X, t) \in \partial \{u > 0\}$
\item $u((Y,s)) = 0$ whenever $(Y-X)\cdot \nu \leq -\sigma_1 \rho$
\item $u((Y,s)) \geq  0$ whenever  $(Y-X)\cdot \nu - \sigma_2\rho \geq 0$.
\item $|\nabla u(Y,s)| \leq h(X,t)(1+\kappa)$
\item $\mathrm{osc}_{(Q,\tau) \in \Delta_\rho(X,t)} h(Q,\tau) \leq \kappa h(X,t)$. 
\end{itemize}
\end{defin}

Finally, in our proofs we will need to consider functions satisfying a ``past flatness" condition (first introduced for constant $h$ in \cite{anderssonweiss}, Definition 4.1). 

\begin{defin}\label{holderpastflatness}
For $0 < \sigma_i \leq 1, \kappa > 0$ we say that $u \in HPF(\sigma_1, \sigma_2, \kappa)$ in $C_\rho(X,t)$ in the direction $\nu \in \mathbb S^{n-1}$ if for $(Y,s) \in C_\rho(X,t)$
\begin{itemize}
\item $(X, t-\rho^2) \in \partial \Omega$
\item $u((Y,s)) = 0$ whenever $(Y-X)\cdot \nu \leq -\sigma_1 \rho$
\item $u((Y,s)) \geq  h(X,t-\rho^2)\left((Y-X)\cdot \nu-\sigma_2\rho\right)$ whenever  $(Y-X)\cdot \nu - \sigma_2\rho \geq 0$.
\item $|\nabla u(Y,s)| \leq h(X,t-\rho^2)(1+\kappa)$
\item $\mathrm{osc}_{(Q,\tau) \in \Delta_\rho(X,t)} h(Q,\tau) \leq \kappa h(X,t-\rho^2)$. 
\end{itemize}
\end{defin}

Proposition \ref{initialholderregularityprop} will be straightfoward once we prove three lemmas. The first two allow us to conclude greater flatness on a particular side given flatness on the other (they are analogues of Lemmas \ref{centeredflatonzerosideisflatonpositiveside} and \ref{centeredflatonbothsidesisextraflatonzeroside} in the H\"older setting). We will postpone their proofs until later subsections. 

\begin{lem}\label{currentholderflatonzerosideisflatonpositiveside}
 Let $0 < \kappa \leq \sigma \leq \sigma_0$ where $\sigma_0$ depends only on dimension. If $u \in \widetilde{HCF}(\sigma, 1/2,\kappa)$ in $C_\rho(Q,\tau)$ in the direction $\nu$, then there is a constant $C_1 > 0$ (depending only on dimension) such that $u\in HCF(C_1\sigma, C_1\sigma, \kappa)$ in $C_{\rho/2}(Q,\tau)$ in the direction $\nu$.
\end{lem}

\begin{lem}\label{currentholderflatonbothsidesisextraflatonzeroside}
 Let $\theta \in (0,1)$ and assume that $u \in HCF(\sigma, \sigma, \kappa)$ in $C_\rho(Q,\tau)$ in the direction $\nu$. There exists a constant $0 < \sigma_\theta < 1/2$ such that if $\sigma < \sigma_\theta$ and $\kappa \leq \sigma_\theta \sigma^2$ then $u \in \widetilde{HCF}(\theta \sigma, \theta \sigma, \kappa)$ in $C_{c(n)\rho \theta}(Q,\tau)$ in the direction $\overline{\nu}$ where $|\overline{\nu} - \nu| \leq C(n) \sigma$. Here $\infty > C(n), c(n) >0$ are constants depending only on dimension.
\end{lem}

The third lemma is an adaptation of Proposition \ref{maingradientbound} and tells us that $|\nabla u(X,t)|$ is bounded above by $h(Q,\tau)$ as $(X,t)$ gets close to $(Q,\tau)$. 

\begin{lem}\label{holdergradientbound}
Let $u, \Omega, h$ be as in Proposition \ref{initialholderregularityprop}. Then there exists a constant $C > 0$, which is uniform in $(Q,\tau) \in \partial \Omega$ on compacta, such that for all $r < 1/4$,  $$\sup_{(X,t) \in  C_r(Q,\tau)} |\nabla u(X,t)|  \leq h(Q,\tau) + Cr^{\min\{3/4, \alpha\}}.$$
\end{lem}

\begin{proof}
Fix an $R >> 1$ and $(Q,\tau) \in \partial \Omega$.  Lemma \ref{gradientrepresentation} says that there is a uniform constant $C > 0$ such that, for any $(X,t) \in \Omega$ with $\|(X,t) - (Q,\tau)\| \leq R/2$, \begin{equation}\label{gradientbound2}
|\nabla u(X,t)| \leq \int_{\Delta_{2R}(Q,\tau)} h(P,\eta) d\hat{\omega}^{(X,t)}(P,\eta) + C\frac{\|(X,t) -(Q,\tau)\|^{3/4}}{R^{1/2}}.
\end{equation}

Let $\|(X,t) - (Q,\tau)\| \leq r$ and let $k_0\in \mathbb N$ be such that $2^{-k_0-1} \leq r < 2^{-k_0}$. The inequality \begin{equation}\label{nottoomuchtoofar}1-\widehat{\omega}^{(X,t)}(C_{2^{-j}}(Q,\tau)) \leq C2^{3j/4} r^{3/4}, \forall j < k_0-2,\end{equation} follows from applying Lemma \ref{growthattheboundary} to $1-\widehat{\omega}^{(Y,s)}(C_{2^{-j}}(Q,\tau))$. We can write \begin{equation}\label{dyadicgradientbound}\begin{aligned} &\int_{\Delta_{2R}(Q,\tau)} h(P,\eta) d\hat{\omega}^{(X,t)}(P,\eta) \leq h(Q,\tau) + C\|h\|_{\mathbb C^{\alpha, \alpha/2}}\int_{\Delta_{4r}(Q,\tau)} (4r)^\alpha d\hat{\omega}^{(X,t)}\\ &+C\|h\|_{\mathbb C^{\alpha, \alpha/2}}\left(\int_{\Delta_{2R}(Q,\tau) \backslash \Delta_1(Q,\tau)} d\hat{\omega}^{(X,t)} +\sum_{j=0}^{k_0-2}\int_{\Delta_{ 2^{-j}}(Q,\tau)\backslash \Delta_{2^{-(j+1)}}(Q,\tau)}2^{-j\alpha}d\hat{\omega}^{(X,t)}\right).\end{aligned}\end{equation} We may bound $$\begin{aligned} \hat{\omega}^{(X,t)}(\Delta_{2R}(Q,\tau)\backslash \Delta_1(Q,\tau)) &\leq 1- \widehat{\omega}^{(X,t)}(\Delta_1(Q,\tau))\stackrel{eq. \eqref{nottoomuchtoofar}}{\leq} C_R r^{3/4}\\
\hat{\omega}^{(X,t)}(\Delta_{ 2^{-j}}(Q,\tau)\backslash \Delta_{2^{-(j+1)}}(Q,\tau)) &\leq 1-\widehat{\omega}^{(X,t)}(\Delta_{2^{-j}}(Q,\tau)) \stackrel{eq. \eqref{nottoomuchtoofar}}{\leq} C2^{3j/4} r^{3/4}.\end{aligned}$$ Plug these estimates into equation \eqref{dyadicgradientbound} to obtain \begin{equation}\label{estimateeachshell}
 \int_{\Delta_{2R}(Q,\tau)} h(P,\eta) d\hat{\omega}^{(X,t)}(P,\eta) \leq h(Q,\tau) + Cr^\alpha + C_R r^{3/4}+ Cr^{3/4}\sum_{j=0}^{[\log_2(r^{-1})]} 2^{j(3/4-\alpha)}.
\end{equation}

If $\alpha > 3/4$ then we can let the sum above run to infinity and evaluate to get the desired result. If $\alpha < 3/4$ then the geometric sum above evaluates to $\approx \frac{r^{\alpha-3/4}-1}{2^{3/4-\alpha}-1} \leq 4r^{\alpha-3/4}$. Plug this into estimate \eqref{estimateeachshell} and we are done. Finally, if $\alpha = 3/4$ then we may apply Lemma \ref{growthattheboundary} with $\varepsilon$ just slightly less than $1/4$ to get a version of equation \eqref{nottoomuchtoofar} with a different exponent, which will allow us to repeat the above argument without issue. 
\end{proof}

These three results allow us to iteratively improve the flatness of the free boundary.

\begin{cor}\label{improveflatnessandgradient}
 For every $\theta \in (0,1)$ there is a $\sigma_{n,\alpha} > 0$  and a constant $c_\theta \in (0,1)$, which depends only on $\theta, \alpha$ and $n$, such that if $u \in \widetilde{HCF}(\sigma, 1/2, \kappa)$ in $C_\rho(Q,\tau)$ in direction $\nu$ then $u \in \widetilde{HCF}(\theta \sigma, \theta\sigma, \theta^2\kappa)$ in $C_{c_\theta \rho}(Q,\tau)$ in direction $\overline{\nu}$ as long as $\sigma \leq \sigma_{n,\alpha}$, $\kappa \leq \sigma_{n,\alpha} \sigma^2$ and $\mathrm{osc}_{\Delta_{s\rho}(0,0)} h \leq \kappa s^\alpha h(0,0)$. Furthermore $\overline{\nu}$ satisfies $|\overline{\nu} - \nu| \leq C\sigma$ where $C$ depends only on dimension. Finally, there is constant $\tilde{C} > 1$, which depends only on $n$, and a number $\gamma \in (0,1)$, which depends only on $n,\alpha$, such that  $\tilde{C} c_\theta^\gamma \geq \theta \geq c_\theta^{\alpha/2}$. 
\end{cor}

\begin{proof}
We may assume that $\rho = 1, (Q,\tau) = (0,0)$ and $\nu = e_n$. By Lemma \ref{currentholderflatonzerosideisflatonpositiveside} we know that $u \in HCF(C_1\sigma, C_1\sigma,  \kappa)$ in $C_{1/2}(0,0)$ in direction $e_n$. Let $\theta_1 \in (0,1)$ be chosen later (to depend on the dimension and $\alpha$), and set $\sigma_{n,\alpha} \colonequals \sigma_{\theta_1}/C_1$ where $\sigma_{\theta_1}$ is the constant given by Lemma \ref{currentholderflatonbothsidesisextraflatonzeroside}.  Then if $\sigma < \sigma_{n,\alpha}$ and $\kappa \leq \sigma_{n,\alpha} \sigma^2$, Lemma \ref{currentholderflatonbothsidesisextraflatonzeroside} implies $u \in \widetilde{HCF}(C_1\theta_1 \sigma, C_1\theta_1\sigma, \kappa)$ in $C_{\tilde{c}\theta_1}(0,0)$ in the direction $\nu_1$ where $|\nu_1 - e_n| \leq C(n)\sigma$.

We turn to improving the bound on $\nabla u$. Observe that $U = \max\{|\nabla u(X,t)| - h(0,0), 0\}$ is an adjoint-subcaloric function in $C_1(0,0)$. Let $V$ be the solution to the adjoint heat equation such that $V = \kappa h(0,0) \chi_{x_n \geq -\sigma}$ on the adjoint parabolic boundary of $C_1(0,0)$. That $u \in \widetilde{HCF}(\sigma, 1/2, \kappa)$ implies $U \leq V$ on $\partial_p C_1(0,0)$. The maximum theorem and Harnack inequality then imply $U \leq V \leq (1-c)\kappa h(0,0)$ on all of $C_{1/2}(0,0)$, where $c$ depends only on dimension. Furthermore, by assumption, $$\mathrm{osc}_{\Delta_{\tilde{c}\theta_1}(0,0)} h \leq \kappa (\tilde{c}\theta_1)^\alpha h(0,0).$$ Hence, if $\theta_0 = \sqrt{1-c}$ and $\theta_1 = \min\{\theta_0/C_1, \theta_0^{2/\alpha}/\tilde{c}\}$ we have that $u \in \widetilde{HCF}(\theta_0\sigma, \theta_0 \sigma, \theta_0^2\kappa)$ in $C_{\tilde{c}\theta_1}(0,0)$ in direction $\nu_1$. 

Iterate this scheme $m$ times to get that $u \in \widetilde{HCF}(\theta_0^m \sigma, \theta_0^m\sigma, \theta_0^{2m}\kappa)$ in $C_{\tilde{c}^m \theta_1^m}(0,0)$ in the direction $\nu_m$ where $|e_n - \nu_m| \leq C \sigma \sum_{j=0}^{\infty} \theta_0^j \leq C(n) \sigma$. Let $m$ be large so that $\theta_0^m \leq \theta \leq \theta_0^{m-1}$. Then, since $\theta_0, \tilde{c}, \theta_1$ are constants which depend only on the dimension, $n$, and $\alpha$, we see that $(\tilde{c}\theta_1)^m = c_\theta$ where $c_\theta$ depends on $\theta, n, \alpha$. By the definition of $\theta_1$, there is some $\chi \geq 2/\alpha > 1$ (but which depends only on $n,\alpha$) such that $\tilde{c}\theta_1 = \theta_0^\chi$. Then $$(c_\theta)^{\alpha/2} \leq (c_\theta)^{1/\chi} = (\tilde{c}\theta_1)^{\frac{m}{\chi}} = \theta_0^m \leq \theta \leq \theta_0^{m-1} = \frac{1}{\theta_0} \theta_0^m = \frac{1}{\theta_0} c_\theta^{1/\chi}.$$

Letting $\tilde{C} = \frac{1}{\theta_0}$ and $\gamma = \frac{1}{\chi}$ implies that \begin{equation}\label{cthetanottoosmall}
c_\theta^{\alpha/2} \leq \theta \leq \tilde{C} c_\theta^{\gamma}.
\end{equation}
which are the desired bounds on $c_\theta$. 
\end{proof}

Proposition \ref{initialholderregularityprop} then follows from a standard argument:

\begin{proof}[Proof of Proposition \ref{initialholderregularityprop}]
We want to apply Corollary \ref{improveflatnessandgradient} iteratively. But before we can start the iteration, we must show that the hypothesis of that result are satisfied. 

By Lemma \ref{holdergradientbound} and that fact that $\log(h)$ is H\"older continuous, there exists a constant $C >0$ for any compact set $K$ such that $\forall (Q,\tau) \in \partial \Omega \cap K$ and $1/4 > \rho > 0$, \begin{equation}\label{oscillationbounds}\begin{aligned}
|\nabla u(X,t)| \leq& h(Q,\tau) + C\rho^{\alpha/2}\\
\mathrm{osc}_{\Delta_{s\rho}(Q,\tau)} h \leq&  C h(Q,\tau) s^\alpha \rho^{\alpha},\; \forall s\in (0,1].
\end{aligned}
\end{equation} 
Fix a compact set $K$ and a $\sigma_0 \leq \sigma_{n,\alpha}$ (where $\sigma_{n,\alpha}$ is as in Corollary \ref{improveflatnessandgradient}). As $\Omega$ is vanishing Reifenberg flat,  there exists an $R \colonequals R_{\sigma, K} > 0$ with the property that for all $\rho < R$ and $(Q,\tau) \in \partial \Omega \cap K$, there is a plane $P$ (containing a line parallel to the time axis and going through $(Q,\tau)$) such that $$D[C_\rho(Q,\tau)\cap P; C_\rho(Q,\tau) \cap \partial \Omega] \leq \rho \sigma.$$ 

Then fix a $\kappa_0 \leq  \sigma_{n,\alpha} \sigma_0^2$. Obviously, we can choose $\rho_0$ small enough (and smaller than $R_{\sigma, K}$ above) such that $h(Q,\tau) + C\rho_0^{\alpha/2} \leq (1+\kappa_0)h(Q,\tau)$ and $Ch(Q,\tau)\rho_0^\alpha \leq \kappa_0$. Furthermore, this $\rho_0$ can be choosen uniformly over all $(Q,\tau) \in K$. These observations,  combined with equation \eqref{oscillationbounds}, means for $0 < \rho \leq \rho_0$, $u \in \widetilde{HCF}(\sigma, \sigma, \kappa_0)$ in $C_\rho(Q,\tau)$ for some direction $\nu$.  

Then for any $(P,\eta) \in C_{\rho_0}(Q,\tau)$ there is a $\nu_0(P,\eta)$ such that $u\in \widetilde{HCF}(\sigma, \sigma, \kappa_0)$ in $C_{\rho_0}(P,\eta)$ in the direction $\nu_0(P,\eta)$. Let $\theta \in (0,1)$ and apply Corollary \ref{improveflatnessandgradient} $m$ times to get that $u \in \widetilde{HCF}(\theta^m \sigma, \theta^m \sigma, \theta^{2m}\kappa_0)$ in $C_{c_\theta^m \rho_0}(P,\eta)$ in the direction $\nu_m(P,\eta)$. We should check that the conditions of Corollary \ref{improveflatnessandgradient} are fulfilled at every step. In particular, that for any $m$, we have $\mathrm{osc}_{\Delta_{sc_\theta^m\rho_0}(P,\eta)} h \leq \theta^{2m}\kappa_0 h(P,\eta).$ Indeed, $$\mathrm{osc}_{\Delta_{sc_\theta^m\rho_0}(P,\eta)} h  \leq Ch(P,\eta)(s\rho_0 c_\theta^m)^\alpha \leq \kappa_0 h(P,\eta) c_\theta^{2m\alpha/2} \leq \kappa_0h(P,\eta)\theta^{2m}.$$ The last inequality above follows from equation \eqref{cthetanottoosmall}, and the penultimate one follows from the definition of $\rho_0$. 

Letting $m \rightarrow \infty$ it is clear that $\partial \Omega$ has a normal vector $\nu(P,\eta)$ at every $(P,\eta) \in C_{\rho_0}(Q,\tau)$ and $|\nu_m(P,\eta) - \nu(P,\eta)| \leq C_\theta \theta^m \sigma$ . Furthermore, if $(P', \eta') \in \Delta_{\rho_0c_\theta^m}(P, \eta) \backslash \Delta_{\rho_0 c_\theta^{m+1}}(P,\eta)$ then $|\nu_m(P,\eta) - \nu_m(P',\eta')| \leq C \theta^m \sigma$. Hence, $|\nu(P,\eta) - \nu(P',\eta')| \leq C\theta^m \sigma$. By equation \eqref{cthetanottoosmall} we know that $C\sigma \theta^m \leq C\theta^m \leq \tilde{C}c_\theta^{\gamma m}$.  Let $\beta \in (0,1)$ be such that $\beta (m+1) = \gamma m$. Hence, $|\nu(P,\eta) - \nu(P', \eta')| \leq C\|(P, \eta) - (P', \eta')\|^\beta$ which is the desired result. 
\end{proof}

\subsection{Flatness of the zero side implies flatness of the positive side: Lemma \ref{currentholderflatonzerosideisflatonpositiveside}}

Before we begin we need two technical lemmas. The first allows us to conclude regularity in the time dimension given regularity in the spatial dimensions. 

\begin{lem}\label{iflipschitzthenregularintime}
If $f$ satisfies the (adjoint)-heat equation in $\mathcal O$ and is zero outside $\mathcal O$ then $$\|f\|_{\mathbb C^{1,1/2}(\R^{n+1})} \leq c \|\nabla f\|_{L^\infty(\mathcal O)},$$ where $0 < c < \infty$ depends only on the dimension. \end{lem}

\begin{proof}
It suffices to show that for any $(X,t), (X,s) \in \mathcal O$ we have $|f(X,t) - f(X,s)| \leq C|s-t|^{1/2}$ where $C$ does not depend on $X, t$ or $s$. Assume $s > t$ and let $r \equiv \sqrt{s-t}$. Before our analysis we need a basic estimate: \begin{equation}\label{timederivative}
\left|\fint_{B'((X,t),r)} f_t dX\right| = \left| \fint_{B'((X,t),r)} \Delta f dX\right| = \left| \frac{c_n}{r} \fint_{\partial B'((X,t),r)} \nabla
 f \cdot \nu \right| \leq \frac{C_n\|\nabla f\|_{L^\infty(\mathcal O)}}{r}\end{equation} as long as $ B'((X,t),r) \colonequals \{(Y,t)\mid |Y-X| \leq r\} \subset \mathcal O$.

There are two cases:

\noindent {\bf Case 1:} $\{(Y, \tau) \mid |Y-X| \leq r, t \leq \tau \leq s\} \subset \mathcal O$. By Lipschitz continuity, $$\left|f(X,t) -\fint_{B'((X,t),r)} f(Y,t) dY\right| \leq C\|\nabla f\|_{L^\infty}r.$$ Note that by Fubini's theorem and the mean value theorem there is a $\tilde{t} \in [t,s]$ such that $$\begin{aligned}\left|\fint_{B'((X,t),r)} f(Y,t) dY- \fint_{B'((X,s),r)} f(Y,s) dY\right| =& \left|\fint_{\{|Y-X| \leq r\}}\int_t^s \partial_\tau f(Y,\tau)d\tau dY\right|\\ =& (s-t)\left|\fint_{B'((X,\tilde{t}),r)} f_\tau(Y,\tilde{t}) dY\right|.\end{aligned}$$ We may combine the two equations above to conclude, \begin{equation*}\begin{aligned}|f(X,t)-f(X,s)| &\leq C\|\nabla f\|_{L^\infty}r+ (s-t) \left|\fint_{B'((X,\tilde{t}),r)} f_\tau(Y,\tilde{t})dY\right|\\ &\stackrel{\mathrm{eqn}\; \eqref{timederivative}}{\leq} C\|\nabla f\|_{L^\infty}(r+  \frac{(s-t)}{r}) = C\|\nabla f\|_{L^\infty(\mathcal O)}\sqrt{s-t}.\end{aligned}\end{equation*}

\noindent {\bf Case 2:} $\{(Y, \tau) \mid |Y-X| \leq r, t \leq \tau \leq s\} \not\subset \mathcal O$. If neither $B'((X,t),r)$ or $B'((X,s),r)$ are contained in $\mathcal O$ then $|f(X,t) - f(X,s)| \leq |f(X,t)| + |f(X,s)| \leq \|\nabla f\|_{L^\infty} r$ by the Lipschitz continuity of $u$. Therefore, without loss of generality we may assume $B'((X,t),r) \subset \mathcal O$ . Let $t \leq t^* \leq s$ be such that $t^* = \inf_{t \leq a} B'((X,a),r) \not\subset \mathcal O$. Since $B'((X,a),r)\subset \mathcal O$ is an open condition, we can argue as in Case 1 so that, $|f(X,t)-f(X,t^*)| \leq C\|\nabla f\|_{L^\infty(\mathcal O)}(t^*-t)/r \leq C\|\nabla f\|_{L^\infty(\mathcal O)}\sqrt{s-t}$. On the other hand, as $B'((X,t^*), r) \not\subset \mathcal O$ we know $|f(X,t^*)| \leq  \|\nabla f\|_{L^\infty}|s-t|^{1/2}$.  Therefore, $|f(X,t)|\leq C \|\nabla f\|_{L^\infty}|s-t|^{1/2}$. Arguing similarly at $s$ we are done.
\end{proof}

This second lemma allows us to bound from below the normal derivative of a solution at a smooth point of $\partial \Omega$. 

\begin{lem}\label{holdernormalderivativeatregularpoint}
Let $(Q,\tau)\in \partial \Omega$ be such that there exists a tangent ball (in the Euclidean sense) $B$ at $(Q,\tau)$ contained in $\overline{\Omega}^c$. Then $$\limsup_{\Omega \ni (X,t) \rightarrow (Q,\tau)} \frac{u(X,t)}{d((X,t), B)} \geq h(Q,\tau).$$
\end{lem}

\begin{proof}
Without loss of generality set $(Q,\tau) = (0,0)$ and let $(X_k, t_k) \in \Omega$ be a sequence that achieves the supremum, $\ell$. Let $(Y_k, s_k) \in B$ be such that $d((X_k,t_k), B) = \|(X_k,t_k)-(Y_k, s_k)\| \equalscolon r_k$. Define $u_k(X,t) \colonequals \frac{u(r_kX + Y_k, r_k^2t + s_k)}{r_k}$, $\Omega_k \colonequals \{(Y,s)\mid Y= (X -Y_k)/r_k, s = (t-s_k)/r_k^2,\; \mathrm{s.t.}\; (X,t)\in \Omega\}$ and $h_k(X,t) \colonequals h(r_kX + Y_k, r_k^2t + s_k)$. Then \begin{equation}\int_{\R^{n+1}} u_k (\Delta \phi - \partial_t \phi)dXdt = \int_{\partial \Omega_k} h_k \phi d\sigma.\end{equation}

As $k\rightarrow \infty$ we can guarentee that $(r_kX + Y_k, r_k^2t + s_k) \in C_{1/100}(0,0)$. Apply Lemma \ref{holdergradientbound} to conclude that, for $(X,t) \in C_1(0,0)$, $$|\nabla u_k(X,t)| = |\nabla u(r_kX + Y_k, r_k^2t + s_k)| \leq Ch(0,0) + \left(\frac{1}{100}\right)^{\beta}.$$ In particular, the $u_k$ are uniformly Lipschitz continuous. By Lemma \ref{iflipschitzthenregularintime} the $u_k$ are bounded uniformly in $\mathbb C^{1,1/2}$. Therefore, perhaps passing to a subsequence, $u_k \rightarrow u_0$ uniformly on compacta. In addition, as there exists a tangent ball at $(0,0)$, $\Omega_k \rightarrow \{x_n >0\}$ in the Hausdorff distance norm (up to a rotation). We may assume, passing to a subsequence, that $\frac{X_k - Y_k}{r_k} \rightarrow Z_0, \frac{t_k - s_k}{r_k^2}\rightarrow t_0$ with $(Z_0, t_0) \in C_1(0,0)\cap \{x_n > 0\}$ and $u_0(Z_0, t_0) = \ell$. Furthermore, by the definition of supremum, for any $(Y,s) \in \{x_n > 0\}$ we have \begin{equation}\begin{aligned} u_0(Y,s) =& \lim_{k\rightarrow \infty} u(r_kY + y_k, r_k^2s + s_k)/r_k \\ \leq& \lim_{k\rightarrow \infty} \ell \frac{\mathrm{pardist}((r_kY + y_k, r_k^2s + s_k), B)}{r_k}\\ =& \lim_{k\rightarrow \infty} \ell \mathrm{pardist}((Y,s), B_k)\\ =& \ell y_n, \end{aligned}\end{equation}where $B_k$ is defined like $\Omega_k$ above.

Let $\phi \in C_0^{\infty}(\mathbb R^{n+1})$ be positive, then \begin{equation}\label{holderu0bound}\begin{aligned} \int_{\{x_n > 0\}} \ell x_n (\Delta \phi - \partial_t \phi)dXdt &\geq \int_{\{x_n > 0\}} u_0(X,t)  (\Delta \phi - \partial_t \phi)dXdt\\ &= \lim_{k \rightarrow \infty} \int_{\Omega_k} u_k(X,t) (\Delta \phi - \partial_t \phi)dXdt \\&=  \lim_{k\rightarrow \infty} \int_{\partial \Omega_k} h_k \phi d\sigma.\end{aligned}\end{equation} Integrating by parts yields \begin{equation*}\begin{aligned} \ell \int_{\{x_n = 0\}} \phi dx dt &=  \int_{\{x_n > 0\}} \ell x_n  (\Delta \phi - \partial_t \phi)dXdt \\ &\stackrel{eqn.\; \eqref{holderu0bound}}{\geq} \lim_{k\rightarrow \infty}\int_{\partial \Omega_k} h_k \phi d\sigma\\ &\geq \lim_{k\rightarrow \infty} \left(\inf_{(P,\eta) \in \supp \phi} h(r_kP+Y_k, r_k^2\eta+ s_k)\right) \int_{\{x_n =0\}} \phi dxdt\end{aligned}\end{equation*} Hence,  $\ell \geq \lim_{k\rightarrow \infty} h(Y_k, s_k) - Cr_k^\alpha,$ by the H\"older continuity of $h$.  As $(Y_k, s_k) \rightarrow (Q,\tau)$ and $r_k \downarrow 0$ the desired result follows. 
\end{proof}

We will first show that for ``past flatness",  flatness on the positive side gives flatness on the zero side.

\begin{lem}\label{holderflatonzerosideisflatonpositiveside}[Compare with Lemma 5.2 in \cite{anderssonweiss}]
Let $0 < \kappa \leq \sigma/4 \leq \sigma_0$ where $\sigma_0$ depends only on dimension. Then if $u\in HPF(\sigma, 1, \kappa)$ in $C_\rho(\tilde{X},\tilde{t})$ in the direction $\nu$, there is a constant $C$ such that $u \in HPF(C\sigma, C\sigma, 3\kappa)$ in $C_{\rho/2}(\tilde{X}+\alpha \nu, \tilde{t})$ in the direction $\nu$ for some $|\alpha| \leq C\sigma \rho$. 
\end{lem}
 
 \begin{proof}
Let $(\tilde{X},\tilde{t})= (0,0), \rho = 1$ and $\nu = e_n$. First we will construct a regular function which touches $\partial \Omega$ at one point.  

Define $$\eta(x, t) = e^{\frac{16(|x|^2 + |t+1|)}{16(|x|^2 + |t+1|)- 1}}$$ for $16(|x|^2 + |t+1|) < 1$ and $\eta(x,t) \equiv 0$ otherwise. Let $D \colonequals \{(x, x_n, t) \in C_1(0,0) \mid x_n > -\sigma + s\eta(x,t)\}$. Now pick $s$ to be the largest such constant that $C_1(0,0) \cap \Omega \subset D$. As $(0,-1) \in \partial \{u > 0\}$, there must be a touching point $(X_0, t_0) \in \partial D \cap \partial \Omega \cap \{-1\leq t \leq -15/16\}$ and $s \leq \sigma$. 

Define the barrier function $v$ as follows:
\begin{equation}
\begin{aligned}
\Delta v + \partial_tv =0 &\: \mathrm{in}\; D, \\
v = 0 &\; \mathrm{in}\; \partial_p D \cap C_1(0,0)\\
v = h(0,-1)(1+\sigma)(\sigma +x_n)&\; \mathrm{in}\; \partial_p D \cap \partial C_1(0,0).
\end{aligned}
\end{equation}

Note that on $\partial_p D\cap C_1(0)$ we have $u = 0$ because $D$ contains the positivity set. Also, as $|\nabla u| \leq h(0,-1)(1+\kappa) \leq h(0,-1)(1+\sigma)$, it must be the case that $u(X,t) \leq h(0,-1)(1+\sigma)\max\{0, \sigma + x_n\}$ for all $(X,t) \in C_1(0,0)$. As $v \geq u$ on $\partial_pD$ it follows that $v \geq u$ on all of $D$ (by the maximum principle for subadjoint-caloric functions). We now want to estimate the normal derivative of $v$ at $(X_0,t_0)$. To estimate from below, apply Lemma \ref{holdernormalderivativeatregularpoint}, \begin{equation}\label{lowerbfondvholder}h(X_0,t_0) \leq \limsup_{(X,t) \rightarrow (X_0,t_0)} \frac{u(X,t)}{\mathrm{pardist}((X,t), B)} \leq -\partial_\nu v(X_0,t_0)\end{equation} where $\nu$ is the normal pointing out of $D$ at $(X_0,t_0)$ and $B$ is the tangent ball at $(X_0,t_0)$ to $D$ contained in $D^c$.

To estimate from above, first consider $F(X,t) \colonequals (1+\sigma)h(0,-1)(\sigma + x_n)-v$.  On $\partial_pD$, \begin{equation}\label{holdersandwichv}-(1+\sigma)h(0,-1)\sigma \leq v- (1+\sigma)h(0,-1)x_n \leq (1+\sigma)h(0,-1)\sigma\end{equation} thus (by the maximum principle) $0 \leq F(X,t) \leq 2(1+\sigma)h(0,-1)\sigma$. As $\partial D$ is piecewise smooth domain, standard parabolic regularity gives $\sup_{D} |\nabla F(X,t)| \leq K(1+\sigma)h(0,-1)\sigma$. Note, since $s \leq \sigma$, that $-\sigma + s\eta(x,t)$ is a function whose $\mathrm{Lip}(1,1)$ norm is bounded by a constant. Therefore, $K$ does not depend on $\sigma$. 

Hence,  \begin{equation}\label{holdernormalderivativeofv}
\begin{aligned}
|\nabla v| - (1+\sigma)h(0,-1) \leq& |\nabla v -(1+\sigma)h(0,-1)e_n|  \leq K(1+\sigma)h(0,-1)\sigma\\
\stackrel{\mathrm{eqn}\; \eqref{lowerbfondvholder}}{\Rightarrow} h(X_0,t_0) \leq&  -\partial_\nu v(Z) \leq (1+K\sigma)(1+\sigma)h(0,-1).\end{aligned}\end{equation}

We want to show that $u \geq v - \tilde{K}(1+\sigma)h(0,-1)\sigma x_n$ for some large constant $\tilde{K}$ to be choosen later, depending only on the dimension. Let $\tilde{Z} \colonequals (Y_0, s_0)$ where $s_0 \in (-3/4, 1), |y_0| \leq 1/2$ and $(Y_0)_n = 3/4$ and assume, in order to obtain a contradiction, that $u \leq v - \tilde{K}(1+\sigma)h(0,-1)\sigma x_n$ at every point in $\{(Y, s_0)\mid |Y-Y_0| \leq 1/8\}$. We construct a barrier function, $w \equiv w_{\tilde{Z}}$, defined by \begin{equation*}
\begin{aligned}
\Delta w+ \partial_tw=0 \: \mathrm{in}\; D\cap \{t < s_0\},& \\
w= x_n \; \mathrm{on}\; \partial_p (D\cap \{t < s_0\}) \cap \{(Y,s_0) \mid |Y-Y_0| < 1/8\},&\\
w = 0 \; \mathrm{on}\; \partial_p (D \cap \{t < s_0\})  \backslash  \{(Y,s_0)\mid |Y-Y_0| < 1/8\}.&
\end{aligned}
\end{equation*}

By our initial assumption (and the definition of $w$), $v - \tilde{K}\sigma(1+\sigma)h(0,-1) w \geq u$ on $\partial_p (D\cap\{t < s_0\})$ and, therefore, $v - \tilde{K}\sigma(1+\sigma)h(0,-1) w\geq u $ on all of $D\cap \{t < s_0\}$. Since $t_0 \leq -15/16$ we know $(X_0,t_0) \in \partial_p(D\cap \{t < s_0\})$. Furthermore,  the Hopf lemma gives an $\alpha > 0$ (independent of $\tilde{Z}$) such that $\partial_\nu w(X_0,t_0) \leq -\alpha$. With these facts in mind, apply Lemma \ref{holdernormalderivativeatregularpoint} at $(X_0,t_0)$ and recall estimate \eqref{holdernormalderivativeofv} to estimate, \begin{equation}\label{holdernormalderivativeofw}\begin{aligned} h(X_0,t_0) =& \limsup_{(X,t)\rightarrow (X_0,t_0)} \frac{u(X,t)}{\mathrm{pardist}((X,t), B)}\\ \leq& -\partial_\nu v(X_0,t_0) +K(1+\sigma)h(0,-1)\sigma \partial_\nu w(X_0,t_0)\\ \leq& (1 + K\sigma)(1+\sigma)h(0,-1) - \tilde{K}\alpha(1+\sigma)h(0,-1) \sigma \leq (1-2\sigma)h(0,-1)\end{aligned}\end{equation} if $\tilde{K} \geq (K+3)/\alpha$. On the other hand, our assumed flatness tells us that $h(X_0,t_0) - h(0,-1)  \geq -\kappa h(0,-1) \geq -\sigma h(0,-1)$. Together with equation \eqref{holdernormalderivativeofw} this implies $-\sigma h(0,-1) \leq -2\sigma h(0,-1)$, which is absurd. 

Hence, there exists  a point, call it $(\overline{Y},s_0)$, such that $|\overline{Y}-Y_0| \leq 1/8$ and $$(u-v)(\overline{Y},s_0) \geq -\tilde{K}\sigma(1+\sigma)h(0,-1) (\overline{Y})_n \stackrel{(\overline{Y})_n \leq 1}{\geq} -\tilde{K}(1+\sigma)h(0,-1)\sigma.$$ Apply the parabolic Harnack inequality to obtain, $$\begin{aligned}&\inf_{\{|X-Y_0| < 1/8\}} (u-v)(X,s_0-1/32) \geq c_n \sup_{\{|\tilde{X}-Y_0|< 1/8\}} (u-v)(\tilde{X},s_0) \geq -c_n\tilde{K}(1+\sigma)h(0,-1)\sigma\\
&\stackrel{\mathrm{eqn}\; \eqref{holdersandwichv}}{\Rightarrow} u(X,s_0-1/32) \geq (1+\sigma)h(0,-1)(x_n-\sigma)- \overline{C}(1+\sigma)h(0,-1)\sigma, \end{aligned}$$ for all $X$ such that $|X-Y_0| < 1/8$ and $\overline{C}$ which depends only on the dimension. Ranging over all $s_0 \in (-3/4, 1)$ and $|y_0| \leq 1/2$ the above implies  $$u(X,t) \geq (1+\sigma)h(0,-1)x_n - C(1+\sigma)h(0,-1)\sigma,$$ whenever $(X,t)$ satisfies $|x| < 1/2, |x_n -3/4| < 1/8, t\in (-1/2, 1/2)$. As $|\nabla u| \leq (1+\sigma)h(0,-1)$ we can conclude, for any $(X,t)$ such that $|x| < 1/2, t\in (-1/2,1/2)$ and $3/4 \geq x_n \geq C\sigma$, that \begin{equation}\label{holderincreasedpositiveflatness} u(X,t) \geq u(x, 3/4, t) - (1+\sigma)h(0,-1)(3/4-x_n) \geq (1+\sigma)h(0,-1)(x_n - C\sigma).\end{equation}

We now need to find an $\alpha$ such that $(0, \alpha, -1/4) \in \partial \Omega$. By the initial assumed flatness, and equation \eqref{holderincreasedpositiveflatness}, $\alpha \in \R$ exists and $-\sigma \leq \alpha \leq C\sigma$ (here we pick $\sigma_0$ is small enough such that $C\sigma_0 < 1/4$).  

Furthermore, by the assumed flatness in $C_1(0,0)$, \begin{equation}\label{supconditiononnablau}\begin{aligned} &h(0,\alpha, -1/4) - h(0,0,-1) \geq -\kappa h(0,0,-1)\\
&\Rightarrow 3 h(0,\alpha, -1/4) \geq (1-\kappa)^{-1} h(0,\alpha,-1/4) \geq h(0,0,-1).\end{aligned}\end{equation} Hence, $$\mathrm{osc}_{C_{1/2}(0,\alpha,0)} h \leq \mathrm{osc}_{C_{1}(0,0)} h \leq \kappa h(0,0,-1) \stackrel{\mathrm{eqn}\; \eqref{supconditiononnablau}}{\leq} 2\kappa h(0,\alpha,-1/4).$$

In summary we know,\begin{itemize}
\item $(0, \alpha, -1/4) \in \partial \Omega,\; |\alpha| < C\sigma$
\item $x_n - \alpha \leq -3C\sigma/2 \Rightarrow x_n \leq -\sigma\Rightarrow u(X,t) = 0$. 
\item When $x_n - \alpha \geq 2C\sigma \Rightarrow x_n \geq C\sigma$ hence equation \eqref{holderincreasedpositiveflatness} implies $u(X,t) \geq ((1+\sigma)h(0,-1))(x_n - C\sigma)  \geq (1+2\kappa)h(0,\alpha, -1/4)(x_n-\alpha - 2C\sigma)$.
\item As written above $\mathrm{osc}_{C_{1/2}(0,\alpha,0)} h \leq 3\kappa h(0,\alpha,-1/4)$. 
\item Finally $\sup_{C_{1/2}(0,\alpha,0)} |\nabla u| \leq \sup_{C_1(0,0)}|\nabla u| \leq (1+\kappa)h(0,-1) \stackrel{\mathrm{eqn}\; \eqref{supconditiononnablau}}{\leq}\frac{1+\kappa}{1-\kappa} h(0,\alpha, -1/4) \leq (1+3\kappa)h(0,\alpha, -1/4)$. 
\end{itemize}
Therefore $u \in HPF(2C\sigma, 2C\sigma, 3\kappa)$ in $C_{1/2}(0, \alpha, 0)$ which is the desired result. 
 \end{proof}
 
Lemma \ref{currentholderflatonzerosideisflatonpositiveside} is the current version of the above and follows almost identically. Thus we will omit the full proof in favor of briefly pointing out the ways in which the argument differs. 

\begin{lem*}[Lemma \ref{currentholderflatonzerosideisflatonpositiveside}]
 Let $0 < \kappa \leq \sigma \leq \sigma_0$ where $\sigma_0$ depends only on dimension. If $u \in \widetilde{HCF}(\sigma, 1/2,\kappa)$ in $C_\rho(Q,\tau)$ in the direction $\nu$, then there is a constant $C_1 > 0$ (depending only on dimension) such that $u\in HCF(C_1\sigma, C_1\sigma, \kappa)$ in $C_{\rho/2}(Q,\tau)$ in the direction $\nu$.
\end{lem*}

\begin{proof}[Proof of Lemma \ref{currentholderflatonzerosideisflatonpositiveside}]
  We begin in the same way; let $(Q,\tau)= (0,0), \rho = 1$ and $\nu = e_n$. Then we recall the smooth function   $$\eta(x, t) = e^{\frac{16(|x|^2 + |t+1|)}{16(|x|^2 + |t+1|)- 1}}$$ for $16(|x|^2 + |t+1|) < 1$ and $\eta(x,t) \equiv 0$ otherwise. Let $D \colonequals \{(x, x_n, t) \in C_1(0,0) \mid x_n > -\sigma + s\eta(x,t)\}$. Now pick $s$ to be the largest such constant that $C_1(0,0) \cap \Omega \subset D$. Since $|x_n| > 1/2$ implies that $u(X,t) > 0$ there must be some touching point  $(X_0, t_0) \in \partial D \cap \partial \Omega \cap \{-1\leq t \leq -15/16\}$. Furthermore, we can assume that $s < \sigma + 1/2 < 2$. 
  
The proof then proceeds as above until equation \eqref{holderincreasedpositiveflatness}. In the setting of  ``past flatness" we need to argue further; the boundary point is at the bottom of the cylinder, so after the cylinder shrinks we need to search for a new boundary point. However, in current flatness the boundary point is at the center of the cylinder so after equation \eqref{holderincreasedpositiveflatness} we have completed the proof of Lemma \ref{currentholderflatonzerosideisflatonpositiveside}. In particular, this explains why there is no increase from $\kappa$ to $3\kappa$ in the current setting. \end{proof}

\subsection{Flatness on Both Sides Implies Greater Flatness on the Zero Side: Lemma \ref{currentholderflatonbothsidesisextraflatonzeroside}} In this section we prove Lemma \ref{currentholderflatonbothsidesisextraflatonzeroside}. The outline of the argument is as follows: arguing by contradiction, we obtain a sequence $u_k$ whose free boundaries, $\partial \{u_k > 0\}$, approach the graph of a function $f$. Then we prove that this function $f$ is $C^{\infty}$ which will lead to a contradiction. 

Throughout this subsection, $\{u_k\}$ is a sequence of adjoint caloric functions such that $\partial \{u_k > 0\}$ is a parabolic regular domain and such that, for all $\varphi \in C^\infty_c(\mathbb R^{n+1})$, $$\int_{\{u_k > 0\}} u_k (\Delta \varphi - \partial_t \varphi) dXdt = \int_{\partial \{u_k > 0\}} h_k \varphi d\sigma.$$ We will also assume the $h_k$ satisfy  $\|\log(h_k)\|_{\mathbb C^{\alpha, \alpha/2}} \leq C\|\log(h)\|_{\mathbb C^{\alpha, \alpha/2}}$ and $h_k(0,0) = h(0,0)$. While we present these arguments for general $\{u_k\}$ it suffices to think of $u_k(X,t) \colonequals \frac{u(r_kX, r_k^2t)}{r_k}$ for some $r_k \downarrow 0$.

\begin{lem}\label{holderblowupfunctionisagraph}[Compare with Lemma 6.1 in \cite{anderssonweiss}]
Suppose that $u_k \in HCF(\sigma_k,\sigma_k, \tau_k)$ in $C_{\rho_k}(0,0)$ in direction $e_n$, with $\sigma_k \downarrow 0$ and $\tau_k/\sigma_k^2\rightarrow 0$. Define $f_k^+(x,t) = \sup \{d\mid (\rho_kx, \sigma_k\rho_kd, \rho_k^2t) \in \{u_k = 0\}\}$ and $f_k^-(x,t) = \inf \{d\mid (\rho_kx, \sigma_k\rho_kd, \rho_k^2t) \in \{u_k > 0\}\}$. Then, passing to subsequences, $f_k^+, f_k^-\rightarrow f$ in $L_{\mathrm{loc}}^\infty(C_1(0,0))$ and $f$ is continuous. 
\end{lem}

\begin{proof}
By scaling each $u_k$ we may assume $\rho_k \equiv 1$. Then define $$D_k \colonequals \{(y,d,t)\in C_1(0,0)\mid (y, \sigma_k d, t) \in \{u_k > 0\}\}.$$ 
Let $$f(x, t) \colonequals \liminf_{\stackrel{(y,s) \rightarrow (x, t)}{ k\rightarrow \infty}} f_k^-(y,s),$$ so that, for every $(y_0, t_0)$, there exists a $(y_k, t_k) \rightarrow (y_0, t_0)$ such that $f_k^-(y_k, t_k) \stackrel{k\rightarrow \infty}{\rightarrow} f(y_0, t_0)$. 

Fix a $(y_0,t_0)$ and note, as $f$ is lower semicontinuous, for every $\varepsilon > 0$, there exists a $\delta > 0, k_0 \in \mathbb N$ such that $$\{(y, d, t) \mid |y-y_0| < 2\delta, |t-t_0| < 4\delta^2, d \leq f(y_0, t_0) -\varepsilon\}\cap \overline{D}_k = \emptyset,\; \forall k \geq k_0.$$
Consequently \begin{equation}\label{fkflatness} x_n - f(y_0, t_0) \leq - \varepsilon \Rightarrow u_k(x, \sigma_k x_n, s) = 0,\; \forall (X,s) \in C_{2\delta}(Y_0, t_0).\end{equation} Together with the definition of $f$, equation \eqref{fkflatness} implies that there exist $\alpha_k \in \mathbb R$ with $|\alpha_k| < 2\varepsilon$ such that $(y_0, \sigma_k(f(y_0,t_0) + \alpha_k), t_0-\delta^2)\in \partial \{u_k > 0\}$. Furthermore, for any $(Y,s) \in C_1(0,0)$ by assumption $h(0,0) - h_k(Y,s) \leq \tau_k h(0,0) \Rightarrow h(0,0) \leq (1+\frac{4}{3}\tau_k)h_k(Y,s)$ for $k$ large enough. This observation, combined with equation \eqref{fkflatness} allows us to conclude, $u_k(\cdot, \sigma_k \cdot, \cdot) \in HPF(3\sigma_k \frac{\varepsilon}{\delta}, 1, 4\tau_k)$ in $C_{\delta}(y_0, \sigma_k (f(y_0, t_0)+\alpha_k) , t_0)$, for $k$ large enough. 

As $\tau_k/\sigma_k^2 \rightarrow 0$ the conditions of Lemma \ref{holderflatonzerosideisflatonpositiveside} hold for $k$ large enough.  Therefore, $u_k(\cdot, \sigma_k \cdot, \cdot) \in HPF(C\sigma_k\frac{\varepsilon}{\delta}, C\sigma_k\frac{\varepsilon}{\delta},8\tau_k)$ in $C_{\delta/2}(y_0, \sigma_k f(y_0, t_0) + \tilde{\alpha}_k, t_0)$ where $|\tilde{\alpha}_k| \leq C\sigma_k\varepsilon$. Thus if $z_n - (\sigma_k f(y_0, t_0) + \tilde{\alpha}_k) \geq C\varepsilon \sigma_k/2$  then $u_k(z, \sigma_k z_n, t) > 0$ for $(Z,t) \in C_{\delta/2}(y_0, \sigma_k f_k^-(y_0, t_0) + \tilde{\alpha}_k, t_0)$. In other words \begin{equation}\label{pluslessthanminus}\sup_{(Z,s) \in C_{\delta/2}(y_0, \sigma_k f(y_0, t_0) + \tilde{\alpha}_k, t_0 )} f^+_k(z, s) \leq f(y_0, t_0) + 3C\varepsilon.\end{equation} As $f^+_k \geq f^-_k$, if $$\tilde{f}(y_0,t_0) \colonequals  \limsup_{\stackrel{(y,s) \rightarrow (y_0, t_0)}{ k\rightarrow \infty}} f_k^+(y,s),$$ it follows (in light of equation \eqref{pluslessthanminus}) that $\tilde{f} = f$.  Consequently, $f$ is continuous and $f^+_k, f^-_k \rightarrow f$ locally uniformly on compacta. \end{proof} 

We now show that $f$ is given by the boundary values of $w$, a solution to the adjoint heat equation in $\{x_n > 0\}$. 

\begin{lem}\label{holderboundaryvaluesofw}[Compare with Proposition 6.2 in \cite{anderssonweiss}]
Suppose that $u_k \in HCF(\sigma_k,\sigma_k, \tau_k)$ in $C_{\rho_k}(0,0)$, with $\rho_k \geq 0, \sigma_k \downarrow 0$ and $\tau_k/\sigma_k^2\rightarrow 0$. Further assume that, after relabeling, $k$ is the subsequence given by Lemma \ref{holderblowupfunctionisagraph}. Define $$w_k(x, d,t) \colonequals \frac{u_k(\rho_kx, \rho_kd, \rho_k^2t) - (1+\tau_k)h(0,0)\rho_kd}{h(0,0)\sigma_k}.$$ Then, $w_k$ is bounded on $C_1(0,0)\cap \{x_n > 0\}$ (uniformly in $k$) and converges, in the $\mathbb C^{2,1}$-norm, on compact subsets of $C_1(0,0)\cap \{x_n > 0\}$ to $w$. Furthermore, $w$ is a solution to the adjoint-heat equation and $w(x, d,t)$ is non-increasing in $d$ when $d > 0$. Finally $w(x,0,t) = -f(x,t)$ and $w$ is continuous in $\overline{C_{1-\delta}(0,0)\cap \{x_n > 0\}}$ for any $\delta > 0$. \end{lem}

\begin{proof}
As before we rescale and set $\rho_k \equiv 1$.  Since $|\nabla u_k| \leq h(0,0)(1+\tau_k)$ and $x_n \leq -\sigma_k \Rightarrow u_k = 0$ it follows that $u_k(X,t) \leq h(0,0)(1+\tau_k)(x_n + \sigma_k)$. Which implies $w_k(X,t) \leq 1+ \tau_k$. On the other hand, when $0 <  x_n \leq \sigma_k$ we have $u_k(X,t) - (1+\tau_k)h(0,0)x_n \geq -(1+\tau_k)h(0,0)x_n \geq -(1+\tau_k)h(0,0)\sigma_k,$ hence $w_k \geq -1-\tau_k$. Finally, if $x_n \geq \sigma_k$ we have $u_k(X,t) - (1+\tau_k)h(0,0)x_n \geq (1+\tau_k)h(0,0)(x_n - \sigma_k) - (1+\tau_k)h(0,0)x_n\Rightarrow w_k \geq -(1+\tau_k)$. Thus, for $k$ large enough, $|w_k| \leq 2$ in $C_1(0,0)\cap\{x_n > 0\}$.

By definition, $w_k$ is a solution to the adjoint-heat equation in $C_1(0,0) \cap \{x_n > \sigma_k\}$. So for any $K \subset \subset \{x_n > 0\}$ the $\{w_k\}$ are, for large enough $k$, a uniformly bounded sequence of solutions to the adjoint-heat equation on $K$. As $|w_k| \leq 2$, standard estimates for parabolic equations tell us that $\{w_k\}$ is uniformly bounded in $\mathbb C^{2+\alpha, 1+\alpha/2}(K)$. Therefore, perhaps passing to a subsequence, $w_k \rightarrow w$ in $\mathbb C^{2,1}(K)$. Furthermore, $w$ must also be a solution to the adjoint heat equation in $K$ and $|w| \leq 1$. A diagonalization argument allows us to conclude that $w$ is adjoint caloric on all of $\{x_n > 0\}$. 

Compute that $\partial_n w_k = (\partial_n u_k - (1+\tau_k)h(0,0))/(h(0,0)\sigma_k) \leq 0$, which implies $\partial_n w \leq 0$ on $\{x_n > 0\}$. As such, $w(x,0,t) \colonequals \lim_{d\rightarrow 0^+} w(x, d, t)$ exists. We will now show that this limit is equal to $-f(x,t)$ (which, recall, is a continuous function). If true, then regularity theory for adjoint-caloric functions immediately implies that $w$ is continuous in $\overline{C_{1-\delta}(0,0)\cap \{x_n > 0\}}$.

First we show that the limit is less that $-f(x,t)$. Let $\varepsilon > 0$ and pick $0 < \alpha \leq 1/2$ small enough so that $|w(x,\alpha, t) - w(x,0,t)| < \varepsilon$. For $k$ large enough we have $\alpha/\sigma_k > f(x,t) + 1 > f^-_k(x,t)$ therefore, \begin{equation}\label{wupperboundfinal}\begin{aligned} w(x,0,t) &\leq w(x,\alpha, t) + \varepsilon = w_k(x, \sigma_k \frac{\alpha}{\sigma_k},t) + \varepsilon + o_k(1)\\
&= (w_k(x, \sigma_k \frac{\alpha}{\sigma_k}, t) - w_k(x, \sigma_k f_k^-(x,t), t)) + w_k(x, \sigma_k f_k^-(x,t), t) + \varepsilon +o_k(1)\\
&\stackrel{\partial_n w_k \leq 0}{\leq} w_k(x, \sigma_k f_k^-(x,t), t)  + o_k(1) + \varepsilon.
\end{aligned}
\end{equation}

By definition, $w_k(x, \sigma_k f_k^-(x,t),t) = -(1+\tau_k)f_k^-(x,t)\rightarrow -f(x,t)$ uniformly in $C_{1-\delta}(0,0)$. In light of equation \eqref{wupperboundfinal}, this observation implies $w(x,0,t) \leq - f(x,t)+ \varepsilon$. Since $\varepsilon > 0$ was arbitrary we have $w(x,0,t) \leq -f(x,t)$. 

To show $w(x,0,t) \geq -f(x,t)$ we first define, for $S> 0, k \in \mathbb N,$ $$\tilde{\sigma}_k =\frac{1}{S} \sup_{(Y,s) \in C_{2S\sigma_k}(x, \sigma_kf_k^-(x, t-S^2\sigma_k^2), t-S^2\sigma_k^2)}(f_k^-(x,t-S^2\sigma_k^2)-f_k^-(y,s)).$$ Observe that if $k$ is large enough (depending on $S, \delta$) then $(x, t-S^2\sigma_k^2) \in C_{1-\delta}(0,0)$.

 Then, by construction, $\forall (Y,s) \in C_{2S\sigma_k}(x, \sigma_kf_k^-(x, t-S^2\sigma_k^2), t-S^2\sigma_k^2)$, $$y_n -\sigma_kf_k^-(x,t-S^2\sigma_k^2) \leq -S\sigma_k\tilde{\sigma}_k  \Rightarrow y_n \leq \sigma_k f_k^-(y, s)\Rightarrow u_k(Y,s) = 0.$$ Bounding the oscillation of $h_k$ as in the proof of Lemma \ref{holderblowupfunctionisagraph}, we have $u_k \in HPF(\tilde{\sigma}_k, 1, 4\tau_k) \subset HPF(\overline{\sigma}_k, 1,4\tau_k)$ in $C_{S\sigma_k}(x, \sigma_kf_k^-(x, t-S^2\sigma_k^2), t)$, where $\overline{\sigma}_k = \max\{16\tau_k, \tilde{\sigma}_k\}$. Note, by Lemma \ref{holderblowupfunctionisagraph}, $\tilde{\sigma}_k \rightarrow 0$ and, therefore, $\overline{\sigma}_k \rightarrow 0$. 

Apply Lemma \ref{holderflatonzerosideisflatonpositiveside} to conclude that \begin{equation}\label{ukispastlfat}\hbox{$u_k \in HPF(C\overline{\sigma}_k, C\overline{\sigma}_k, 8\tau_k)$ in $C_{S\sigma_k/2}(x, \sigma_k f_k^-(x,t-S^2\sigma_k^2)+\alpha_k, t)$ where $|\alpha_k| \leq CS\sigma_k\overline{\sigma}_k$.}
\end{equation} Define $D_k \equiv f^-_k(x, t-S^2\sigma_k^2) + \alpha_k/\sigma_k + S/2$. Pick $S > 0$ large such that $D_k \geq 1$ and then, for large enough $k$, we have $D_k - \alpha_k/\sigma_k- f_k^-(x, t-S^2\sigma_k^2) - CS\overline{\sigma}_k > 0$ and $(x, \sigma D_k, t) \in C_{S\sigma_k/2}(x, \sigma_k f_k^-(x,t-S^2\sigma_k^2)+\alpha_k, t)$. The flatness condition, \eqref{ukispastlfat}, gives

\begin{equation}\label{ukwithdk}\begin{aligned}
u_k(x, \sigma_kD_k, t) &\geq h_k(x, \sigma_kf_k^-(x,t-S^2\sigma_k^2) + \alpha_k, t-S^2\sigma_k^2/4)(S\sigma_k - CS\overline{\sigma}_k\sigma_k)/2\\
& \geq (1-\tau_k)h(0,0)S\sigma_k(1-C\overline{\sigma}_k)/2.
\end{aligned}\end{equation}
Plugging this into the definition of $w_k$,
\begin{equation}\label{wkwithdk}\begin{aligned}
w_k(x, \sigma_k D_k, t) &\geq (1-\tau_k)S(1-C\overline{\sigma}_k)/2 - (1+\tau_k)D_k\\
& = (1-\tau_k)S(1-C\overline{\sigma}_k)/2 - (1+\tau_k) (f^-_k(x, t-S^2\sigma_k^2) + \alpha_k/\sigma_k + S/2)\\
& = -f_k^-(x, t-S^2\sigma_k^2)+o_k(1) = -f(x,t) + o_k(1).
\end{aligned}
\end{equation}

We would like to replace the left hand side of equation \eqref{wkwithdk} with $w_k(x, \alpha, t)$, where $d$ does not depend on $k$. We accomplish this by means of barriers; for $\varepsilon > 0$ define $z_\varepsilon$ to be the unique solution to \begin{equation}\label{defofz}\begin{aligned}\partial_t z_\varepsilon + \Delta z_\varepsilon &=0,\: \mathrm{in}\; C_{1-\delta}(0,0)\cap \{x_n > 0\}\\
z_\varepsilon &= g_\varepsilon,\: \mathrm{on}\; \partial_p (C_{1-\delta}(0,0)\cap \{x_n > 0\})\cap \{x_n =0\}\\
z_\varepsilon &= -2,\: \mathrm{on}\; \partial_p (C_{1-\delta}(0,0)\cap \{x_n > 0\})\cap \{x_n > 0\},
\end{aligned}\end{equation} where $g_\varepsilon \in C^\infty(C_{1-\delta}(0,0))$ and $-f(x, t) - 2\varepsilon< g_\varepsilon(x,t) <- f(x,t) - \varepsilon$. By standard parabolic theory, for any $\varepsilon > 0$ there exists an $\alpha > 0$ (which depends on $\varepsilon > 0$) such that $|x_n| < \alpha$ implies $|z_\varepsilon(x,x_n,t) - z_\varepsilon(x,0,t)| < \varepsilon/2$. Pick $k$ large enough so that $\sigma_k < \alpha$. We know $w_k$ solves the adjoint heat equation on $\{x_n \geq \sigma_k\}$ and, by equations \eqref{defofz} and \eqref{wkwithdk}, $w_k \geq z_\varepsilon$ on $\partial_p (C_{1-\delta}(0,0)\cap \{x_n > \sigma_k\})$. Therefore, $w_k \geq z_\varepsilon$ on all of $C_{1-\delta}(0,0)\cap \{x_n > \sigma_k\}$. 

Consequently,  $$w_k(x,\alpha,t) \geq z_\varepsilon(x,\alpha,t) \geq z_\varepsilon(x,0,t) - \varepsilon/2 \geq -f(x,t) - 3\varepsilon.$$ As $k \rightarrow \infty$ we know $w_k(x,\alpha,t) \rightarrow w(x,\alpha,t) \leq w(x,0,t)$. This gives the desired result. 
\end{proof} 

The next step is to prove that the normal derivative of $w$ on $\{x_n = 0\}$ is zero. This will allow us to extend $w$ smoothly over $\{x_n = 0\}$ and obtain regularity for $f$. 

\begin{lem}\label{holdernormalderivativeiszero}
Suppose the assumptions of Lemma \ref{holderblowupfunctionisagraph} are satisfied and that $k$ is the subsequence identified in that lemma. Further suppose that $w$ is the limit function identified in Lemma \ref{holderboundaryvaluesofw}. Then $\partial_n w = 0$, in the sense of distributions, on $C_{1/2}(0,0)\cap \{x_n =0\}$.
\end{lem}

\begin{proof}
Rescale so $\rho_k \equiv 1$ and define $g(x,t)= 5 - 8(|x|^2 + |t|)$. For $(x,0,t) \in C_{1/2}(0,0)$ we observe $f(x,0,t) \leq1 \leq  g(x,0,t)$. We shall work in the following set $$Z\colonequals \{(x, x_n, t)\mid |x|, |t| \leq 1, x_n \in \R\}.$$ For any, $\phi(x,t)$ define $Z^+(\phi)$ to be the set of points in $Z$ above the graph $\{(X,t)\mid x_n = \phi(x,t)\}$, $Z^-(\phi)$ as set of points below the graph and $Z^0(\phi)$ as the graph itself. Finally, let $\Sigma_k \colonequals \{u_k > 0\} \cap Z^0(\sigma_k g)$. 

Recall, for any Borel set $A$, we define the ``surface measure", $\mu(A) = \int_{-\infty}^\infty \mathcal{H}^{n-1}(A\cap \{s=t\}) dt$. If $k$ is sufficiently large, and potentially adding a small constant to $g$, $\mu(Z^0(\sigma_kg)\cap \partial \{u_k > 0\} \cap C_{1/2}(0,0)) = 0$. 

There are three claims, which together prove the desired result.

\medskip

\noindent {\bf Claim 1:} $$\mu(\partial \{u_k > 0\} \cap Z^-(\sigma_kg)) \leq \frac{1}{(1-\tau_k)h_k(0,0)}\left(\int_{\Sigma_k} \partial_n u_k - 1dxdt +\mu(\Sigma_k)\right) + C\sigma_k^2$$

{\it Proof of Claim 1}: For any positive $\phi \in C_0^\infty(C_1(0,0))$ we have \begin{equation}\label{firststeptoclaim1} \begin{aligned} \int_{\partial \{u_k > 0\}} \phi d\mu &\leq \int_{\partial \{u_k > 0\}} \phi \frac{h_k(Q,\tau)}{(1-\tau_k)h_k(0,0)} d\mu(Q,\tau)\\&= \frac{1}{(1-\tau_k)h_k(0,0)}\int_{\{u_k > 0\}} u_k(\Delta \phi - \partial_t \phi)dXdt\\ &= -\frac{1}{(1-\tau_k)h_k(0,0)}\int_{\{u_k > 0\}} \nabla u_k \cdot \nabla \phi + u_k \partial_t\phi dXdt\end{aligned}\end{equation} (we can use integration by parts because, for almost every $t$, $\{u_k > 0\} \cap \{s =t\}$ is a set of finite perimeter). Let $\phi \rightarrow \chi_{Z^-(\sigma_kg)} \chi_{C_1}$ (as functions of bounded variation) and, since $|t| > 3/4$ or $|x|^2 > 3/4$ implies $u(x, \sigma_kg(x,t), t) = 0$, equation \eqref{firststeptoclaim1} becomes \begin{equation}\label{integrateoverE} \mu(\partial \{u_k > 0\} \cap Z^-(\sigma_kg)) \leq - \frac{1}{h_k(0,0)(1-\tau_k)}\left(\int_{\Sigma_k} \frac{\nabla u_k \cdot \nu + \sigma_k u_k\; \mathrm{sgn}(t)}{\sqrt{1+\sigma_k^2(|\nabla_x g(x,t)|^2 + 1)}} d\mu\right),\end{equation} where $\nu(x,t) = (\sigma_k\nabla g(x,t), -1)$ points outward spatially in the normal direction. 

We address the term with $\mathrm{sgn}(t)$ first; the gradient bound on $u_k$ tells us that $|u_k| \leq C\sigma_k(1+\tau_k)h_k(0,0)$ on $\Sigma_k$, so \begin{equation}\label{timepartsilly}\left|\frac{\sigma_k}{(1-\tau_k)h_k(0,0)}\int_{\Sigma_k} \frac{\sigma_k u_k\; \mathrm{sgn}(t)}{\sqrt{1+\sigma_k^2(|\nabla_x g(x,t)|^2 + 1)}} d\mu \right| \leq C\sigma_k^2.\end{equation} 

To bound the other term note that $\frac{d\mu}{\sqrt{1+\sigma_k^2(|\nabla_x g(x,t)|^2 + 1)}} = dx dt$ where the latter integration takes place over $E_k = \{(x,t)\mid (x, \sigma_kg(x,t), t)\in \Sigma_k\} \subset \{x_n =0\}$. Then integrate by parts in $x$ to obtain \begin{equation}\label{integratingbypartsonek}\begin{aligned}\int_{E_k} &(\sigma_k\nabla g(x,t), -1)\cdot \nabla u_k(x, \sigma_kg(x,t), t) dx dt = \int_{\partial E_k} \sigma_ku_k(x, \sigma_kg(x,t),t)\partial_\eta gd\mathcal H^{n-2}dt\\&-\int_{E_k} \sigma_k u_k(x,\sigma_kg(x,t), t) \Delta_x g(x,t) +\sigma_k^2 \partial_nu_k(x,\sigma_kg(x,t),t)|\nabla g|^2dxdt\\ &-\int_{E_k} \partial_n u_k(x,\sigma_kg(x,t),t)- 1dxdt +\mathcal L^n(E_k), \end{aligned}\end{equation} where $\eta$ is the outward space normal on $\partial E_k$. Since $u_k = 0$ on $\partial \Sigma_k$ the first term zeroes out. 

The careful reader may object that $E_k$ may not be a set of finite perimeter and thus our use of integration by parts is not justified. However, for any $t_0$, we may use the coarea formula with the $L^1$ function $\chi_{\{u(x,\sigma_kg(x,t_0),t_0) > 0\}}$ and the smooth function $\sigma_kg(-, t_0)$ to get $$\infty > \int \sigma_k |\nabla g(x,t_0)| \chi_{\{u(x,\sigma_kg(x,t_0),t_0) > 0\}}dx = \int_{-\infty}^\infty \int_{\{(x,t_0)\mid \sigma_kg(x,t_0) = r\}} \chi_{\{u(x,r,t_0) >0\}} d\mathcal H^{n-2}(x) dr.$$ Thus $\{(x,t_0)\mid \sigma_kg(x,t_0) > r\}\cap \{(x,t)\mid u(x,\sigma_kg(x,t_0),t_0) >0\}$ is a set of finite perimeter for almost every $r$. Equivalently, $\{(x,t_0)\mid \sigma_k (g(x,t_0)+\varepsilon) > 0\}\cap \{(x,t)\mid u(x,\sigma_k(g(x,t_0)+\varepsilon),t_0) >0\}$ is a set of finite perimeter for almost every $\varepsilon \in \mathbb R$. Hence, there exists a $\varepsilon >0$ aribtrarily small such that if we replace $g$ by $g+\varepsilon$ then $E_k\cap \{t = t_0\}$ will be a set of finite perimeter for almost every $t_0$. Since we can perturb $g$ slightly without changing the above arguments, we may safely assume that $E_k$ is a set of finite perimeter for almost every time slice. 

Observe that $\Delta g$ is bounded above by a constant, $|u_k| \leq C h(0,0)(1+\tau_k)\sigma_k$ on $\Sigma_k$, $|\partial_n u_k| \leq h(0,0)(1+\tau_k)$ and finally $\mu(\Sigma_k) \geq \mathcal L^n(E_k)$. Thus, $$\mu(\partial \{u_k > 0\} \cap Z^-(\sigma_kg)) \leq \frac{1}{(1-\tau_k)h_k(0,0)}\left(\int_{E_k} \partial_n u_k - 1dxdt + \mu(\Sigma_k)\right) + C\sigma_k^2.$$ As the difference between integrating over $E_k$ and integrating over $\Sigma_k$ is a factor of $\sqrt{1+\sigma_k^2} \simeq 1 + \sigma_k^2$ (for $\sigma_k$ small) we can conclude   $$\mu(\partial \{u_k > 0\} \cap Z^-(\sigma_kg)) \leq \frac{1}{(1-\tau_k)h_k(0,0)}\left(\int_{\Sigma_k} \partial_n u_k - 1dxdt + \mu(\Sigma_k)\right) + C\sigma_k^2,$$ which is of course the claim. 

Note, arguing as in equations \eqref{timepartsilly} and \eqref{integratingbypartsonek},  \begin{equation}\label{onlyenmatters}
\int_{\Sigma_k}\frac{(\sigma_k \nabla_x g(x,t), 0, \sigma_k \mathrm{sgn}(t))}{\sqrt{1+\|(\sigma_k \nabla_x g(x,t), 0, \sigma_k \mathrm{sgn}(t))\|^2}} \cdot (\nabla_x w_k, 0, w_k) d\mu \stackrel{k\rightarrow \infty}{\rightarrow} 0,
\end{equation}
which will be useful to us later. 

\medskip

\noindent {\bf Claim 2:} $$\mu(\Sigma_k) - C_2 \sigma_k^2 \leq \mu(\partial \{u_k > 0\} \cap Z^-(\sigma_kg)).$$

{\it Proof of Claim 2:}  Let $\nu_k(x,t)$ the inward pointing measure theoretic space normal to $\partial \{u_k > 0\} \cap \{s= t\}$ at the point $x$. Note that for almost every $t$ it is true that $\nu_k$ exists $\mathcal H^{n-1}$ almost everywhere. Defining $\nu_{\sigma_k g}(X,t) = \frac{1}{\sqrt{1+256\sigma_k^2|x|^2}}(-\sigma_k 16x, 1, 0)$,  we have $$\mu(\partial \{u_k > 0\} \cap Z^-(\sigma_kg)) = \int_{\partial \{u_k > 0\} \cap Z^-(\sigma_kg)} \nu_k \cdot \nu_k d\mu \geq \int_{\partial^* \{u_k > 0\} \cap Z^-(\sigma_kg)} \nu_k \cdot \nu_k d\mu \geq $$$$\int_{\partial^* \{u_k > 0\} \cap Z^-(\sigma_kg)} \nu_k \cdot \nu_{\sigma_kg} d\mu \stackrel{\mathrm{div}\;\mathrm{thm}}{=} -\int_{Z^-(\sigma_kg)\cap \{u_k > 0\}} \mathrm{div}\; \nu_{\sigma_kg}dXdt + \int_{\Sigma_k} 1d\mu.$$ In the last equality above we use the fact that on $Z^0(\sigma_kg)$, $\nu_{\sigma_kg}$ agrees with upwards pointing space normal. 

We compute $|\mathrm{div}\; \nu_{\sigma_kg}| = \left|\frac{-16\sigma_k(n-1)}{\sqrt{1+ 256\sigma_k^2|x|^2}} + \frac{3\sigma_k^3(16*256)|x|^2}{\sqrt{1+ 256\sigma_k^2|x|^2}^3}\right|  \leq C\sigma_k$. As the ``width" of $Z^-(\sigma_kg)\cap \{u_k > 0\}$ is of order $\sigma_k$ we get the desired result.

\medskip

\noindent {\bf Claim 3:} $$\int_{\Sigma_k} |\partial_n w_k| \stackrel{k\rightarrow \infty}{\rightarrow} 0.$$

{\it Proof of Claim 3:} Recall that $\partial_n u_k \leq (1+\tau_k)h(0,0)$, which implies, $\partial_n w_k \leq 0$. To show the limit above is at least zero we compute \begin{equation}\label{limitingtozero} \begin{aligned} \int_{\Sigma_k} \partial_n w_k d\mu =& \int_{\Sigma_k} \frac{\partial_n u_k - 1}{\sigma_k h_k(0,0)} d\mu + \frac{\mu(\Sigma_k)}{\sigma_kh_k(0,0)} -\frac{(1+\tau_k)\mu(\Sigma_k)}{\sigma_k}\\
\stackrel{\mathrm{Claim 1}} \geq&\frac{(1-\tau_k) \mu(\partial \{u_k > 0\} \cap Z^-(\sigma_kg))}{\sigma_k} -\frac{(1+\tau_k)\mu(\Sigma_k)}{\sigma_k} - C\sigma_k \\
\stackrel{\mathrm{Claim 2}}{\geq}& \frac{(1-\tau_k) \mu(\Sigma_k)}{\sigma_k} -\frac{(1+\tau_k)\mu(\Sigma_k)}{\sigma_k} - \tilde{C}\sigma_k\\
\geq& -C'(\sigma_k+ \frac{4\tau_k}{\sigma_k}) \rightarrow 0.
\end{aligned}
\end{equation}

\medskip

We can now combine these claims to reach the desired conclusion. We say that $\partial_n w = 0$, in the sense of distributions on $\{x_n = 0\}$, if, for any $\zeta \in C^\infty_0(C_{1/2}(0,0))$, $$\int_{\{x_n = 0\}} \partial_n w\zeta = 0.$$ 

 Claim 3 implies \begin{equation}\label{whatdoesclaim3say} 0 = \lim_{k\rightarrow \infty} \int_{\Sigma_k} \zeta \partial_n w_kd\mu .\end{equation} On the other hand equation \eqref{onlyenmatters} (and $\zeta$ bounded) implies \begin{equation}\label{finalmanipulationstozero} \lim_{k\rightarrow \infty} \int_{\Sigma_k} \zeta \partial_n w_kd\mu=
 \lim_{k\rightarrow \infty} \int_{\Sigma_k} \zeta \nu_{\Sigma_k} \cdot (\nabla_X w_k, w_k)d\mu,\end{equation} where $\nu_{\Sigma_K}$ is the unit normal to $\Sigma_k$ (thought of as a Lispchitz graph in $(x,t)$) pointing upwards.  Together, equations \eqref{whatdoesclaim3say}, \eqref{finalmanipulationstozero} and the divergence theorem in the domain $Z^+(\sigma_kg) \cap C_{1/2}(0,0)$ have as a consequence \begin{equation*}\begin{aligned} 0 =& \lim_{k\rightarrow \infty} \int_{Z^+(\sigma_kg)} \mathrm{div}_{X,t} (\zeta (\nabla_X w_k, w_k)) dXdt \\ =&\lim_{k\rightarrow \infty}  \int_{Z^+(\sigma_kg)} \nabla_X \zeta \cdot \nabla_X w_k + (\partial_t \zeta)w_k + \zeta (\Delta_X w_k + \partial_t w_k) dXdt\\ \stackrel{\Delta w_k + \partial_t w_k = 0}{=}& \int_{\{x_n > 0\}} \nabla_X w\cdot \nabla_X \zeta + (\partial_t \zeta) w dXdt \\ \stackrel{\mathrm{integration}\; \mathrm{by}\; \mathrm{parts}}{=}& \int_{\{x_n = 0\}} w_n \zeta dxdt - \int_{\{x_n > 0\}} \zeta (\Delta_X w + \partial_t w)dXdt.\end{aligned}\end{equation*} As $w$ is adjoint caloric this implies that $\int_{\{x_n = 0\}} \partial_n w\zeta = 0$ which is the desired result. 
\end{proof}

From here it is easy to conclude regularity of $f$. 

\begin{cor}\label{holderfhashighregularity}
Suppose the assumptions of Lemma \ref{holderblowupfunctionisagraph} are satisfied and that $k$ is the subsequence identified in that lemma. Then $f\in C^{\infty}(C_{1/2}(0,0))$ and in particular the $\mathbb C^{2+\alpha,1+\alpha}$ norm of $f$ in $C_{1/4}(0,0)$ is bounded by an absolute constant.
\end{cor}

\begin{proof}
Extend $w$ by reflection across $\{x_n = 0\}$. By Lemma \ref{holdernormalderivativeiszero} this new $w$ satisfies the adjoint heat equation  in all of $C_{1/2}(0,0)$ (recall a continuous weak solution to the adjoint heat equation in the cylinder is actually a classical solution to the adjoint heat equation). Since $\|w\|_{L^\infty(C_{3/4}(0,0))} \leq 2$, standard regularity theory yields the desired results about $-f = w|_{x_n =0}$. 
\end{proof}

We can use this regularity to prove Lemma \ref{currentholderflatonbothsidesisextraflatonzeroside}.

\begin{lem*}[Lemma \ref{currentholderflatonbothsidesisextraflatonzeroside}]
 Let $\theta \in (0,1)$ and assume that $u \in HCF(\sigma, \sigma, \kappa)$ in $C_\rho(Q,\tau)$ in the direction $\nu$. There exists a constant $0 < \sigma_\theta < 1/2$ such that if $\sigma < \sigma_\theta$ and $\kappa \leq \sigma_\theta \sigma^2$ then $u \in \widetilde{HCF}(\theta \sigma, \theta \sigma, \kappa)$ in $C_{c(n)\rho \theta}(Q,\tau)$ in the direction $\overline{\nu}$ where $|\overline{\nu} - \nu| \leq C(n) \sigma$. Here $\infty > C(n), c(n) >0$ are constants depending only on dimension.
\end{lem*}

\begin{proof}[Proof of Lemma \ref{currentholderflatonbothsidesisextraflatonzeroside}]
Without loss of generality, let $(Q,\tau)  = (0,0)$ and we will assume that the conclusions of the lemma do not hold. Choose a $\theta \in (0,1)$ and, by assumption, there exists $\rho_k, \sigma_k \downarrow 0$ and $\kappa_k/\sigma_k^2 \rightarrow 0$ such that $u \in HCF(\sigma_k, \sigma_k, \kappa_k)$  in $C_{\rho_k}(0,0)$ in the direction $\nu_k$ (which after a harmless rotation we can set to be $e_n$) but so that $u$ is not in $\widetilde{HCF}(\theta \sigma_k, \theta \sigma_k, \kappa_k)$ in $C_{c(n)\theta \rho_k}(0,0)$ in any direction $\nu$ with $|\nu_k-\nu| \leq C\sigma_k$ and for any constant $c(n)$. Let $u_k(X,t) = \frac{u(\rho_k X, \rho_k^2t)}{\rho_k}$. It is clear that $u_k$ is adjoint caloric, that its zero set is a parabolic regular domain and that it is associated to an $h_k$ which satisfies $\|\log(h_k)\|_{\mathbb C^{1,1/2}} \leq C\|\log(h)\|_{\mathbb C^{1,1/2}}$ and $h_k(0,0) \equiv h(0,0)$. 

By Lemma \ref{holderblowupfunctionisagraph} we know that there exists a continuous function $f$ such that $\partial \{u_k > 0\} \rightarrow \{(X,t)\mid x_n = f(x,t)\}$ in the Hausdorff distance sense. Corollary \ref{holderfhashighregularity} implies that there is a universal constant, call it $K$, such that  \begin{equation}\label{fdoesntgetbig} f(x,t) \leq f(0,0) + \nabla_x f(0,0) \cdot x  + K(|t| + |x|^2)\end{equation} for $(x,t) \in C_{1/4}(0,0)$. Since $(0,0) \in \partial \{u_k > 0\}$ for all $k$, $f(0,0) = 0$. If $\theta \in (0,1)$, then there exists a $k$ large enough (depending on $\theta$ and the dimension) such that  $$f_k^+(x,t) \leq  \nabla_x f(0,0) \cdot x + \theta^2/4K,\; \forall (x,t) \in C_{\theta/(4K)}(0,0),$$  where $f_k^+$ is as in Lemma \ref{holderblowupfunctionisagraph}.

Let $$\nu_k \colonequals \frac{(-\sigma_k \nabla_x f(0,0), 1)}{\sqrt{1+ |\sigma_k  \nabla_x f(0,0)|^2}}$$ and compute \begin{equation}\label{takingadotproductwithholdernu}
x \cdot \nu_k \geq \theta^2 \sigma_k/4K \Rightarrow
 x_n \geq \sigma_k x' \cdot \nabla_x f(0,0) + \theta^2 \sigma_k/4K \geq f_k^+(x,t).
\end{equation} Therefore, if $(X,t) \in C_{\theta/(4K)}(0,0)$ and $x\cdot \nu_k \geq \theta^2\sigma_k/4K$, then $u_k(X,t) > 0$. Arguing similarly for $f_k^-$ we can see that $u_k \in \widetilde{HCF}(\theta \sigma_k, \theta \sigma_k, \kappa_k)$ in $C_{\theta/4K}(0,0)$ in the direction $\nu_k$. It is easy to see that $|\nu_k - e_n| \leq C\sigma_k$ and so we have the desired contradiction. 
\end{proof}

\section{Higher Regularity}\label{higherregularity}

 We begin by recalling the partial hodograph transform (see \cite{kinderlehrerstampacchia}, Chapter 7 for a short introduction in the elliptic case).  Here, and throughout the rest of the paper, we assume that $(0,0) \in \partial \Omega$ and that, at $(0,0)$, $e_n$ is the inward pointing normal to $\partial \Omega \cap \{t= 0\}$. Before we can use the hodograph transform, we must prove that $\nabla u$ extends smoothly to the boundary.
 
 \begin{lem}\label{gradientsmoothatboundary}
Let $s\in (0,1)$ and $\partial \Omega$ be a $\mathbb C^{1+s, (1+s)/2}$ domain such that $\log(h) \in \mathbb C^{s, s/2}(\mathbb R^{n+1})$. Then $u \in \mathbb C^{1+s, (1+s)/2}(\overline{\Omega})$.
 \end{lem}
 
 \begin{proof}
We will show that $u$ has the desired regularity in a neighborhood of $(0,0)$. For any $R > 0$ let $H_R$ be as in Corollary \ref{nontangentiallimiteverywhere}: $H_R(X,t) = \varphi_R(X,t)\nabla u(X,t) - w_R(X,t)$. Therefore, we may estimate:

\begin{equation}\label{provingregularityatboundary}
\begin{aligned}
|\nabla u(X,t) - h(0,0)e_n| =& |H_R(X,t) - h(0,0)e_n| + |w_R(X,t)|\\
 \leq& \int_{\partial \Omega} |h(Q,\tau)\hat{n}(Q,\tau)\varphi_R(Q,\tau) - h(0,0)e_n| d\hat{\omega}^{X,t} + |w_R(X,t)|.
\end{aligned}
\end{equation}

Since $\partial \Omega$ is locally given by a $\mathbb C^{1+s, (1+s)/2}$  graph, for any $\delta > 0$ there exists an $R_\delta > 0$ such that $\partial \Omega \cap C_r(Q,\tau)$ is $\delta$-flat for any $(Q,\tau) \in C_{100}(0,0)$ and any $r < R_\delta$. In particular, we can ensure that Lemma \ref{growthattheboundary} applies at all $r < R_\delta/4$ for a $\varepsilon > 0$ such that $1-\varepsilon > \frac{1+s}{2}$. Arguing in the same way as in the proof of Lemma \ref{cutoffargument}, we can deduce that \begin{equation}\label{revisedWestimate} |w_R(X,t)| \leq C \frac{\|(X,t)\|^{(1+s)/2}}{R^{1/2}}.
\end{equation}
Furthermore, we may deduce (in much the same manner as equation \eqref{nottoomuchtoofar}), that if $\|(X,t)\| \leq r$ and $k_0\in \mathbb N$ is such that $2^{-k_0-1} \leq r < 2^{-k_0}$, then \begin{equation}\label{nottoomuchtoofarpart2}1-\widehat{\omega}^{(X,t)}(C_{2^{-j}}(0,0)) \leq C2^{j(1+s)/2} r^{(1+s)/2}, \forall j < k_0-2.\end{equation} 

In light of the estimates \eqref{provingregularityatboundary}, \eqref{revisedWestimate} and \eqref{nottoomuchtoofarpart2}, we may conclude \begin{equation}\label{dyadicgradientbound}\begin{aligned} &|\nabla u(X,t) - h(0,0)e_n| \leq C\int_{\Delta_{4r}(0,0)} (4r)^s d\hat{\omega}^{(X,t)} + C\frac{r^{(1+s)/2}}{R^{1/2}} + C\hat{\omega}^{(X,t)}(\partial \Omega\backslash C_{R}(0,0)) \\ &+C\left(\int_{\Delta_{R}(0,0) \backslash \Delta_1(0,0)} d\hat{\omega}^{(X,t)} +\sum_{j=0}^{k_0-2}\int_{\Delta_{ 2^{-j}}(0,0)\backslash \Delta_{2^{-(j+1)}}(0,0)}2^{-js}d\hat{\omega}^{(X,t)}\right),\end{aligned}\end{equation} where $C$ depends on the H\"older norm of $h$, of $\hat{n}$ and on $R$. Also as above, $k_0$ is such that $2^{-k_0-1} \leq r < 2^{-k_0}$.

 We may bound $$\begin{aligned} \hat{\omega}^{(X,t)}(\Delta_{R}(0,0)\backslash \Delta_1(0,0)) &\leq 1- \widehat{\omega}^{(X,t)}(\Delta_1(0,0))\stackrel{eq. \eqref{nottoomuchtoofarpart2}}{\leq} C_R r^{(1+s)/2}\\
 \hat{\omega}^{(X,t)}(\partial \Omega \backslash \Delta_{R}(0,0)) &\leq 1- \widehat{\omega}^{(X,t)}(\Delta_R(0,0))\stackrel{eq. \eqref{nottoomuchtoofarpart2}}{\leq} C_R r^{(1+s)/2}\\
\hat{\omega}^{(X,t)}(\Delta_{ 2^{-j}}(Q,\tau)\backslash \Delta_{2^{-(j+1)}}(0,0)) &\leq 1-\widehat{\omega}^{(X,t)}(\Delta_{2^{-j}}(0,0)) \stackrel{eq. \eqref{nottoomuchtoofarpart2}}{\leq} C2^{j(1+s_/2} r^{(1+s)/2}.\end{aligned}$$ Plug these estimates into equation \eqref{dyadicgradientbound} to obtain \begin{equation}\label{estimateeachshell}\begin{aligned}
|\nabla u(X,t) - h(0,0)e_n| \leq& Cr^s +Cr^{(1+s)/2} + Cr^{(1+s)/2}\sum_{j=0}^{[\log_2(r^{-1})]} 2^{j(1-s)/2)}\\
 \leq & Cr^s\left(1+ r^{(1-s)/2}\frac{(2^{[\log_2(r^{-1})]})^{(1-s)/2} - 1}{2^{(1-s)/2} - 1}\right) \leq C_s r^s.
\end{aligned}
\end{equation}
Since $r = \|(X,t)\|$ we have proven that $\nabla u \in \mathbb C^{s, s/2}(\overline{\Omega})$.
 \end{proof}
 
With this regularity in hand, we may define $F: \Omega \rightarrow \mathbb H^+ \colonequals \{(x, x_n, t) \mid x_n > 0\}$ by $(x, x_n, t) = (X,t) \mapsto (Y,t) = (x, u(x,t),t)$. In a neighborhood of $0$ we know that $u_n \neq 0$ and so $DF$ is invertible on each time slice (since $\nabla u \in C(\overline{\Omega})$). By the inverse function theorem, there is some neighborhood, $\mathcal O$, of $(0,0)$ in $\Omega$ that is mapped diffeomorphically to $U$, a neighborhood of $(0,0)$ in the upper half plane. Furthermore, this map extends in a $\mathbb C^{s, s/2}$ fashion from $\overline{\mathcal O^+}$ to $\overline{U}$ (by Proposition \ref{initialholderregularityprop} and Lemma \ref{gradientsmoothatboundary}). Let $\psi: \overline{U} \rightarrow \R$ be given by $\psi(Y,t) = x_n$, where $F(X,t) = (Y,t)$. Because $F$ is locally one-to-one, $\psi$ is well defined. 

If $\nu_{Q,\tau}$ denotes the spatial unit normal pointing into $\Omega$ at $(Q,\tau)$ then $u$ satisfies \begin{eqnarray}
u_t(X,t)+ \Delta u(X,t) &=& 0, \; (X,t)\in \Omega^+\nonumber\\
u_{\nu_{Q,\tau}}(Q,\tau) &=& h(Q,\tau),\; (Q,\tau) \in \partial \Omega.\nonumber
\end{eqnarray}

After our change of variables these equations become \begin{equation}\label{transformedequations}
0 = -\frac{\psi_t}{\psi_n} + \frac{1}{2}\left(\frac{1}{\psi_n^2}\right)_n + \sum_{i=1}^{n-1} \left(-\left(\frac{\psi_i}{\psi_n}\right)_i + \frac{1}{2}\left(\frac{\psi_i^2}{\psi_n^2}\right)_n\right)
\end{equation} on $U$ and  \begin{equation}\label{transformedboundary}
\psi_n(y,0,t) h(y, \psi(y,0,t), t) = \sqrt{1+\sum_{i=1}^{n-1} \psi_i^2(y,0,t)}
\end{equation} on the boundary. 

\begin{rem}\label{transformremarks} The following are true of $\psi$:
\begin{itemize}
\item Let $k \geq 1$. If $\partial \Omega$ is a $\mathbb C^{k+s, (k+s)/2}$ graph and $\log(h) \in \mathbb C^{k-1+s, (k-1+s)/2}$ then $u \in \mathbb C^{k+s, (k+s)/2}(\overline{\Omega})$.
\item Let $k \geq 0$ be such that $h \in \mathbb C^{k+\alpha, (k+\alpha)/2}(\partial \Omega)$ and $\psi|_{\{y_n=0\}} \in \mathbb C^{k+1+s, (k+1+s)/2}$ for $0 < s \leq \alpha$. If $\tilde{h}(y,t) = h(y, \psi(y,0,t), t)$, then $\tilde{h} \in \mathbb C^{k+\alpha, (k+\alpha)/2}(\{y_n =0\})$.
\item $\psi_n > 0$ in $\overline{U}$. 
\end{itemize}
\end{rem}

\begin{proof}[Justification]
Let us address the first claim. When $k \geq 2$ we note that $\overline{u}(x,x_n t) = u(x, \psi(x,0,t)+ x_n, t)$ is the strong solution of an adjoint parabolic equation in the upper half space with zero boundary values and coefficients in $\mathbb  C^{k-1+s, (k+s)/2-1}$. Standard parabolic regularity theory then gives the desired result. When $k = 1$, this is simply Lemma \ref{gradientsmoothatboundary}. 

When $k =0$ the second claim follows from a difference quotient argument and the fact that $|\psi(x+y, t+s) -\psi(x, t)| \leq C(|y| + |s|^{1/2})$. We $k \geq 1$ take a derivative and note $\partial_i \tilde{h} = \partial_i h + \partial_nh \partial_i \psi$. As $\psi$ has one more degree of differentiability than $h$ it is clear that $ \partial_nh \partial_i \psi$ is just as regular as $\partial_n h$. We can argue similarly for higher spatial derivatives and for difference quotients or derivatives in the time direction. 

Our third claim follows from the assumption that $e_n$ is the inward pointing normal at $(0,0)$ and that $\partial_n u(0,0) > 0$ in $\mathcal O$. 
\end{proof}

To prove higher regularity we will use two weighted Schauder-type estimates due to Lieberman \cite{liebermanintermediateschauder} for parabolic equations in a half space. Before we state the theorems, let us introduce weighted H\"older spaces (the reader should be aware that our notation here is non-standard).

\begin{defin}\label{weightedholderspace}
Let $\mathcal O \subset \mathbb R^{n+1}$ be a bounded open set. For $a, b \notin \mathbb Z$ define $$\|u\|_{\mathbb C^{a, a/2}_b(\mathcal O)} =  \sup_{\delta > 0} \delta^{a+b} \|u\|_{\mathbb C^{a, a/2}(\mathcal O_\delta)}$$ where $\mathcal O_\delta = \{(X,t) \in \mathcal O \mid \mathrm{dist}((X,t), \partial \mathcal O) \geq \delta\}$. It should be noted that $\mathbb C^{a, a/2}_{-a} \equiv \mathbb C^{a, a/2}$. 
\end{defin}

For the sake of brevity, the following are simplified versions of \cite{liebermanintermediateschauder}, Theorems 6.1 and 6.2 (the original theorems deal with a more general class of domains, operators and boundary values that we do not need here).

\begin{thm}\label{liebermandirichletestimates}
Let $v$ be a solution to the boundary value problem \begin{equation}\begin{aligned} v_t(X,t)- \sum_{ij}p_{ij}(X,t)D_{ij}v(X,t) &= f(X,t),\; (X,t) \in U\\ v(x, 0, t) &= g(x,0,t),\; (x,0,t) \in \overline{U},\end{aligned}\end{equation} where $P = (p_{ij})$.  Let $ a> 2, b> 1$ be non-integral real numbers such that $p_{ij} \in \mathbb C^{a-2, a/2-1}(U), f\in \mathbb C^{a-2, a/2-1}_{2-b}(U)$ and $g\in \mathbb C^{b, b/2}(\overline{U} \cap \{x_n=0\})$. Then, if $v \in \mathbb C^{a, a/2}_{-b}(U)$ and $v|_{\partial U\backslash \{x_n = 0\}} = 0$, there exists a constant $C > 0$ (depending on the ellipticity, the H\"older norms of $p_{ij}, m$ and the dimension) such that $$\|v\|_{\mathbb C^{a, a/2}_{-b}(U)} \leq C\left(\|g\|_{\mathbb C^{b, b/2}} + \|v\|_{L^\infty} + \|f\|_{\mathbb C^{a-2, a/2-1}_{2-b}}\right).$$
\end{thm}

\begin{thm}\label{liebermanschauderestimate}
Let $v$ be a solution to the boundary value problem \begin{equation}\begin{aligned} v_t(X,t)- \sum_{ij}p_{ij}(X,t)D_{ij}v(X,t) &= f(X,t),\; (X,t) \in U\\ \vec{m}(x,0,t)\cdot \nabla_x v(x, 0, t) &= g(x,0,t),\; (x,0,t) \in \overline{U},\end{aligned}\end{equation} where $P = (p_{ij})$ is uniformly elliptic and $m_n(x,0,t) \geq  c >0$.  Let $ a> 2, b> 1$ be non-integral real numbers with $p_{ij} \in \mathbb C^{a-2, a/2-1}(U), m(x,0,t) \in \mathbb C^{b-1, (b-1)/2}(\overline{U}\cap \{x_n =0\}), f\in \mathbb C^{a-2, a/2-1}_{2-b}(U)$ and $g\in \mathbb C^{b-1, b/2-1/2}(\overline{U} \cap \{x_n=0\})$. Then, if $v \in \mathbb C^{a, a/2}_{-b}(U)$ and $v|_{\partial U\backslash \{x_n = 0\}} = 0$, there exists a constant $C > 0$ (depending on the ellipticity, the H\"older norms of $p_{ij}, m$ and the dimension) such that $$\|v\|_{\mathbb C^{a, a/2}_{-b}(U)} \leq C\left(\|g\|_{\mathbb C^{b-1, b/2-1/2}} + \|v\|_{L^\infty} + \|f\|_{\mathbb C^{a-2, a/2-1}_{2-b}}\right).$$
\end{thm}

With this theorem in hand we can use an iterative argument (modeled after one in \cite{jerison} to prove optimal H\"older regularity. To reduce clutter we define, for $f$ a function and $\vec{v} \in \mathbb R^{n+1}$ with $v_n =0$, \begin{equation}\label{seconddifference} \delta^2_{\vec{v}}f(X,t) = f((X,t) + \vec{v}) + f((X,t) - \vec{v}) - 2f(X,t).\end{equation} It will also behoove us to define the parabolic length of a vector $\vec{v} = (v_x, v_n, v_t)$ by \begin{equation}\label{paraboliclength} \|\vec{v}\|_p = |(v_x, v_n)| + |v_t|^{1/2}.\end{equation} Recall that $$\|f\|_{L^\infty} +\ \sup_{\vec{v}, (X,t) \in \mathbb R^{n+1}} \frac{|\delta^2_{\vec{v}}f(X,t)|}{\|\vec{v}\|_p^\alpha} = \|f\|_{\mathbb C^{\alpha, \alpha/2}},$$ (see, e.g. \cite{steinsingularintegrals}, Chapter 5, Proposition 8). 

\begin{prop}\label{sharpholderestimates}
Let $\Omega\subset \mathbb R^{n+1}$ be a parabolic regular domain with $\log(h) \in \mathbb C^{k+ \alpha, (k+\alpha)/2}(\mathbb R^{n+1})$ for $k \in \mathbb N$ and $\alpha \in (0,1)$. There is a $\delta_n > 0$ such that if $\delta_n \geq \delta > 0$ and $\Omega$ is $\delta$-Reifenberg flat then $\Omega$ is a $\mathbb C^{k+1+\alpha, (k+1+\alpha)/2}(\mathbb R^{n+1})$ domain. 
\end{prop}

\begin{proof}
Let us first prove the theorem for $k = 0$. By Proposition \ref{initialholderregularityprop} and Remark \ref{transformremarks} we know that $\psi \in \mathbb C^{1 + s, (1+s)/2}(\overline{U})$. For any vector $\vec{v} \in \mathbb R^{n+1}$ with $v_n =0$ define $w^{\varepsilon}(X,t) = \psi((X, t) + 2\varepsilon \vec{v}) - \psi((X,t))$. It is then easy to check that $w^{\varepsilon}(X,t)$ satisfies the following oblique derivative problem: \begin{equation}\label{differencequotientequation} \begin{aligned} -\partial_t w^{\varepsilon}(X,t)-\sum_{ij}p_{ij}((X,t)+\varepsilon\vec{v})D_{ij}w^{\varepsilon}(X,t) &= f^{\varepsilon}(X,t),\; (X,t) \in U\\ \vec{m}(y,0,t)\cdot \nabla w^{\varepsilon}(y, 0, t) &= g^{\varepsilon}(y,0,t),\; (x,0,t) \in \overline{U}.\end{aligned}\end{equation} Where \begin{equation*}\begin{aligned}(p_{ij}(X,t)) &= \left( \begin{array}{ccccc}
-1 & 0 & 0 &\ldots & \frac{\psi_1(X, t)}{\psi_n(X, t)} \\
0 & -1 & 0 &\ldots & \frac{\psi_2(X,t)}{\psi_n(X,t)} \\
\vdots & 0 & \ddots & \ldots & \vdots \\
\frac{\psi_1(X, t)}{\psi_n(X, t)} & \ldots & \frac{\psi_i(X, t)}{\psi_n(X, t)} & \ldots & -\left(\frac{\sqrt{1+ \sum_{i=1}^{n-1} \psi_i(X, t)^2}}{\psi_n(X, t)}\right)^2  \end{array} \right),\\ 
\vec{m}_j(y,0,t) &= \int_0^1 \frac{\partial}{\partial \psi_j}G(\nabla_X\psi(y,0,t) + s(\nabla_X\psi((y,0,t) +\varepsilon\vec{v}) - \nabla_X\psi(y,0,t)))ds,\\
G(\nabla_X \psi) &= \frac{\sqrt{1+\sum_{i=1}^{n-1}\psi_i^2}}{\psi_n},\\
f^{\varepsilon}(X,t)&= \sum_{ij}(p_{ij}(X, t) - p_{ij}((X,t)+\varepsilon \vec{v}))D_{ij}\psi(X,t),\\
g^{\varepsilon}(y, 0,t)&= \tilde{h}((y,0,t)+\varepsilon\vec{v}) - \tilde{h}(y, 0, t).
\end{aligned}
\end{equation*}

It is a consequence of $\psi_n > 0$ (see Remark \ref{transformremarks}) that $(p_{ij})$ is uniformly elliptic and $m_n(y,0,t) \geq c  > 0$ is uniformly in $(y,t)$. Furthermore, $p_{ij}(X+\varepsilon{v},t)$ and $m(y+\varepsilon\vec{v},0,t)$ are $C^{s, s/2}$ H\"older continuous, uniformly in $\varepsilon, \vec{v}$.  

To apply Theorem \ref{liebermanschauderestimate} we need that $w^{\varepsilon}$ is in a weighted H\"older space. Apply Theorem \ref{liebermandirichletestimates} with $a = 2+s, b = 1+s$ to $\psi$ in the space $U \cap \{x_n \geq \delta\}$ and let $\delta\downarrow 0$ to obtain the {\it a priori} estimate that $\psi \in \mathbb C^{2+s, 1+s/2}_{-1-s}(U)$ (note that $\psi|_{x_n = \delta}$ is uniformly in $\mathbb C^{1+s, (1+s)/2}$ by Remark \ref{transformremarks}). As such,  $w^{\varepsilon} \in \mathbb C^{2+\eta, 1+\eta/2}_{-1-s}$, uniformly in $\varepsilon > 0$ for any $0 < \eta << s$.

We will now compute the H\"older norm of $g^{\varepsilon}$ and the weighted H\"older norm of $f^{\varepsilon}$. For any $(x,0,t),(y,0,r)\in \overline{U},$ \begin{equation}\label{htildeestimate} \begin{aligned} 2\|\tilde{h}\|_{\mathbb C^{\alpha, \alpha/2}} (|x-y|^{s} +|t-r|^{s/2}) \varepsilon^{\alpha-s} \geq&\\
 \min\{2(|x-y|^\alpha + |t-r|^{\alpha/2})\|\tilde{h}\|_{\mathbb C^\alpha}, 2\varepsilon^\alpha\|\tilde{h}\|_{\mathbb C^{\alpha}}\}\geq&\\ |\tilde{h}(x,0,t) -\tilde{h}((x,0,t)+\varepsilon\vec{v}) - \tilde{h}(y,0,r) +\tilde{h}((y,0,r)+\varepsilon\vec{v})|.&\end{aligned}\end{equation} Thus $\|\tilde{h}(-)-\tilde{h}(-+\varepsilon\vec{v})\|_{\mathbb C^{s, s/2}}\leq C\varepsilon^{\alpha-s}$. 
 
For $d > 0$ we have that $$|f^{\varepsilon}(x, d, t) - f^{\varepsilon}(y,d, r)| \leq |p_{ij}((x,d,t) + \varepsilon\vec{v}) - p_{ij}(x,d,t) - p_{ij}((y,d,r)+ \varepsilon \vec{v}) + p_{ij}(y,d,r)||D_{ij}\psi(x,d,t)| $$$$+ |p_{ij}((y,d,r)+\varepsilon \vec{v}) - p_{ij}(y,d,r)||D_{ij}\psi(x,d,t) - D_{ij}\psi(y,d,r)|.$$ Arguing as in equation \eqref{htildeestimate} we can bound the first term on the right hand side by $$
d^{s-\eta-1}\|\psi\|_{\mathbb C^{2+\eta, 1+\eta/2}_{-1-s}(U)}(|x-y|^\eta + |t-r|^{\eta/2})\|\nabla \psi\|_{\mathbb C^{s, s/2}}\varepsilon^{s-\eta}.$$ The second term is more straightforward and can be bounded by $$\|\nabla \psi\|_{\mathbb C^{s, s/2}} \varepsilon^s d^{s-\eta-1}\|\psi\|_{\mathbb C^{2+\eta, 1+\eta/2}_{-1-s}(U)}(|x-y|^\eta + |t-r|^{\eta/2}).$$ Therefore $\|f^{\varepsilon}\|_{\mathbb C^{\eta, \eta/2}_{1-s}} \leq C(\varepsilon^{s-\eta} + \varepsilon^s)$. 

We may apply Theorem \ref{liebermanschauderestimate} to obtain \begin{equation}\label{weightedwepsilonestimate}\|w^{\varepsilon}\|_{\mathbb C^{2+\eta, 1+\eta/2}_{-1-s}} \leq C(\varepsilon^{s-\eta}+ \varepsilon^s + \varepsilon^{\alpha-s} + \varepsilon)\end{equation} (as $\|w^\varepsilon\|_{L^\infty} \leq C\varepsilon$). The reader may be concerned that $w^{\varepsilon}$ is not zero on the boundary of $U$ away from $\{x_n = 0\}$. However, we can rectify this by multiplying with a cutoff function and hiding the resulting error on the right hand side (see, e.g. the proof of Theorem 6.2 in \cite{adn1}).

We claim that if $w^{\varepsilon} \in \mathbb C^{2+\eta, 1+\eta/2}_{-1-s}$ then in fact $w_\varepsilon|_{\{x_n = 0\}}\in \mathbb C^{1+s, (1+s)/2}$ for any $\eta > 0$. For any $i = 1,...,n-1$, the fundamental theorem of calculus gives (recall the notation in equation \eqref{seconddifference} and \eqref{paraboliclength}) \begin{equation}\label{weightedgivesboundary}\begin{aligned} &|\delta^2_{\vec{v}}D_iw^{\varepsilon}(x, 0 ,t)| \leq |\delta^2_{\vec{v}} D_iw^{\varepsilon}(x, \|\vec{v}\|_p, t)| + \int_0^{\|\vec{v}\|_p}|\delta^2_{\vec{v}}w_{in}^{\varepsilon}(x, r, t)| dr\\ & \leq C\|w^{\varepsilon}\|_{\mathbb C^{2+\eta, 1+\eta/2}_{-1-s}}\|\vec{v}\|_p^{-1+s-\eta}\|\vec{v}\|_p^{1+\eta}+ C\|w^{\varepsilon}\|_{\mathbb C^{2+\eta, 1+\eta/2}_{-1-s}}\int_0^{\|\vec{v}\|_p} \frac{\|\vec{v}\|_p^{\eta}}{r^{1-(s-\eta)}}dr \\ &\leq C_{s,\eta}\|w^{\varepsilon}\|_{\mathbb C^{2+\eta, 1+\eta/2}_{-1-s}} \|\vec{v}\|_p^{s}.\end{aligned}\end{equation} 

Therefore,  for $i = 1,...,n-1$ we have \begin{equation}\label{iterationstep} \begin{aligned} |\delta^2_{\varepsilon \vec{v}} D_i\psi(x,0,t)| =&  |D_i w^\varepsilon((x,0,t)) - D_i w^{\varepsilon}((x,0,t)-\varepsilon\vec{v})|\\ &\leq \|D_x w^{\varepsilon}\|_{\mathbb C^{s, s/2}} \|\varepsilon \vec{v}_x\|_p^{s}\\ &\stackrel{\mathrm{eqn}\; \eqref{weightedgivesboundary}}{\leq} C_{s,\eta}\|w^\varepsilon\|_{\mathbb C^{2+\eta, 1+\eta/2}_{-1-s}} \|\varepsilon \vec{v}_x\|_p^{s}\\
&\stackrel{\mathrm{eqn}\;\eqref{weightedwepsilonestimate}}{\leq} C_{s,\eta}(\varepsilon^{s-\eta} + \varepsilon^{\alpha-s}+ \varepsilon + \varepsilon^{s})\|\varepsilon\vec{v}\|_p^{s}.\end{aligned}\end{equation} If $\|\vec{v}\|_p = 1$ and points either completely in a spacial or the time direction then we can conclude that $\psi \in \mathbb C^{1+\beta, (1+\beta)/2}(\overline{U}\cap \{x_n =0\})$, where $\beta = \min\{\alpha, 2s-\eta\}$. As such, $\partial \Omega$ is a $\mathbb C^{1+\beta, (1+\beta)/2}$ domain and, invoking Remark \ref{transformremarks}, we can conclude that $\psi \in \mathbb C^{1+\beta, (1+\beta)/2}(\overline{U})$. Repeat this argument until $\beta = \alpha$ to get optimal regularity.  

 When $k =1$ (that is $\log(h) \in \mathbb C^{1+\alpha, (1+\alpha)/2}$) we want to show that $\psi \in \mathbb C^{2+s, 1+s/2}$ for some $s$. Then we will invoke classical Schauder theory. We can argue almost exactly as above, except that equation \eqref{htildeestimate} cannot detect regularity in $\log(h)$ above $\mathbb C^{1,1/2}$. The argument above tells us that $\psi \in \mathbb C^{1+s, (1+s)/2}$ for any $s < 1$, thus, $\tilde{h} \in \mathbb C^{1+\alpha, (1+\alpha)/2}$. 

Let $\vec{\xi}, \vec{v} \in \mathbb R^{n+1}$ be such that $\xi_n, v_n = 0$. Then, in the same vein as equation \eqref{htildeestimate}, we can estimate \begin{equation}\label{htildeestimatebutbetter} \begin{aligned} |\delta^2_{\vec{\xi}}\tilde{h}((x,0,t)+\vec{v}) - \delta^2_{\vec{\xi}}\tilde{h}((x,0,t))|\leq& \|h\|_{\mathbb C^{1+\alpha, (1+\alpha)/2}} \min\{2\|\vec{\xi}\|_p^{1+\alpha}, 3\|\vec{v}\|_p\}\\ \leq& 3\|h\|_{\mathbb C^{1+\alpha, (1+\alpha)/2}}\|\vec{\xi}\|_p^s\|\vec{v}\|_p^{1-\frac{s}{1+\alpha}}.\end{aligned}\end{equation}
 Consequently, $\|\tilde{h}(-)-\tilde{h}(-+\vec{v})\|_{\mathbb C^{s, s/2}}\leq C\|\vec{v}\|_p^{1-\frac{s}{1+\alpha}}$. We may then repeat the argument above until we reach equation \eqref{weightedwepsilonestimate}, which now reads \begin{equation}\label{weightedwepsilonestimatepart2}
 \|w^\varepsilon\|_{\mathbb C^{2+\eta, 1+\eta/2}_{-1-s}} \leq C(\varepsilon^{s-\eta} + \varepsilon^s + \varepsilon^{1-\frac{s}{1+\alpha}} + \varepsilon).
 \end{equation}
 
Resume the argument until equation \eqref{iterationstep}, which is now, \begin{equation}\label{iterationstep2}
|\delta_{\varepsilon \vec{v}}^2D_i\psi(x,0,t)| \leq C_{s,\eta}(\varepsilon^{s-\eta} + \varepsilon^s + \varepsilon^{1-\frac{s}{1+\alpha}} + \varepsilon)\|\varepsilon \vec{v}\|^s_p.
\end{equation}
If $\|\vec{v}\|_p =1$ and points either completely in the spacial or time direction then we can conclude that $\psi\in \mathbb C^{1+\beta, (1+\beta)/2}(\{x_n =0\})$ where $\beta =  \min\{2s-\eta, 1+s, 1+s-\frac{s}{1+\alpha}\}$. Pick $\eta, s$ such that $2s-\eta > 1$ and we have that there is some $\gamma \in (0,1)$ such that $\psi\in \mathbb C^{2+\gamma, 1+\gamma/2}(\{x_n =0\})$.  Remark \ref{transformremarks} ensures that  \begin{equation}\label{aprioriestimate} \psi \in \mathbb C^{2+\gamma,1+ \gamma/2}(\overline{U}),\end{equation} for some $\gamma \in (0,1)$.

Now that we have the {\it a priori} estimate \eqref{aprioriestimate}, we may apply Theorem \ref{liebermanschauderestimate} to $w^\varepsilon$, the solution of equation \eqref{differencequotientequation}, but with $a = 2+\beta, b = 2+\beta$.  In this form, Theorem \ref{liebermanschauderestimate} comports with classical Schauder theory.  An iterative argument one similar to the above, but substantially simpler, yields optimal regularity. In fact, we can use the same iterative argument to prove that $\psi \in \mathbb C^{k+\alpha,(k+ \alpha)/2}$ given $\psi \in \mathbb C^{k-1+\alpha,(k-1+ \alpha)/2}$. Thus the full result follows. 
\end{proof}

We have almost completed a proof of Theorem \ref{loghholdercontinuous}--we need only to discuss what happens when $\log(h)$ is analytic.  However,  the anisotropic nature of the heat equation makes analyticity the wrong notion of regularity to consider here. Instead we recall the second Gevrey class:

\begin{defin}\label{gevrydef}
We say that a function $F(X,t)$ is analytic in $X$ and of the second Gevrey class in $t$ if there are constants $C, \kappa$ such that $$|D_X^\ell \partial_t^m F| \leq C\kappa^{|\ell| + 2m}(|\ell| +2m)!.$$
\end{defin}

Let us recall Theorem \ref{loghholdercontinuous}:
\begin{thm*}[Theorem \ref{loghholdercontinuous}]
Let $\Omega\subset \mathbb R^{n+1}$ be a parabolic regular domain with $\log(h) \in \mathbb C^{k+\alpha, (k+\alpha)/2}(\mathbb R^{n+1})$ for $k \geq 0$ and $\alpha \in (0,1)$.  There is a $\delta_n > 0$ such that if $\delta_n \geq \delta > 0$ and $\Omega$ is $\delta$-Reifenberg flat then $\Omega$ is a $\mathbb C^{k+1+\alpha, (k+1+\alpha)/2}(\mathbb R^{n+1})$ domain. 

Furthermore, if $\log(h)$ is analytic in $X$ and in the second Gevrey class in $t$ then, under the assumptions above, we can conclude that $\Omega$ is the region above the graph of a function which is analytic in the spatial variables and in the second Gevrey class in $t$. Similarly, if $\log(h) \in C^{\infty}$ then $\partial \Omega$ is locally the graph of a $C^\infty$ function.
\end{thm*}

It is clear that Proposition \ref{sharpholderestimates} implies the above theorem except for the statement when $\log(h)$ analytic in $X$ and second Gevrey class in $t$. This follows from a theorem of Kinderlehrer and Nirenberg \cite{kinderlehrernirenberganalyticity}, Theorem 1, which we present (in a modified version, below).

\begin{thm}\label{gevreyclasstheorem}[Modified Theorem 1 in \cite{kinderlehrernirenberganalyticity}]
Let $v \in C^\infty(\overline{U})$ be a solution to $$\begin{aligned} -F(D^2v, Dv, v, X,t) + v_t = 0\; &(X,t) \in U\\
v_n = \Phi(\partial_1v,....,\partial_{n-1}v, v, x, t),\; &(x,0,t) \in \overline{U}\cap \{x_n=0\}\end{aligned}$$ where $(F_{v_{ij}})$ is a positive definite form. Assume that $F$ is analytic in $D^2v, Dv$ and $\Phi$ is analytic in $\partial_x v$ and $v$.  If $F, \Phi$ are analytic in $X$ and in the second Gevrey class in $t$ then $v$ is analytic in $X$ and in the second Gevrey class in $t$. 
\end{thm}

 In \cite{kinderlehrernirenberganalyticity}, Theorem 1 is stated for Dirichlet boundary conditions. But the remarks after the theorem (and, especially, equation 2.11 there) show that the result applies to Neumann conditions also. Finally, it is easy to see that if $\log(h)$ is analytic in $X$ and in the second Gevrey class in $t$ then $\psi$ satisfies the hypothesis of Theorem \eqref{gevreyclasstheorem}. This  finishes the proof of Theorem \ref{loghholdercontinuous}. 
 
\appendix

\section{Classification of ``Flat Blowups"}\label{proofofflatprop} A key piece of the blowup argument is a classification of ``flat blowups"; Theorem \ref{poissonkernel1impliesflat}.  The proof of this theorem follows from two lemmas which are modifications of results in Andersson and Weiss \cite{anderssonweiss}. Before we begin, let us try to clarify the relationship between this work and that of \cite{anderssonweiss}.

As was mentioned in the introduction, the results in \cite{anderssonweiss} are for solutions in the sense of domain variations to a problem arising in combustion. Although this is the natural class of solutions to consider when studying their problem (see the introduction in \cite{weissregularity}), the definition of these solutions is quite complex and it is unclear whether the parabolic Green function at infinity satisfies it. For example, neither the integral bounds on the growth of time and space derivatives (see the first condition in Definition 6.1 in \cite{weissregularity}) nor the monotonicity formula (see the second condition in Definition 6.1 in \cite{weissregularity}), clearly hold {\it a priori} for functions $u$ which satisfy the conditions of Theorem \ref{poissonkernel1impliesflat}.

 Upon careful examination of \cite{anderssonweiss} we identified which properties of solutions in the sense of domain variations were crucial to the proof. The first of these was the following ``representation" formula which holds in the sense of distributions for almost every time, $t_0$, (see Theorem 11.1 in \cite{weissregularity}) \begin{equation}\label{awrepresentation}\Delta u - \partial_t u|_{t_0} = \mathcal H^{n-1}|_{R(t_0)} + 2\theta(t_0, -)\mathcal H^{n-1}|_{\Sigma_{**}(t_0)} + \lambda(t)|_{\Sigma_z(t_0)}.\end{equation} Without going too deep into details, we should think of $R(t_0) \subset \partial \{u > 0\}$ as boundary points which are best behaved. In particular, blowups at these points are plane solutions with slope 1. $\Sigma_{**}(t_0)$ points are also regular in the sense that the set is countable rectifiable and the blowups are planes, but the slope of the blowup solution is $2\theta(t_0, -) < 1$. $\Sigma_z(t_0)$ consists of singular points at which blowups may not be planes.  We observe that this situation is much more complicated than the elliptic one in \cite{altcaf}, but, as is pointed out in the introduction \cite{weissregularity}, this is an unavoidable characteristic of solutions to the combustion problem. 

Lemma \ref{normalderivativeatregularpoint} (which corresponds to Lemma 5.1 in \cite{anderssonweiss} and is a parabolic version of Lemma 4.10 in \cite{altcaf}) bounds from below the slope of blowup solutions at points which have an exterior tangent ball. For solutions in the sense of domain variation, points in $\Sigma_{**}(t_0)$ complicate matters and thus more information than is given by \eqref{awrepresentation} is needed. Indeed, it is mentioned in \cite{anderssonweiss} that defining solutions merely as those which satisfy \eqref{awrepresentation} would not be sufficient to implement their approach. However, in our setting, Lemma \ref{basicallyhbiggerthan1} says that $k\geq 1$ at almost every point on the free boundary which suffices to show that any blowup at a regular point must have slope at least 1. 

Another property of solutions in the sense of domain variations which is critical to \cite{anderssonweiss} is that $|\nabla u(X,t)|^2$ always approaches a value less than or equal to 1 as $(X,t)$ approaches the free boundary (Lemma 8.2 in \cite{anderssonweiss}). In our setting we know that $|\nabla u| \leq 1$ everywhere (by Proposition \ref{maingradientbound}) and so we need not worry. That $|\nabla u| \leq 1$ (along with Lemma \ref{iflipschitzthenregularintime}, proven in Section \ref{initialholderregularity} above), implies that blowups of $u$ are precompact in the $\Lip(1,1/2)$ norm--which is another property of domain variations that is used in \cite{anderssonweiss}. Finally, it is important to the arguments in \cite{anderssonweiss} that the set $R(t_0)$ is rectifiable (e.g. in order to apply integration by parts).  In the setting of Proposition \ref{poissonkernel1impliesflat}, Ahlfors regularity lets us apply integration by parts as well (albeit, we must be more careful. See, e.g., the proof of Lemma \ref{normalderivativeiszero} below). 

By finding appropriate substitutes for the relevant properties of domain variations (as described above) we were able to prove that the results of \cite{anderssonweiss} apply mostly unchanged to $u, k, \Omega$ which satisfy the hypothesis of Proposition \ref{poissonkernel1impliesflat}. However, we needed to make an additional modification, as the conclusions of Proposition \ref{poissonkernel1impliesflat} are global whereas the main theorem in \cite{anderssonweiss} is a local regularity result (see Corollary 8.5 there). In particular, Theorem 5.2 and Lemma 8.1 in \cite{anderssonweiss} roughly state that if a solution is flat in a certain sense in a cylinder, then it is even flatter in another sense in a smaller cylinder {\it whose center has been translated in some direction}. This translation occurs because the considered cylinders contain a free boundary point that is centered in the space variables but in the ``past" in the time coordinate. We want to improve flatness at larger and larger scales, so we cannot allow our cylinders to move in this manner (otherwise our cylinder might drift off to infinity). 

To overcome this issue, we introduce the concept of ``current flatness" (see Definition \ref{strongcurrentflatness}). However, the parabolic equation is anisotropic, so centering our cylinder in time means that we have to accept weaker results. This leads to the notion of ``weak current flatness" (Definition \ref{weakcurrentflatness}). Unfortunately, the qualitative nature of weak flatness is not always sufficient so we still need to prove some results for ``past flatness" (as introduced in \cite{anderssonweiss}). We also remark that this idea of ``current flatness" could be used in the setting of domain variations to analyze the global properties of those solutions.

Let us recall Theorem \ref{poissonkernel1impliesflat};

\begin{thm*}[Theorem \ref{poissonkernel1impliesflat}]
Let $\Omega_\infty$ be a $\delta$-Reifenberg flat parabolic regular domain with Green function at infinity $u_\infty$ and associated parabolic Poisson kernel $h_\infty$ (i.e.  $h_\infty = \frac{d\omega_\infty}{d\sigma}$). Furthermore, assume that $|\nabla u_\infty| \leq 1$ in $\Omega_\infty$ and $|h_\infty| \geq 1$ for $\sigma$-almost every point on $\partial \Omega_\infty$. There exists a $\delta_n > 0$ such that if $\delta_n \geq \delta > 0$ we may conclude that, after a potential rotation and translation, $\Omega_\infty = \{(X,t)\mid x_n > 0\}$. 
\end{thm*}

We define three notions of ``flatness" for solutions. The definition of ``past flatness" is taken from Andersson and Weiss \cite{anderssonweiss} (who in turn adapted it from the corresponding elliptic definitions in \cite{altcaf}). As mentioned above, we also introduce two types of ``current flatness". The first type is quantitative and we call it ``strong current flatness."
\begin{defin}\label{strongcurrentflatness}
For $0 < \sigma_i \leq 1/2$ we say that $U \in CF(\sigma_1, \sigma_2)$ in $C_\rho(Q,\tau)$ in the direction $\nu \in \mathbb S^{n-1}$ if 
\begin{itemize}
\item $(Q, \tau) \in \partial \{U > 0\}$
\item $U((Y,s)) = 0$ whenever $(Y-Q)\cdot \nu \leq -\sigma_1 \rho$ and $(Y,s) \in C_\rho(Q,\tau)$. 
\item $U((Y,s)) \geq  (Y-Q)\cdot \nu-\sigma_2\rho$ whenever  $(Y-Q)\cdot \nu - \sigma_2\rho \geq 0$ and $(Y,s) \in C_\rho(Q,\tau)$.
\end{itemize}
\end{defin}

The parabolic nature of our problem means that it behooves us to introduce a ``past" version of this flatness:
\begin{defin}\label{pastflatness}
For $0 < \sigma_i \leq 1$ we say that $u \in PF(\sigma_1, \sigma_2)$ in $C_\rho(X,t)$ in the direction $\nu \in \mathbb S^{n-1}$ if for $(Y,s) \in C_\rho(X,t)$
\begin{itemize}
\item $(X, t-\rho^2) \in \partial \{U > 0\}$
\item $U((Y,s)) = 0$ whenever $(Y-X)\cdot \nu \leq -\sigma_1 \rho$
\item $U((Y,s)) \geq  (Y-X)\cdot \nu-\sigma_2\rho$ whenever  $(Y-X)\cdot \nu - \sigma_2\rho \geq 0$.
\end{itemize}
\end{defin}
 
Our final notion of flatness is qualitative and weaker than strong current flatness. We call it ``weak current flatness." 

\begin{defin}\label{weakcurrentflatness}
For $0 < \sigma_i \leq 1/2$ we say that $U \in \CF(\sigma_1, \sigma_2)$ in $C_\rho(Q,\tau)$ in the direction $\nu \in \mathbb S^{n-1}$ if 
\begin{itemize}
\item $(Q, \tau) \in \partial \{U > 0\}$
\item $U((Y,s)) = 0$ whenever $(Y-Q)\cdot \nu \leq -\sigma_1 \rho$
\item $U((Y,s)) > 0$ whenever  $(Y-Q)\cdot \nu > \sigma_2\rho $.
\end{itemize}
\end{defin}

We may now state our two main lemmas, the first of which allows us to conclude flatness on the positive side of $\partial \{u_\infty > 0\}$ given flatness on the zero side. 

\begin{lem}\label{centeredflatonzerosideisflatonpositiveside}[``Current" version of Theorem 5.2 in \cite{anderssonweiss}]
Let $\Omega, u_\infty$ satisfy the assumptions of Proposition \ref{poissonkernel1impliesflat}. Furthermore, assume $u_\infty \in \CF(\sigma, 1/2)$ in $C_\rho(Q,\tau)$ in the direction $\nu$. Then there is a constant $K_1 > 0$ (depending only on dimension) such that $u_\infty \in CF(K_1\sigma, K_1\sigma)$ in $C_{\rho/2}(Q,\tau)$ in the direction $\nu$.
\end{lem}

The second lemma provides greater flatness on the zero-side under the assumption of flatness on both sides. 

\begin{lem}\label{centeredflatonbothsidesisextraflatonzeroside}[``Current" version of Lemma 8.1 in \cite{anderssonweiss}]
Let $u_\infty, \Omega$ satisfy the assumptions of Proposition \ref{poissonkernel1impliesflat} and assume, for some $\sigma, \rho > 0$ that $u_\infty \in CF(\sigma,\sigma)$ in $C_\rho(Q,\tau)$ in the direction $\nu$. For $\theta\in (0,1)$, there exists a constant $1/2 > \sigma_\theta > 0$ which depends only on $\theta, n$, such that if $0 <\sigma < \sigma_\theta$ then $u_\infty \in \CF(\theta\sigma, \theta\sigma)$ in $C_{K_2\theta \rho}(0, 0)$ in the direction $\tilde{\nu}$, where $\tilde{\nu}$ satisfies $|\tilde{\nu} - \nu| < K_3 \sigma$. Here $0 < K_2, K_3 < \infty$ are constants depending only on the dimension.
\end{lem}

Our proof that these two lemmas imply that $\Omega_\infty$ is a half-space is based on the analogous result in the elliptic setting proven by Kenig and Toro in \cite{kenigtoro2}. 

\begin{proof}[Proof of the Proposition \ref{poissonkernel1impliesflat} assuming the two lemmas]
Pick $\theta' \in (0,1)$ small enough so that $\max\{\theta', K_1^2\theta', K_2\theta'/4\} < 1/2$. Then let $\delta_n < \min\{1/2, \sigma_{\theta'}/K_1\}$. Here, and through the rest of this proof, $K_1, K_2, K_3$ and $\sigma_{\theta'}$ are as in Lemmas \ref{centeredflatonzerosideisflatonpositiveside} and \ref{centeredflatonbothsidesisextraflatonzeroside}. 

Assume, without loss of generality, that $(0,0) \in \partial \Omega$. For every $\rho > 0$, there is an $n$-plane, $V(\rho)$, containing a line parallel to the $t$-direction, such that $\frac{1}{\rho}D[V(\rho)\cap C_\rho(0,0), \partial \Omega \cap C_\rho(0,0)] < \delta$. Let $\nu_\rho$ be the unit normal vector to $V(\rho)$ correctly oriented so that if $(X,t) \in C_\rho(0,0)$ and $\left\langle X, \nu_\rho \right\rangle \leq -\delta \rho$ then $(X,t) \in \Omega^c$ (similarly, $(X,t) \in C_\rho(0,0), \left\langle X, \nu_\rho \right\rangle \geq \delta \rho$ implies $(X,t) \in \Omega$). Translated into the language of weak current flatness, $u_\infty \in \CF(\delta, \delta)$ in $C_\rho(0,0)$ in the direction $\nu_\rho$. 

Apply Lemma \ref{centeredflatonzerosideisflatonpositiveside} so that $u_\infty \in CF(K_1\delta, K_1\delta)$ in $C_{\rho/2}(0,0)$ in the direction $\nu_\rho$. Then Lemma \ref{centeredflatonbothsidesisextraflatonzeroside} implies that $u_\infty \in \CF(K_1\theta' \delta, K_1\theta'\delta)$ in $C_{K_2\theta' \rho/2}(0,0)$ in the direction $\nu^{(1)}_\rho$ where $|\nu^{(1)}_\rho - \nu_\rho| < K_1K_3\delta$. Returning to Lemma \ref{centeredflatonzerosideisflatonpositiveside} yields $u_\infty \in CF(K_1^2\theta' \delta, K_1^2\theta' \delta)$ in $C_{K_2 \theta' \rho/4}(0,0)$ in the direction $\nu^{(1)}_\rho$. Note that $\theta, \delta_n$ were chosen small enough to justify repeating this procedure arbitrarily many times. After $m$ iterations we have shown $u_\infty \in CF(\theta^m \delta, \theta^m \delta)$ in $C_{\eta^m \rho}(0,0)$ in the direction $\nu^{(m)}_\rho$ where $\eta \equiv K_2\theta'/4 < 1/2$. Additionally, for $m \geq 1$, $|\nu^{(m)}_\rho - \nu^{(m+1)}_\rho| < K_3\theta^m\delta$. From now on we will abuse notation and refer to all constants which only depend on the dimension by $C$. 

Define $\overline{\nu}_\rho \colonequals \lim_{m\rightarrow \infty} \nu^{(m)}_\rho$ and compute \begin{equation}\label{closenormals}|\overline{\nu}_\rho - \nu^{(m)}_\rho| < C \delta \theta^m \frac{1}{1-\theta} < C\delta \theta^m.\end{equation} For any $(P,\xi) \in C_\rho(0,0)$; there is some $m$ such that $(P,\xi) \in C_{\eta^m\rho}(0,0)$ but $(P,\xi) \notin C_{\eta^{m+1}\rho}(0,0)$. As $u_\infty \in CF(\theta^m \delta, \theta^m \delta)$ in the direction $\nu^{(m)}_\rho$ we can conclude \begin{equation}\label{flatatscalem}\begin{aligned} &\left\langle P, \nu^{(m)}_\rho\right\rangle \leq -\theta^m \delta \eta^m\rho \Rightarrow u_\infty(P,\xi) = 0\\
&\left\langle P, \nu^{(m)}_\rho \right\rangle \geq \theta^m \delta \eta^m\rho \Rightarrow u_\infty(P,\xi) > 0. \end{aligned} \end{equation}

We may write $\left\langle P, \overline{\nu}_\rho\right\rangle = \left\langle P, \nu_\rho^{(m)}\right\rangle + \left\langle P, \overline{\nu}_\rho - \nu^{(m)}_\rho\right\rangle$ and estimate, using equation \eqref{closenormals},  $| \left\langle P, \overline{\nu}_\rho - \nu^{(m)}_\rho\right\rangle| < C \delta \theta^m \eta^m \rho$. Then equation \eqref{flatatscalem} implies, \begin{equation}\label{flatatscalem2}\begin{aligned} &\left\langle P, \overline{\nu}_\rho\right\rangle \leq -C\theta^m \delta \eta^m\rho \Rightarrow u_\infty(P,\xi) = 0\\
&\left\langle P, \overline{\nu}_\rho \right\rangle \geq C \theta^m \delta \eta^m\rho \Rightarrow u_\infty(P,\xi) > 0. \end{aligned} \end{equation}
Hence, $$\frac{1}{\eta^m\rho} D[\Lambda(\rho)\cap C_{\eta^m \rho}(0,0), \partial \Omega \cap C_{\eta^m\rho}(0,0)] < C \theta^m \delta,$$ where $\Lambda(\rho)$ is the $n$-hyperplane containing $(0,0)$ that is perpendicular to $\overline{\nu}_\rho$. 

If $\eta^{m+1} \rho\leq s < \eta^m \rho$ one computes $$\frac{1}{s}D[\Lambda(\rho) \cap C_{s}(0,0), \partial \Omega\cap C_s(0,0)] \leq \frac{1}{\eta} D[\Lambda(\rho) \cap C_{\eta^m \rho}(0,0), C_{\eta^m\rho}(0,0)\cap \partial \Omega] \leq C \delta \frac{\theta^m}{\eta}.$$ As $\theta, \eta < 1$ we can write $\theta = \eta^\beta$ for some $\beta > 0$ and the above becomes \begin{equation}\label{distanceatstor}
\frac{1}{s}D[\Lambda(\rho)\cap C_s(0,0), \partial \Omega \cap C_s(0,0)]   \leq C_\eta \delta \left(\frac{s}{\rho}\right)^\beta, \forall s < \rho.
\end{equation}

Pick any $\rho_j \rightarrow \infty$. By compactness, $\Lambda(\rho_j) \rightarrow \Lambda_\infty$ in the Hausdorff distance (though we may need to pass to a subsequence that the limit plane may depend on the subsequence chosen). Then for any $s> 0$, equation \eqref{distanceatstor} implies $$\frac{1}{s}D[\Lambda_\infty\cap C_s(0,0), \partial \Omega\cap C_s(0,0)] = 0.$$ We conclude that $\partial \Omega = \Lambda_\infty$, and that $\Lambda_\infty$ is, in fact, independent of $\{\rho_j\}$. After a rotation and translation $\Omega = \{(X,t)\mid x_n > 0\}$. \end{proof}

\subsection{Flatness of the zero side implies flatness of the positive side: Lemma \ref{centeredflatonzerosideisflatonpositiveside}}

Before we begin we need two technical lemmas. The first allows us to conclude regularity in the time direction given regularity in the spatial directions.  We stated and proved this Lemma in Section \ref{initialholderregularity} so we will just state it again here.

\begin{lem*}
If $f$ satisfies the (adjoint)-heat equation in $\mathcal O$ and is zero outside $\mathcal O$ then $$\|f\|_{\mathbb C^{1.1/2}(\R^{n+1})} \leq c \|\nabla f\|_{L^\infty(\mathcal O)},$$ where $0 < c < \infty$ depends only on the dimension. \end{lem*}

This second lemma allows us to bound from below the normal derivative of a solution at a smooth point of $\partial \Omega_\infty$.  For ease of notation we will drop the subscript $\infty$ from $u_\infty, \Omega_\infty$ and $h_\infty$. However, all these results are proven with the same assumptions as Theorem \ref{poissonkernel1impliesflat}. 

\begin{lem}\label{normalderivativeatregularpoint}
Let $(Q,\tau)\in \partial \Omega$ be such that there exists a tangent ball (in the Euclidean sense) $B$ at $(Q,\tau)$ contained in $\overline{\Omega}^c$. Then $$\limsup_{\Omega \ni (X,t) \rightarrow (Q,\tau)} \frac{u(X,t)}{d((X,t), B)} \geq 1.$$
\end{lem}

\begin{proof}
Without loss of generality set $(Q,\tau) = (0,0)$ and let $(X_k, t_k) \in \Omega$ be a sequence that achieves the supremum, $\ell$. Let $(Y_k, s_k) \in B$ be such that $d((X_k,t_k), B) = \|(X_k,t_k)-(Y_k, s_k)\| \equalscolon r_k$. Define $u_k(X,t) \colonequals \frac{u(r_kX + Y_k, r_k^2t + s_k)}{r_k}$, $\Omega_k \colonequals \{(Y,s)\mid Y= (X -Y_k)/r_k, s = (t-s_k)/r_k^2,\; \mathrm{s.t.}\; (X,t)\in \Omega\}$ and $h_k(X,t) \colonequals h(r_kX + Y_k, r_k^2t + s_k)$. Then $$\int_{\R^{n+1}} u_k (\Delta \phi - \partial_t \phi)dXdt = \int_{\partial \Omega_k} h_k \phi d\sigma.$$

By assumption, the $u_k$ have uniform Lipschitz bound 1. Thus Lemma \ref{iflipschitzthenregularintime} implies that the $u_k$ are bounded uniformly in $\mathbb C^{1,1/2}$.  Therefore, perhaps passing to a subsequence, $u_k \rightarrow u_0$ uniformly on compacta. In addition, as there exists a tangent ball at $(0,0)$, $\Omega_k \rightarrow \{x_n >0\}$ in the Hausdorff distance norm (up to a rotation). We may assume, passing to a subsequence, that $\frac{X_k - Y_k}{r_k} \rightarrow Z_0, \frac{t_k - s_k}{r_k^2}\rightarrow t_0$ with $(Z_0, t_0) \in C_1(0,0)\cap \{x_n > 0\}$ and $u_0(Z_0, t_0) = \ell$. Furthermore, by the definition of supremum, for any $(Y,s) \in \{x_n > 0\}$ we have \begin{equation}\begin{aligned} u_0(Y,s) =& \lim_{k\rightarrow \infty} u(r_kY + y_k, r_k^2s + s_k)/r_k \\ \leq& \lim_{k\rightarrow \infty} \ell \frac{\mathrm{pardist}((r_kY + y_k, r_k^2s + s_k), B)}{r_k}\\ =& \lim_{k\rightarrow \infty} \ell \mathrm{pardist}((Y,s), B_k)\\ =& \ell y_n, \end{aligned}\end{equation}where $B_k$ is defined like $\Omega_k$ above.

Let $\phi \in C_0^{\infty}(\mathbb R^{n+1})$ be positive, then \begin{equation}\label{u0bound}\begin{aligned} \int_{\{x_n > 0\}} \ell x_n (\Delta \phi - \partial_t \phi)dXdt &\geq \int_{\{x_n > 0\}} u_0(X,t)  (\Delta \phi - \partial_t \phi)dXdt\\ &= \lim_{k \rightarrow \infty} \int_{\Omega_k} u_k(X,t) (\Delta \phi - \partial_t \phi)dXdt \\&=  \lim_{k\rightarrow \infty} \int_{\partial \Omega_k} h_k \phi d\sigma.\end{aligned}\end{equation} Integrating by parts yields \begin{equation*}\begin{aligned} \ell \int_{\{x_n = 0\}} \phi dx dt &=  \int_{\{x_n > 0\}} \ell x_n (\Delta \phi - \partial_t \phi)dXdt \\ &\stackrel{eqn.\; \eqref{u0bound}}{\geq}\lim_{k\rightarrow \infty} \int_{\partial \Omega_k} h_k \phi d\sigma\\ &\stackrel{h_k \geq 1}{\geq} \lim_{k\rightarrow \infty}  \int_{\{x_n =0\}} \phi dxdt\end{aligned}\end{equation*} Hence,  $\ell \geq 1$, the desired result. 
\end{proof}

We will first show that for ``past flatness",  flatness on the positive side gives flatness on the zero side.

\begin{lem}\label{flatonzerosideisflatonpositiveside}[Compare with Lemma 5.2 in \cite{anderssonweiss}]
Let $0 < \sigma \leq \sigma_0$ where $\sigma_0$ depends only on dimension. Then if $u\in PF(\sigma, 1)$ in $C_\rho(\tilde{X},\tilde{t})$ in the direction $\nu$, there is a constant $C$ such that $u \in PF(C\sigma, C\sigma)$ in $C_{\rho/2}(\tilde{X}+\alpha \nu, \tilde{t})$ in the direction $\nu$ for some $|\alpha| \leq C\sigma \rho$. 
\end{lem}
 
 \begin{proof}
Let $(\tilde{X},\tilde{t})= (0,0), \rho = 1$ and $\nu = e_n$. First we will construct a regular function which touches $\partial \Omega$ at one point.  

Define $$\eta(x, t) = e^{\frac{16(|x|^2 + |t+1|)}{16(|x|^2 + |t+1|)- 1}}$$ for $16(|x|^2 + |t+1|) < 1$ and $\eta(x,t) \equiv 0$ otherwise. Let $D \colonequals \{(x, x_n, t) \in C_1(0,0) \mid x_n > -\sigma + s\eta(x,t)\}$. Now pick $s$ to be the largest such constant that $C_1(0,0) \cap \Omega \subset D$. As $(0,-1) \in \partial \{u > 0\}$, there must be a touching point $(X_0, t_0) \in \partial D \cap \partial \Omega \cap \{-1\leq t \leq -15/16\}$ and $s \leq \sigma$.

Define the barrier function $v$ as follows:
\begin{equation}
\begin{aligned}
\Delta v + \partial_tv =0 &\: \mathrm{in}\; D, \\
v = 0 &\; \mathrm{in}\; \partial_p D \cap C_1(0,0)\\
v = (\sigma +x_n)&\; \mathrm{in}\; \partial_p D \cap \partial C_1(0,0).
\end{aligned}
\end{equation}

Note that on $\partial_p D\cap C_1(0)$ we have $u = 0$ because $D$ contains the positivity set. Also, as $|\nabla u| \leq 1$, it must be the case that $u(X,t) \leq \max\{0, \sigma + x_n\}$ for all $(X,t) \in C_1(0,0)$. Since, $v \geq u$ on $\partial_pD$ it follows that $v \geq u$ on all of $D$ (by the maximum principle for subadjoint-caloric functions). We now want to estimate the normal derivative of $v$ at $(X_0,t_0)$. To estimate from below, apply Lemma \ref{normalderivativeatregularpoint}, \begin{equation}\label{lowerboundforvnormal}1 \leq \limsup_{(X,t) \rightarrow (X_0,t_0)} \frac{u(X,t)}{\mathrm{pardist}((X,t), B)} \leq -\partial_\nu v(X_0,t_0)\end{equation} where $\nu$ is the normal pointing out of $D$ at $(X_0,t_0)$ and $B$ is the tangent ball at $(X_0,t_0)$ to $D$ contained in $D^c$.

To estimate from above, first consider $F(X,t) \colonequals (\sigma + x_n)-v$.  On $\partial_pD$, $$-\sigma \leq x_n - v \leq \sigma$$ thus (by the maximum principle) $0 \leq F(X,t) \leq 2\sigma$. As $\partial D$ is piecewise smooth domain, standard parabolic regularity gives $\sup_{D} |\nabla F(X,t)| \leq K\sigma$. Note, since $s\leq \sigma$, that $-\sigma + s\eta(x,t)$ is a function whose $\Lip(1,1)$ norm is bounded by a constant. Therefore, $K$ does not depend on $\sigma$.

Hence, \begin{equation}\label{normalderivativeofv}\begin{aligned} |\nabla v|-1 \leq& |\nabla v - e_n|  \leq K\sigma\\ \stackrel{\mathrm{eqn}\; \eqref{lowerboundforvnormal}}{\Rightarrow} 1 \leq&  -\partial_\nu v(Z) \leq 1+K\sigma.\end{aligned}\end{equation}

We want to show that $u \geq v - \tilde{K}\sigma x_n$ for some large constant, $\tilde{K}$, to be choosen later depending only on the dimension. Let $\tilde{Z} \colonequals (Y_0,s_0)$, where $s_0 \in (-3/4, 1), |y_0| \leq 1/2$ and $(Y_0)_n = 3/4$, and assume, in order to obtain a contradiction, that $u \leq v-\tilde{K}\sigma x_n$ at every point in $\{(Y,s_0)\mid |Y-Y_0| \leq 1/8\}$. We construct a barrier function, $w \equiv w_{\tilde{Z}}$, defined by \begin{equation*}
\begin{aligned}
\Delta w+ \partial_tw=0 \: \mathrm{in}\; D\cap \{t < s_0\},& \\
w= x_n \; \mathrm{on}\; \partial_p (D\{t < s_0\}) \cap \{(Y,s_0) \mid |Y-Y_0| < 1/8\},&\\
w = 0 \; \mathrm{on}\; \partial_p (D \cap \{t < s_0\})  \backslash  \{(Y,s_0)\mid |Y-Y_0| < 1/8\}.&
\end{aligned}
\end{equation*}

By our initial assumption (and the definition of $w$), $u \leq v-\tilde{K}\sigma x_n$ on $\partial_p(D\cap \{t < s_0\})$ and, therefore, $u\leq v-\tilde{K}\sigma x_n$ on all of $D \cap \{t < s_0\}$. Since $t_0 \leq -15/16$ we know $(X_0,t_0) \in \partial_p(D\cap \{t < s_0\})$. Furthermore,  the Hopf lemma gives an $\alpha > 0$ (independent of $\tilde{Z}$) such that $\partial_\nu w(X_0,t_0) \leq -\alpha$. With these facts in mind,  apply Lemma \ref{normalderivativeatregularpoint} at $(X_0,t_0)$ and recall equation \eqref{normalderivativeofv} to estimate, \begin{equation}\label{usinghopftocontradict} \begin{aligned}1 =& \limsup_{(X,t)\rightarrow (X_0,t_0)} \frac{u(X,t)}{\mathrm{pardist}((X,t), B)}\\ \leq& -\partial_\nu v(X_0,t_0) +\tilde{K}\sigma \partial_\nu w(X_0,t_0)\\ \leq& (1 + K\sigma)- \tilde{K}\alpha \sigma.\end{aligned}\end{equation} Setting $\tilde{K} \geq K/\alpha$ yields the desired contradiction. 

Hence, there exists  a point, call it $(\overline{Y},s_0)$, such that $|\overline{Y}-Y_0| \leq 1/8$ and $$(u-v)(\overline{Y},s_0) \geq -\tilde{K}\sigma(\overline{Y})_n \stackrel{(\overline{Y})_n \leq 1}{\geq} -\tilde{K}\sigma.$$ Apply the parabolic Harnack inequality to obtain,  $$\begin{aligned} \inf_{\{|X-Y_0| < 1/8\}} (u-v)(X,s_0-1/32) &\geq c_n\sup_{\{|\tilde{X}-Y_0|< 1/8\}} (u-v)(\tilde{X},s_0) \geq -c_n\tilde{K}\sigma\\
\stackrel{v \geq x_n -\sigma}{\Rightarrow}u(X,s_0-1/32) &\geq x_n-\sigma - c_n\tilde{K}\sigma,\end{aligned}$$  for all $X$ such that $|X-Y_0| < 1/8$. Ranging over all $s_0 \in (-3/4, 1)$ and $|y_0| \leq 1/2$, the above implies $$u(X,t) \geq x_n - C\sigma,$$ whenever $(X,t)$ satisfies $|x| < 1/2, |x_n -3/4| < 1/8, t\in (-1/2, 1/2)$ and for some constant $C > 0$ which depends only on the dimension. As $|\nabla u| \leq 1$ we can conclude, for any $(X,t)$ such that $|x| < 1/2, t\in (-1/2,1/2)$ and $3/4 \geq x_n \geq C\sigma$, that \begin{equation}\label{increasedpositiveflatness} u(X,t) \geq u(x, 3/4, t) - (3/4-x_n) \geq (x_n - C\sigma).\end{equation}

We now need to find an $\alpha$ such that $(0, \alpha, -1/4) \in \partial \Omega$. By the initial assumed flatness, and equation \eqref{increasedpositiveflatness}, $\alpha \in \R$ exists and $-\sigma \leq \alpha \leq C\sigma$ (here we pick $\sigma_0$ is small enough such that $C\sigma_0 < 1/4$).  

In summary we have shown that,\begin{itemize}
\item $(0, \alpha, -1/4) \in \partial \Omega,\; |\alpha| < C\sigma$
\item $x_n - \alpha \leq -3C\sigma/2 \Rightarrow x_n \leq -\sigma\Rightarrow u(X,t) = 0$. 
\item When $x_n - \alpha \geq 2C\sigma \Rightarrow x_n \geq C\sigma$, hence equation \eqref{increasedpositiveflatness}) implies $u(X,t) \geq (x_n - C\sigma)  \geq (x_n-\alpha - 2C\sigma)$.
\end{itemize}
Therefore $u \in PF(2C\sigma, 2C\sigma)$ in $C_{1/2}(0, \alpha, 0)$ which is the desired result. 
 \end{proof}
 
Lemma \ref{centeredflatonzerosideisflatonpositiveside} is the current version of the above and follows almost identically. Thus we will omut the full proof in favor of briefly pointing out the ways in which the argument differs.  

\begin{lem*}[Lemma \ref{centeredflatonzerosideisflatonpositiveside}]
Let $\Omega, u_\infty$ satisfy the assumptions of Proposition \ref{poissonkernel1impliesflat}. Furthermore, assume $u_\infty \in \CF(\sigma, 1/2)$ in $C_\rho(Q,\tau)$ in the direction $\nu$. Then there is a constant $K_1 > 0$ (depending only on dimension) such that $u_\infty \in CF(K_1\sigma, K_1\sigma)$ in $C_{\rho/2}(Q,\tau)$ in the direction $\nu$.
\end{lem*}

\begin{proof}[Proof of Lemma \ref{centeredflatonzerosideisflatonpositiveside}]
    We begin in the same way; let $(Q,\tau)= (0,0), \rho = 1$ and $\nu = e_n$. Then we recall the smooth function   $$\eta(x, t) = e^{\frac{16(|x|^2 + |t+1|)}{16(|x|^2 + |t+1|)- 1}}$$ for $16(|x|^2 + |t+1|) < 1$ and $\eta(x,t) \equiv 0$ otherwise. Let $D \colonequals \{(x, x_n, t) \in C_1(0,0) \mid x_n > -\sigma + s\eta(x,t)\}$. Now pick $s$ to be the largest such constant that $C_1(0,0) \cap \Omega \subset D$. Since $|x_n| > 1/2$ implies that $u(X,t) > 0$ there must be some touching point  $(X_0, t_0) \in \partial D \cap \partial \Omega \cap \{-1\leq t \leq -15/16\}$. Furthermore, we can assume that $s < \sigma + 1/2 < 2$. 
  
The proof then proceeds as above until equation \eqref{increasedpositiveflatness}. In the setting of  ``past flatness" we need to argue further; the boundary point is at the bottom of the cylinder, so after the cylinder shrinks we need to search for a new boundary point. However, in current flatness the boundary point is at the center of the cylinder, so after equation \eqref{holderincreasedpositiveflatness} we have completed the proof of Lemma \ref{centeredflatonzerosideisflatonpositiveside}. \end{proof}

\subsection{Flatness on Both Sides Implies Greater Flatness on the Zero Side: Lemma \ref{centeredflatonbothsidesisextraflatonzeroside}} In this section we prove Lemma \ref{centeredflatonbothsidesisextraflatonzeroside}. The outline of the argument is as follows: arguing by contradiction, we obtain a sequence $u_k$ whose free boundaries $\partial \{u_k > 0\}$ approaches the graph of a function $f$. Then we prove that this function $f$ is $C^{\infty}$ and this smoothness leads to a contradiction. 

Throughout this subsection, $\{u_k\}$ is a sequence of adjoint caloric functions such that $\partial \{u_k > 0\}$ is a parabolic regular domain and such that, for all $\varphi \in C^\infty_c(\mathbb R^{n+1})$, $$\int_{\{u_k > 0\}} u_k (\Delta \varphi - \partial_t \varphi) dXdt = \int_{\partial \{u_k > 0\}} h_k \varphi d\sigma.$$ We will also assume the $h_k \geq 1$ at $\sigma$-a.e. point on $\partial \{u_k > 0\}$ and $|\nabla u_k| \leq 1$. While we present these arguments for general $\{u_k\}$ it suffices to think of $u_k(X,t) \colonequals \frac{u(r_kX, r_k^2t)}{r_k}$ for some $r_k \downarrow 0$. 

\begin{lem}\label{blowupfunctionisagraph}[Compare with Lemma 6.1 in \cite{anderssonweiss}]
Suppose that $u_k \in CF(\sigma_k,\sigma_k)$ in $C_{\rho_k}(0,0)$, in direction $e_n$, with $\sigma_k \downarrow 0$. Define $f_k^+(x,t) = \sup \{d\mid (\rho_kx, \sigma_k\rho_kd, \rho_k^2t) \in \{u_k = 0\}\}$ and $f_k^-(x,t) = \inf \{d\mid (\rho_kx, \sigma_k\rho_kd, \rho_k^2t) \in \{u_k > 0\}\}$. Then, passing to subsequences, $f_k^+, f_k^-\rightarrow f$ in $L_{\mathrm{loc}}^\infty(C_1(0,0))$ and $f$ is continuous. 
\end{lem}

\begin{proof}
By scaling each $u_k$ we may assume $\rho_k \equiv 1$. Then define $$D_k \colonequals \{(y,d,t)\in C_1(0,0)\mid (y, \sigma_k d, t) \in \{u_k > 0\}\}.$$ Let $$f(x, t) \colonequals \liminf_{\stackrel{(y,s) \rightarrow (x, t)}{ k\rightarrow \infty}} f_k^-(y,s),$$ so that, for every $(y_0, t_0)$, there exists a $(y_k, t_k) \rightarrow (y_0, t_0)$ such that $f_k^-(y_k, t_k) \stackrel{k\rightarrow \infty}{\rightarrow} f(y_0, t_0)$. 

Fix a $(y_0,t_0)$ and note, as $f$ is lower semicontinuous, for every $\varepsilon > 0$, there exists a $\delta > 0, k_0 \in \mathbb N$ such that $$\{(y, d, t) \mid |y-y_0| < 2\delta, |t-t_0| < 4\delta^2, d \leq f(y_0, t_0) -\varepsilon\}\cap \overline{D}_k = \emptyset,\; \forall k \geq k_0.$$
Consequently \begin{equation}\label{vmofkflatness} x_n - f(y_0, t_0) \leq - \varepsilon \Rightarrow u_k(x, \sigma_k x_n, s) = 0,\; \forall (X,s) \in C_{2\delta}(Y_0, t_0).\end{equation} Together with the definition of $f$, equation \eqref{vmofkflatness} implies that there exist $\alpha_k \in \mathbb R$ with $|\alpha_k| < 2\varepsilon$ such that $(y_0, \sigma_k(f(y_0,t_0) + \alpha_k), t_0-\delta^2)\in \partial \{u_k > 0\}$. This observation, combined with equation \eqref{vmofkflatness} allows us to conclude, $u_k(\cdot, \sigma_k \cdot, \cdot) \in PF(3\sigma_k \frac{\varepsilon}{\delta}, 1)$ in $C_{\delta}(y_0, \sigma_k (f(y_0, t_0)+\alpha_k) , t_0)$, for $k$ large enough. 

As $\tau_k/\sigma_k^2 \rightarrow 0$ the conditions of Lemma \ref{flatonzerosideisflatonpositiveside} hold for $k$ large enough.  Thus, $u_k(\cdot, \sigma_k \cdot, \cdot) \in PF(C\sigma_k\frac{\varepsilon}{\delta}, C\sigma_k\frac{\varepsilon}{\delta})$ in $C_{\delta/2}(y_0, \sigma_k f(y_0, t_0) + \tilde{\alpha}_k, t_0)$ where $|\tilde{\alpha}_k| \leq C\sigma_k\varepsilon$. Thus if $z_n - (\sigma_k f(y_0, t_0) + \tilde{\alpha}_k) \geq C\varepsilon \sigma_k/2$  then $u_k(z, \sigma_k z_n, t) > 0$ for $(Z,t) \in C_{\delta/2}(y_k, \sigma_k f_k^-(y_k, t_k) + \alpha, t_k + \delta^2)$. In other words \begin{equation}\label{vmopluslessthanminus}\sup_{(Z,s) \in C_{\delta/2}(y_0, \sigma_k f(y_0, t_0) + \tilde{\alpha}_k, t_0 )} f^+_k(z, s) \leq f(y_0, t_0) + 3C\varepsilon.\end{equation} As $f^+_k \geq f^-_k$, if $$\tilde{f}(y_0,t_0) \colonequals  \limsup_{\stackrel{(y,s) \rightarrow (y_0, t_0)}{ k\rightarrow \infty}} f_k^+(y,s),$$ it follows (in light of equation \eqref{vmopluslessthanminus} that $\tilde{f} = f$.  Consequently, $f$ is continuous and $f^+_k, f^-_k \rightarrow f$ locally uniformly on compacta. \end{proof} 

We now show that $f$ is given by the boundary values of $w$, a solution to the adjoint heat equation in $\{x_n > 0\}$. 

\begin{lem}\label{boundaryvaluesofw}[Compare with Proposition 6.2 in \cite{anderssonweiss}]
Suppose that $u_k \in CF(\sigma_k,\sigma_k)$ in $C_{\rho_k}(0,0)$, in direction $e_n$ with $\rho_k \geq 0, \sigma_k \downarrow 0$. Further assume that, after relabeling, $k$ is the subsequence given by Lemma \ref{blowupfunctionisagraph}. Define $$w_k(x, d,t) \colonequals \frac{u_k(\rho_kx, \rho_kd, \rho_k^2t) - \rho_kd}{\sigma_k}.$$ Then, $w_k$ is bounded on $C_1(0,0)\cap \{x_n > 0\}$ (uniformly in $k$) and converges, in the $\mathbb C^{2,1}$-norm, on compact subsets of $C_1(0,0)\cap \{x_n > 0\}$ to $w$. Furthermore, $w$ is a solution to the adjoint-heat equation and $w(x, d,t)$ is non-increasing in $d$ when $d > 0$. Finally $w(x,0,t) = -f(x,t)$ and $w$ is continuous in $\overline{C_{1-\delta}(0,0)\cap \{x_n > 0\}}$ for any $\delta > 0$. \end{lem}

\begin{proof}
As before we rescale and set $\rho_k \equiv 1$.  Since $|\nabla u_k| \leq 1$ and $x_n \leq -\sigma_k \Rightarrow u_k = 0$ it follows that $u_k(X,t) \leq (x_n + \sigma_k)$. Which implies $w_k(X,t) \leq 1+ \tau_k$. On the other hand when $0 <  x_n \leq \sigma_k$ we have $u_k(X,t) - x_n \geq -x_n \geq -\sigma_k$ which means $w_k \geq -1$. Finally, if $x_n \geq \sigma_k$ we have $u_k(X,t) - x_n \geq (x_n - \sigma_k) - x_n\Rightarrow w_k \geq -1$. Thus, $|w_k| \leq 1$ in $C_1(0,0)\cap\{x_n > 0\}$.

By definition, $w_k$ is a solution to the adjoint-heat equation in $C_1(0,0) \cap \{x_n > \sigma_k\}$. So for any $K \subset \subset \{x_n > 0\}$ the $\{w_k\}$ are, for large enough $k$, a uniformly bounded sequence of solutions to the adjoint-heat equation on $K$. As $|w_k| \leq 1$, standard estimates for parabolic equations tell us that $\{w_k\}$ is uniformly bounded in $\mathbb C^{2+\alpha, 1+\alpha/2}(K)$. Therefore, perhaps passing to a subsequence, $w_k \rightarrow w$ in $\mathbb C^{2,1}(K)$. Furthermore, $w$ is also be a solution to the adjoint heat equation in $K$ and $|w| \leq 1$. A diagonalization argument allows us to conclude that $w$ is adjoint caloric on all of $\{x_n > 0\}$. 

Compute that $\partial_n w_k = (\partial_n u_k - 1)/(\sigma_k) \leq 0$, which implies, $\partial_n w \leq 0$ on $\{x_n > 0\}$. As such, $w(x,0,t) \colonequals \lim_{d\rightarrow 0^+} w(x, d, t)$ exists. We will now show that this limit is equal to $-f(x,t)$. If true, then regularity theory for adjoint-caloric functions implies that $w$ is continuous is $\overline{C_{1-\delta}\cap \{x_n > 0\}}$.

First we show that the limit is less than $-f(x,t)$. Let $\varepsilon > 0$ and pick $0 < \alpha \leq 1/2$ small enough so that $|w(x,\alpha,t) - w(x,0,t)| < \varepsilon$. For $k$ large enough we have $\alpha/\sigma_k > f(x,t) + 1 > f_k^-(x,t)$, therefore, 

\begin{equation}\label{upperboundonwregular}\begin{aligned}
&w(x,0,t) \leq w(x,\alpha, t) + \varepsilon = w_k(x, \sigma_k\frac{\alpha}{\sigma_k}, t) + \varepsilon + o_k(1)\\
& = \left(w_k(x, \sigma_k\frac{\alpha}{\sigma_k}, t) - w_k(x, \sigma_kf_k^-(x,t), t)\right) + w_k(x, \sigma_k f_k^-(x,t),t) + \varepsilon + o_k(1)\\
&\stackrel{\partial_n w_k \leq 0}{\leq} w_k(x, \sigma_kf_k^-(x,t), t) + o_k(1) + \varepsilon.
\end{aligned}
\end{equation}

By definition, $w_k(x, \sigma_k f_k^-(x,t),t) = -f_k^-(x,t)\rightarrow -f(x,t)$ uniformly in $C_{1-\delta}(0,0)$. In light of equation \eqref{upperboundonwregular}, this observation implies $w(x,0,t) \leq - f(x,t) + \varepsilon$. Since $\varepsilon > 0$ is arbitrary, we have $w(x,0,t) \leq -f(x,t)$. 

To show $w(x,0,t) \geq -f(x,t)$ we first define, for $S > 0, k\in \mathbb N$, $$\tilde{\sigma}_k =\frac{1}{S} \sup_{(Y,s) \in C_{2S\sigma_k}(x, \sigma_kf_k^-(x, t-S^2\sigma_k^2), t-S^2\sigma_k^2)}(f_k^-(x,t-S^2\sigma_k^2)-f_k^-(y,s)).$$ Observe that if $k$ is large enough (depending on $S,\delta$), then $(x, t-S^2\sigma_k^2) \in C_{1-\delta}(0,0)$. By construction, $\forall (Y,s) \in C_{2S\sigma_k}(x, \sigma_k f_k^-(x, t-S^2\sigma_k^2), t-S^2\sigma_k^2)$, $$y_n - \sigma_kf_k^-(x,t-S^2\sigma_k^2) \leq -S\sigma_k \tilde{\sigma}_k \Rightarrow y_n \leq \sigma_kf_k^-(y,s) \Rightarrow u_k(Y,s) = 0.$$
 In other words, $u_k \in PF(\tilde{\sigma}_k, 1)$ in $C_{S\sigma_k}(x, \sigma_kf_k^-(x, t-S^2\sigma_k^2), t)$. Note, by Lemma \ref{blowupfunctionisagraph}, $\tilde{\sigma}_k \rightarrow 0$. 

Apply Lemma \ref{flatonzerosideisflatonpositiveside}, to conclude that \begin{equation}\label{ukispastflatregular}
\hbox{$u_k \in PF(C\tilde{\sigma}_k, C\tilde{\sigma}_k)$ in $C_{S\sigma_k/2}(x, \sigma_k f_k^-(x,t-S^2\sigma_k^2)+\alpha_k, t)$ where $|\alpha_k| \leq CS\sigma_k\tilde{\sigma}_k$.}
\end{equation}

Define $D_k \equiv f_k^-(x,t-S^2\sigma_k^2) + \alpha_k/\sigma_k + S/2$. Pick $S > 0$ large such that $D_k \geq 1$ and then, for large enough $k$, we have $D_k - \alpha_k/\sigma_k - f_k^-(x,t-S^2\sigma_k^2) - CS\tilde{\sigma}_k > 0$ and $(x, \sigma_k D_k, t) \in C_{S\sigma_k/2}(x, \sigma_k f_k^-(x,t-S^2\sigma_k^2)+\alpha_k, t)$. As such, the flatness condition, \eqref{ukispastflatregular}, gives

\begin{equation}\label{ukwithdkregular}\begin{aligned}
u_k(x,\sigma_kD_k, t) &\geq \left(\sigma_k D_k -\sigma_k f_k^-(x,t-S^2\sigma_k^2)-\alpha_k -CS\tilde{\sigma}_k\sigma_k/2\right)^+\\
&= \frac{S\sigma_k}{2}\left(1-C\tilde{\sigma}_k\right).
\end{aligned}
\end{equation}

Plugging this into the definition of $w_k$, 

\begin{equation}\label{wkwithdkregular}\begin{aligned}
w_k(x, \sigma_k D_k, t) \geq& \frac{S}{2}(1-C\tilde{\sigma}_k) - D_k\\
=&  \frac{S}{2}(1-C\tilde{\sigma}_k) - (f_k^-(x,t-S^2\sigma_k^2) + \alpha_k/\sigma_k + S/2)\\
=& -f_k^-(x,t-S^2\sigma_k^2) + o_k(1) = -f(x,t) + o_k(1).
\end{aligned}
\end{equation}

We would like to replace the left hand side of equation \eqref{wkwithdkregular} with $w_k(x,\alpha, t)$ where $\alpha$ does not depend on $k$. We accomplish this by means of barriers; for $\varepsilon > 0$ define $z_\varepsilon$ to be the unique solution to \begin{equation}\label{vmodefofz}\begin{aligned}\partial_t z_\varepsilon + \Delta z_\varepsilon &=0,\: \mathrm{in}\; C_{1-\delta}(0,0)\cap \{x_n > 0\}\\
z_\varepsilon &= g_\varepsilon,\: \mathrm{on}\; \partial_p (C_{1-\delta}(0,0)\cap \{x_n > 0\})\cap \{x_n =0\}\\
z_\varepsilon &= -2,\: \mathrm{on}\; \partial_p (C_{1-\delta}(0,0)\cap \{x_n > 0\})\cap \{x_n > 0\},
\end{aligned}\end{equation} where $g_\varepsilon \in C^\infty(C_{1-\delta}(0,0))$ and $-f(x, t) - 2\varepsilon< g_\varepsilon(x,t) <- f(x,t) - \varepsilon$. By standard parabolic theory for any $\varepsilon > 0$ there exists an $\alpha > 0$ (which depends on $\varepsilon > 0$) such that $|x_n| < \alpha$ implies $|z_\varepsilon(x,x_n,t) - z_\varepsilon(x,0,t)| < \varepsilon/2$. Pick $k$ large enough so that $\sigma_k < \alpha$. We know $w_k$ solves the adjoint heat equation on $\{x_n \geq \sigma_k\}$ and, by equations \eqref{vmodefofz} and \eqref{wkwithdkregular}, $w_k \geq z_\varepsilon$ on $\partial_p (C_{1-\delta}(0,0)\cap \{x_n > \sigma_k\})$. Therefore, $w_k \geq z_\varepsilon$ on all of $C_{1-\delta}(0,0)\cap \{x_n > \sigma_k\}$. 

Consequently,  $$w_k(x,\alpha,t) \geq z_\varepsilon(x,\alpha,t) \geq z_\varepsilon(x,0,t) - \varepsilon/2 \geq -f(x,t) - 3\varepsilon.$$ As $k \rightarrow \infty$ we know $w_k(x,\alpha,t) \rightarrow w(x,\alpha,t) \leq w(x,0,t)$. This gives the desired result. 
\end{proof} 

The next step is to prove that the normal derivative of $w$ on $\{x_n =0\}$ is zero. This will allow us to extend $w$ smoothly over $\{x_n = 0\}$ and obtain regularity for $f$. 

\begin{lem}\label{normalderivativeiszero}
Suppose the assumptions of Lemma \ref{blowupfunctionisagraph} are satisfied and that $k$ is the subsequence identified in that lemma. Further suppose that $w$ is the limit function identified in Lemma \ref{boundaryvaluesofw}. Then $\partial_n w = 0$, in the sense of distributions, on  $\{x_n = 0\}\cap  C_{1/2}(0,0)$.
\end{lem}

\begin{proof}
Rescale so $\rho_k \equiv 1$ and define $g(x,t)= 5 - 8(|x|^2 + |t|)$. For $(x,0,t) \in C_{1/2}(0,0)$ we observe $f(x,0,t) \leq 1 \leq g(x,0,t)$. We shall work in the following set $$Z\colonequals \{(x, x_n, t)\mid |x|, |t| \leq 1, x_n \in \R\}.$$ For any $\phi(x,t)$, define $Z^+(\phi)$ to be the set of points in $Z$ above the graph $\{(X,t)\mid x_n = \phi(x,t)\}$, $Z^-(\phi)$ as set of points below the graph and $Z^0(\phi)$ as the graph itself. Finally, let $\Sigma_k \colonequals \{u_k > 0\} \cap Z^0(\sigma_k g)$. 

Recall, for any Borel set $A$, we define the ``surface measure", $\mu(A) = \int_{-\infty}^\infty \mathcal{H}^{n-1}(A\cap \{s=t\}) dt$. If $k$ is sufficiently large, and potentially adding a small constant to $g$, $\mu(Z^0(\sigma_kg)\cap \partial \{u_k > 0\} \cap C_{1/2}(0,0)) = 0$. 

There are three claims, which together prove the desired result.

\medskip

\noindent {\bf Claim 1:} $$\mu(\partial \{u_k > 0\} \cap Z^-(\sigma_kg)) \leq \int_{\Sigma_k} \partial_n u_k - 1dxdt +\mu(\Sigma_k)+ C\sigma_k^2$$

{\it Proof of Claim 1}: For any positive $\phi \in C_0^\infty(C_1(0,0))$ we have \begin{equation}\label{vmofirststeptoclaim1} \begin{aligned} \int_{\partial \{u_k > 0\}} \phi d\mu \leq \int_{\partial \{u_k > 0\}} \phi d\mu =& \int_{\{u_k > 0\}} u_k(\Delta \phi - \partial_t \phi)dXdt\\  =& -\int_{\{u_k > 0\}} \nabla u_k \cdot \nabla \phi + u_k \partial_t\phi\end{aligned}\end{equation} (we can use integration by parts because, for almost every $t$, $\{u_k > 0\} \cap \{s =t\}$ is a set of finite perimeter). Let $\phi \rightarrow \chi_{Z^-(\sigma_kg)} \chi_{C_1}$ (as functions of bounded variation) and, since $|t| > 3/4$ or $|x|^2 > 3/4$ implies $u(x, \sigma_kg(x,t), t) = 0$, equation \eqref{firststeptoclaim1} becomes \begin{equation}\label{vmointegrateoverE} \mu(\partial \{u_k > 0\} \cap Z^-(\sigma_kg)) \leq - \int_{\Sigma_k} \frac{\nabla u_k \cdot \nu + \sigma_k u_k\; \mathrm{sgn}(t)}{\sqrt{1+\sigma^2_k(|\nabla_xg(x,t)|^2 + 1)}} d\mu,\end{equation} where $\nu(x,t) = (\sigma_k\nabla g(x,t), -1)$ points outward spatially in the normal direction. 

We address the term with $\mathrm{sgn}(t)$ first; the gradient bound on $u_k$ tells us that $|u_k| \leq C\sigma_k$ on $\Sigma_k$, so \begin{equation}\label{vmotimepartsilly}\left|\sigma_k\int_{\Sigma_k} \frac{u_k\; \mathrm{sgn}(t)}{\sqrt{1+\sigma^2_k(|\nabla_xg(x,t)|^2 + 1)}} d\mu\right| \leq C\sigma_k^2.\end{equation} 

To bound the other term, note that $\frac{d\mu}{\sqrt{1+\sigma^2_k(|\nabla_xg(x,t)|^2 + 1)}} = dxdt$ where the latter integration takes place over $E_k = \{(x,t) \mid (x, \sigma_kg(x,t), t) \in \Sigma_k\} \subset \{x_n = 0\}$. Then integrate by parts in $x$ to obtain  \begin{equation}\label{vmointegratingbypartsonek}\begin{aligned}&\int_{E_k} (\sigma_k\nabla g(x,t), -1)\cdot \nabla u_k(x, \sigma_kg(x,t), t) dx dt = \int_{\partial E_k} \sigma_ku_k(x,\sigma_kg(x,t),t)\partial_\eta gd\mathcal H^{n-2}dt\\&-\int_{E_k} \sigma_k u_k (x,\sigma_kg(x,t),t)\Delta g(x,t) +\sigma_k^2 \partial_nu_k(x,\sigma_kg(x,t),t)|\nabla g|^2dxdt\\
&-\int_{E_k}\partial_n u_k(x,\sigma_kg(x,t),t)- 1dxdt +\mathcal L^n(E_k), \end{aligned}\end{equation} where $\eta$ is the outward space normal on $\partial E_k$. Since $u_k(x,\sigma_kg(x,t),t) = 0$ for $(x,t) \in \partial E_k$ the first term zeroes out. 

The careful reader may object that $E_k$ may not be a set of finite perimeter and thus our use of integration by parts is not justified. However, for any $t_0$, we may use the coarea formula with the $L^1$ function $\chi_{\{u(x,\sigma_kg(x,t_0),t_0) > 0\}}$ and the smooth function $\sigma_kg(-, t_0)$ to get $$\infty > \int \sigma_k |\nabla g(x,t_0)| \chi_{\{u(x,\sigma_kg(x,t_0),t_0) > 0\}}dx = \int_{-\infty}^\infty \int_{\{(x,t_0)\mid \sigma_kg(x,t_0) = r\}} \chi_{\{u(x,r,t_0) >0\}} d\mathcal H^{n-2}(x) dr.$$ Thus $\{(x,t_0)\mid \sigma_kg(x,t_0) > r\}\cap \{(x,t)\mid u(x,\sigma_kg(x,t_0),t_0) >0\}$ is a set of finite perimeter for almost every $r$. Equivalently, $\{(x,t_0)\mid \sigma_k (g(x,t_0)+\varepsilon) > 0\}\cap \{(x,t)\mid u(x,\sigma_k(g(x,t_0)+\varepsilon),t_0) >0\}$ is a set of finite perimeter for almost every $\varepsilon \in \mathbb R$. Hence, there exists a $\varepsilon >0$ aribtrarily small such that if we replace $g$ by $g+\varepsilon$ then $E_k\cap \{t = t_0\}$ will be a set of finite perimeter for almost every $t_0$. Since we can perturb $g$ slightly without changing the above arguments, we may safely assume that $E_k$ is a set of finite perimeter for almost every time slice. 

Observe that $\Delta g$ is bounded above by a constant, $|u_k| \leq C\sigma_k$ on $\Sigma_k$, $|\partial_n u_k| \leq 1$ and finally $\mu(\Sigma_k) \geq \mathcal L^n(E_k)$. Hence, $$\mu(\partial \{u_k > 0\} \cap Z^-(\sigma_kg)) \leq \int_{E_k} \partial_n u_k - 1dxdt + \mu(\Sigma_k) + C\sigma_k^2.$$ Note integrating over $E_k$ is the same as integrating over $\Sigma_k$ modulo a factor of $\sqrt{1+|\sigma_k\nabla_x g|^2} \simeq 1+ \sigma_k^2$.  As $ |\partial_n u_k - 1|$ is bounded we may rewrite the above as $$\mu(\partial \{u_k > 0\} \cap Z^-(\sigma_kg)) \leq \int_{\Sigma_k} \partial_n u_k - 1d\mu + \mu(\Sigma_k) + C\sigma_k^2,$$ which is the desired inequality.

Before moving on to Claim 2, observe that arguing as in equations \eqref{vmotimepartsilly} and \eqref{vmointegratingbypartsonek},  \begin{equation}\label{vmoonlyenmatters}
\int_{\Sigma_k}\frac{(\sigma_k \nabla_x g(x,t), 0, \sigma_k \mathrm{sgn}(t))}{\sqrt{1+\|(\sigma_k \nabla_x g(x,t), 0, \sigma_k \mathrm{sgn}(t))\|^2}} \cdot (\nabla_x w_k, 0, w_k) d\mu \stackrel{k\rightarrow \infty}{\rightarrow} 0,
\end{equation}
which will be useful to us later. 

\medskip

\noindent {\bf Claim 2:} $$\mu(\Sigma_k) - C_2 \sigma_k^2 \leq \mu(\partial \{u_k > 0\} \cap Z^-(\sigma_kg)).$$

{\it Proof of Claim 2:}  Let $\nu_k(x,t)$ the inward pointing measure theoretic space normal to $\partial \{u_k > 0\} \cap \{s= t\}$ at the point $x$. For almost every $t$ it is true that $\nu_k$ exists $\mathcal H^{n-1}$ almost everywhere. Defining $\nu_{\sigma_k g}(X,t) = \frac{1}{\sqrt{1+256\sigma_k^2|x|^2}}(-\sigma_k16x, 1,0)$, we have $$\mu(\partial \{u_k > 0\} \cap Z^-(\sigma_kg)) = \int_{\partial \{u_k > 0\} \cap Z^-(\sigma_kg)} \nu_k \cdot \nu_k d\mu \geq \int_{\partial^* \{u_k > 0\} \cap Z^-(\sigma_kg)} \nu_k \cdot \nu_k d\mu \geq $$$$\int_{\partial^* \{u_k > 0\} \cap Z^-(\sigma_kg)} \nu_k \cdot \nu_{\sigma_kg} d\mu \stackrel{\mathrm{div}\;\mathrm{thm}}{=} -\int_{Z^-(\sigma_kg)\cap \{u_k > 0\}} \mathrm{div}\; \nu_{\sigma_kg}dXdt + \int_{\Sigma_k} 1d\mu.$$ 

We compute $|\mathrm{div}\; \nu_{\sigma_kg}| = \left|\frac{-16\sigma_k(n-1)}{\sqrt{1+ 256\sigma_k^2|x|^2}} + \frac{3\sigma_k^3(16*256)|x|^2}{\sqrt{1+ 256\sigma_k^2|x|^2}^3}\right|  \leq C\sigma_k$. Since the ``width" of $Z^-(\sigma_kg)\cap \{u_k > 0\}$ is of order $\sigma_k$, the claim follows. 

\medskip

\noindent {\bf Claim 3:} $$\int_{\Sigma_k} |\partial_n w_k| \stackrel{k\rightarrow \infty}{\rightarrow} 0.$$

{\it Proof of Claim 3:} We first notice that $\partial_n u_k \leq 1$ and, therefore, $\partial_n w_k \leq 0$. To show the limit above is at least zero we compute \begin{equation}\label{vmolimitingtozero} \begin{aligned} \int_{\Sigma_k} \partial_n w_k d\mu =& \int_{\Sigma_k} \frac{\partial_n u_k - 1}{\sigma_k } d\mu \\
\stackrel{\mathrm{Claim 1}} \geq&\frac{\mu(\partial \{u_k > 0\} \cap Z^-(\sigma_kg))}{\sigma_k} - \frac{\mu(\Sigma_k)}{\sigma_k} - C\sigma_k \\
\stackrel{\mathrm{Claim 2}}{\geq}& \frac{\mu(\Sigma_k)}{\sigma_k} -\frac{\mu(\Sigma_k)}{\sigma_k} - C\sigma_k\rightarrow 0.\\
\end{aligned}
\end{equation}

\medskip

We can now combine these claims to reach the desired conclusion. We say that $\partial_n w = 0$, in the sense of distributions on $\{x_n = 0\}$, if, for any $\zeta \in C^\infty_0(C_{1/2}(0,0))$, $$\int_{\{x_n = 0\}} \partial_n w\zeta = 0.$$ 

By claim 3 \begin{equation}\label{vmowhatdoesclaim3say} 0 = \lim_{k\rightarrow \infty} \int_{\Sigma_k} \zeta \partial_n w_k .\end{equation} On the other hand equation \eqref{vmoonlyenmatters} (and $\zeta$ bounded) implies \begin{equation}\label{vmofinalmanipulationstozero} \lim_{k\rightarrow \infty} \int_{\Sigma_k} \zeta \partial_n w_k =\lim_{k\rightarrow \infty} \int_{\Sigma_k} \zeta \nu_{\Sigma_k} \cdot (\nabla w_k, w_k),\end{equation} where $\nu_{\Sigma_K}$ is the unit normal to $\Sigma_k$ (thought of as a Lispchitz graph in $(x,t)$) pointing upwards.  Together, equations \eqref{vmowhatdoesclaim3say}, \eqref{vmofinalmanipulationstozero} and the divergence theorem in the domain $Z^+(\sigma_kg) \cap C_{1/2}(0,0)$ have as a consequence \begin{equation*}\begin{aligned} 0 =& \lim_{k\rightarrow \infty} \int_{Z^+(\sigma_kg)} \mathrm{div}_{X,t} (\zeta (\nabla_X w_k, w_k)) dXdt \\ =&\lim_{k\rightarrow \infty}  \int_{Z^+(\sigma_k)} \nabla_X \zeta \cdot \nabla_X w_k + (\partial_t \zeta)w_k + \zeta (\Delta_X w_k + \partial_t w_k) dXdt\\ \stackrel{\Delta w + \partial_t w = 0}{=}& \int_{\{x_n > 0\}} \nabla_X w\cdot \nabla_X \zeta + (\partial_t \zeta) w_k dXdt \\ \stackrel{\mathrm{integration}\; \mathrm{by}\; \mathrm{parts}}{=}& \int_{\{x_n = 0\}} w_n \zeta dxdt - \int_{\{x_n > 0\}} \zeta (\Delta_X w + \partial_t w).\end{aligned}\end{equation*} As $w$ is a solution to the adjoint heat equation this implies that $\int_{\{x_n = 0\}} \partial_n w\zeta = 0$ which is the desired result. 
\end{proof}

From here it is easy to conclude regularity of $f$. 

\begin{cor}\label{fhashighregularity}
Suppose the assumptions of Lemma \ref{blowupfunctionisagraph} are satisfied and that $k$ is the subsequence identified in that lemma. Then $f\in C^{\infty}(C_{1/2}(0,0))$ and in particular the $\mathbb C^{2+\alpha,1+\alpha}$ norm of $f$ in $C_{1/4}$ is bounded by an absolute constant.
\end{cor}

\begin{proof}
Extend $w$ by reflection across $\{x_n = 0\}$. By Lemma \ref{normalderivativeiszero} this new $w$ satisfies the adjoint heat equation  in all of $C_{1/2}(0,0)$ (recall a continuous weak solution to the adjoint heat equation in the cylinder is actually a classical solution to the adjoint heat equation). Since $\|w\|_{L^\infty(C_{3/4}(0,0))} \leq 2$, standard regularity theory yields the desired results about $-f = w|_{x_n =0}$. 
\end{proof}

We can use this regularity to prove Lemma \ref{centeredflatonbothsidesisextraflatonzeroside}.

\begin{proof}[Proof of Lemma \ref{centeredflatonbothsidesisextraflatonzeroside}]
Without loss of generality let $(Q,\tau)  = (0,0)$ and we will assume that the conclusions of the lemma do not hold. Choose a $\theta \in (0,1)$ and, by assumption, there exists $\rho_k$ and  $\sigma_k \downarrow 0$ such that $u \in CF(\sigma_k, \sigma_k)$  in $C_{\rho_k}(0,0)$ in the direction $\nu_k$ (which after a harmless rotation we can set to be $e_n$) but so that $u$ is not in $\widetilde{CF}(\theta \sigma_k, \theta \sigma_k)$ in $C_{c(n)\theta \rho_k}(0,0)$ in any direction $\nu$ with $|\nu_k-\nu| \leq C\sigma_k$ and for any constant $c(n)$. Let $u_k = \frac{u(\rho_k X, \rho_k^2t)}{\rho_k}$. It is clear that $u_k$ is adjoint caloric, that its zero set is a parabolic regular domain, that $|\nabla u_k| \leq 1$ and $h_k \geq 1$.

By Lemma \ref{blowupfunctionisagraph} we know that there exists a continuous function $f$ such that  $\partial \{u_k > 0\} \rightarrow \{(X,t)\mid x_n = f(x,t)\}$ in the Hausdorff distance sense. Corollary \ref{fhashighregularity} implies that there is a universal constant, call it $K$, such that \begin{equation}\label{vmofdoesntgetbig} f(x,t) \leq f(0,0) + \nabla_x f(0,0) \cdot x  + K(|t| + |x|^2)\end{equation} for $(x,t) \in C_{1/4}(0,0)$. Since $(0,0) \in \partial \{u_k > 0\}$ for all $k$, $f(0,0) = 0$. If $\theta \in (0,1)$, then there exists a $k$ large enough (depending on $\theta$ and the dimension) such that $$f_k^+(x,t) \leq  \nabla_x f(0,0) \cdot x + \theta^2/4K,\; \forall (x,t) \in C_{\theta/(4K)}(0,0),$$ where $f_k^+$ is as in Lemma \ref{blowupfunctionisagraph}. Let $$\nu_k \colonequals \frac{(-\sigma_k \nabla_x f(0,0), 1)}{\sqrt{1+ |\sigma_k  \nabla_x f(0,0)|^2}}$$ and compute \begin{equation}\label{changingthedirecitonregular}x \cdot \nu_k \geq \theta^2 \sigma_k/4K \Rightarrow x_n \geq \sigma_k x' \cdot \nabla_x f(0,0) + \theta^2 \sigma_k/4K \geq f^+_k(x,t).\end{equation} Therefore, if $(X,t) \in  C_{\theta/(4K)}(0,0)$ and $x\cdot \nu_k \geq \theta^2 \sigma_k/4K$, then $u_k(X,t) > 0$. Arguing similarly for $f_k^-$ we can see that $u_k \in \widetilde{CF}(\theta \sigma_k, \theta \sigma_k)$ in $C_{\theta/4K}(0,0)$ in the direction $\nu_k$. It is easy to see that $|\nu_k - e_n| \leq C\sigma_k$ and so we have the desired contradiction. 
\end{proof}

\section{Non-tangential limit of $F(Q,\tau)$: proof of Lemma \ref{equaltoh}}\label{appendix:nontangentiallimit}

If $\partial \Omega$ is smooth and $F(Q,\tau)$ is the non-tangential limit of $\nabla u(Q,\tau)$ then $F(Q,\tau) = h(Q,\tau)\hat{n}(Q,\tau)$. In this section we prove this is true when $\partial \Omega$ is a parabolic chord arc domain; this is Lemma \ref{equaltoh}. Easy modifications of our arguments will give the finite pole result, Lemma \ref{equaltohfinitepole}. Before we begin, we establish two geometric facts about parabolic regular domains which will be useful.  The first is on the existence of ``tangent planes" almost everywhere. 

\begin{lem}\label{prdomainshavetangentplanes}
Let $\Omega$ be a parabolic regular domain. For $\sigma$-a.e. $(Q,\tau) \in \partial \Omega$ there exists a $n$-plane $P\equiv P(Q,\tau)$ such that $P$ contains a line parallel to the $t$ axis and such that for any $\varepsilon > 0$ there exists a $r_\varepsilon > 0$ where $$0 < r < r_\varepsilon \Rightarrow \left|\left\langle \hat{n}(Q,\tau), Z-Q\right\rangle \right| < r\varepsilon,\; \forall (Z,t) \in \partial \Omega \cap C_r(Q,\tau).$$
\end{lem}

\begin{proof}
For $(Q,\tau) \in \partial \Omega, r > 0$ define $$\gamma_\infty(Q,\tau, r) = \inf_P \sup_{(Z,\eta) \in C_r(Q,\tau)\cap \partial \Omega} \frac{d((Z,\eta), P)}{r}$$ where the infinum is taken over all $n$-planes through $(Q,\tau)$ with a line parallel to the $t$-axis. Let $P(Q,\tau, r)$ be a plane which achieves the minimum.  For $(Q,\tau) \in \partial \Omega, r > 0$ we have $$\gamma_\infty(Q,\tau, r)^{n+3} \leq 16^{n+3}\gamma(Q,\tau, 2r)$$ (see \cite{hlnbigpieces}, equation (2.2)). 

Parabolic uniform rectifiability demands that $\frac{\gamma(Q,\tau, 2r)}{2r}$ be integrable at zero for $\sigma$-a.e. $(Q,\tau)$. Thus it is clear that $\gamma_\infty(Q,\tau, r) \stackrel{r\downarrow 0}{\rightarrow} 0$ for $\sigma$-a.e. $(Q,\tau)$. Let $r_j \downarrow 0$ and $P_j \colonequals P(Q,\tau, r_j)$, which approximate $\partial \Omega$ near $(Q,\tau)$ increasingly well. Passing to a subsequence, compactness implies $P_j \rightarrow P_\infty(Q,\tau)$. Our lemma is proven if $P_\infty(Q,\tau)$ does not depend on the sequence (or subsequence) chosen. 

If $K$ is any compact set then for $\mathcal L^1$-a.e. $t$, $\partial \Omega \cap K \cap \{s = t\}$ is a set of locally finite perimeter.  The theory of sets of locally finite perimeter (see e.g. \cite{evansandgariepy}) says that for $\mathcal H^{n-1}$-a.e. point on this time slice there exists a unique measure theoretic space-normal. Therefore, for $\sigma$-a.e. $(Q,\tau) \in \partial \Omega$ there is a measure theoretic space normal $\hat{n} \colonequals \hat{n}((Q,\tau))$ (this vector is normal to the time slice as opposed to the whole surface). Let $(Q,\tau)$ be a point both with a measure theoretic normal and such that $\gamma_\infty(Q,\tau, r) \rightarrow 0$. We claim $P_\infty(Q,\tau) = \hat{n}(Q,\tau)^\perp$, and thus is independent of the sequence $r_j \downarrow 0$.

As $\sup_{(Z,\eta) \in C_{r_j}(Q,\tau)\cap \partial \Omega}\frac{d((Z,\eta), P(Q,\tau, r_j))}{r_j}\rightarrow 0$ it follows that $\sup_{(Z,\tau) \in C_{r_j}(Q,\tau)\cap \partial \Omega}\frac{d((Z,\tau), P(Q,\tau, r_j))}{r_j}\rightarrow 0$. Since a point-wise tangent plane must also be a measure theoretic tangent plane, $P_j \rightarrow \hat{n}^\perp$. 
 \end{proof}

 A consequence of the above lemma is a characterization of the infinitesimal behavior of $\sigma$. 
 
 \begin{cor}\label{densityofsigma}
 Let $\Omega$ be a parabolic chord regular domain. For $\sigma$-a.e. $(Q,\tau) \in \partial \Omega$ and any $\varepsilon > 0$, we can choose an $R \equiv R(\varepsilon, Q, \tau) > 0$ such that for $r< R$, $$\left|\frac{\sigma(\Delta_r(Q,\tau))}{r^{n+1}} - 1\right| < \varepsilon.$$
 \end{cor}
 
 \begin{proof}
 Let $\Omega_{r,(Q,\tau)} = \{(P,\eta)\mid (Q+rP, r^2\eta+\tau) \in \Omega\}$.  Lemma \ref{prdomainshavetangentplanes} tells us that, for any compact set $K$ containing $0$, after a rotation which depends on $(Q,\tau)$, we have $K \cap \Omega_{r,(Q,\tau)} \rightarrow K \cap \{(X,t)\mid x_n > 0\}$ in the Hausdorff distance sense. In particular $\chi_{\Omega_{r,(Q,\tau)}} \equalscolon \chi_r \rightarrow \chi_{\{x_n > 0\}}$ in $L^1_{\mathrm{loc}}$. This convergence immediately gives, using the divergence theorem (on each time slice), that $\hat{n}_r\sigma_r \colonequals \hat{n}_r\sigma|_{\partial \Omega_{r,(Q,\tau)}}$ converges weakly to $e_n\sigma|_{\{x_n > 0\}}$ (here $\hat{n}_r$ is the measure theoretic space normal to $\partial \Omega_{r,(Q,\tau)}$). 
 
 $\hat{n}(Q,\tau) \in L_{\mathrm{loc}}^1(d\sigma)$, therefore, $\sigma$-a.e. $(Q,\tau) \in \partial \Omega$ is a Lebesgue point for $\hat{n}(Q,\tau)$.  Assume that $(Q,\tau)$ is a Lesbesgue point and that the tangent plane at $(Q,\tau)$ is $\{x_n =0\}$. Then, $$\lim_{r\downarrow 0} \fint_{\Delta_r(Q,\tau)}  \hat{n}d\sigma = e_n \Leftrightarrow \lim_{r\downarrow 0} \fint_{C_1(0,0)\cap \partial \Omega_{r, (Q,\tau)}} \hat{n}_r d\sigma_r = e_n.$$ Weak convergence implies (recall that $\sigma(C_1(0,0)\cap \{x_n =0\}) = 1$) $$ \liminf_{r\downarrow 0} \frac{1}{\sigma_r(C_1(0,0)\cap \partial \Omega_{r, (Q,\tau)})} \leq 1 \leq \limsup_{r\downarrow 0} \frac{1}{\sigma_r(C_1(0,0)\cap \partial \Omega_{r,(Q,\tau)})}.$$
 
As $C_1(0,0)$ is a set of continuity for $\sigma|_{\{x_n = 0\}}$ we can conclude that $\lim_{r\downarrow 0} \sigma_r(C_1(0,0)\cap \partial \Omega_{r, (Q,\tau)}) = \sigma(C_1(0,0)\cap \{x_n = 0\}) = 1$. Remembering $\sigma_r(C_1(0,0)\cap \partial \Omega_{r, (Q,\tau)}) = \frac{1}{r^{n+1}}\sigma(\Delta_{r}(Q,\tau))$ (see \cite{evansandgariepy}, Chapter 5.7, pp. 202-204 for more details) the result follows. 
\end{proof}

Our proof of Lemma \ref{equaltoh} is in two steps; first, we show that $F(Q,\tau)$ points in the direction of the measure theoretic space normal $\hat{n}(Q,\tau)$. Here we follow very closely the proof of Lemma 3.2 in \cite{kenigtorotwophase}.

\begin{lem}\label{onlynormaldirection}
For $\sigma$-a.e. $(Q,\tau)\in \partial \Omega$ we have $F(Q,\tau) = \left\langle F(Q,\tau), \vec{n}(Q,\tau)\right\rangle \vec{n}(Q,\tau)$. 
\end{lem}

\begin{proof}
Let $(Q,\tau) \in \partial \Omega$ be a point of density for $F$ and $M_1(h)$, be such that there is a tangent plane at $(Q,\tau)$ (in the sense of Lemma \ref{prdomainshavetangentplanes}), satisfy $F(Q,\tau), M_1(h)(Q,\tau) < \infty$ and be such that $\nabla u$ converges non-tangentially to $F$ at $(Q,\tau)$. 

In order to discuss the above conditions in a quantative fashion we introduce, for $\varepsilon, \eta > 0$:  \begin{equation}\label{thedeltas}
\begin{aligned}
\delta_\varepsilon(r) &\colonequals \frac{1}{r^{n+1}} \sigma(\{ (P, \zeta) \in C_{2r}(Q,\tau) \cap \partial \Omega \mid |F(P,\zeta) - F(Q,\tau)| > \varepsilon\})\\
\delta'(r) &\colonequals \frac{1}{r^{n+1}} \sigma(\{(P, \zeta) \in C_{2r}(Q,\tau) \cap \partial \Omega \mid M_1(h)(P,\zeta) \geq 2M_1(h)(Q,\tau)\})\\
\delta_\eta''(r) &\colonequals \frac{1}{r^{n+1}}\sigma((P, \zeta) \in C_{2r}(Q,\tau) \cap \partial \Omega\backslash E(\eta)).\\
\end{aligned}
\end{equation}
To define $E(\eta)$, first, for any $\varepsilon > 0$ and $\lambda > 0$, let $$H(\lambda, \varepsilon)\colonequals \{(P,\zeta) \in \partial \Omega \mid |F(P,\zeta) - \nabla u(X,t)| < \varepsilon,\; \forall (X,t) \in \Gamma_{10}^\lambda(P,\zeta)\}.$$

By Corollary \ref{nontangentiallimiteverywhere}, for each $\varepsilon > 0$ and $\sigma$-a.e. $(P,\zeta) \in \partial \Omega$ there is some $\lambda$ such that $(P,\zeta) \in H(\lambda,\varepsilon)$. Arguing as in the proof of Egoroff's theorem (see, e.g., Theorem 3, Chapter 1.2 in \cite{evansandgariepy}), for any $\eta > 0$ we can find a $\lambda(\varepsilon, \eta)$ such that $\sigma(\partial \Omega\backslash H(\lambda(\varepsilon,\eta), \varepsilon)) < \eta$.  Let $\varepsilon_n = 2^{-n}$ and $\eta_n = 2^{-n-1}\eta$ to obtain $\lambda_n \colonequals \lambda(\varepsilon_n, \eta_n)$ as above. Then define $E(\eta) = \bigcap_{n \geq 0} H(\lambda_n, \varepsilon_n)$. Note, that $\sigma(\partial \Omega \backslash \bigcup_{\eta > 0} E(\eta)) = 0,$ as such, for $\sigma$-a.e. $(Q,\tau) \in \partial \Omega$ there is some $\eta > 0$ such that $(Q,\tau)$ is a density point for $E(\eta)$. At those points (for the relevant $\eta$) we have $\delta_\eta''(r) \rightarrow 0$. 

$(Q,\tau)$ is a point of density for $M_1(h), F$, hence $\delta'(r), \delta_\varepsilon(r) \rightarrow 0$ for any $\varepsilon > 0$. Additionally, $\partial \Omega$ has a tangent plane at $(Q,\tau)$, so there exists an $n$-plane $V$ (containing a line parallel to the $t$-axis), a function $\ell(r)$ and $R, \eta > 0$  such that \begin{equation}\label{definitionofell}\begin{aligned}
\lim_{r\downarrow 0} \frac{\ell(r)}{r} =& 0\\
\sup_{(P,\zeta) \in C_{2r}(Q,\tau)\cap \partial \Omega} d((P,\zeta), V\cap C_{2r}(Q,\tau)) \leq& \ell(r)\\
\sigma(C_{\ell(r)}(P,\zeta)) \geq 2[\delta_\varepsilon(r) + \delta'(r) + \delta_\eta''(r)]r^{n+1},\; &\forall (P,\zeta) \in \partial \Omega \cap C_{2r}(Q,\tau).
\end{aligned}
\end{equation}

For ease of notation, assume that $(Q,\tau) = (0,0)$ and $V$ (the tangent plane at $(Q,\tau)$) is $\{x_n = 0\}$. Let $D(r) \colonequals \{(x,x_n,\tau) \mid |x| < r, x_n = \frac{1}{2}C_0\ell(r)\}$ where $C_0$ is a large constant satisfying the following constraint: 
\begin{center}
If $(Y,0) \in D(r)$ and $(P,\zeta) \in C_{2r}(0,0)$ are such that $\|(y, 0, 0) - (p, 0, \zeta)\| \leq 2\ell(r)$ then $D(r) \cap C_{\ell(r)}(Y,0) \subset \Gamma_{10}^{C_0\ell(r)}(P,\zeta)$.
\end{center}
We make the following claim, whose proof, for the sake of continuity, will be delayed until later. 

\medskip

\noindent {\bf Claim 1:} {\it Under the assumptions above, if $r > 0$ is small enough, $(Y,0)\in D(r)$, we have: \begin{equation}\label{boundonthedisc}
|u(Y,0)| \leq CM_1(h)(0,0)\ell(r)
\end{equation}
}

The general strategy is to take two points $(Y_1,0), (Y_2,0)\in D(r)$ and estimate $\left\langle F(0,0), Y_2-Y_1\right\rangle$ in terms of $u(Y_2,0) - u(Y_1,0)$. Define $R(Y_1, Y_2) =  u(Y_2,0) - u(Y_1,0) - \left\langle F(0,0), Y_2-Y_1\right\rangle.$ Equations \eqref{definitionofell} and \eqref{boundonthedisc} imply that, for $r > 0$ small, we have $|u(Y_1,0)|, |u(Y_2,0)| < CM_1(h)(0,0)\ell(r) < C\varepsilon r$. Therefore, in order to show that $\left\langle F(0,0), Y_2-Y_1\right\rangle$ is small, it suffices to show that $R(Y_1, Y_2)$ is small.

Write $u(Y_2, 0) - u(Y_1,0) = \int_0^1 \left\langle \nabla u(Y_1 + \theta(Y_2-Y_1), 0), Y_2-Y_1\right\rangle d\theta$ which implies \begin{equation}\label{estimateonR}
|R(Y_1,Y_2)| \leq |Y_2-Y_1|\int_0^1 |\nabla u(Y_1 + \theta(Y_2-Y_1), 0) - F(0,0)| d\theta.
\end{equation}

If we define $I(r) = \frac{1}{r} \fint_{D(r)}\fint_{D(r)} |R(Y_1,Y_2)| dY_1dY_2$, then Fubini's theorem and equation \eqref{estimateonR} yields \begin{equation}\label{estimateonI}
I(r) \leq C_n \fint_{D(r)} |\nabla u(X,0) - F(0,0)| dX
\end{equation}
(for more details see \cite{kenigtoro}, Appendix A.2). We now arrive at our second claim: 

\medskip

\noindent {\bf Claim 2:}  {\it For $(Y,0) \in D(r)$ we have $|\nabla u(Y,0) - F(0,0)| < 2\varepsilon$. }

\medskip

Claim 2 immediately implies that $I(r) < C\varepsilon$, which, as $|u| \leq C\varepsilon r$ on $D(r)$, gives \begin{equation}\label{estimateofaveragedf}
\fint_{D(r)}\fint_{D(r)} |\left\langle F(0,0), Y_2-Y_1\right\rangle| dY_1dY_2 \leq rI(r) + 2C\varepsilon r< C\varepsilon r.
\end{equation}
Pick any direction $e \perp e_n$ such that $e \in \R^n$ (i.e. has no time component) and let $M \colonequals |\left\langle F(0,0), e\right\rangle|$. Now consider the cone of directions $\tilde{\Gamma}$ in $\mathbb S^{n-2}$ (i.e. perpendicular to both the time direction and $e_n$) such that $|\left\langle F(0,0), \tilde{e} \right\rangle|\geq M/2$ for $\tilde{e} \in \tilde{\Gamma}$. Note that $\mathcal H^{n-2}(\tilde{\Gamma})/\mathcal H^{n-2}(\mathbb S^{n-2}) = c_n$ a constant depending only on dimension. Thus from equation \eqref{estimateofaveragedf} we obtain $$C\varepsilon r > \fint_{D(r)}\fint_{D(r)} |\left\langle F(0,0), Y_2-Y_1\right\rangle| dY_1dY_2 \geq$$$$ \frac{c(n)}{r^{n-1}}\fint_{D(r)\cap \{(Y,t)\mid |y| < r/2\}} \int_0^{r/2} \int_{\theta \in \tilde{\Gamma}}\left\langle F(0,0), \rho \theta\right\rangle \rho^{n-2}d\mathbb S^{n-2}d\rho \geq \tilde{C}Mr$$  which of course implies that $M \leq C\varepsilon$. In other words $\left\langle F(0,0), x'\right\rangle \leq C\varepsilon |x'|$ for any $x'\in \R^{n-1}$, which is the desired result. \end{proof}

\begin{proof}[Proof of Claim 1]
We have $(Y,0) \in D(r)$ and want to show that $|u(Y,0)| \leq C M_1(h)(0,0)\ell(r)$. As $\partial \Omega$ is well approximated by $\{x_n = 0\}$ in $C_{2r}(0,0)$, there must be a $(P,0) \in \partial \Omega \cap C_{3r/2}(0,0)$ such that $p = y$, and, hence, $|p_n| < \ell(r)$.  Let $C(Y) \colonequals C_{\ell(r)}(P, 0)$. Note that equation \eqref{definitionofell} implies $\sigma(C(Y)) \geq 2 \delta'(r) r^{n+1}$. Given the definition of $\delta'(r)$ we can conclude the existence of $(\tilde{P}, \tilde{\zeta}) \in C(Y)\cap \partial \Omega$ such that $M_1(h)(\tilde{P},\tilde{\zeta}) < 2M_1(h)(0,0)$. Furthermore $(\tilde{P},\tilde{\zeta}) \in C_{2r}(0,0)$ and $\|(\tilde{p},0, \tilde{\zeta}) - (y,0, 0)\| < 2\ell(r)$. 

By the condition on $C_0$ and the aperture of the cone we can conclude that $(Y,0) \in \Gamma_{10}^{C_0\ell(r)}(\tilde{P},\tilde{\zeta})$. Then, arguing as in Lemma \ref{nontangetiallybounded} (i.e. using Lemma \ref{biggeratNTApoint}, the backwards Harnack inequality for the Green's function, $\|(Y,0)-(\tilde{P},\tilde{\zeta})\| \sim C_0\ell(r)$, Lemma \ref{comparisonformeasureatinfinity} and that $\omega$ is doubling) we can conclude $$u((Y,0)) \leq C u(A^-_{C\ell(r)}(\tilde{P},\tilde{\zeta})) \leq$$$$ \frac{C}{\ell(r)^n} \omega(C_{C\ell(r)}(\tilde{P},\tilde{\zeta})) \leq C\ell(r) \fint_{C_{c\ell(r)}(\tilde{P},\tilde{\zeta})} h(Z,t) d\sigma(Z,t) \leq C\ell(r)M_1(h)(\tilde{P},\tilde{\zeta}).$$ As $M_1(h)(\tilde{P},\tilde{\zeta}) \leq 2M_1(h)(0,0)$ we are done. 
\end{proof}

\begin{proof}[Proof of Claim 2]
We want to show that for $(Y,0) \in D(r)$ we have $|\nabla u(Y,0) - F(0,0)| < 2\varepsilon$. Arguing exactly like in Claim 1 produces a $(P,0) \in \partial \Omega$ and then $C(Y)$. This time we use equation \eqref{definitionofell} to give the bound $\sigma(C(Y)) \geq 2(\delta''(r)+\delta_\varepsilon(r))r^{n+1}$. We can then conclude that there exists a $(\tilde{P}, \tilde{\zeta}) \in C(Y)\cap \partial \Omega$ such that $(\tilde{P},\tilde{\zeta}) \in E(\eta, R)$ and $|F(\tilde{P},\tilde{\zeta}) - F(0,0)| < \varepsilon$. 

Recall that $(\tilde{P},\tilde{\zeta}) \in E(\eta) \Rightarrow (\tilde{P},\tilde{\zeta}) \in H(\lambda_n, 2^{-n})$ for all $n$. So pick $n$ large enough that $2^{-n} < \varepsilon$ and then $r$ small enough so that $C_0\ell(r) < \lambda_n$. Thus $|\nabla u(Y,0) - F(0,0)| < |\nabla u(Y,0) - F(\tilde{P},\tilde{\zeta})| + |F(\tilde{P},\tilde{\zeta}) - F(0,0)| < 2\varepsilon$.
\end{proof}

We now want to show that $|F(Q,\tau)| = h(Q,\tau)$ $d\sigma$-almost everywhere.  Here, again, we follow closely the approach of Kenig and Toro (\cite{kenigtorotwophase}, Lemma 3.4)  who prove the analogous elliptic result. One difference here is that the time and space directions scale differently. To deal with this difficulty, we introduce a technical lemma. 

\begin{lem}\label{anisotropicmaximalfunction}
Let $1 < p < \infty$ and $g\in L^p_{\mathrm{loc}}(d\sigma)$. Then $$\frac{1}{\sigma(\Delta((Q,\tau), r,s))}\int_{\Delta((Q,\tau),r,s)\cap \partial \Omega} gd\sigma \stackrel{s,r\downarrow 0}{\rightarrow} g(Q,\tau)$$  for $\sigma$-a.e. $(Q,\tau)\in \partial \Omega.$ (Here, and from now on, $C((Q,\tau),r,s) = \{(X,t)\mid |X-Q| \leq r, |t-\tau| < s^{1/2}\}$ and $\Delta((Q,\tau), r,s) = C((Q,\tau),r,s)\cap \partial \Omega$.) 
\end{lem}

\begin{proof}
Follows from the work of Zygmund, \cite{zygmund}, and the fact that $(\partial \Omega, \sigma)$ is a space of homogenous type.
\end{proof}

\begin{proof}[Proof of Lemma \ref{equaltoh}]
We will prove the theorem for all $(Q,\tau) \in \partial \Omega$ for which Proposition \ref{onlynormaldirection} holds, such that there is a tangent plane to $\partial \Omega$ at $(Q,\tau)$ and such that \begin{equation}\label{goodpointconditions2}
\begin{aligned}
\lim_{r\downarrow 0} \frac{\sigma(\Delta_r(Q,\tau))}{r^{n+1}} &= 1\\
\lim_{r,s \downarrow 0}\fint_{C((Q,\tau),r,s)\cap \partial \Omega} h d\sigma&= h(Q,\tau)\\
\lim_{r\downarrow 0} \fint_{C_r(Q,\tau)\cap \partial \Omega} F d\sigma &= F(Q,\tau)\\
\lim_{r\downarrow 0} \fint_{C_r(Q,\tau)\cap \partial \Omega} M_1(h)d\sigma &= M_1(h)(Q,\tau)\\
\lim_{\stackrel{(X,t)\rightarrow (Q,\tau)}{(X,t) \in \Gamma(Q,\tau)}} \nabla u(X,t) &= F(Q,\tau)\\
M_1(h)((Q,\tau)), F(Q,\tau), h(Q,\tau) &< \infty.
\end{aligned}
\end{equation}
That this is $\sigma$-a.e. point follows from Proposition \ref{onlynormaldirection}, Lemmas \ref{prdomainshavetangentplanes}, \ref{densityofsigma}, \ref{anisotropicmaximalfunction} and \ref{nontangentiallimiteverywhere},  and $F, h, M_1(h)\in L^2_{\mathrm{loc}}(d\sigma)$. 

$\partial \Omega \cap \{s = t\}$ is a set of locally finite perimeter for almost every $t$, hence, for any $\phi \in C_c^\infty(\R^{n+1})$, \begin{equation}\label{bypartsinu}\int_{\partial \Omega} \phi h d\sigma = \int_{\Omega} u(\Delta \phi - \phi_t)dXdt =  - \int_{\Omega} \nabla \phi \cdot \nabla u + u \phi_tdXdt. \end{equation} Let $\rho_1, \rho_2 > 0$ and set $\phi(X,t) = \zeta(|X-Q|/\rho_1)\xi(|t-\tau|/\rho_2^2)$. We calculate that $\nabla \phi(X,t) = \xi(|t-\tau|/\rho_2^2) \zeta'(|X-Q|/\rho_1) \frac{X-Q}{|X-Q|\rho_1}$ and also that$\frac{d}{d\rho_1} \zeta(|X-Q|/\rho_1) = -\frac{|X-Q|}{\rho_1^2} \zeta'(|X-Q|/\rho_1)$. Together this implies \begin{equation}\label{equationforzeta}
-\nabla \phi(X,t) =  \xi(|t-\tau|/\rho_2^2)\rho_1\frac{X-Q}{|X-Q|^2}\frac{d}{d\rho_1} \zeta(|X-Q|/\rho_1).
\end{equation}
Similarly $\partial_t \phi(X,t) = \zeta(|X-Q|/\rho_1)\xi'(|t-\tau|/\rho_2^2) \frac{\mathrm{sgn}(t-\tau)}{\rho_2^2}$ and $\frac{d}{d\rho_2}\xi(|t-\tau|/\rho_2^2) = -\frac{2|t-\tau|}{\rho_2^3} \xi'(|t-\tau|/\rho_2^2)$, therefore, 
\begin{equation}\label{equationforxi}
-\partial_t \phi(X,t) = \zeta(|X-Q|/\rho_1) \frac{\rho_2}{2(t-\tau)}\frac{d}{d\rho_2}\xi(|t-\tau|/\rho_2^2).
\end{equation}

Plugging equations \eqref{equationforzeta}, \eqref{equationforxi} into equation \eqref{bypartsinu} and letting $\xi, \zeta$ approximate $\chi_{[0,1]}$ we obtain $$\int_{\partial \Omega\cap C((Q,\tau),\rho_1,\rho_2)} hd\sigma = \rho_1\frac{d}{d\rho_1}\int_{\Omega\cap C((Q,\tau), \rho_1,\rho_2)} \left\langle \nabla u(X,t), \frac{X-Q}{|X-Q|^2}\right\rangle dXdt +$$$$ \rho_2\frac{d}{d\rho_2} \int_{\Omega\cap C((Q,\tau), \rho_1,\rho_2)} \frac{u(X,t)}{2(t-\tau)}dXdt.$$ Differentiating under the integral and then integrating $\rho_1, \rho_2$ from $0$ to $\rho>0$ yields \begin{equation}\label{afterintegration}\begin{aligned}\underbrace{\int_0^\rho \int_0^\rho \int_{\partial \Omega\cap C((Q,\tau),\rho_1, \rho_2)} hd\sigma d\rho_1d\rho_2}_{(I)}  =&\\ \underbrace{\int_0^\rho\int_{\Omega \cap C((Q,\tau), \rho,\rho_2)} \left\langle \nabla u, \frac{X-Q}{|X-Q|}\right\rangle dXdt d\rho_2}_{(II)} +&\underbrace{\int_0^\rho \int_{\Omega \cap C((Q,\tau), \rho_1, \rho)} u dXdt d\rho_1}_{(III)}.\end{aligned}\end{equation}

For any $\varepsilon > 0$ we will prove that there is a $\delta > 0$ such that if $\rho < \delta$ we have $$\begin{aligned} &|\mathrm{(III)}| < \varepsilon \rho^{n+3}\\
&|\mathrm{(I)}-\frac{1}{3n} h(Q,\tau) \rho^{n+3}|  < \varepsilon \rho^{n+3}\\
&|\mathrm{(II)} - \left\langle F(Q,\tau), \hat{n}(Q,\tau)\right\rangle \frac{1}{3n} \rho^{n+3}| < \varepsilon \rho^{n+3},\end{aligned}$$ which implies the desired result. Note that throughout the proof the constants may seem a little odd due to our initial normalization of Hausdorff measure.

\medskip

{\bf Analysis of (III):} $u$ is continuous in $\overline{\Omega}$, hence for any $\varepsilon' > 0$ there is a $\delta = \delta(\varepsilon') > 0$ such that if $\delta > \rho$ then $(X,t) \in C_\rho(Q,\tau) \Rightarrow u(X,t) < \varepsilon'$. It follows that, $|\mathrm{(III)}| \leq C\varepsilon' \int_0^\rho \rho_1^n\rho^2d\rho_1 = \varepsilon \rho^{n+3}$, choosing $\varepsilon'$ small enough. 

\medskip

{\bf Analysis of (I):} $(Q,\tau)$ is a point of density for $h$, so for any $\varepsilon' > 0$ there exists a $\delta > 0$ such that if $\rho < \delta$ then $$|(I) -  h(Q,\tau)\int_0^\rho \int_0^\rho \sigma(\Delta((Q,\tau), \rho_1,\rho_2))d\rho_1d\rho_2| < \varepsilon' \int_0^\rho \int_0^\rho \sigma(\Delta((Q,\tau), \rho_1, \rho_2))d\rho_1d\rho_2.$$ Switching the order of integration, $\int_0^\rho \int_0^\rho \sigma(\Delta((Q,\tau), \rho_1, \rho_2))d\rho_1d\rho_2 = \int_{\Delta_\rho(Q,\tau)} (\rho - |X-Q|)(\rho - \sqrt{|t-\tau|})d\sigma$. Consider the change of coordinates $X = \rho Y + Q$ and $t = s\rho^2 + \tau$. As $\partial \Omega$ has a tangent plane, $V$, at $(Q,\tau)$ the set $\{(Y,s)\mid (X,t) \in C_\rho(Q,\tau)\cap \partial \Omega\}$ converges (in the Hausdorff distance sense) to $C_1(0,0)\cap V$. Therefore, $$\frac{1}{\rho^{n+3}}\int_{\Delta_\rho(Q,\tau)} (\rho - |X-Q|)(\rho - \sqrt{|t-\tau|})d\sigma \stackrel{\rho\downarrow 0}{\rightarrow}\int_{|y| < 1, |s| <1} (1-|y|)(1-\sqrt{s})dyds = \frac{1}{3n}.$$
Which, together with the above arguments, yields the desired inequality.

\medskip

{\bf Analysis of (II):} Writing $\nabla u(X,t) = (\nabla u(X,t) - F(Q,\tau)) + F(Q,\tau)$ we obtain $$\mathrm{(II)} = \underbrace{\int_0^\rho \int_{\Omega \cap C((Q,\tau), \rho,\rho_2)} \left\langle \nabla u(X,t) - F(Q,\tau), \frac{X-Q}{|X-Q|}\right\rangle dXdt d\rho_2}_{\mathrm{(E)}} $$$$+ |F(Q,\tau)|\int_0^\rho\int_{\Omega \cap C((Q,\tau), \rho,\rho_2)} \left\langle \hat{n}(Q,\tau), \frac{X-Q}{|X-Q|}\right\rangle dXdt d\rho_2.$$

In the second term above, divide the domain of integration into points within $\varepsilon' \rho$ of the tangent plane $V$ at $(Q,\tau)$ and those distance more than $\varepsilon'\rho$ away. By the Ahlfors regularity of $\partial \Omega$, the former integral will have size $< C\varepsilon' \rho^{n+3}$. The latter (without integrating in $\rho_2$) is $$\int_{\Omega \cap C((Q,\tau), \rho,\rho_2)\cap \{(X,t)\mid \left\langle X-Q, \hat{n}(Q,\tau)\right\rangle \geq \varepsilon' \rho\}} \left\langle \hat{n}(Q,\tau), \frac{X-Q}{|X-Q|}\right\rangle dXdt.$$ Again we change variables so that $X = \rho Y + Q$ and $t = s\rho^2 + \tau$, and recall that, under this change of variables, our domain $\Omega$ converges to a half space. Arguing as in our analysis of (I), a simple computation (see equation 3.72 in \cite{kenigtorotwophase}) yields $$\left|\int_{\Omega \cap C((Q,\tau), \rho,\rho_2)\cap \{(x,t)\mid \left\langle x-Q, \hat{n}(Q,\tau)\right\rangle \geq \varepsilon' \rho\}} \left\langle \hat{n}(Q,\tau), \frac{x-Q}{|x-Q|}\right\rangle dxdt-\frac{\rho_2^2 \rho^n}{n}\right| <  C\varepsilon' \rho^n\rho_2^2.$$ Integration in $\rho_2$ gives $$|\mathrm{(II)} - |F(Q,\tau)|\frac{1}{3n}\rho^{n+3}| < C\varepsilon'\rho^{n+3}+|\mathrm{(E)}|.$$ The desired result then follows if we can show $|(E)| < \varepsilon' \rho^{n+3}$. To accomplish this we argue exactly as in proof of Lemma 3.4, \cite{kenigtorotwophase} but include the arguments here for completeness. 

We first make the simple estimate $$|(E)| \leq \rho \left|\int_{\Omega \cap C_\rho(Q,\tau)} \left\langle \nabla u(X,t) - F(Q,\tau), \frac{X-Q}{|X-Q|}\right\rangle dXdt\right|.$$  If $(X,t) \in C_\rho(Q,\tau) \cup \{(X,t)\mid \left\langle X-Q, \hat{n}(Q,\tau)\right\rangle \geq 4\varepsilon' \rho \}$ then  $\delta(X,t) > 2\varepsilon'\rho$ for small enough $\rho$ (because we have a tangent plane at $(Q,\tau)$). On the other hand we have $\varepsilon' \|(X,t) - (Q,\tau)\| \leq 2\varepsilon' \rho < \delta(X,t)$ which implies that $(X,t)$ is in some fixed non-tangential region of $(Q,\tau)$. 

By the definition of non-tangential convergence (which says we have convergence for all cones of all apertures), for any $\eta > 0$ if we make $\rho > 0$ even smaller we have  $|\nabla u(X,t) - F(Q,\tau)| < \eta$. Therefore, $$\mathrm{(E)} \leq C\eta \rho^{n+3} + \rho\int_{C_\rho(Q,\tau)\cap \Omega \cap \{(X,t)\mid |\left\langle X-Q, \hat{n}(Q,\tau)\right\rangle| \leq 4\varepsilon' \rho\}} |\nabla u(X,t)| + |F(Q,\tau)|dXdt.$$

Standard parabolic regularity results imply that $|\nabla u(X,t)| \leq C\frac{u(X,t)}{\delta(X,t)}$. As the closest point to $(X,t)$ on $\partial \Omega$ is in $C_{2\rho}(Q,\tau)$ we may apply Lemma \ref{growthattheboundary} and then Lemma \ref{biggeratNTApoint} to get $|\nabla u(X,t)| \leq C\frac{u(A^-_{8\rho}(Q,\tau))}{\sqrt{\delta(X,t)\rho}}$.  Continue arguing as in Lemma \ref{nontangetiallybounded} to conclude $$(X,t) \in C_\rho(Q,\tau)\cap \Omega \cap \{\frac{\rho}{2^{i+1}} < \delta(X,t) \leq \frac{\rho}{2^i}\} \Rightarrow |\nabla u(X,t)| \leq C2^{i/2} M_1(h)(Q,\tau).$$

Let $i_0 \geq 1$ be such that $\frac{1}{2^{i_0 + 1}} < 4\varepsilon' < \frac{1}{2^{i_0}}$ and recall $|F(Q,\tau)| < \infty$ to obtain,
$$\rho\int_{C_\rho(Q,\tau)\cap \Omega \cap \{(X,t)\mid |\left\langle X-Q, \hat{n}(Q,\tau)\right\rangle| \leq 4\varepsilon' \rho\}} |\nabla u(X,t)| + |F(Q,\tau)|dXdt $$$$< C\varepsilon' \rho^{n+3} + \rho\sum_{i = i_0}^\infty \int_{C_\rho(Q,\tau)\cap \Omega \cap \{\frac{\rho}{2^{i+1}} < \delta(X,t) \leq \frac{\rho}{2^i}\}}|\nabla u(X,t)| dX dt$$$$< C\varepsilon'\rho^{n+3} + C\rho M_1(h)(Q,\tau) \sum_{i=i_0}^\infty 2^{i/2} |C_\rho(Q,\tau)\cap \Omega \cap \{\frac{\rho}{2^{i+1}} < \delta(X,t) \leq \frac{\rho}{2^i}\}|.$$

A covering argument (using the Ahlfors regularity of $\partial \Omega$) yields $$|C_\rho(Q,\tau)\cap \Omega \cap \{\frac{\rho}{2^{i+1}} < \delta(X,t) \leq \frac{\rho}{2^i}\}| \leq C2^{i(n+1)} \left(\frac{\rho}{2^i}\right)^{n+2} \leq C\frac{\rho^{n+2}}{2^i}.$$ As $\sum_{i = i_0}^\infty 2^{-i/2} < C\sqrt{\varepsilon'}$ the desired result follows.
 \end{proof}
 
 \section{Caloric Measure at $\infty$}\label{appendix:caloricmeasure}
 
 We recall the existence and uniquness of the Green function and caloric measure with pole at infinity (done by Nystr\"om \cite{nystrom1} ). We also establish some estimates in the spirit of Section \ref{preliminaryestimates}. 
  
 \begin{lem}\label{existenceanduniquenesspoleatinfinity1}[Lemma 14 in \cite{nystrom1}]
 Let $\Omega \subset \R^{n+1}$ be an unbounded $\delta$-Reifenberg flat domain, with $\delta > 0$ small enough (depending only on dimension), and $(Q,\tau) \in \partial \Omega$. There exists a unique function $u$ such that,
 \begin{equation}\label{greensfunctionatinfinity1}
\begin{aligned}
u(X,t) > 0, \;&(X,t)\in \Omega \\
u(X,t) = 0, \;&(X,t)\in \R^{n+1}\backslash \Omega\\
\Delta u(X,t)+ u_t(X,t) = 0, \;& (X,t) \in \Omega\\
u(A^-_1(Q,\tau)) = 1.
\end{aligned}
\end{equation} 
Furthermore $u$ satisfies a backwards Harnack inequality at any scale with constant $c$ depending only on dimension and $\delta > 0$ (see Lemma \ref{backwardsharnack}).
 \end{lem}
 
 \begin{proof}
 Without loss of generality let $(Q,\tau) = (0,0)$ and for ease of notation write $A \equiv A^-_1(0,0)$. Any $u$ which satisfies equation \eqref{greensfunctionatinfinity1} also satisfies a backwards Harnack inequality at any scale with constant $c$ depending only on dimension and $\delta > 0$ (see the proof of Lemma 3.11 in \cite{hlncaloricmeasure}). The proof then follows as in \cite{nystrom1}, Lemma 14. 
 \end{proof}

 \begin{cor}\label{existenceanduniquenesspoleatinfinity2}[Lemma 15 in \cite{nystrom1}]
 Let $\Omega, u$ be as in Lemma \ref{existenceanduniquenesspoleatinfinity1}. There exists a unique Radon measure $\omega$, supported on $\partial \Omega$ satisfying \begin{equation}\label{defofmeasureatinfinity} \int_{\partial \Omega} \varphi d\omega = \int_{\Omega} u(\Delta \varphi - \partial_t\varphi),\; \forall \varphi \in C_c^\infty(\R^{n+1}).
 \end{equation} 
 \end{cor}
 
 \begin{proof}
 Uniqueness is immediate from equation \eqref{defofmeasureatinfinity}; for any $(Q,\tau) \in \partial \Omega$ and $r > 0$ let $\varphi$ approximate $\chi_{C_r(Q,\tau)}(X,t)$. Existence follows as in Lemma 15 in \cite{nystrom1}.
 \end{proof}
 
 \begin{lem}\label{doublingformeasureatinfinity}[Lemma \ref{doublingmeasure} for the caloric measure at infinity]
Let $u, \Omega, \omega$, be as in Corollary \ref{existenceanduniquenesspoleatinfinity2}. There is a universal constant $c > 0$ such that for all $(Q,\tau) \in \partial \Omega, r > 0$ we have $$\omega(\Delta_{2r}(Q,\tau)) \leq c \omega(\Delta_r(Q,\tau)).$$
\end{lem}

\begin{proof}
Follows as in Lemma 15 in \cite{nystrom1}. 
\end{proof}
 
 \begin{cor}\label{comparisonformeasureatinfinity}[Lemma \ref{comparisontheorem} for caloric measure and Green's function at infinity]
 Let $u, \Omega, \omega$ be as in Corollary \ref{existenceanduniquenesspoleatinfinity2}. There is a universal constant $c > 0$ such that for all $(Q,\tau) \in \partial \Omega, r > 0$ we have \begin{equation}\label{uislikeomegaatinfinity}
 c^{-1}r^n u(A^+_r(Q,\tau)) \leq \omega(\Delta_{r/2}(Q,\tau)) \leq cr^n u(A^-_r(Q,\tau))
 \end{equation}
 \end{cor}
 
 \begin{proof}
 The inequality on the right hand side follows from equation \eqref{defofmeasureatinfinity}; let $\chi_{C_{r/2}(Q,\tau)}(X,t) \leq \varphi(X,t) \leq \chi_{C_r(Q,\tau)}(X,t)$ and $|\Delta \varphi|, |\partial_t \varphi| < C/r^2$. Inequality then follows from Lemma \ref{biggeratNTApoint}. 
 
 The left hand side is more involved: as in the proof of Lemmas  \ref{existenceanduniquenesspoleatinfinity1} and \ref{existenceanduniquenesspoleatinfinity2} we can write $u$ as the uniform limit of $u_j$s (multiples of Green's functions with finite poles) and $\omega$ as the weak limit of $\omega_j$s (multiplies of caloric measures with finite poles) which satisfy Lemma \ref{comparisontheorem} at $(Q,\tau)$ for larger and larger scales. Taking limits gives that $$c^{-1}r^nu(A^+_r(Q,\tau)) \leq \omega(\overline{\Delta_{r/2}(Q,\tau)}).$$ That $\omega$ is doubling implies the desired result. 
 \end{proof}
 
 \begin{prop}\label{hlnformeasureatinfinity}[see Theorem 1 in \cite{hlncaloricmeasure}]
 Let $\Omega$ be a parabolic regular domain with Reifenberg constant $\delta > 0$. There is some $\overline{\delta} = \overline{\delta}(M, \|\nu\|_+) > 0$ such that if $\delta < \overline{\delta}$ then $\omega \in A^{\infty}(d\sigma)$. That is to say, there exists a $p >1$ and a constant $c = c(n, p) > 0$ such that $\omega$ satisfies a reverse Harnack inequality with exponent $p$ and constant $c$ at any $(Q,\tau) \in \partial \Omega$ and at any scale $r > 0$.
 \end{prop}
 
 \begin{proof}[Proof Sketch]
 The proof follows exactly as in \cite{hlncaloricmeasure}, with Lemma \ref{doublingformeasureatinfinity}, Corollary \ref{comparisonformeasureatinfinity} and the fact that the Green function at infinity satisfies the strong Harnack inequality substituting for the corresponding facts for the Green function with a finite pole. 
 \end{proof}
 
 \section{Boundary behaviour of caloric functions in parabolic Reifenberg flat domains}\label{fatouetcforparaboliccase}

In this appendix we prove some basic facts about the boundary behviour of caloric functions in parabolic Reifenberg flat domains, culminating in an analogue of Fatou's theorem (Lemma \ref{nontangentiallimitsdae}) and a representation formula for adjoint caloric functions with integrable non-tangential maximal functions (Proposition \ref{generalrepresentationtheorem}). Often the theorems and proofs mirror those in the elliptic setting; in these cases we follow closely the presentation of Jerison and Kenig (\cite{jerisonandkenig}) and Hunt and Wheeden (\cite{huntwheedenboundaryvalues} and \cite{huntwheeden}).

In the elliptic setting, the standard arguments rely heavily on the fact that harmonic measure (on, e.g. NTA domains) is doubling. Unfortunately, we do not know if caloric measure is doubling for parabolic NTA domains. However, in Reifenberg flat domains, Lemma \ref{doublingmeasure} tells us that caloric measure is doubling in a certain sense. Using this, and other estimates in Section \ref{preliminaryestimates}, we can follow Hunt and Wheeden's argument to show that the Martin boundary of a Reifenberg flat domain is equal to its topological boundary. The theory of Martin Boundaries (see Martin's original paper \cite{martin} or Part 1 Chapter 19 in Doob \cite{doob}) then allows us to conclude the following  representation formula for bounded caloric functions.

\begin{lem}\label{nontangentiallimitsdae}
Let $\Omega$ be a parabolic $\delta$-Reifenberg flat domain with $\delta > 0$ small enough. Then for any $(X_0,t_0) \in \Omega$ the adjoint-Martin boundary of $\Omega$ relative to $(X_0,t_0)$ is all of $\partial \Omega \cap \{t  > t_0\}$. Furthermore, for any bounded solution to the adjoint-heat equation, $u$, and any $s > t_0$ there exists a $g(P,\eta) \in L^\infty(\partial \Omega)$ such that \begin{equation}\label{boundedrepresentation}u(Y,s) = \int_{\partial \Omega} g(P,\eta)K^{(X_0,t_0)}(P,\eta, Y,s) d\hat{\omega}^{(X_0,t_0)}(P,\eta).\end{equation} Here $K^{(X_0,t_0)}(P,\eta, Y,s) \equiv \frac{d\hat{\omega}^{(Y,s)}}{d\hat{\omega}^{(X_0,t_0)}}(P,\eta)$ which exists for all $s > t_0$ and $\omega^{(X_0,t_0)}$-a.e. $(P,\eta) \in \partial \Omega$ by the Harnack inequality. 

Finally, if $u(Y,s) = \int_{\partial \Omega} g(P,\eta)K^{(X_0,t_0)}(P,\eta, Y,s) d\hat{\omega}^{(X_0,t_0)}(P,\eta)$, then $u$ has non-tangential limit $g(Q,\tau)$ for $d\hat{\omega}^{(X_0,t_0)}$-almost every $(Q,\tau) \in \partial \Omega$.
\end{lem}

\begin{proof}
Recall that the Martin boundary $\partial^M \Omega$ of $\Omega$ with respect to $(X_0,t_0)$ is the largest subset of $\partial \Omega$ such that the Martin kernel $V^{(X_0,t_0)}(X,t,Y,s) \colonequals \frac{G(X,t,Y,s)}{G(X,t,X_0,t_0)}$ has a continuous extension $V^{(X_0,t_0)} \in C(\partial^M \Omega \cap \{(X,t) \in \Omega \mid t > t_0\} \times \{(Y,s) \in \Omega \mid s > t_0\})$. Martin's representation theorem (see the theorem on page 371 of \cite{doob}) states that for any bounded solution to the adjoint-heat equation, $u$, there exists a measure, $\mu_u$, such that $u(Y,s) = \int_{\partial^M \Omega} V^{(X_0, t_0)}(Q,\tau, Y,s) d\mu_u(Q,\tau)$ where $\partial^M\Omega$ is the Martin boundary of $\Omega$. 

That $V^{(X_0, t_0)}(Q,\tau, Y,s)$ exists for all $\tau, s > t_0$ (and is, in fact, H\"older continuous in $(Q,\tau)$ for $\tau > s$) follows from Lemmas \ref{quotienttheorem} and \ref{growhtofquotientatboundary}. When $s > \tau$ it is clear that $V^{(X_0,t_0)}(Q,\tau, Y,s) = 0$. So indeed the Martin boundary is equal to the whole boundary (after time $t_0$).  

We will now prove, for a bounded solution $u$ to the adjoint-heat equation, $\mu_u << \hat{\omega}^{(X_0,t_0)}$ on any compact $K \subset \subset \{t > t_0\}$. To prove this, first assume that $u$ is positive (if not, add a constant to $u$ to make it positive). Let $(Q,\tau) \in \partial \Omega$ be such that $\tau > t_0$. Then there exists an $A \geq 100$ and an $r_0 > 0$ such that for all $r < r_0$ we have $(X_0,t_0) \in T_{A,r}^-(Q,\tau)$. Applying Lemmas \ref{comparisontheorem} and \ref{backwardsharnack} and observing that $\lim_{(X,t)\rightarrow (Q',\tau') \in \Delta_{r/4}(Q,\tau)} \omega^{(X,t)}(\Delta_{r/2}(Q,\tau)) = 1$ we can conclude that there is some constant $\gamma > 0$ such that $$\hat{\omega}^{(X_0,t_0)}(\Delta_r(Q,\tau)) V^{(X_0,t_0)}(Q',\tau', A^-_r(Q,\tau)) \geq \gamma,$$ for all $(Q',\tau') \in \Delta_{r/4}(Q,\tau)$. 

It follows that, \begin{equation}\label{absolutelycont}\begin{aligned} \frac{\mu_u(\Delta_{r/4}(Q,\tau))}{\hat{\omega}^{(X_0,t_0)}(\Delta_{r/4}(Q,\tau))} \leq& C \frac{\mu_u(\Delta_{r/4}(Q,\tau))}{\hat{\omega}^{(X_0,t_0)}(\Delta_{r}(Q,\tau))}\\ \leq& C\gamma^{-1} \int_{\Delta_{r/4}(Q,\tau)}V^{(X_0,t_0)}(Q',\tau', A^-_r(Q,\tau))d\mu_u(Q',\tau')\\ \leq& C\gamma^{-1}u(A^-_r(Q,\tau)) \leq C\gamma^{-1}\|u\|_{L^\infty}.\end{aligned}\end{equation} Therefore, there is a $g_u(P,\eta) =\frac{d\mu_u}{d\hat{\omega}^{(X_0,t_0)}}$ such that (by Martin's representation theorem) \begin{equation}\label{representationwithv} u(Y,s) = \int_{\partial \Omega} g_u(P,\eta)V^{(X_0,t_0)}(P,\eta, Y,s) d\hat{\omega}^{(X_0,t_0)}(P,\eta).\end{equation}

Assume that $V^{(X_0, t_0)}(Q,\tau, Y,s) = K^{(X_0,t_0)}(Q,\tau, Y,s)$. Then equation \eqref{representationwithv} is equation \eqref{boundedrepresentation}. The existence of a non-tangential limit follows from a standard argument (see e.g. \cite{huntwheedenboundaryvalues}) which requires three estimates. First, that $\omega^{(X_0,t_0)}$ is doubling, which we know is true after some scale for any point $(Q,\tau)$ with $\tau > t_0$. Second we need, for $(Q_0,t_0) \in \partial \Omega, r > 0$, $$\lim_{(X,t) \rightarrow (Q_0,t_0)} \sup_{(Q,\tau) \notin \Delta_r(Q_0,t_0)} K^{(X_0,t_0)}(Q,\tau, X,t) = 0.$$ This follows from Harnack chain estimates and Lemma \ref{growthattheboundary} (see the proof of Lemma 4.15 in \cite{jerisonandkenig} for more details). Finally, for some $\alpha > 0$, which depends on the flatness of $\Omega$, we want $$K^{(X_0,t_0)}(P,\eta, A^-_{4R}(Q,\tau)) \leq \frac{C2^{-\alpha j}}{\omega^{(X_0,t_0)}(\Delta_{2^{-j}R}(Q,\tau))},$$ for all $(P,\eta)\in \Delta_{R2^{-j}}(Q,\tau) \backslash \Delta_{R2^{-j-1}}(Q,\tau)$ and for values of $R$ small. This follows from Lemmas \ref{growthattheboundary} and Lemmas \ref{comparisontheorem} (see \cite{jerisonandkenig}, Lemma 4.14 for more details). 

So it suffices to show that $V^{(X_0, t_0)}(Q,\tau, Y,s) = K^{(X_0,t_0)}(Q,\tau, Y,s)$. Fix $r > 0$, $(Q,\tau)\in \partial \Omega$ and consider the adjoint-caloric function $U(Y,s) = \hat{\omega}^{(Y,s)}(\Delta_r(Q,\tau))$. By the Martin representation theorem and equation \eqref{absolutelycont} there is a function $g \equiv g_{Q,\tau, r}$ such that \begin{equation}\label{absolutelycontrep}U(Y,s) = \int V^{(X_0,t_0)}(P,\eta,Y,s) g(P,\eta) d\hat{\omega}^{(X_0,t_0)}.\end{equation} We are going to show that $g(P,\eta) = \chi_{\Delta_r(Q,\tau)}$. If true then, by the definition of caloric measure, we conclude $$\int_{\Delta_r(Q,\tau)} V^{(X_0,t_0)}(P,\eta,Y,s)  d\hat{\omega}^{(X_0,t_0)}= \hat{\omega}^{(Y,s)}(\Delta_r(Q,\tau)) = \int_{\Delta_r(Q,\tau)} K^{(X_0,t_0)}(P,\eta, Y,s) d\hat{\omega}^{(X_0,t_0)},$$ for all surface balls. It would follow that  $V = K$. 

For a closed $E \subset \partial \Omega$, following the notation of \cite{martin} Section 3, let $U_E$ be the unique adjoint-caloric function in $\partial \Omega$ given as the limit inferior of super adjoint-caloric functions which agree with $U$ on open sets, $\mathcal O$, containing $E$ and which are adjoint-caloric on $\Omega\backslash \overline{\mathcal O}$ with zero boundary values on $\partial \Omega \backslash \mathcal O$.  In a $\delta$-Reifenberg flat domain (where the Martin boundary agrees with and has the same topology as the topological boundary) and if $E = \overline{\Delta_\rho(P,\eta)}$, for some $(P,\eta) \in \partial \Omega$ and $\rho > 0$, it is easy to compute that $U_E(Y,s) = \hat{\omega}^{(Y,s)}(E)$. By the uniqueness of distributions (Theorem III on page 160 in \cite{martin}) it must be the case that $$U_E(Y,s) = \int_E V^{(X_0,t_0)}(P,\eta, Y,s)g(P,\eta) d\hat{\omega}^{(X_0,t_0)},$$ where $g$ is as in equation \eqref{absolutelycontrep}. If $(Y,s) = (X_0,t_0)$ and $E = \overline{\Delta_\rho(P_0,\eta_0)} \subset \Delta_r(Q,\tau)$ then the above equation becomes $$\hat{\omega}^{(X_0,t_0)}(\overline{\Delta_\rho(P_0,\eta_0)}) = \hat{\omega}^{(X_0,t_0)}(\Delta_\rho(P_0,\eta_0)) = \int_{\Delta_\rho(P_0,\eta_0)} g(P,\eta) d\hat{\omega}^{(X_0,t_0)}.$$ Letting $(P_0,\eta_0)$ and $\rho > 0$ vary it is clear that $g(P,\eta) = \chi_{\Delta_r(Q,\tau)}$ and we are done. 
\end{proof}

Our approach differs most substantially from the elliptic theory in the construction of sawtooth domains. In particular, Jones' argument using ``pipes" (\cite{jonesadv}, see also Lemma 6.3 in \cite{jerisonandkenig}) does not obviously extend to the parabolic setting. The crucial difference is that parabolic Harnack chains move forward through time, whereas elliptic Harnack chains are directionless.  The best result in the parabolic context is the work of Brown \cite{brown}, who constructed sawtooth domains inside of $\mathrm{Lip}(1,1/2)$-graph domains. Our argument below, which is in the same spirit as Brown's, works for $\delta$-Reifenberg flat domains. Before the proof we make the following observation. 

\begin{rem}\label{closestatsametimepoint}
Let $\Omega$ be a $\delta$-Reifenberg flat domain and let $(X,t) \in \Omega$. If  $(P,\eta) \in \partial \Omega$ such that $$\|(P,\eta) - (X,t)\| = \delta_\Omega(X,t) \colonequals \inf_{(Q,\tau) \in \partial \Omega} \|(Q,\tau) - (X,t)\|$$ then $\eta = t$. That is, every ``closest" point to $(X,t)$ has time coordinate $t$. 
\end{rem}

\begin{proof}[Justification]
By Reifenberg flatness, for any $(P,\eta) \in \partial \Omega$ the point $(P,t)$ is within distance $\delta |t-\eta|^{1/2}$ of $\partial \Omega$. Then $\delta_\Omega(X,t) \leq |P-X| + \delta |t-\eta|^{1/2} < \|(P,\eta) - (X,t)\|.$
\end{proof}

We are now ready to construct sawtooth domains. Recall, for $\alpha > 0$ and $F\subset \partial \Omega$ closed, we define $S_\alpha(F) = \{(X,t) \in \partial \Omega \mid \exists (Q,\tau) \in F,\; \mathrm{s.t.}\; (X,t) \in \Gamma_\alpha(Q,\tau)\}$. 

 \begin{lem}\label{parabolicsawtoothdomains}
Let $\Omega$ be a $(\delta^{10})$-Reifenberg flat parabolic NTA domain and $F \subset \partial \Omega \cap C_{s}(0,0)$ be a closed set.  There is a universal constant $c \in (0,1)$ such that if $c > \delta > 0$ then $S_\alpha(F)$ is a parabolic $\delta$-Reifenberg flat domain for almost every $\alpha \geq\alpha_0(\delta) > 0$. Furthermore, if $(X,t) \in S_\alpha(F)$ then, on $F$, $\hat{\omega}_{S_\alpha(F)}^{(X,t)} << \hat{\omega}^{(X,t)} << \hat{\omega}_{S_\alpha(F)}^{(X,t)}.$ Here $\hat{\omega}_{S_\alpha(F)}^{(X,t)}$ is the adjoint-caloric measure of $S_\alpha(F)$ with a pole at $(X,t)$. 
 \end{lem}
 
 \begin{proof}
To prove that $S_\alpha(F)$ is $\delta$-Reifenberg flat first consider $(Q,\tau) \in F \subset \partial S_\alpha(F)$ and $\rho > 0$. Let $V$ be an $n$-plane through $(Q,\tau)$ containing a vector in the time direction such that $D[\partial \Omega \cap C_{2\rho}(Q,\tau), V\cap C_{4\rho}(Q,\tau)] \leq 2\delta^{10} \rho$. If $(X,t) \in \partial S_\alpha(F) \cap C_{\rho}(Q,\tau)$ then (for $\alpha \geq 1$) $(P,t) \in C_{2\rho}(Q,\tau)\cap \partial \Omega$ where $(P,t) \in \partial \Omega$ satisfies $\delta_\Omega(X,t) = \mathrm{dist}((P,t), (X,t))$. We can compute that $$\mathrm{dist}((X,t), V) \leq \delta_\Omega(X,t) + \mathrm{dist}((P,t), V) \leq \frac{\rho}{1+\alpha} + 4\rho \delta^{10} \leq \rho \delta/11$$ for $\delta < 1/2$  and $1+\alpha \geq 20/\delta$. 

One might object that the above is a one sided estimate, and that Reifenberg flatness requires a two sided estimate. However, Saard's theorem tells us that, for almost every $\alpha$, $S_\alpha(F)$ is a closed set such that $\mathbb R^{n+1}\backslash S_\alpha(F)$ is disjoint union of two open sets. Since this argument rules out the presence of ``holes" in $S_\alpha(F)$, our one sided estimates are enough to conclude Reifenberg flatness. Hence, $D[C_\rho(Q,\tau)\cap \partial S_\alpha(F), C_\rho(Q,\tau) \cap V] \leq \rho \delta/11$ for almost every $\alpha \geq \frac{20}{\delta} - 1$. 

We need to also show that $\partial S_\alpha(F)$ is flat at points not in $F$. Let $(Q,\tau) \in \partial S_\alpha(F) \backslash F$ and $R \colonequals \delta_\Omega(Q,\tau)$. Our proof of has four cases, depending on the scale, $\rho$, for which we are trying to show $\partial S_\alpha(F)$ is flat. 
\medskip

\noindent {\bf Case 1:} $\rho \geq \frac{(1+\alpha)R}{10}$. Observe that $C_{\rho}(Q,\tau) \subset C_{11\rho}(P,\eta)$ for some $(P,\eta) \in F$. The computation above then implies that $D[C_{\rho}(Q,\tau)\cap \partial S_{\alpha}(F), C_\rho(Q,\tau)\cap \{L + (Q,\tau)-(P,\eta)\}] \leq \rho \delta$ where $L \equiv L(P,\eta, 11\rho)$, is the plane through $(P,\eta)$ which best approximates $C_{11\rho}(P,\eta)\cap \partial \Omega$. 

\medskip

\noindent {\bf Case 2:} $\frac{5}{\delta}R \leq \rho < \frac{(1+\alpha)R}{10}$. We should note that this case may be vacuous for certain values of $\alpha$ (i.e, if $1+\alpha < \frac{50}{\delta}$).  Without loss of generality let $(Q,\tau) = (Q,0)$ and let $(0,0)\in \partial \Omega$ be a point in $\partial \Omega$ closest to $(Q,0)$ (which is at time zero by Remark \ref{closestatsametimepoint}). If $L(0,0,4\rho)$ is the plane which best approximates $\partial \Omega$ at $(0,0)$ for scale $4\rho$, we will prove that $D[C_{\rho}(Q,0)\cap \partial S_\alpha(F), C_{\rho}(Q,0)\cap \{L(0,0,4\rho)+ Q\}] \leq \delta \rho$.

We may assume $L(0,0,4\rho) = \{x_n =0 \}$. Note that $\delta_F(-)$ is a 1-Lipschitz function. Thus, if $(Z_1, t_1), (Z_2, t_2)\in C_{\rho}((Q,0)) \cap \partial S_\alpha(F)$ then $|\delta_\Omega(Z_1,t_1) - \delta_\Omega(Z_2, t_2)| < \frac{2\rho}{1+\alpha} < \frac{R}{5}$. We may conclude that $\delta_\Omega(Y,s) < 2R$ for all $(Y,s) \in C_\rho(Q,0)\cap S_\alpha(F)$ and, therefore, if $(P,s)\in \partial \Omega$ is a point in $\partial \Omega$ closest to $(Y,s)$ then $(P,s) \in C_{2\rho}(0,0)$. By Reifenberg flatness, $$||y_n| -4\rho\delta^{10}| \leq ||y_n| - |p_n|| \leq \delta_\Omega(Y,s) \leq 2R \leq \frac{2\rho \delta}{5} \Rightarrow |y_n| \leq \frac{3\rho\delta}{5}.$$ On the other hand $q_n \leq |Q| = R \leq \frac{\delta \rho}{5}$. Therefore, $|y_n - q_n| \leq |y_n| + |q_n| \leq \frac{3\rho\delta}{5}+ \frac{\rho\delta}{5} \leq \rho \delta,$ the desired result. 

\medskip

\noindent {\bf Case 3:} $\delta^2R \leq \rho < \frac{5}{\delta}R$. Let $(X,t)$ be the point in $\partial S_{\alpha}(F)\cap C_{\rho}(Q,\tau)$ which is furthest from $\partial \Omega$ and set $\delta_\Omega(X,t) = \tilde{R}$.  We may assume that $(X,t) = (0, \tilde{R}, 0)$ and $(0,0)$ is a point in $\partial \Omega$ which minimizes the distance to $(X,t)$. We will show that $D[C_{\rho}(Q,\tau) \cap \partial S_\alpha(F), C_{\rho}(Q,\tau) \cap \{x_n =\tilde{R}\} ] < \rho\delta/2$, which of course implies that $D[C_\rho(Q,\tau)\cap \partial S_\alpha(F), C_\rho(Q,\tau)\cap \{x_n =q_n\}] < \delta \rho$. 

First we prove that $L(0,0, 4\rho)$, the plane which best approximates $\partial \Omega$ at $(0,0)$ for scale $4\rho$, is close to $\{x_n = 0\}$. If $\theta$ is the minor angle between the two planes then the law of cosines (and the fact that $\delta_\Omega((0,\tilde{R},0)) = \tilde{R}$) produces $$(\tilde{R} - 4\rho\delta^{10})^2 \leq L^2 + \tilde{R}^2 -2L\tilde{R}\sin(\theta) \Rightarrow 2L\tilde{R}\sin(\theta)\leq L^2 + 8\rho\tilde{R}\delta^{10}$$ for any $L \leq 2\rho$. If $L \equiv \delta^4 \tilde{R}$ then $$2\delta^4\tilde{R}^2\sin(\theta) \leq \delta^8\tilde{R}^2 + 8\rho\tilde{R}\delta^{10} \stackrel{\delta\rho < 5\tilde{R}}{\Rightarrow} 2\delta^4\tilde{R}^2\sin(\theta) \leq \frac{3\delta^8 \tilde{R}^2}{2}.$$ For small enough $\delta$, we conclude $\theta < \delta^4$. 

Therefore, for any $(Y,s)\in C_\rho(Q,\tau)\cap S_\alpha(F)$ the distance between $(Y,s)$ and  $L(0,0,4\rho)$ is $\leq y_n + 2\delta^4\rho$. On the other hand, $\Omega$ is $(\delta^{10})$-Reifenberg flat so that means $\delta_\Omega(Y,s) \leq y_n + 2\delta^4\rho + 4\delta^{10}\rho \leq y_n + 3\delta^4\rho$. Recall, from above, that $\delta_F(-)$ is a 1-Lipschitz function. Therefore, if $(Z_1, t_1), (Z_2, t_2)\in C_{\rho}((Q,\tau)) \cap \partial S_\alpha(F)$ then $|\delta_\Omega(Z_1,t_1) - \delta_\Omega(Z_2, t_2)| < \frac{2\rho}{1+\alpha} < \frac{\delta \rho}{10}$ if we pick $1+\alpha \geq 20/\delta$. An immediate consequence of this observation is that $\tilde{R} < 2R$ (and thus $\delta^4\tilde{R} \leq 2\rho$ above). This also implies that $\delta_\Omega(Y,s) \geq \tilde{R} - \frac{\delta \rho}{10}$. Therefore, $$\tilde{R} - \frac{\delta \rho}{10} \leq y_n + 3\delta^4\rho \Rightarrow \tilde{R} - \frac{\delta \rho}{2} \leq y_n,$$ which is one half of the desired result. 

 On the other hand, it might be that there is a $(Y,s) \in C_{\rho}(Q,\tau) \cap \partial S_\alpha(F)$ such that $y_n > \tilde{R} + \rho\delta/2$. Arguing similarly to above we can see that $L(0,0, 4\tilde{R} + 4\rho)$ satisfies $D[C_1(0,0)\cap L(0,0,4\tilde{R} + 4\rho), C_1(0,0) \cap \{x_n =0\}] < 2\delta^4$. But if $(P,s)$ is the point on $\partial \Omega$ closest to $(Y,s)$ it must be the case that $(P,s) \in \partial \Omega \cap C_{2\tilde{R} + 2\rho}(0,0)$. Therefore, $p_n < 4\delta^4(\tilde{R} + \rho) + 4\delta^{10}(\tilde{R} + \rho) < 5\delta^4(\tilde{R} + \rho) < 6\delta^2\rho$. Of course this implies that $\|(Y,s) - (P,s)\| \geq \tilde{R}+ \rho\delta/2 - 6\delta^2\rho > \tilde{R}$ for $\delta$ small enough. This is a contradiction as no point in $C_\rho(Q,\tau)\cap \partial S_\alpha(F)$ can be a distance greater than $\tilde{R}$ from $\partial \Omega$. 
 
 \medskip
 
 \noindent {\bf Case 4:} $\rho \leq \delta^2R$. Again let $(X,t)$ be the point in $\partial S_{\alpha}(F)\cap C_{\rho}(Q,\tau)$ which is furthest from $\partial \Omega$. Let $\delta_\Omega(X,t) = \tilde{R}$ and without loss of generality, $(X,t) = (0, \tilde{R}, 0)$ and $(0,0)$ is the point in $\Omega$ closest to $(X,t)$. We will show that $D[C_{\rho}(Q,\tau) \cap \partial S_\alpha(F), C_{\rho}(Q,\tau) \cap \{x_n =\tilde{R}\} ] < \rho\delta/2$, which of course implies that $D[C_\rho(Q,\tau)\cap \partial S_\alpha(F), C_\rho(Q,\tau)\cap \{\{x_n =0\} + q_n\}] < \delta \rho$. 
 
 Assume, in order to obtain a contradiction, that there is a $(Y,s) \in C_\rho(Q,\tau) \cap \partial S_\alpha(F)$ such that $y_n < \tilde{R} - \rho\delta/2$ (we will do the case when $y_n \geq \tilde{R} + \rho\delta/2$ shortly). Examine the triangle made by $(Y,s), (0,\tilde{R},0)$ and the origin. The condition on $y_n$ implies that the cosine of the angle between the segments $\overline{(0,\tilde{R},0)(Y,s)}$ and $\overline{(0,\tilde{R},0)(0,0)}$ times the length of the segment between $(Y,s)$ and $(0,\tilde{R},0)$ must be at least $\rho\delta/2$. Consequently, by the law of cosines \begin{equation}\label{lawofconsinesynsmall}|Y|^2 \leq 4\rho^2 + \tilde{R}^2 -\tilde{R}\rho\delta.\end{equation} When $1+\alpha \geq 30/\delta$ and $\delta < 1/10$ it is easy to see that \begin{equation}\label{thisisterrible3} \begin{aligned} (\tilde{R}-2\delta^{10}\rho-\frac{2\rho}{1+\alpha})^2 &\geq (\tilde{R} - \frac{\delta\rho}{10})^2 \geq \tilde{R}^2 - \frac{\delta\rho\tilde{R}}{5}\\ &\geq \tilde{R}^2  - \frac{\delta\rho\tilde{R}}{5}(5- 20\delta) = \tilde{R}^2 + 4\delta^2\rho\tilde{R}-\delta\rho\tilde{R}\end{aligned}\end{equation}. 
 
Equations \eqref{thisisterrible3} and \eqref{lawofconsinesynsmall} give $|Y| < (\tilde{R}-2\delta^{10}\rho-\frac{2\rho}{1+\alpha})$. Hence, Reifenberg flatness implies $\delta_\Omega(Y,s) \leq \|(Y,s) - (0,s)\| +\delta_\Omega(0,s) \leq (\tilde{R} -2\delta^{10}\rho- \frac{2\rho}{1+\alpha}) + 2\delta^{10}\rho \leq \tilde{R} - \frac{2\rho}{1+\alpha},$ which, as we saw in {\bf Case 2}, is a contradiction.  
  
 Finally, it might be that there is a  $(Y,s) \in C_\rho(Q,\tau) \cap \partial S_\alpha(F)$ such that $y_n > \tilde{R} + \rho\delta/2$. Let $(P,s)$ be the point in $\partial \Omega$ closest to $(Y,s)$ and note that $(Y,s) \in C_{3\tilde{R}}(0,0)$. Let $L(0,0, 5\tilde{R})$ be the plane that best approximates $\partial \Omega$ at $(0,0)$ for scale $5\tilde{R}$. If $\theta$ is the angle between this plane and $\{x_n = 0\}$ then $\delta_\Omega((0,\tilde{R}, 0)) = \tilde{R}$ implies that $$\tilde{R} \leq \tilde{R}\sin(\pi/2-\theta) + 5\delta^{10}\tilde{R} \Rightarrow (1-5\delta^{10}) < \cos(\theta) \Rightarrow \theta < \delta^4.$$ 
 
Therefore, $p_n < 3\tilde{R}\delta^4 + 5\tilde{R}\delta^{10} < 4\tilde{R}\delta^4$. Thus, if $\beta$ is the angle between the segment from $Y$ to $P$ and the segment from $(0,\tilde{R})$ to $Y$ it must be that $\beta < \frac{\pi}{2} - \delta/4 + 10 \delta^4$. The law of cosines (on the triangle with vertices $(0, \tilde{R},0), (P,0), (Y,0)$) gives \begin{equation}\label{lawofcosinesrandp} |(0, \tilde{R})- P|^2 \leq 4\rho^2 + \tilde{R}^2 - 4\rho \tilde{R}\sin (\delta/4- 10\delta^4) < 4\rho^2 + \tilde{R}^2 + 40\rho \tilde{R}\delta^4 - \delta\rho\tilde{R}/2.\end{equation} 

Note that \begin{equation}\label{thisisterrible} 8\rho+80\tilde{R}\delta^2+8\tilde{R}\delta^{10} \leq \delta \tilde{R}\end{equation} because $\rho \leq \delta^2\tilde{R}$ by assumption and we can let $\delta < 1/100$. With this in mind we can estimate \begin{equation}\label{thisisterrible2}\begin{aligned} 4\rho^2 + \tilde{R}^2 + 40\rho \tilde{R}\delta^4 - \delta\rho\tilde{R}/2 &< (\tilde{R}-2\delta^{10}\rho)^2 + (4\delta^{10}\rho\tilde{R} +40\rho\tilde{R}\delta^4 + 4\rho^2 - \delta\rho\tilde{R}/2)\\
\stackrel{eq.\; \eqref{thisisterrible}}{\leq}& (\tilde{R}-2\delta^{10}\rho)^2 + (\rho/2)(\delta \tilde{R}) - \delta\rho\tilde{R}/2 = (\tilde{R}-2\delta^{10}\rho)^2.\end{aligned}\end{equation} 

Combine equation \eqref{thisisterrible2} with equation \eqref{lawofcosinesrandp} to conclude that $|(0, \tilde{R})- P| \leq (\tilde{R}-2\delta^{10}\rho)$. On the other hand, by Reifenberg flatness, $(P,0)$ is distance $< 2\delta^{10}\rho$ from a point on $\partial \Omega$. Hence, by the triangle inequality,  $\delta_{\Omega}((0,\tilde{R}, 0)) < \tilde{R}$ a contradiction. 

\medskip

Finally, we need to show the mutual absolute continuity of the two adjoint-caloric measures. In this we follow very closely the proof of Lemma 6.3 in \cite{jerisonandkenig}. The maximum principle implies that $\hat{\omega}_{S_\alpha(F)}^{(X_0,t_0)} << \hat{\omega}^{(X_0,t_0)}$, for any $(X_0,t_0) \in S_\alpha(F)$. Now take any $E\subset F$ such that $\hat{\omega}_{S_\alpha(F)}^{(X_0,t_0)}(E) = 0$. First we claim that there is a constant $1 > C > 0$ (depending on $\alpha, n, \Omega$) such that $\hat{\omega}^{(Y,s)}(\partial \Omega \backslash F)  > C$ for all $(Y,s) \in \partial S_{\alpha}(F)\backslash F$. Indeed, let $(Q,s) \in \partial \Omega$ be a point in $\partial \Omega$ closest to $(Y,s)$. Then there is a constant $C(\alpha, n) > 0$ such that $\hat{\omega}^{(Y,s)}(C_{\alpha\delta(Y,s)}(Q,s)) \geq C(\alpha, n),$ (see equation 3.9 in \cite{hlncaloricmeasure}). As $C_{\alpha\delta(Y,s)}(Q,s)\cap F = \emptyset$ (by the triangle inequality) the claim follows.

Armed with our claim we recall that a {\it lower function} $\Phi(X,t)$, for a set $E \subset \partial \Omega$, is a subsolution to the adjoint heat equation in $\Omega$ such that $\limsup_{(X,t) \rightarrow (Q,\tau) \in \partial \Omega} \Phi(X,t) \leq \chi_{E}(Q,\tau)$. Potential theory tells us that $\hat{\omega}^{(Y,s)}(E) = \sup_\Phi \Phi(Y,s)$ where the supremum is taken over all lower functions for $E$ in $\Omega$. By our claim, $\Phi(X,t) \leq \hat{\omega}^{(X,t)}(E) \leq 1-C$ for $(X,t)\in \partial S_\alpha(F)\backslash F$ for any lower function, $\Phi$, of $E$ in $\Omega$. In particular, $\Phi(X,t) - 1 + C$ is a lower function for $E$ inside of $S_\alpha(F)$. Therefore, $$\sup_{\Phi} \Phi(X,t) - 1 + C \leq \hat{\omega}_{S_\alpha(F)}^{(X,t)}(E) = 0 \Rightarrow \Phi(X,t) \leq 1-C,\; \forall (X,t) \in S_\alpha(F).$$ This in turn implies that $\hat{\omega}^{(Y,s)}(E) \leq 1-C$ for every $(Y,s) \in S_\alpha(F)$. By Lemma \ref{nontangentiallimitsdae}, the non-tangential limit of $\hat{\omega}^{(Y,s)}(E)$ must be equal to $1$ for $d\hat{\omega}^{(X_0,t_0)}$-almost every point in $E$. Therefore, $\hat{\omega}^{(X_0,t_0)}(E) = 0$ and we have shown mutual absolute continuity. 
\end{proof}

From here, a representation theorem for solutions to the adjoint heat equation with integrable non-tangential maximal function follows just as in the elliptic case. For details, see the proof of Lemma A.3.2 in \cite{kenigtoro}.
 
 \begin{lem}\label{generalrepresentationtheorem}[Compare with Lemma A.3.2 in \cite{kenigtoro}]
 Let $\Omega, \delta$ be as in Lemma \ref{parabolicsawtoothdomains} above, and let $u$ be a solution to the adjoint heat equation on $\Omega$. Assume also that for some $\alpha > 0$ and all $(X,t) \in \Omega,\; N^\alpha(u) \in L^1(d\hat{\omega}^{(X,t)})$.  Then there is some $g \in L^1(d\hat{\omega}^{(X,t)})$ such that $$u(X,t) = \int_{\partial \Omega} g(P,\eta) d\hat{\omega}^{(X,t)}(P,\eta).$$
 \end{lem}


\begin{thebibliography}{widest-label}
\bibitem[ADN59]{adn1} S. Agmon, A. Douglis and L. Nirenberg, Estimates Near the Boundary for Solutions of Elliptic Partial Differential Equations Satisfying General Boundary Conditions I. {\it Comm. Pure App. Math.} {\bf 17} (1959), pp 623-727.

\bibitem[AC81]{altcaf} H. Alt and L. Caffarelli, Existence and Regularity for a minimum problem with free boundary. {\it J. reine angew. Math.} {\bf 325} (1981), pp 105-144.

\bibitem[AW09]{anderssonweiss} J. Andersson and G. S. Weiss, A Parabolic Free Boundary Problem with Bernoulli Type Condition on the Free Boundary. {\it J. Reine angew. Math.} {\bf 627} (2009), pp 213-235.

\bibitem[B89]{brown} R. Brown, Area Integral Estimates for Caloric Functions. {\it Tran. Am. Math. Soc.} {\bf 315} (1989), pp 565-589. 

\bibitem[Da77]{dahlberg} B. Dahlberg, Estimates of harmonic measure. {\it Arch. Rational Mech. Anal.} {\bf 65} (1977), 275-288. 

\bibitem[DJ90]{davidandjerison} G. David and D. Jerison, Lipschitz approximation to hypersurfaces, harmonic measure, and singular integrals, {\it Indiana Univ. Math. J.} {\bf 39} (1990), 831-845. 

\bibitem[DS93]{davidandsemmes} G. David and S. Semmes, ``Analysis of and on Uniformly Rectifiable Sets." Math. Surveys Monog. {\bf 38} Amer. Math. Soc., Providence, 1993. 

\bibitem[Do84]{doob} J. L. Doob. ``Classical Potential Theory and Its Probabilistic Counterpart." Springer-Verlag, New York, 1984. 

\bibitem[E15]{nequals2} M. Engelstein, Parabolic Chord-Arc Domains in $\mathbb R^2$. {\it Preprint} (2015). 

\bibitem[EG92]{evansandgariepy} L.C. Evans and R.F. Gariepy. ``Measure Theory and Fine Properties of Functions." CRC Press, Florida, 1992. 

\bibitem[FGS86]{fabesgarofalosalsa} E. B. Fabes, N. Garofalo and S. Salsa, A Backward Harnack Inequality and Fatou Theorem for Nonnegative Solutions of Parabolic Equations. {\it Ill. J. of Math.} {\bf 30} (1986), pp 536-565. 

\bibitem[GJ78]{garnettandjones} J.B. Garnett and P.W. Jones, The distance from BMO to $L^\infty$. {\it Ann. of Math.} {\bf 108} (1978) pp 373-393.

\bibitem[HLN04]{hlncaloricmeasure} S. Hofmann, J.L. Lewis and K. Nystr\"om, Caloric measure in parabolic flat domains. {\it Duke Math. J.} {\bf 122} (2004), pp 281-346.

\bibitem[HLN03]{hlnbigpieces}  S. Hofmann, J.L. Lewis and K. Nystr\"om, Existence of big pieces of graphs for parabolic problems.  {\it Ann. Acad. Sci. Fenn. Math.} {\bf 28} (2003), pp 355-384. 

\bibitem[HW70]{huntwheeden} R. Hunt and R. Wheeden, Positive Harmonic Functions on Lipschitz Domains. {\it Trans. Amer. Math. Soc.} {\bf 147} (1970), pp 507-527. 

\bibitem[HW68]{huntwheedenboundaryvalues} R. Hunt and R. Wheeden, On the Boundary Values of Harmonic Functions. {\it Trans. Amer. Math. Soc.} {\bf 132} (1968), pp 307-322. 

\bibitem[Je90]{jerison} D. Jerison, Regularity of the Poisson Kernel and Free Boundary Problems. {\it Coll. Math.} {\bf 60/61} (1990), pp 547-568. 

\bibitem[JK82]{jerisonandkenig} D. Jerison and C. Kenig, Boundary Behavior of Harmonic Functions in Non-Tangentially Accessible Domains. {\it Advances in Math.} {\bf 46} (1982), pp. 80-147. 

\bibitem[JN16]{johnnirenberg} F. John and L. Nirenberg, On functions of bounded mean oscillation, {\it Comm. Pure Appl. Math.} {\bf 14} (1961), pp 415-426. 

\bibitem[Jo90]{jones} P. Jones, Rectifiable Sets and the Travling Salesman Problem. {\it Inv. Math.} {\bf 102} (1990), pp 1-15. 

\bibitem[Jo82]{jonesadv} P. Jones, A Geometric Localization Theorem. {\it Adv Math.} {\bf 46} (1982), pp. 71-79.

\bibitem[KW88]{kaufmanandwu} R. Kaufman and J.-M. Wu, Parabolic measure on domains of class Lip $1/2$. {\it Compositio Math.} {\bf 65} (1988), 201-207. 

\bibitem[KL37]{keldyshandlavrentiev} M. Keldysh and M. Lavrentiev, Sur la repr\'esentation conforme des domaines limit\'es par des courbes rectifiables. {\it Ann. Sci. ENS} {\bf 54} (1937), pp. 1-38. 

\bibitem[KT06]{kenigtorotwophase} C. Kenig and T. Toro, Free boundary regularity below the continuous threshold: 2-phase problems. {\it J. reine angew. Math.} {\bf 596} (2006), pp 1-44. 


\bibitem[KT04]{kenigtoro2} C. Kenig and T. Toro, On the Free Boundary Regularity Theorem of Alt and Caffarelli. {\it Disc. and Cont. Dyn. Systems} {\bf 10} (2004), pp 397-422.

\bibitem[KT03]{kenigtoro} C. Kenig and T. Toro, Poisson Kernel characterization of Reifenberg flat chord arc domains. {\it Ann. Sci. \'Ecole Norm. Sup.} {\bf 36} (2003), pp 323-401. 

\bibitem[KT99]{kenigtoroannals} C. Kenig and T. Toro, Free Boundary Regularity for Harmonic Measures and Poisson Kernels. {\it Ann. of Math.} {\bf 150} (1999), pp 369-454. 

\bibitem[KT97]{kenigtoroduke} C. Kenig and T. Toro, Harmonic Measure on Locally Flat Domains. {\it Duke Math. J.} {\bf 87} (1997), pp 509-550.

\bibitem[KN78]{kinderlehrernirenberganalyticity} D. Kinderlehrer and L. Nirenberg, Analyticity at the Boundary of Solutions to Nonlinear Second-Order Parabolic Equations. {\it CPAM} {\bf 31} (1978), pp 283-338. 

\bibitem[KS80]{kinderlehrerstampacchia} D. Kinderlehrer and G. Stampacchia. ``An Introduction To Variational Inequalities and Their Applications." Academic Press, New York, 1980.

\bibitem[LM95]{lewisandmurray} J.L. Lewis and M.A. Murray, The Method of Layer Potentials for the Heat Equation in Time-Varying Domains. {\it Mem. Am. Math. Soc.} {\bf 114} (1995), no 545. 

\bibitem[L86]{liebermanintermediateschauder} G. M. Lieberman, Intermediate Schauder Theory for Second Order Parabolic Equations. I. Estimates. {\it J. Diff. Eq.} {\bf 63} (1986), pp. 1-31. 

\bibitem[M41]{martin} R. Martin, Minimal Positive Harmonic Functions. {\it Trans. Amer. Math. Soc.}{\bf 49} (1941), 137-172.

\bibitem[N12]{nystromgraphdomains} K. Nystr\"om, On an inverse type problem for the heat equation in parabolic regular graph domains. {\it Math. Zeitschrift} {\bf 270} (2012), pp 197-222.

\bibitem[N06a]{nystromindiana} K. Nystr\"om, On Blow-ups and the Classification of Global Solutions to a Parabolic Free Boundary Problems. {\it Indiana Univ Math. J.} {\bf 55} (2006), pp 1233-1290. 

\bibitem[N06b]{nystrom1} K. Nystr\"om, Caloric Measure and Reifenberg Flatness. {\it Ann. Acad. Sci. Fenn. Math.} {\bf 31} (2006), pp 405-436.

\bibitem[Se91]{semmeslipschitz} S. Semmes, Chord-arc surfaces with small constant, I. {\it Adv. Math.} {\bf 85} (1991), 198-223. 

\bibitem[St70]{steinsingularintegrals} E. Stein. ``Singular Integrals and Differentiability Properties of Functions." Princeton University Press, Princeton, NJ, 1970. 

\bibitem[W03]{weissregularity} G. Weiss, A singular limit arising in combustion theory: Fine properties of the free boundary. {\it Calc. Var. PDE} {\bf 17} (2003), 311-340.

\bibitem[Z34]{zygmund} A. Zygmund, On the differentiability of multiple integrals, {\it Fund. Math.} {\bf 23} (1934), pp 143-149. 

\end{thebibliography}
\end{document}